\documentclass[a4paper,11pt]{book}

\usepackage[utf8]{inputenc}
\usepackage[french]{babel}
\usepackage[T1]{fontenc}
\usepackage{amsmath}
\usepackage{amsfonts}
\usepackage{amssymb}
\usepackage{amsthm}
\usepackage{stmaryrd}
\usepackage{tikz-cd}
\usepackage{enumitem}
\usepackage{todonotes}
\usepackage{graphicx}
\usepackage{longtable}
\usepackage{geometry}

\DeclareMathOperator{\Gal}{Gal}
\DeclareMathOperator{\Spec}{Spec}
\DeclareMathOperator{\Hom}{Hom}
\DeclareMathOperator{\End}{End}
\DeclareMathOperator{\Aut}{Aut}
\DeclareMathOperator{\Ker}{Ker}

\DeclareMathOperator{\Ind}{Ind}

\DeclareMathOperator{\ind}{ind}

\DeclareMathOperator{\Tr}{Tr}
\DeclareMathOperator{\Trd}{Trd}
\DeclareMathOperator{\Nrd}{Nrd}
\DeclareMathOperator{\Br}{Br}

\DeclareMathOperator{\disc}{disc}
\DeclareMathOperator{\Inv}{Inv}

\DeclareMathOperator{\Id}{Id}
\DeclareMathOperator{\Pf}{Pf}
\DeclareMathOperator{\Quad}{Quad}

\DeclareMathOperator{\Alt}{Alt}
\DeclareMathOperator{\Hyp}{Hyp}
\DeclareMathOperator{\ext}{ext}

\newcommand{\inj}{\hookrightarrow}

\newcommand{\impl}{\Longrightarrow}

\newcommand{\Isom}{\stackrel{\sim}{\longrightarrow}}

\newcommand{\ie}{\textit{i.e.}}
\newcommand{\Z}{\mathbb{Z}}
\newcommand{\N}{\mathbb{N}}
\newcommand{\Q}{\mathbb{Q}}
\newcommand{\R}{\mathbb{R}}

\newcommand{\pfis}[1]{\langle\langle #1\rangle\rangle}
\newcommand{\To}{\longrightarrow}

\newcommand{\fdiag}[1]{\langle #1\rangle}
\newcommand{\ens}[2]{\{ #1\,|\, #2\}}
\newcommand{\transp}[1]{{}^t \! #1}
\newcommand{\tld}{\widetilde}
\newcommand{\eps}{\varepsilon}
\newcommand{\p}{\mathfrak{p}}

\newcommand{\bitem}{\item[$\bullet$]}
\newcommand{\pgq}{\geqslant}
\newcommand{\ppq}{\leqslant}

\newcommand{\CBr}[1][K]{\mathbf{Br}(#1)}
\newcommand{\CBrh}[1][K]{\mathbf{Br}_h(#1)}
\newcommand{\Zd}{\Z/2\Z}
\newcommand{\can}{\overline{\phantom{a}}}

\renewcommand{\phi}{\varphi}

\newcommand{\foncdef}[5]{\begin{array}{rrcl}
#1 : & #2 & \To & #3 \\
 & #4 & \longmapsto & #5
\end{array}}

\newcommand{\anonfoncdef}[4]{\begin{array}{rcl}
#1 & \To & #2 \\
#3 & \longmapsto & #4
\end{array}}

\newcommand{\isomdef}[5]{\begin{array}{rrcl}
#1 : & #2 & \Isom & #3 \\
 & #4 & \longmapsto & #5
\end{array}}

\newcommand{\anonisomdef}[4]{\begin{array}{rcl}
#1 & \Isom & #2 \\
#3 & \longmapsto & #4
\end{array}}

\numberwithin{equation}{section}
\newtheorem{thm}[equation]{Théorème}
\newtheorem{prop}[equation]{Proposition}
\newtheorem{coro}[equation]{Corollaire}
\newtheorem{lem}[equation]{Lemme}

\newtheorem{propdef}[equation]{Proposition-définition}

\theoremstyle{definition}
\newtheorem{defi}[equation]{Définition}
\newtheorem{rem}[equation]{Remarque}
\newtheorem{ex}[equation]{Exemple}

\author{Nicolas Garrel}
\title{Invariants cohomologiques de groupes algébriques et d'algèbres à involution}

\begin{document}

\thispagestyle{empty}

\newgeometry{top=6em}

\vspace{-50em}

\includegraphics[scale=0.65]{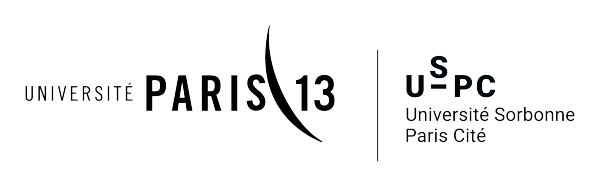}%
\hfill
\includegraphics[scale=0.45]{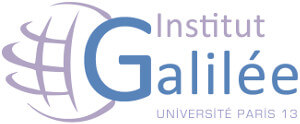}

{\large

\vspace*{1cm}

\begin{center}

{\bf THÈSE DE DOCTORAT DE \\ l'UNIVERSITÉ PARIS 13}

\vspace*{0.5cm}

Spécialité \\ [2ex]
{\bf Mathématiques}\ \\ 

\vspace*{0.5cm}

École doctorale Galilée

\vspace*{1cm}

Présentée par \ \\

\vspace*{0.5cm}

{\Large {\bf Nicolas GARREL}}

\vspace*{1cm}
Pour obtenir le grade de \ \\[1ex]
{\bf DOCTEUR de l'UNIVERSITÉ PARIS 13} \ \\

\vspace*{1cm}

Sujet de la thèse : \\
{\Large {\bf Invariants cohomologiques de groupes algébriques et d'algèbres à involution \\ }}

\vspace*{1.5cm} 
Soutenue le 17 octobre 2018\\[2ex]
devant le jury composé de :\\[3ex]
\begin{tabular}{r@{\ }lll}
  & Mme Anne {\sc Quéguiner} & Directrice de thèse & Université Paris 13\\
  & M. Philippe {\sc Gille} & Rapporteur & CNRS, Université Lyon 1\\
  & M. Michel {\sc Brion} & Examinateur & CNRS, Université Grenoble-Alpes \\
  & M. Charles {\sc De Clercq} & Examinateur & Université Paris 13  \\
  & M. Bruno {\sc Kahn} & Examinateur & CNRS, Université Paris 7  \\
  & M. Jean-Pierre {\sc Tignol} & Examinateur & Université Catholique de Louvain \\
\end{tabular}

\end{center}

\pagebreak

\restoregeometry

\tableofcontents

\chapter*{Remerciements}
\addcontentsline{toc}{chapter}{Remerciements}

Il était une fois un petit garçon qui aimait bien les nombres, et les énigmes.
Il était une fois un adolescent qui nourrissait quelque aspiration de devenir
chercheur, peut-être, qui sait. Il était une fois un jeune adulte qui,
au coeur de la nuit, à la lueur vacillante d'une lampe Ikea™, se prenait la tête entre les mains,
se demandant avec désespoir s'il était judicieux de séparer la proposition
2.3.9 en deux lemmes distincts, ou
s'il valait mieux conserver une preuve plus longue mais un énoncé plus clair,
ou bien encore aller dormir et trancher le problème le lendemain, J-3 avant la
sacro-sainte date limite.

Quelle aventure que tout ceci. Comme toute belle épopée, elle a ses origines, et ses
détours. Des épreuves, de l'amour, des rires, des tragédies. Et ce sentiment, en tournant la dernière
page, que voilà, ah, ça y est, c'est fini. C'est un peu triste, un livre qui
s'achève, on voudrait toujours qu'il y ait une suite. Il y en aura une. Mais
dans ce livre-ci, vous n'avez pas droit aux épreuves, à l'amour, aux rires ou aux tragédies.
Rien de tout cela n'est écrit ici, vous n'y pourrez lire que quelques théorèmes,
quelques remarques, quelques exemples. Et à vrai dire, il y a des chances que vous
ne les lisiez pas. Personne ne veut retenir de la Quête du Graal qu'une interminable
description d'un vase en or, ou peut-être d'une pierre incandescente, ou d'un bocal
à anchois. On rêve de chevalerie, et pas d'artisanat, si minutieux soit-il.
C'est pourquoi, dans ces rares pages que vous lirez, je veux au moins présenter quelques
personnages, quelques companions, qui m'ont accompagné dans cette quête.
Dire merci aux Merlin, aux Lancelot, aux Guenièvre
et aux Perceval de cette légende, qui n'y auraient pas autrement figuré. 
\\

Il me faut avant tout remercier mon jury, Charles De Clercq, Bruno Kahn,
Michel Brion, Jean-Pierre Tignol, d'avoir pris le temps et la peine
de venir m'écouter et de me lire, en particulier mes rapporteurs, Philipe
Gille et Alexander Merkurjev, qui ont dû éplucher mon manuscrit pendant des jours
que l'on consacre plus volontier aux vacances. Bien entendu il reste un nom,
le plus évident, celui d'Anne, ma directrice, à qui cette thèse doit énormément,
plus que je n'ai de place ici à y consacrer.
Elle lui doit des mathématiques, bien sûr, mais aussi et surtout de l'écoute,
du soutien, des encouragements, toutes ces choses sans lesquelles on
ne peut pas envisager de passer trois ans avec un manuscrit pour horizon.

Je dois ensuite remercier ma famille, et en premier lieu mon père, qui m'a élevé,
m'a toujours soutenu même lorsqu'il n'avait pas la plus petite idée de ce dans
quoi je me lançais, et m'a consacré un gros bout de sa vie. Mina aussi, qui tient
une place particulière dans mon coeur. Je ne peux pas citer tous les noms, grand-parents, oncles,
tantes, cousins et cousines de toutes sortes, de sang ou de coeur, toutes
ces personnes qui n'ont pas toutes été nécessairement à mes côtés pendant cette thèse,
mais qui toutes m'ont vu grandir, m'ont fait grandir; et il faut bien grandir avant de partir
à l'aventure.

Mes ami-e-s de Nice, Damien, Jonjon, le Baobab Rousseau, Foued-Barnabé,
et bien d'autres, avec qui j'ai été très jeune et stupide,
puis de moins en moins jeune, et tout aussi stupide. Je persiste à penser que l'UNSS Sloubi
avait une vraie chance de succès.

Celleux de prépa, qui ont permis à deux années que l'on qualifiera pudiquement de
studieuses d'être parmi les meilleures de ma vie. Que ce soient les internes de
HX1 (Paul/Pierre, Fiona, Anthony, Pauline, Miguel, Benoît et les autres) ou le clan des calorifères
par la suite (Charlotte lumière de mes jours, Ric le panda, ...).
Je suis arrivé en prépa comme on arrive à Poudlard, les yeux écarquillés et le
coeur serré, j'en suis parti avec plus de souvenirs et d'amis qu'à savoir quoi en faire.

Mes innombrables et inoubliables ami-e-s de la période ENS, qu'il m'est physiquement
impossible de citer nommément; que ce soit justement
à l'École, en DG, au volley (la première équipe 2 de l'Histoire, la meilleure équipe 1
de l'Histoire), en K-Fêt, au BOcal, au COF, au BDS, au Mega, en AG, en Courô,
à Montrouge, en cave 26, parfois même dans l'amphi Rataud, dans ces mille et un lieux que j'ai hantés,
où j'ai hurlé, ri, bu, pleuré, trollé et où surtout je me suis senti chez moi ; dans
d'autres écoles aussi puisqu'une ne me suffisait pas, et notamment à l'X (Morgane,
Hugo, Bastien, ...) où j'ai si souvent trouvé bon accueil.

Enfin pour la chronologie, mes camarades de thèse, cette bande de braves qui se retrouvent au 12,
Pierre, Tom, Mattia, Carlos, Anna, Marta, Jean-Michel, notre regrettée cheffe Anna-Laura,
Delphin, Daniel, Didier et les autres, passés et présents, mais aussi les chercheuses
et chercheurs qui ont accompagné mes premiers pas, et les membres de l'équipe de support,
notamment Isabelle et Yolande qui ont compensé comme elles ont pu mes maladresses
administratives.
\\

Je n'ai pas pu rendre justice à toute la galerie de personnages qu'il m'aurait fallu
évoquer, par manque de temps, par manque de place, par défaillance de ma mémoire,
et si des yeux venaient à parcourir ces lignes et à ne s'y reconnaître nulle part, je
tiens à m'en excuser. Mais il est un groupe qui doit figurer plus que tous les autres,
ma Table Ronde, ma deuxième famille : les grotas.

Louis, seigneur de la nuit, des profondeurs,
et guérisseur de machines. Paul Odieux et Puant Simon, l'ours le plus brave dont un couple
de saumons pouvait rêver. Grogrodile, à qui je reconnais que en vrai, Grenoble, ça passe.
Machin, qui doit être en train de devenir fou en triant les pages de cette thèse par couleur.
Pandarion, dont le demi-médaillon est forgé non dans l'acier mais dans mon coeur. Amiel,
et ses divers alter-ego qui sont autant de frères et de soeurs. Béné, mi-loutre, mi-lapin, et re-mi-loutre
derrière ; nous aurons toujours Philadelphie. Grotillon, le sanglier au coeur le plus tendre
parmi les hommes. Surée, qui m'a appris beaucoup sur moi-même, et aussi un peu sur le café.
Hélène, qui m'a offert à la Scène, à qui j'appartiens désormais. Jonas, qui m'a autant fait
découvrir de jeux que d'astuces sur le bon emploi d'un PEL. Léa, puisses-tu voler sur les
ailes de tes oiseaux quand mon souffle t'appelle. Maud, pour qui je pourrais presque envisager
d'aller poser les pieds en Bretagne. Marc, l'inexplicable motard sans moto, à l'hospitalité
indéfectible, quoique mousticale. Marine la Bio, qui est depuis longtemps devenue simplement
Marine. P4bl0, dont je suis obligé de dire du bien vu que c'est mon proprio, et mon
compagnon de magret. Picomango, California Girl et co-dégé belle, grande et forte.
Snoopy, mon complémentaire exemplaire, est-il besoin d'en dire davantage. Ted,
dont je reconnais l'existence. Laetitia, non-malicieuse au sens très, très fort.

Plus que quiconque, vous êtes celles et ceux sur qui repose cette thèse, tout à la
fois mon épée légendaire et le rocher d'où je l'extirpe.

\chapter*{Introduction}
\addcontentsline{toc}{chapter}{Introduction}

\section*{Algèbres à involution}

Comme l'indique son titre, les objets centraux de cette thèse sont les algèbres
à involution. On entend par là des algèbres simples centrales sur un corps
munies d'un anti-automorphisme involutif (souvent noté $\sigma$ ou $\tau$, parfois
$\theta$ ou $\rho$), l'exemple le plus élémentaire étant une algèbre de matrices
munie de la transposition. \emph{On fera toujours l'hypothèse que le corps de base $k$
est de caractéristique différente de 2.} Excepté brièvement dans le dernier
chapitre, les involutions que nous considérerons seront de \emph{première espèce},
c'est-à-dire qu'elles agissent trivialement sur le centre. Bien que ne faisant
pas nécessairement partie du bagage mathématique standard de tout algébriste,
les algèbres à involution apparaissent naturellement dans de nombreux contextes,
et les chemins qui y conduisent sont de nature variée.

Historiquement, les algèbres à involution ont été introduites par Albert
dans les années 30 dans le contexte de la géométrie riemannienne (voir par
exemple \cite{Alb}). Albert a initié une étude systématique des algèbres
à involution afin de caractériser les algèbres à division sur $\mathbb{Q}$
qui apparaissent comme algèbre de multiplication d'une surface de Riemann.
Cette motivation historique est toutefois très éloignée des considérations
qui apparaîtront dans cette thèse.
\\

Une porte d'entrée vers les algèbres à involution revêtant une bien plus
grande importance en ce qui nous concerne est à trouver dans les travaux de
Weil sur les groupes algébriques (\cite{Wei}). Se fondant sur les travaux de
Killing, Élie Cartan avait, à la fin du 20\up{e} siècle, classifié les groupes
de Lie simples sur $\mathbb{C}$ (\cite{Car}), en passant par la classification,
plus simple à établir, des algèbres de Lie correspondantes. La classification se fait
en terme des fameux \emph{diagrammes de Dynkin}, associés aux systèmes de
racines. On y retrouve en particulier les groupes dits \emph{classiques},
à savoir $PSL_{n+1}$ (correspondant au diagramme $A_n$), $SO_{2n+1}$ (correspondant
à $B_n$), $PSp_{2n}$ (correspondant à $C_n$) et $PSO_{2n}$ (correspondant à
$D_n$) pour $n\neq 4$ (les groupes de type $D_4$ ne sont généralement pas
considérés comme classiques à cause des phénomènes exceptionnels de \emph{trialité}).

En ce qui concerne les groupes algébriques, Chevalley montre dans les années
50 que la classification des groupes semi-simples sur un corps algébriquement
clos ne dépend pas du corps, et notamment pas de la caractéristique, et on
retrouve essentiellement la même classification en terme de diagrammes
de Dynkin (il faut également tenir compte des isogénies, par exemple
les groupes de type $A_n$ sont de la forme $SL_{n+1}/\mu_d$
où $d$ divise $n+1$, le groupe $PSL_{n+1}$ correspondant à $d=n+1$). En réalité,
il apparaît rapidement que la classification est identique si on suppose
non pas que le corps est algébriquement clos, mais que le groupe est
\emph{déployé} (c'est-à-dire qu'il contient un tore maximal déployé).

Si on veut tenir compte des groupes non déployés, on se rend aisément
compte que les formes bilinéaires jouent un rôle primordial dans
la description attendue. C'est particulièrement manifeste pour les groupes
de type $B_n$ et $D_n$, puisque le groupe $SO(q)$ où $q$ est une forme
quadratique est de toute évidence une forme non triviale du groupe
déployé correspondant. Les groupes de type $C_n$ sont par construction
liés aux formes alternées, mais en revanche celles-ci sont entièrement
caractérisées par leur dimension, même sur un corps quelconque. Quant
aux groupes de type $A_n$, il est un peu moins évident de constater qu'ils
sont liés aux formes hermitiennes : en effet, si $K'/K$ est une extension
quadratique, on dispose pour toute forme hermitienne $h$ sur $(K',\can)$
(où $\can$ est l'automorphisme non trivial de $K'/K$) du groupe unitaire
$SU(h)$, et $SL_{n+1}$ correspond au cas où $K'=K\times K$, donc ces groupes
sont bien des formes des groupes $A_n$ déployés.

L'observation de Weil est qu'on peut obtenir tous les groupes classiques
si on déplace notre attention des formes bilinéaires vers les involutions
qu'elles induisent sur des algèbres de matrices. En effet, toute forme
$b$ symétrique, alternée ou hermitienne sur un $K$-espace vectoriel $V$
induit sur $\End_K(V)$ une involution $\sigma_b$ dite \emph{adjointe} à la forme,
qui vérifie pour tout $u\in \End_K(V)$ et tous $x,y\in V$
\[ b(u(x),y) = b(x,\sigma_b(u)(y)), \]
et on peut définir le groupe d'isométrie de la forme comme
\begin{equation}\label{eq_gpe_isom}
   G = \ens{u\in \End_K(V)}{u\sigma_b(u)=\Id}. 
\end{equation}
De plus, toute involution sur $\End_K(V)$ correspond à une certaine forme
$b$, qui est unique à un facteur près. Si l'involution est de première
espèce (donc si elle fixe $K$), elle correspond soit à une forme symétrique
(on dit que l'involution est \emph{orthogonale}) soit à une forme
alternée (on dit que l'involution est \emph{symplectique}). Dans le cas
contraire, l'involution est dite \emph{unitaire}.

Une fois qu'on dispose de cette description, il apparaît qu'on peut
remplacer $(\End_K(V),\sigma_b)$ par une algèbre à involution $(A,\sigma)$
quelconque, et définir un groupe algébrique par l'analogue direct
de la formule (\ref{eq_gpe_isom}). Après extension à une clôture
séparable, l'algèbre $A$ est déployée, donc $\sigma$ correspond
à un certain $\sigma_b$, et le groupe est donc une forme d'un
des groupes déployés classiques. Le type de $\sigma$ est défini comme
le type du $\sigma_b$ obtenu après déploiement. Le point crucial,
suivant Weil, est qu'on obtient ainsi tous les groupes classiques
(on renvoie au chapitre VI de \cite{BOI} pour une étude complète
de la question).
\\

Comme le met en exergue la discussion précédente sur les groupes algébriques,
les algèbres à involution entretiennent des liens étroits avec les
formes bilinéaires ou hermitiennes, ce qui constitue un autre point d'entrée
possible dans la théorie. Précisément, on a vu qu'on pouvait naturellement
associer une algèbre à involution à une forme symétrique, alternée ou
hermitienne ; on peut regrouper ces trois catégories en une même définition :
si $K$ est muni d'un automorphisme $\tau$ d'ordre 2 et si $\eps=\pm 1$,
on peut définir une forme $\eps$-hermitienne sur $(K,\tau)$ comme
vérifiant $h(ax,by)=\tau(a)h(x,y)b$ et $h(y,x)=\eps\tau(h(x,y))$.
Le cas orthogonal est donné par $\tau=\Id$ et $\eps=1$, le cas
symplectique par $\tau=\Id$ et $\eps=-1$, et le cas unitaire par
$\tau=\can$ et $\eps=1$. À partir d'une paire $(K,\tau)$, on peut
donc définir des formes hermitiennes $(V,h)$ qui induisent des algèbres
à involution $(\End_K(V),\sigma_h)$. Le fait de se placer dans le
cadre des algèbres à involution permet de symétriser complètement
la situation : sur une algèbre à involution $(A,\sigma)$, on peut
définir des formes $\eps$-hermitiennes $(V,h)$, et alors
$(\End_A(V),\sigma_h)$ est encore une algèbre à involution. On renvoie
à \cite{Knu} ainsi qu'à la partie \ref{sec_morita} pour plus de détails
sur ces considérations, mais on peut déjà constater que les
algèbres à involution forment un cadre stable pour étudier les
formes hermitiennes et leurs algèbres d'endomorphismes.

D'autre part, un certain nombre d'autres constructions classiques
autour des formes quadratiques trouvent naturellement leur place au
sein de la théorie des algèbres à involution. Un exemple naturel est
celui des algèbres de Clifford : à toute forme quadratique $q$ on
associe une certaine algèbre $C(q)$, qui est canoniquement munie
d'une involution. On ne peut pas directement étendre cette définition
aux algèbres à involution orthogonale à cause d'une obstruction
omniprésente dans ce type de construction : l'involution associée
à une forme quadratique ne la caractérise qu'à similitude près,
donc on ne peut étendre aux involutions que les constructions qui
sont invariantes par similitude. Dans le cas de l'algèbre de Clifford,
on peut ainsi définir l'algèbre de Clifford \emph{paire} d'une algèbre
à involution orthogonale, et ainsi à une algèbre à involution on
associe une autre algèbre à involution (\cite[§8]{BOI}). On peut également citer
l'exemple de la forme d'Albert associée à une algèbre de biquaternions,
qui se généralise en l'algèbre discriminante d'une algèbre à involution
unitaire (\cite[§10]{BOI}). La correspondance classique entre formes
quadratique de dimension 6 et de discriminant 1 et algèbres de biquaternions
se généralise ainsi naturellement en une correspondance entre algèbres
de type $D_3$ et $A_3$ (voir \cite[§15.D]{BOI}, ainsi que la partie
\ref{sec_a3_d3} de cette thèse pour des calculs explicites).

Enfin, parmi les autres situations où les algèbres à involution
peuvent jouer un rôle déterminant, on peut notamment citer la
théorie des représentations de groupes, puisque l'algèbre de
groupe $K[G]$ d'un groupe fini est, si $K$ est de caractéristique
nulle, une algèbre semi-simple naturellement munie d'une involution
induite par $g\mapsto g^{-1}$.

\section*{Invariants}

On a décrit le lien étroit entre deux des termes du titre de cette thèse,
à savoir les groupes algébriques (qu'on ne présente plus) et les algèbres
à involution (qu'on a présentées). Il reste donc à préciser ce qu'on entend
par \emph{invariants cohomologiques}, et à décrire le contexte dans lequel
s'inscrit cette thèse.
\\

L'origine de l'étude de ces invariants se trouve sans nul doute dans
la théorie des formes quadratiques. C'est un fait très élémentaire que
si le déterminant de la matrice d'une forme bilinéaire symétrique dans une
certaine base n'est pas un invariant bien défini, en revanche il est
bien défini modulo les carrés : c'est ce qu'on appelle le déterminant
d'une forme quadratique, et c'est le premier invariant non trivial
que l'on définit ; il est donc à valeurs dans $K^*/(K^*)^2$. On peut
définir un deuxième invariant, à valeurs dans le groupe de Brauer
$\Br(K)$, de deux façons différentes. Une première possibilité
est de considérer la classe de Brauer $[C(q)]$ de l'algèbre de Clifford
de $q$ si $q$ est de dimension paire, et la classe $[C_0(q)]$ si
$q$ est de dimension impaire ; c'est ce qu'on appelle l'invariant
de Clifford $c(q)\in \Br(K)$. Une deuxième idée est de partir d'une
diagonalisation $q=\fdiag{a_1,\dots,a_r}$, et de considérer
$w_2(q)=\sum_{i<j}(a_i,a_j)$ où $(a_i,a_j)\in \Br(K)$ est la classe
de Brauer de l'algèbre de quaternions correspondante ; c'est ce qu'on
appelle l'invariant de Hasse-Witt de $q$. Ces deux invariants à
valeur dans le groupe de Brauer ne sont pas égaux, mais on peut
les relier par des formules simples faisant intervenir la dimension
et le déterminant (voir par exemple \cite[prop 3.20]{Lam}).

Ces deux invariants représentent les deux directions classiques
dans lesquelles on peut poursuivre la construction. Tout d'abord,
on observe que le déterminant est à valeurs dans le premier groupe
de cohomologie $H^1(K,\mu_2)$ (pour les prérequis sur la cohomologie galoisienne,
on renvoie par exemple à \cite{Ser}, \cite{GS}, ou au chapitre VII de
\cite{BOI}), tandis que les invariants de Clifford
et de Hasse-Witt sont à valeurs dans $H^2(K,\mu_2)$ (qui est naturellement
isomorphe à la 2-torsion de $\Br(K)$). Il est donc naturel
de chercher à construire des invariants à valeurs dans
$H^d(K,\mu_2)$ pour $d$ de plus en plus grand ; on qualifiera
logiquement ces invariants de \emph{cohomologiques}. Notons qu'en
degré 0, on peut définir $\dim_0(q)\in \Zd=H^0(K,\mu_2)$ comme la
dimension modulo 2 de $q$.

Si on poursuit dans la direction de l'invariant de Hasse-Witt, on
arrive à la construction des invariants de Stiefel-Whitney (analogues
aux classes de Stiefel-Whitney en K-théorie) : pour tout $d\in \N$,
si $q=\fdiag{a_1,\dots,a_r}$, on peut poser
\begin{equation}\label{eq_stiefel}
   w_d(q) = \sum_{i_1<\dots <i_d}(a_{i_1},\dots,a_{i_d}) \in H^d(K,\mu_2) 
\end{equation}
où $(a_{i_1},\dots,a_{i_d})\in H^d(K,\mu_2)$ est le symbole galoisien
naturel. On revient sur ces invariants dans la partie \ref{sec_dim_fix}.

Si on poursuit dans la direction de l'invariant de Clifford : on commence
par observer que $\dim_0(q)$ ne dépend que de la classe de Witt de $q$
(voir par exemple \cite{EKM} ou \cite{Lam} pour toutes les définitions
et tous les résultats standard sur l'anneau de Witt),
et définit donc une application $e_0: W(K)\to H^0(K,\mu_2)$, qui est clairement
un morphisme de groupes additifs.
On peut ensuite observer que si $q$ est de dimension paire $2r$, et qu'on modifie
la définition de $\det(q)$ par $\disc(q)=(-1)^r\det(q)$ (ce qu'on appelle
le \emph{discriminant} de $q$, parfois aussi appelé \emph{déterminant signé}
dans la littérature), alors $\disc(q)$ ne dépend que de la classe de Witt
de $q$, et donc définit une application $e_1: I(K)\to H^1(K,\mu_2)$, où
$I(K)\subset W(K)$ est l'idéal fondamental de l'anneau de Witt $W(K)$
constitué des formes de dimension paire. De plus, on voit facilement
que cette application est en réalité un morphisme de groupes additifs,
et que son noyau est précisément $I^2(K)$ (\cite[thm 13.7]{EKM}).
On constate alors que sur $I^2(K)$, $c(q)$ ne dépend encore que de
la classe de Witt de $q$ et définit un morphisme additif $e_2: I^2(K)\to H^2(K,\mu_2)$
(\cite[thm 14.3]{EKM}). On voit alors la généralisation se dessiner: on
souhaiterait construire pour tout $n\in \N$ un morphisme additif
\begin{equation}\label{eq_milnor}
   e_n: I^n(K)\To H^n(K,\mu_2) 
\end{equation}
tel que $e_{n+1}(q)$ soit défini lorsque $e_n(q)=0$. Cela constitue une
partie de la \emph{conjecture de Milnor} (voir \cite{Mil}), qui
affirme notamment que $e_n$ est surjective de noyau $I^{n+1}(K)$,
et donc qu'il induit un isomorphisme de $I^n(K)/I^{n+1}(K)$ vers
$H^n(K,\mu_2)$. La construction des invariants $e_n$ est \emph{considérablement}
plus difficile que celle des invariants $w_n$ ; l'invariant $e_3$ a
été construit par Arason en 1975 (\cite{Ara}), et l'invariant $e_4$ par
Jacob et Rost en 1989 (\cite{JR}). La conjecture de Milnor est entièrement
résolue en 2002 par Voevodsky (\cite{Voe}), notamment grâce à
la construction de la \emph{cohomologie motivique}, ce qui lui vaudra
la médaille Fields.
\\

\label{par_inv}Une fois la situation des formes quadratiques en tête pour servir
de motivation, on peut établir une notion générale d'invariant, et
notamment d'invariant cohomologique, due à Serre dans \cite{GMS}.
On se fixe donc un corps de base $k$, et on considère $\mathbf{Field}_{/k}$
la catégorie des extensions de corps de $k$. Alors si $F$ est un foncteur
de $\mathbf{Field}_{/k}$ vers la catégorie des ensembles, et si $H$ est
un foncteur de $\mathbf{Field}_{/k}$ vers la catégorie des groupes abéliens,
on définit
\[ \Inv(F,H) \]
comme étant le groupe abélien des transformations naturelles de $F$ vers
$H$ (vu par composition avec le foncteur d'oubli comme un foncteur vers
les ensembles). Autrement dit, on demande simplement d'avoir pour
toute extension $K/k$ une fonction $F(K)\to H(K)$, avec pour seule
condition d'avoir une compatibilité avec les extensions de scalaires.
On dispose toujours au moins des invariants constants :
$H(k)\subset \Inv(F,H)$. Si $F$ est un foncteur vers les ensembles \emph{pointés},
on définit le groupe des invariants \emph{normalisés} $\Inv_{norm}(F,H)$
comme le sous-groupe constitué des invariants qui envoient le point
distingué sur 0. On a dans ce cas toujours une décomposition
\[ \Inv(F,H) = \Inv_{norm}(F,H) \oplus H(k). \]
Les invariants \emph{cohomologiques} de $F$ correspondent au cas où
$H(K) = \linebreak H^*(K,C)$ pour un certain module galoisien $C$ défini
sur $k$ (on dit que les invariants sont \og{} à valeurs dans $C$\fg{}).
On note plus succintement $\Inv(F,C)$ ce groupe, et $\Inv^d(F,C)$
pour la partie de degré $d$ (c'est-à-dire correspondant à $H(K) = H^d(K,C)$).
L'exemple qui nous intéressera le plus est $C=\mu_2=\Zd$.

On parle d'invariant (par exemple cohomologique) d'un groupe algébrique
$G$ lorsque $F(K) = H^1(K,G)$ est l'ensemble (pointé) des classes
d'isomorphisme de $G$-torseurs sur $K$ ; on note alors $\Inv(G,H)$
(et $\Inv(G,C)$ pour les invariants cohomologiques à valeurs dans $C$).
On parle d'invariant d'algèbres
à involution lorsque $F(K)$ est un sous-ensemble des algèbres à
involution sur $K$ (on peut par exemple se restreindre à un certain
type d'involution, ou à un degré fixé, ou un indice borné, etc.).
Le lien évoqué ci-dessus entre groupes classiques et algèbres à
involution permet d'établir que ces deux notions sont en réalité
extrêmement proches. Notamment, si $(A,\sigma)$ est une algèbre
à involution, alors on dispose du groupe algébrique $G=\Aut(A,\sigma)$
(groupe absolument presque simple de type adjoint)
dont les points sont donnés par les automorphismes d'algèbre de $A$
commutant avec $\sigma$ et il est établi dans \cite[19.14]{BOI}
que $H^1(K,G)$ est naturellement identifié aux classes d'isomorphismes
d'algèbres à involution sur $K$ de même type que $\sigma$ et de même
degré que $A$ (l'ensemble étant bien entendu pointé par $(A,\sigma)$).

La définition d'invariants cohomologiques non triviaux d'algèbres
à involution ou de groupes algébriques est globalement un exercice
difficile, et relativement peu d'invariants ont été construits jusqu'ici.
Les seuls cas traités de façon satisfaisante sont les invariants de
degré au plus 3. Il est montré dans \cite[31.15]{BOI} que si $G$
est un groupe algébrique sur $k$,
$\Inv_{norm}^1(G,C)$ est canoniquement isomorphe à $\Hom_\Gamma(\pi_0(G_{sep}),C)$
où $\Gamma$ est le groupe de Galois absolu de $k$, et
$\pi_0(G_{sep})$ est le groupe des composantes connexes de $G_{sep}$ ;
en particulier, un groupe connexe n'a pas d'invariant de degré 1 non
constant. Il est également montré (\cite[31.19]{BOI}) que si $n$ est
premier à la caractéristique de $k$ et si $G$ est connexe
alors $\Inv^2_{norm}(G,\mu_n)$ est naturellement
isomorphe à la $n$-torsion du groupe de Picard $Pic(G)$ ; en particulier un groupe
simplement connexe n'a pas d'invariant modulo $n$ non constant de degré 1 ou 2.

Pour le cas des invariants de degré 3, jusqu'à récemment le mieux
que l'on savait faire était de traiter le cas des groupes simplement
connexes : $\Inv^3_{norm}(G,\mu_n)$ est alors un groupe cyclique
muni d'un générateur distingué appelé \emph{invariant de Rost}
(voir la deuxième partie de \cite{GMS} pour le détail des constructions).
Mais très récemment, Merkurjev a dans une série d'articles déterminé
$\Inv_{norm}^3(G,\mu_n)$ pour n'importe quel groupe $G$ semi-simple
(voir \cite{Mer1} et \cite{Mer2}).

On renvoie à \cite{Tig} pour un recensement relativement récent des
invariants cohomologiques d'algèbres à involution connus. Il apparaît
que les invariants de degré au moins 4 sont quasiment inexistants excepté
dans des contextes très restreints, notamment en indice au plus 2 (voir
\cite{Ber}, sur lequel on reviendra dans la partie \ref{sec_ind_2}).

\section*{Contenu de la thèse}

Maintenant qu'on a exposé le contexte mathématique sous-jacent, on
peut décrire l'objectif de cette thèse : construire des invariants
cohomologique d'algèbres à involution (et donc de groupes algébriques)
de degré supérieur. La thèse suit une progression logique en trois
chapitres : le premier revient sur le cas des formes quadratiques et
propose une étude relativement complète des invariants de $I^n$, le
deuxième construit des outils pour étendre ces méthodes au cas des
algèbres à involution, et le troisième met ces outils en application.
On décrit un peu plus en détail le contenu de chaque chapitre.
\\

Le point de départ est la volonté d'étendre des invariants définis
sur $I^n$ à des algèbres à involution en indice 2 suivant la méthode
de Berhuy (\cite{Ber}). Il apparaît que pour rentabiliser cette méthode,
il faut connaître le plus possible d'invariants de $I^n$ ; or mis à
part $e_n$, on ne disposait que de quelques constructions, notamment
un invariant à valeurs dans $H^{2n}(K,\mu_2)$, construit à partir d'une
opération de \og carré divisé\fg{} $I^n\to I^{2n}$ (voir \cite[§19]{Gar}).
De plus, pour étendre ces invariants il faut un certain contrôle sur
leur comportement vis-à-vis des similitudes (le problème récurrent pour
passer des formes quadratiques aux involutions orthogonales) et
de la ramification pour des valuations discrètes de rang 1 (une
considération standard quand on déploie génériquement une algèbre).
On a donc entrepris d'étendre la méthode décrite dans \cite[§19]{Gar}
pour classifier entièrement les invariants dans $\Inv(I^n,\mu_2)$
et décrire leurs propriétés, en utilisant des opérations $I^n\to I^{nd}$
pour tout $d\in \N$.

L'observation de base pour construire ces opérations est que les
formes de Pfister ont un comportement régulier vis-à-vis des
$\lambda$-opérations sur l'anneau de Grothendieck-Witt (voir
la proposition \ref{prop_pi}), ce qui permet de \og déformer\fg{}
la structure de $\lambda$-anneau en une autre pour laquelle les
opérations envoient $I^n$ dans $I^{nd}$. Comme toutes ces constructions
sont essentiellement formelles, on développe un formalisme général
décrivant comment on peut déformer naturellement une structure de
pré-$\lambda$-anneau, et en particulier comment obtenir une
structure qui a un comportement naturel vis-à-vis d'une notion
abstraite de forme de Pfister (voir le théorème \ref{thm_pi}).

On démontre ensuite (voir le théorème \ref{thm_g})
que ces opérations $I^n\to I^{nd}$ forment une
\og base\fg{} des invariants $\Inv(I^n,W)$, et que de même les
invariants cohomologiques $I^n(K)\to H^{nd}(K,\mu_2)$ qui s'en déduisent
par composition avec $e_{nd}$ forment
une \og{}base\fg{} de $\Inv(I^n,\mu_2)$ (les guillemets sont
présents parce que des combinaisons infinies sont nécessaires).
Pour éviter de dupliquer toutes les preuves on met en place un
formalisme commun aux deux types d'invariants (invariants de Witt
et invariants cohomologiques). Le reste de la section est consacrée
à l'étude des propriétés élémentaires de ces invariants, notamment leur
comportement par produit, similitude, et restriction de $I^n$ à
$I^{n+1}$. On consacre également une partie à l'étude des invariants
de formes dans $I^n$ qui ont une $r$-forme de Pfister en facteur,
afin de généraliser la construction donnée dans \cite[20.8]{Gar}.

Enfin, une courte section expose une idée, qui sera à développer ultérieurement,
consistant à décrire les invariants qui sont définis lorsqu'un invariant
précédemment défini s'annule, de sorte à construire des \og tours\fg{}
d'invariants, à la manière des invariants $e_n$.
\\

Le deuxième chapitre, indépendant du premier, décrit les outils
nécessaires pour étendre les méthodes précédentes aux algèbres
à involution. Un des problèmes récurrents pour passer des
formes quadratiques aux involutions orthogonales, outre
évidemment la complexité accrue de l'objet étudié, est le
manque de structure dont on dispose. Les formes quadratiques
peuvent s'additionner, se multiplier, et on dispose également
des opérations $\lambda$, dont on a vu dans le chapitre précédent
qu'elles étaient une clé possible pour définir des invariants.
À l'inverse, on peut multiplier des algèbres à involution (par
le produit tensoriel), et \cite[§10]{BOI} décrit comment définir
des opération $\lambda$, mais on ne peut pas combiner ces opérations
car on ne dispose pas a priori d'addition. Il a déjà été observé
(notamment dans \cite{KPS} et \cite{Dej}) que si en revanche on dispose
d'une équivalence de Morita entre deux algèbres à involution,
alors on peut définir leur somme, par exemple en les représentant
comme des formes $\eps$-hermitiennes sur une même algèbre à
involution ; si on ne dispose pas d'une équivalence de Morita
explicite, les formes hermitiennes sont seulement définies
à un facteur près, et la somme n'est pas bien définie.

On met en place dans ce chapitre un formalisme qui permet
d'effectuer toutes les opérations que l'on souhaite sans
avoir à garder un contrôle permanent sur des équivalences
de Morita explicite. Précisément, pour toute algèbre à
involution $(A,\sigma)$ de première espèce, on définit
un anneau commutatif gradué $\tld{GW}(A,\sigma)$ qui
contient toutes les classes d'isométrie de formes
$\eps$-hermitiennes sur $(A,\sigma)$, ainsi que l'équivalent
$\tld{W}(A,\sigma)$ de l'anneau de Witt pour les formes
quadratiques. On utilise pour
cela des équivalences de Morita canoniques données par la
forme trace d'involution $T_\sigma$. On définit également
une structure de pré-$\lambda$-anneau sur $\tld{GW}(A,\sigma)$,
qui correspond aux opérations sur les involutions décrites
dans \cite[§10]{BOI}.

Enfin, afin de préparer l'application de ces outils à la
construction d'invariants cohomologiques, on définit
un analogue de la filtration fondamentale de l'anneau de Witt,
et on donne une description du gradué associé, qui remplace
la cohomologie modulo 2 dans ce contexte (par analogie avec
la conjecture de Milnor).
\\

Pour finir, le dernier chapitre est celui où on confronte les deux
premiers pour enfin construire des invariants cohomologiques
d'algèbres à involution. C'est l'objet de la première partie,
où les opérations contruites dans le premier chapitre sont appliquées
à la filtration de $\tld{GW}(A,\sigma)$ (voir la proposition
\ref{prop_inv_herm}). On discute après la proposition de comment
exploiter ces constructions pour construire des invariants supérieurs,
notamment des invariants relatifs, et on propose une étude plus
approfondie du cas des algèbres d'indice 2.

Le reste du chapitre est consacré à des phénomènes en petite dimension.
La deuxième partie est dédiée à des calculs explicites concernant
l'équivalence entre algèbres de type $A_3$ et $D_3$, ce qui est
mis en application dans la troisième partie où on fait le lien
entre nos constructions et un invariant de degré 4 défini dans
\cite{RST}, et où on exploite les outils du premier chapitre
pour étendre au cas de l'indice 2 un certain invariant défini
sur les formes quadratiques de degré 12 dans $I^3$ (voir
\cite[20.8]{Gar}).

\chapter{Invariants de classes de Witt}

\label{par_pfister}Dans ce chapitre on propose un traitement détaillé des invariants
de Witt et des invariants cohomologiques de $I^n$. On rappelle qu'on
s'est fixé un corps de base $k$ de caractéristique différente de 2,
et que $K/k$ désigne une extension de corps. Par \og forme quadratique\fg{}
on entendra \emph{toujours} forme quadratique \emph{non dégénérée}.
On note $GW(K)$ l'anneau
de Grothendieck-Witt et $W(K)$ l'anneau de Witt. Dans $GW(K)$ on note
$\hat{I}(K)$ l'idéal fondamental des formes de dimension virtuelle nulle,
dont la projection dans $W(K)$ est l'idéal fondamental $I(K)$. On
fait la remarque importante que la projection naturelle est en réalité
un isomorphisme $\hat{I}(K)\simeq I(K)$, et en particulier on identifiera
systématiquement $\hat{I}^n(K)$ et $I^n(K)$. Notamment, une $n$-forme
de Pfister $\pfis{a_1,\dots,a_n}$ désignera sauf mention explicite du
contraire un élément de $\hat{I}^n(K)$ ; précisément, $\pfis{a} = \fdiag{1}-\fdiag{a}$
(et non comme on en a l'habitude $\fdiag{1,-a}$), ce qui
ne change rien pour les classes de Witt.

Notre objectif est de construire des applications naturelles
$\pi_n^d: I^n(K)\to I^{nd}(K)$, vérifiant
\[ \pi_n^d\left( \sum_{i=1}^r \phi_i \right) = \sum_{i_1<\dots<i_d}\phi_{i_1}\cdots \phi_{i_d} \]
pour toutes $n$-formes de Pfister $\phi_i$ (voir la proposition
\ref{prop_somme_pfis}), et d'en déduire par composition avec $e_{nd}$
des invariants $u_{nd}^{(n)}: I^n(K)\to H^{nd}(K,\mu_2)$.
La construction des $\pi_n^d$ est essentiellement formelle, ne dépendant
que des relations des formes de Pfister dans le $\lambda$-anneau $GW(K)$.
On donne donc une définition valable dans n'importe quel $\lambda$-anneau
(voir le théorème \ref{thm_pi}).

Le résultat central du chapitre est le théorème \ref{thm_g}, qui montre
en particulier (voir le corollaire \ref{cor_ecr_unique}) que tout invariant
de Witt (resp. cohomologique) peut s'écrire de façon unique comme combinaison
infinie des $\pi_n^d$ (resp. des $u_{nd}^{(n)}$). Pour éviter de dupliquer
les preuves, on adopte un formalisme commun pour ces deux types d'invariants
(voir la partie \ref{sec_contexte}), dans lequel $\pi_n^d$ et $u_{nd}^{(n)}$
sont tous les deux notés $f_n^d$. La description exacte donnée dans le
théorème se fait en terme d'invariants notés $g_n^d$, qui sont décrits
explicitement (proposition \ref{prop_f_g}) et ont la propriété remarquable
que toute combinaison, même infinie, donne un invariant.

Une fois ce théorème établi, on peut utiliser les propriétés particulières
des $f_n^d$ pour démontrer un certain nombre de propriétés des invariants,
notamment comment exprimer le produit de deux invariants (proposition \ref{prop_pi_prod}),
la restriction d'un invariant de $I^n$ à $I^{n+1}$ (proposition \ref{prop_restr_f}),
le comportement d'un invariant par similitude (proposition \ref{prop_simil}),
et par résidus pour une valuation discrète (proposition \ref{prop_ram_witt}).
On étudie également le lien entre notre description des invariants de $I$
et $I^2$ avec la description de Serre des invariants de formes quadratiques
en dimension fixée, et on fait le lien avec une description de Vial des
opérations sur la cohomologie modulo 2.

Dans la dernière partie du chapitre, on propose une réflexion sur une
notion de \og tours d'invariants\fg{}, à savoir des suites d'invariants
tels que chaque invariant est défini si les précédents sont nuls (sur le
modèle des invariants $e_n$).

\section{Anneaux grecs et éléments de Pfister}\label{sec_grec}

On présente ici des définitions et résultats classiques (voir \cite{Yau}) sur
les $\lambda$-anneaux, ainsi que quelques constructions qui nous seront utiles.
En revanche, on change un peu le point de vue en ne privilégiant pas les
$\lambda$-opérations telles que $\lambda^d(1)=0$ pour $d\pgq 2$ ; on réservera
le symbole $\lambda$ à celles-ci, et on utilisera plutôt le symbole $\alpha$
pour des opérations qui ont une action quelconque sur 1.

Pour cette raison, on choisit d'adopter une terminologie plus neutre en appelant
\emph{anneau grec} ce qui est appelé \emph{pré-$\lambda$-anneau} dans la littérature
moderne (et était initialement appelé \emph{$\lambda$-anneau} chez Grothendieck).
À un anneau grec est associée une famille d'opérations 
vérifiant les axiomes de pré-$\lambda$-anneau au sens de \cite{Yau}, chacune caractérisée
par son action sur 1.

La finalité est de contruire des opérations ayant un bon comportement
sur une famille fixée d'éléments (en ce qui nous concerne, les
\emph{éléments de Pfister}).

\subsection{Anneaux grecs}

\subsubsection{Opérations grecques}

\begin{defi}\label{defi_ope_grec}
  Soit $R$ un anneau commutatif. Une \emph{opération grecque} sur
  $R$ est une suite d'applications
  \[ \alpha^d : R \To R \]
  pour tout $d\in \N$ vérifiant les propriétés suivantes :
  \begin{enumerate}[label=(\roman*)]
  \item $\forall x\in R,\, \alpha^0(x) = 1$ ;
    
  \item $\forall x\in R,\, \alpha^1(x) = x$ ;
    
  \item $\forall x,y\in R,\, \forall d\in \N,\, \alpha^d(x + y) = \sum_{k=0}^d \alpha^k(x)\alpha^{d-k}(y)$.
  \end{enumerate}

  On note $\Gamma(R)$ l'ensemble des opérations grecques de $R$.
\end{defi}

La définition usuelle d'un pré-$\lambda$-anneau consiste à dire dans ce langage
que c'est un anneau muni d'une opération grecque ; notre définition consistera
plutôt à se donner une classe d'équivalence de telles opérations.

\begin{ex}\label{ex_Z_lambda}
  La suite $(\lambda^d)$ où $\lambda^d(n) = \binom{n}{d}$ est une opération grecque sur $\Z$.
\end{ex}

\begin{ex}\label{ex_GW_lambda}
  L'exemple qui nous intéresse en premier lieu : les $\lambda$-opérations usuelles
  sur les formes quadratiques définissent une opération grecque sur $GW(K)$. Il est
  à noter que si ces opérations s'étendent bien à $GW(K)$, elles ne passent
  en revanche pas au quotient sur $W(K)$, ce qui explique pourquoi on ne travaille
  pas directement sur $W(K)$.
\end{ex}

On propose une caractérisation plus algébrique de la définition précédente :
pour tout anneau commutatif $R$,
on définit
\begin{equation}\label{eq_gr}
 G(R) = 1 + tR[[t]] \subset R[[t]]. 
\end{equation}

Alors $G(R)$ est un groupe abélien pour la multiplication dans $R[[t]]$,
et on obtient ainsi un foncteur $G: \mathbf{Ring}\to \mathbf{Ab}$.
On a un morphisme naturel de groupes
\begin{equation}\label{eq_eps_r}
\foncdef{\eps_R}{G(R)}{R}{1+\sum a_k t^k}{a_1.} 
\end{equation}

\begin{prop}\label{prop_alpha_t}
  Soit $\alpha_t: R\To G(R)$ une application. On note
  \[ \alpha_t(x) = \sum_d \alpha^d(x)t^d. \]
  Alors la suite d'applications $(\alpha^d)$ est une opération
  grecque si et seulement si $\alpha_t$ est un morphisme de groupe
  qui est une section de $\eps_R$.
\end{prop}

\begin{proof}
  C'est une simple reformulation de la définition précédente.
\end{proof}

On se donnera donc une opération grecque indifféremment sous
la forme d'une suite d'applications de $R$ dans $R$, ou d'une unique
application à valeurs dans $G(R)$.

\subsubsection{Lettres grecques}

L'idée des définitions qui suivent est la suivante : si on s'est donné une 
opération grecque $(\alpha^d)$, peut-on la \og{} déformer\fg{} en $(\beta^d)$ avec
\[ \beta^d = \sum_{k=0}^d a_{k,d} \alpha^k \]
où les $a_{k,d}$ sont des coefficients dans $R$, de sorte que les $\beta^d$
forment encore une opération grecque ?

Il s'avère que c'est possible, et que les $a_{k,d}$ sont entièrement
déterminés par les $\beta^d(1)\in R$ pour $d\pgq 2$. De plus, toutes les
suites $(\beta^d(1))_d$ sont atteignables de cette façon, et en
particulier on peut inverser le processus pour retrouver les $\alpha^d$
à partir des $\beta^d$ par le même genre de déformation.

On prend alors le parti de dire que toutes ces opérations
définissent une même structure d'\emph{anneau grec}, et que le choix d'une opération
spécifique $(\alpha^d)$ correspond au choix d'une suite d'éléments (qui est alors
$(\alpha^d(1))$), qu'on appelle la \emph{lettre grecque} de l'opération. On
peut ainsi pour un même anneau grec passer d'une lettre grecque à l'autre pour
faire varier les opérations qu'on utilise selon les propriétés qu'on désire.
\\

Plus précisément : pour tout anneau commutatif $R$, on pose
\begin{equation}\label{eq_l_r}
  L(R) = tG(R) = t + t^2R[[t]] \subset R[[t]].
\end{equation}

Alors $L(R)$ est un groupe (pas abélien) pour la composition dans $R[[t]]$,
ce qui définit un foncteur $L: \mathbf{Ring}\to \mathbf{Gpe}$.

\begin{defi}\label{def_lettre}
  Soit $R$ un anneau commutatif, et $\alpha_t: R\to G(R)$ une opération
  grecque. On appelle \emph{lettre grecque} de l'opération l'élément
  $\alpha_t(1)-1\in L(R)$, ce qui définit une application naturelle
  $\Lambda_R : \Gamma(R)\to L(R)$ (où on rappelle que $\Gamma(R)$ désigne
  l'ensemble des opérations grecques sur $R$).
\end{defi}

On pose aussi
\begin{equation}\label{eq_a_r}
    A(R) = \ens{\Phi\in \Aut(G(R))}{ \eps_R\circ \Phi = \eps_R}. 
\end{equation}
Autrement dit, un élément de $A(R)$ doit préserver le terme de degré $1$.

\begin{prop}
L'application
\[ \anonfoncdef{L(R)^{op}}{A(R)}{\tau}{(f\mapsto f\circ \tau)} \]
définit un morphisme de groupes injectif.
\end{prop}

\begin{proof}
Quel que soit $\tau\in tR[[t]]$, l'application $f\mapsto f\circ \tau$
est un morphisme de groupe de $G(R)$.

Clairement, l'application de l'énoncé est un morphisme de monoïdes vers les endomorphismes de $G(R)$,
et comme $L(R)$ est un groupe, on a bien un morphisme de groupes vers les
automorphismes de $G(R)$. Il est injectif puisque l'action de $\tau$
sur $1+t$ est $1+\tau$.

De plus, comme le terme de degré $1$ de $\tau \in L(R)$ est $t$, composer 
par $\tau$ à droite préserve le terme de degré $1$, donc le morphisme
est à valeurs dans $A(R)$.
\end{proof}

\begin{coro}
  Le groupe $L(R)$ agit à droite sur $\Gamma(R)$,
  par $\alpha_t\cdot \tau = \alpha_{\tau(t)}$, et
  $\Lambda_R: \Gamma(R)\To L(R)$ est un morphisme de
  $L(R)$-ensembles à droite.

  En particulier, l'action de $L(R)$ sur $\Gamma(R)$ est
  libre.
\end{coro}

\begin{proof}
  Le groupe $A(R)$ agit à gauche sur $\Gamma(R)$ par
  $\Phi\cdot \alpha_t = (x\mapsto \Phi(\alpha_t(x)))$ ; le
  fait que $\Phi$ soit un automorphisme de groupe montre
  que $\Phi\cdot \alpha_t$ est bien un morphisme de groupes
  $R\to G(R)$, et le fait que $\Phi$ préserve le coefficient
  de degré 1 montre que $\Phi\cdot \alpha_t$ est bien une
  section de $\eps_R$. Comme $L(R)^{op}$ s'injecte dans $A(R)$,
  $L(R)$ agit à droite sur $\Gamma(R)$.

  De plus, la lettre grecque de $\alpha_t\cdot \tau$ est par
  contruction $\alpha_{\tau(t)}(1)-1 = (\alpha_t(1)-1)\circ \tau$,
  donc l'action de $L(R)$ sur $\Gamma(R)$ est bien compatible
  avec $\Lambda_R$.

  De façon générale, pour tout groupe $G$, si un $G$-ensemble admet un morphisme vers
  un $G$-ensemble libre, il est lui-même libre. Or $L(R)$ est
  bien entendu un $L(R)$-ensemble libre, donc $\Gamma(R)$ aussi.
\end{proof}

\subsubsection{Anneaux grecs}

On en vient enfin à la définition qui nous intéresse :

\begin{defi}
  Un \emph{anneau grec} est un anneau commutatif $R$ muni
  du choix d'une orbite d'opérations grecques sous l'action
  de $L(R)$.
\end{defi}

\begin{propdef}
  Soit $R$ un anneau grec, et soit $\alpha\in L(R)$ une lettre grecque.
  Alors il existe une unique opération grecque dans l'orbite donnée
  par la définition qui a $\alpha$ pour lettre grecque. On l'appelle
  la $\alpha$-opération de $R$, et on note $\alpha_t: R\To G(R)$
  et $\alpha^d: R\To R$ les applications correspondantes.
\end{propdef}

\begin{proof}
  Une orbite est un $L(R)$-ensemble transitif, donc l'action étant
  libre la restriction de $\Lambda_R$ à cette orbite doit être
  bijective vers $L(R)$.
\end{proof}

Le symbole $\alpha$ sera généralement utilisé pour désigner une
lettre grecque quelconque.

\begin{ex}
  On notera $\lambda$ l'élément neutre de $L(\Z)$, à savoir $t\in \Z[[t]]$.
  La $\lambda$-opération d'un anneau grec est donc celle vérifiant $\lambda^d(1)=0$
  si $d\pgq 2$.

  Les exemples \ref{ex_Z_lambda} et \ref{ex_GW_lambda} étaient des
  $\lambda$-opérations, comme la notation employée le suggérait.
\end{ex}

\begin{ex}
  Dans son étude de la K-théorie, Grothendieck a introduit des opérations
  sur tout $\lambda$-anneau, dites $\gamma$-opérations,  qui sont devenues
  classiques dans la théorie. Elles correspondent bien à notre notion de
  $\gamma$-opérations sur un anneau grec si on note $\gamma$ la lettre grecque $\sum_{d\pgq 1} t^d$
  (autrement dit c'est l'opération caractérisée par $\gamma^d(1)=1$ pour tout $d$). 
\end{ex}

Le fait de définir la structure d'anneau grec à partir d'une orbite d'opérations
est justifié par le fait que de nombreuses constructions ne dépendent pas du choix
d'une lettre grecque particulière.

\begin{propdef}
  Soient $R$ et $S$ deux anneaux grecs, et soit $f:R\To S$ un morphisme
  d'anneaux. Soit $\alpha\in L(R)$ ; on note $f(\alpha)\in L(S)$ la lettre
  grecque obtenue par fonctorialité. Alors le fait que pour tout $d\in \N$
  et tout $x\in R$ on ait $f(\alpha^d(x)) = f(\alpha)^d(f(x))$ ne dépend
  pas du choix de $\alpha$.

  Quand cette condition est vérifiée, on dit que $f$ est un morphisme
  d'anneaux grecs.
\end{propdef}

\begin{proof}
  Soit $\beta\in L(R)$ une autre lettre grecque, et soit $\tau\in L(R)$
  tel que $\beta = \alpha\cdot \tau$. Considérons le diagramme commutatif
  suivant :
  \[ \begin{tikzcd}
      R \rar{\alpha_t} \dar{f} & G(R) \rar{\cdot \tau} \dar{f_*} & G(R) \dar{f_*} \\
      S \rar{f(\alpha)_t} & G(S) \rar{\cdot f(\tau)} & G(S)
    \end{tikzcd} \]
  Alors le fait que le carré de gauche commute traduit la condition voulue pour $\alpha$,
  le carré de droite commute toujours par fonctorialité, et le fait que le grand
  rectangle commute traduit la condition pour $\beta$.
\end{proof}

\begin{propdef}
  Soit $R$ un anneau grec, et soit $I$ un idéal. Soit $\alpha\in L(R)$.
  Alors le fait que pour tout $d\in \N^*$ on ait $\alpha^d(I)\subset I$
  ne dépend pas du choix de $\alpha$.

  Quand cette condition est vérifiée, on dit que $I$ est un idéal grec.
\end{propdef}

\begin{proof}
  Si $\beta$ est une autre lettre grecque, alors $\beta^d$ est une
  combinaison des $\alpha^k$ pour $1\ppq k\ppq d$, donc si chaque
  $\alpha^k$ préserve $I$, $\beta^d$ aussi.
\end{proof}

\subsubsection{Dimension}

Dans cette thèse, on utilisera le fait de changer de lettre
grecque pour contrôler la dimension des éléments.

\begin{defi}
  Soit $R$ un anneau grec et $\alpha\in L(R)$. La $\alpha$-dimension
  d'un élément $x\in R$ est le degré de $\alpha_t(x)$ (qui peut être
  infini).
\end{defi}

\begin{ex}
  Dans $GW(K)$, la $\lambda$-dimension d'une forme quadratique $q$
  est sa dimension au sens usuel (aussi appelé son rang).
\end{ex}

\begin{ex}
  Dans tout anneau grec, $1$ est de $\lambda$-dimension $1$.
  De façon générale, la $\alpha$-dimension de $1$ est le degré
  de $\alpha$.
\end{ex}

Bien évidemment, la $\alpha$-dimension d'un élément dépend fortement du
choix de $\alpha$. On s'intéresse particulièrement aux éléments
de dimension 1 :

\begin{prop}\label{prop_dim1}
  Soient $R$ un anneau grec, $X\subset R$ un sous-ensemble et $\alpha\in L(R)$.
  Alors il existe $\beta\in L(R)$ tel que tout $x\in X$ soit de $\beta$-dimension
  1 si et seulement si il existe $\tau\in L(R)$ tel que pour tout $x\in X$ 
  \[ \alpha_t(x) = 1 + x\tau(t), \]
  c'est-à-dire qu'il existe $a_d\in R$ tels que pour tout $d\pgq 1$ et tout
  $x\in X$ on ait $\alpha^d(x) = a_dx$.
\end{prop}

\begin{proof}
  Si $\beta$ existe, on pose $\tau = \beta^{-1}\circ \alpha$, qui
  donne trivialement $\alpha_t(x) = 1 + x\tau(t)$ pour tout $x\in X$,
  et si $\tau$ existe on pose $\beta = \alpha\circ \tau^{-1}$.
\end{proof}

\begin{rem}
  La preuve montre qu'on peut prendre $\beta = \alpha\circ \tau^{-1}$.
  Tout autre choix de $\beta$ s'écrit $\beta+\delta$ avec $\delta\in t^2Ann_R(X)[[t]]$
  (où $Ann_R(X)$ est l'idéal annulateur de $X$ dans $R$).
\end{rem}

\begin{prop}\label{prop_somme_dim_1}
  Soit $R$ un anneau grec, et soient $x_1,\dots,x_r\in R$ de $\alpha$-dimension 1.
  Alors pour tout $d\in \N$, on a
  \[ \alpha^d(x_1+\cdots +x_r) = \sum_{i_1<\dots < i_d} x_{i_1}\cdots x_{i_d}. \]
  En particulier, si $r<d$ alors $\alpha^d(x_1+\cdots +x_r)=0$. 
\end{prop}

\begin{proof}
  De façon générale, on voit par récurrence que quels que soient
  $x_1,\dots,x_r$ on a
  \[ \alpha^d(x_1+\cdots +x_r) = \sum_{k_1+\cdots +k_r=d} \alpha^{k_1}(x_1)\cdots \alpha^{k_r}(x_r). \]
  On spécialise alors au cas où les $x_i$ sont de $\alpha$-dimension 1,
  et donc il ne reste que les $k_i=0$ ou $1$.
\end{proof}

\subsubsection{Semi-anneaux et graduations}\label{sec_semi_ann}

On aura besoin dans le chapitre suivant de travailler avec des semi-anneaux,
et en particulier des semi-anneaux gradués. On consacre donc
cette courte partie à généraliser les définitions précédentes
à ce cadre, qui ne servira pas dans la suite de ce chapitre.

Soit $S$ un semi-anneau commutatif, et soit $\Gamma$ un groupe
commutatif, noté additivement. Une $\Gamma$-graduation sur $S$
est une décomposition $S = \bigoplus_{g\in \Gamma}S_g$ en tant que monoïde
additif telle que $S_gS_h\subset S_{g+h}$. Notamment, tout semi-anneau
commutatif peut tautologiquement être vu comme gradué sur le groupe
trivial.

\begin{lem}\label{lem_semi_grad}
  Soit $R$ le groupe de Grothendieck de $S$, muni de sa structure
  naturelle d'anneau commutatif. Alors toute $\Gamma$-graduation
  de $S$ s'étend de façon unique à $R$.
\end{lem}

\begin{proof}
  Si $R_g$ est le groupe de Grothendieck de $S_g$, alors la propriété
  universelle montre que $R=\bigoplus_{g\in G}R_g$, et aussi que
  $R_gR_h\subset R_{g+h}$.
\end{proof}

Si $S$ est $\Gamma$-gradué, on définit une opération grecque homogène
sur $S$ comme un morphisme de monoïdes $\alpha_t:S\To G(S)$ tel
que $\alpha^d(S_g)\subset S_{dg}$. 

\begin{lem}\label{lem_semi_grec}
  Soit $S$ un semi-anneau commutatif $\Gamma$-gradué, et
  soit $R$ son anneau de Grothendieck. Alors toute opération
  grecque homogène sur $S$ s'étend de façon unique en une
  opération grecque homogène sur $R$.
\end{lem}

\begin{proof}
  On dispose d'un morphisme naturel de monoïdes $G(S)\to G(R)$,
  ce qui induit $\alpha_t:S\to G(R)$ par composition, et par
  propriété universelle on a une unique extension $R\to G(R)$
  en un morphisme de groupes. Comme $R_g$ est engendré par $S_g$,
  on voit facilement que $\alpha^d(R_g)\subset R_{dg}$.
\end{proof}

Si un anneau gradué $R$ est muni d'une opération grecque homogène,
alors sa lettre grecque est dans $L(R_0)$. De plus, si $R$ est
un anneau grec, et si $\alpha,\beta\in L(R_0)$, alors $\alpha_t$
est homogène si et seulement si $\beta_t$ l'est. On peut donc adapter
les constructions précédentes au cas gradué en définissant un anneau
grec gradué comme un anneau gradué muni d'une structure d'anneau
grec telle que les opérations dont la lettre grecque est dans $L(R_0)$
soient homogènes. Le cas non gradué se retrouve directement
en prenant la graduation triviale.

On définit un morphisme d'anneaux grecs $\Gamma$-gradués comme un
morphisme d'anneau grecs qui est un morphisme homogène,
ce qui définit clairement une catégorie des anneaux
grecs $\Gamma$-gradués.

\begin{ex}
  Si on pose $GW^\pm(K)=GW(K)\oplus GW^-(K)$, alors c'est
  un anneau grec $\Zd$-gradué.
\end{ex}

\begin{ex}
  Si $R$ est un anneau grec et $A$ un groupe commutatif, alors l'anneau de groupe
  $R[A]$ est naturellement un anneau grec $A$-gradué,
  en posant pour tout $\alpha\in L(R)$ $\alpha^d(xa)=\alpha^d(x)(da)$
  si $x\in R$ et $a\in A$.
\end{ex}

\begin{ex}
  De façon plus précise, si $R$ est un anneau grec $B$-gradué,
  alors $R[A]$ est un anneau grec $A\times B$-gradué. Notamment,
  $GW^\pm(K)[\Zd]$ est naturellement un anneau grec gradué sur
  le groupe de Klein.
\end{ex}

\subsection{$\lambda$-anneaux}

La plupart des anneaux grecs que l'on rencontre \og dans la nature\fg{}
(et en particulier $GW(K)$) satisfont des propriétés supplémentaires,
qui cette fois s'expriment spécifiquement en fonction des opérations $\lambda$.

\begin{prop}\label{prop_prod_gr}
Le groupe abélien $(G(R), +_G)$ (où $+_G$ est le produit des séries formelles)
est muni d'un unique produit $\times_G$ naturel en $R$ tel
que $(G(R),+_G,\times_G)$ soit un anneau commutatif, et tel que
pour tout $f(t)\in G(R)$ et tout $a\in R$, on ait
\[  f(t) \times_G (1+at) = f(at).   \] 
\end{prop}

\begin{proof}
  Pour montrer l'unicité, il suffit par naturalité de montrer que pour
  $R = \Z[\sigma_1,\sigma_2,\dots ; s_1,s_2,\dots]$ (où les $\sigma_i,s_i$
  sont des indéterminées) le produit $f(t)\times_G g(t)$ est uniquement
  déterminé, avec $f(t) = 1+\sum_i \sigma_it^i$ et $g(t) = 1+\sum_i s_it^i$.

  Pour cela, il suffit de montrer que $p_n(f(t)\times_G g(t))$ est
  uniquement déterminé pour tout $n$, où $p_n : R\to R_n = \Z[\sigma_1,\dots,\sigma_n
  ; s_1,\dots,s_n]$ est la projection naturelle. Par naturalité, $p_n(f(t)\times_G g(t))$
  est
  \[ (1+\sum_{i=1}^n \sigma_i t^i)\times_G (1 +\sum_{i=1}^n s_i t^i)\in R_n. \]
  Or on sait qu'on a une inclusion $R_n\subset \Z[x_1,\dots,x_n ; y_1,\dots, y_n]$
  de sorte que $\sigma_i$ soit la $i$-ème fonction symétrique sur les $x_i$,
  et $s_i$ celle sur les $y_i$. On a alors
  \[  1+\sum_{i=1}^n \sigma_i t^i = \prod_{i=1}^n (1+x_it)  \]
  et de même pour les $s_i$, d'où, en se souvenant que le produit des polynômes
  est la somme $+_G$ dans $G(R_n)$ et que donc $\times_G$ est distributif sur
  cette opération :
  \[ (1+\sum_{i=1}^n \sigma_i t^i)\times_G (1 +\sum_{i=1}^n s_i t^i)
    = \prod_{i,j} (1+x_it)\times_G (1+y_jt)  = \prod_{i,j} (1+x_iy_jt).\]

  Pour l'existence, il faut vérifier que cela définit bien un produit
  associatif naturel sur $G(R)$, voir \cite[thm 2.5]{Yau}.
\end{proof}

\begin{prop}
  L'anneau $G(R)$ est muni d'une unique structure d'anneau grec naturelle
  en $R$ telle que pour tout $a\in R$ l'élément $1+at$ soit de
  $\lambda$-dimension 1.
\end{prop}

\begin{proof}
  Comme pour la proposition précédente, on se ramène à montrer
  que les $\lambda^d(1+\sum_{i=1}^n \sigma_i t^i)$ sont uniquement
  déterminés, et comme $1+\sum_{i=1}^n \sigma_i t^i = (1+x_1t) +_G\dots +_G (1+x_nt)$,
  on a
  \[ \lambda^d(1+\sum_{i=1}^n \sigma_i t^i) = \prod_{i_1<\dots<i_d} (1+x_{i_1}\cdots x_{i_d}t). \]

  Le fait que cela définisse bien une structure d'anneau grec est
  encore montré dans \cite[thm 2.6]{Yau}.
\end{proof}

\begin{defi}
  Soit $R$ un anneau grec. On dit que $R$ est un $\lambda$-anneau
  si $\lambda_t: R\To G(R)$ est un morphisme d'anneaux grecs.
\end{defi}

\begin{rem}
  Cela correspond bien à la définition usuelle : on requiert que
  $\lambda_t(xy)=\lambda_t(x)\times_G \lambda_t(y)$ et
  $\lambda_t(\lambda^d(x)) = \lambda^d(\lambda_t(x))$, ce qui donne
  les formules usuelles pour $\lambda^d(xy)$ et $\lambda^n(\lambda^m(x))$
  étant donné la structure d'anneau grec $G(R)$.
\end{rem}

\begin{ex}
  Les anneaux $\Z$ et $GW(K)$ sont des $\lambda$-anneaux. De façon
  plus générale, pour une large classe de catégories abéliennes
  monoïdales symétriques $\mathbf{C}$, l'anneau de Grothendieck $K(\mathbf{C})$
  est un $\lambda$-anneau ($\Z$ correspondant par exemple à une
  catégorie d'espaces vectoriels, et $GW(K)$ à la catégorie des espaces
  bilinéaires symétriques non dégénérés sur $K$).
\end{ex}

\begin{ex}\label{ex_lambda_libre}
  L'anneau $\Z[x_1,x_2,\dots]$ admet une unique structure de $\lambda$-anneau
  telle que $\lambda^d(x_1)=x_d$. Cette structure fait de cet anneau le $\lambda$-anneau
  libre sur un élément (voir \cite[thm 2.25]{Yau}).

  De même, on peut construire le $\lambda$-anneau libre sur $n$ éléments en prenant
  $n$ jeux de variables indépendantes.
\end{ex}

\subsection{Éléments de Pfister}

Notre but est de montrer que les $n$-formes de Pfister sont
de $\pi_n$-dimension 1 dans $GW(K)$ pour une certaine lettre
grecque $\pi_n$ (indépendante de $K$). On va aboutir à ce résultat en
abstrayant la notion de $n$-forme de Pfister à un anneau grec quelconque,
à partir de na notion de $1$-forme de Pfister.

\begin{defi}
  Soient $R$ un anneau grec, $\alpha\in L(R)$ et $x\in R$.
  On dit que $x$ est un $\alpha$-élément de Pfister si $x^2=2x$
  et pour tout $d\in \N^*$, $\alpha^d(x)=x$.

  Si $\alpha=\lambda$, on parlera simplement d'élément de Pfister.

  Un $n$-élément de Pfister est le produit de $n$ élements de Pfister (en
  particulier 1 est l'unique $0$-élément de Pfister, et un 1-élément de Pfister
  est simplement un élément de Pfister). On note $\Pf_n(R)$ l'ensemble
  des $n$-éléments de Pfister de $R$ ; $\Pf_n$ est un foncteur des anneaux
  grecs dans les ensembles.
\end{defi}

\begin{rem}
  En utilisant la proposition \ref{prop_dim1}, on voit que les éléments
  de $R$ qui sont des $\alpha$-éléments de Pfister
  pour un certain $\alpha$ sont exactement ceux qui vérifient
  $x^2=2x$ et sont de $\beta$-dimension 1 pour un certain $\beta$.
\end{rem}

On peut préciser ça dans le cas $\alpha=\lambda$ :

\begin{prop}\label{prop_pi}
  Soit $R$ un anneau grec. Les éléments de Pfister de $R$
  sont exactement les $x\in R$ vérifiant $x^2=2x$ et
  de $\pi$-dimension au plus 1, où $\pi = \sum (-1)^{d+1}t^d$.

  De plus, ce sont exactement les éléments de la forme $1-a$
  où $a$ est de $\lambda$-dimension 1 et vérifie $a^2 = 1$. 
\end{prop}

\begin{proof}
  Les éléments de Pfister sont ceux qui vérifient $x^2=2x$
  et $\lambda_t(x) = 1 + x\tau(t)$ où $\tau(t) = \sum_{d\pgq 1}t^d$.
  Cette dernière condition est équivalente à $\pi_t(x)=1+xt$
  si $\pi=\tau^{-1}$. Or $\tau = \frac{t}{1-t}$, donc il est
  facile de voir que sa réciproque est $\pi = \frac{t}{1+t}$,
  comme annoncé.

  Soit $x\in R$, et $a=1-x$. Alors $x^2=2x$ est équivalent à
  $a^2=1$. Comme $\lambda_t$ est un morphisme de groupe,
  $\lambda_t(a)= \frac{\lambda_t(1)}{\lambda_t(x)} =  \frac{1+t}{\lambda_t(x)}$ donc
  le fait que $a$ soit de $\lambda$-dimension 1 est équivalent
  à $\lambda_t(x)=\frac{1+t}{1+at}$ et comme $a^2=1$ cela donne
  $\lambda_t(x)=1+\sum_{d\pgq 1}(1-a)t^d$,
  donc à $\lambda^d(x)=x$ pour tout $d\pgq 1$.
\end{proof}

\begin{ex}
  Soit $\phi\in I(K)$ la classe de Witt d'une $n$-forme de Pfister,
  et soit $x\in \hat{I}(K)$ son unique relevé. Alors $x$ est un
  $n$-élément de Pfister dans $GW(K)$ (et ce sont en fait les seuls).
  C'est ce qui explique qu'on choisisse comme convention dans cette
  thèse d'appeler \og forme de Pfister\fg{} les éléments de cette forme
  plutôt que la classe dans $GW(K)$ d'une forme de Pister au sens
  usuel.
\end{ex}

\begin{ex}
  Le foncteur $\Pf$ est représentable si on le restreint
  aux $\lambda$-anneaux : si on considère
  $U = \Z[x_1,x_2,\dots]$ le $\lambda$-anneau libre sur un élément
  (voir exemple \ref{ex_lambda_libre}), on peut construire $P$
  comme le quotient de $U$ par l'idéal grec engendré par $x_1^2-2x_1$
  et $x_d-x_1$ (pour $d\pgq 2$). Alors l'image de $x_1$ dans $P$
  est l'élément de Pfister universel.

  On peut préciser la structure de $P$ : si on considère $P'$
  le quotient de $U$ par l'idéal engendré par $x_1^2-2x_1$
  et $x_d-x_1$, alors $P$ est en tant qu'anneau un quotient de $P'$.
  Or $P' = \Z \oplus \Z x_1$, donc si $P$ n'est pas isomorphe à $P'$
  on doit avoir $a,b\in \Z$ tels que pour tout élément de Pfister $x$
  dans un $\lambda$-anneau on ait $ax=b$ ; mais l'exemple des formes
  de Pfister dans un $GW(K)$ nous permet de voir que c'est faux. Donc
  $P = \Z \oplus \Z x_1$.
\end{ex}

\begin{ex}\label{ex_repr_pfis}
  De même, le foncteur $\Pf^n$ est représentable sur les
  $\lambda$-anneaux, et l'anneau $P_n$ qui le représente est isomorphe
  à $(\Z\oplus \Z x)^{\otimes n}$.

  Cela signifie que tout $n$-élément de Pfister dans un $\lambda$-anneau
  est l'image d'un élément de cet anneau par un certain morphisme,
  mais qui n'est pas unique, donc on ne peut pas en conclure que
  $\Pf_n$ est représentable. On peut parler d'un $n$-élément
  de Pfister \emph{versel}.
\end{ex}

Pour généraliser la proposition \ref{prop_pi} aux $n$-éléments de Pfister,
on va se restreindre aux $\lambda$-anneaux.

\begin{thm}\label{thm_pi}
  Soit $n\in \N^*$. Il existe une unique lettre grecque $\pi_n\in L(\Z)$
  telle que pour tout $\lambda$-anneau $R$, tout $n$-élément de Pfister
  de $R$ est de $\pi_n$-dimension 1.

  Précisément :
  \[ \pi_n = 1 - \frac{2}{1 + (1+2^nt)^{\frac{1}{2^{n-1}}}}, \]
  soit :
  \[ \pi_{n+1} = \pi_n \circ (tC(-2^{n-1}t)), \quad \pi_n = \pi_{n+1}\circ (t + 2^{n-1}t^2)   \]
  où $C(t)$ est la série génératrice des nombres de Catalan, soit
  $C(t) = \frac{1-\sqrt{1-4t}}{2t}$.
\end{thm}

\begin{proof}
  On vérifie que si $\pi_n$ est donné par la formule ci-dessus, alors
  \[ \pi_n^{-1}= \frac{1}{2^n}\left[ \left( \frac{1+t}{1-t} \right)^{2^{n-1}} -1\right]. \]
  On doit donc montrer que si $\phi$ est un $n$-élément de Pfister
  alors $\lambda_t(\phi) = 1 + \phi \pi_n^{-1}(t)$ avec $\pi_n^{-1}$
  donné ci-dessus.

  On procède par récurrence sur $n$. Pour $n=1$, il s'agit juste
  de la proposition \ref{prop_pi}. On suppose que la propriété
  est vraie jusqu'à $n\pgq 1$. D'après l'exemple \ref{ex_repr_pfis},
  on peut se ramener au cas d'un $n$-élément de Pfister versel dans
  l'anneau $P_n$, qui a la particularité d'être sans torsion. En
  particulier, $P_n$ se plonge dans $R=P_n\otimes_\Z \Q$, et d'après
  \cite[prop 3.50]{Yau}, on peut étendre la structure de $\lambda$-anneau de $P_n$
  à $R$ ; on s'est donc ramené au cas d'un $n$-élément de Pfister
  dans un anneau où $2$ est inversible.

  On peut alors facilement observer par définition d'un élément de
  Pfister que pour tout $n$-élément de Pfister $x$, $(\frac{x}{2^n})^i = \frac{x}{2^n}$,
  donc pour tout $f\in tR[[t]]$,
  \begin{equation}\label{eq_pfis_polyn}
    \frac{x}{2^n}f(t) = f(\frac{x}{2^n}t).
  \end{equation}
  En particulier, si $x$ est un $1$-élément de Pfister, on a 
  \[ \lambda_t(x) =  1+ x\pi_1^{-1}(t) = 1 + 2\pi_1^{-1}(\frac{xt}{2}) = \frac{1+\frac{xt}{2}}{1-\frac{xt}{2}}.   \]

  De là, soit $\phi\in \Pf_{n+1}(R)$, avec $\phi = xy$ où
  $x\in \Pf_1(R)$ et $y\in \Pf_n(R)$. Alors:
  \begin{align*}
    \lambda_t(\phi) &= \lambda_t(xy) \\
                    &= \lambda_t(x)\times_G \lambda_t(y) \\
                    &= \left[ (1+\frac{x}{2}t) -_G (1-\frac{x}{2}t)\right] \times_G (1+y\pi_n^{-1}(t)) \\
                    &= (1 + y\pi_n^{-1}(\frac{xt}{2})) -_G (1 + y\pi_n^{-1}(-\frac{xt}{2})) \\
                    &= \frac{1 + 2^n\pi_n^{-1}(\frac{\phi t}{2^{n+1}})}{1 + 2^n\pi_n^{-1}(-\frac{\phi t}{2^{n+1}})} \\
                    &= 1 + \frac{\phi}{2^{n+1}}\left( \frac{1+2^n\pi_n^{-1}(t)}{1+2^n\pi_n^{-1}(-t)} -1 \right)
  \end{align*}
  où la dernière égalité vient de (\ref{eq_pfis_polyn}) appliqué à
  $f= \frac{1 + 2^n\pi_n^{-1}(t)}{1 + 2^n\pi_n^{-1}(-t)}-1$. Or
  \begin{align*}
    \frac{1+2^n\pi_n^{-1}(t)}{1+2^n\pi_n^{-1}(-t)} &= \frac{\left( \frac{1+t}{1-t} \right)^{2^{n-1}}}{\left( \frac{1-t}{1+t} \right)^{2^{n-1}}} \\
                                                   &= \left( \frac{1+t}{1-t} \right)^{2^n}
  \end{align*}
  donc on trouve bien $\lambda_t(\phi) = 1 +\phi\pi_{n+1}^{-1}(t)$.

  Il faut ensuite vérifier que les formules de récurrence donnent
  bien la bonne formule pour $\pi_n$. Tout d'abord, on montre
  facilement que $tC(-2^{n-1}t)$ et $t + 2^{n-1}t^2$ sont
  réciproques l'une de l'autre, puis :
  \begin{align*}
    \pi_{n+1}\circ (t + 2^{n-1}t^2) &= 1 - \frac{2}{1 + (1+2^{n+1}(t + 2^{n-1}t^2))^{\frac{1}{2^n}}} \\
                                    &= 1 - \frac{2}{1 + ((1+2^nt)^2)^{\frac{1}{2^n}}} \\
                                    &= \pi_n.
  \end{align*}
\end{proof}

\begin{coro}\label{cor_relat_pi}
  On a les relations explicites :
  \[ \pi_1^d = \sum_{k=1}^d (-1)^{d-k}\binom{d-1}{k-1}\lambda^k \]
  \[ \pi_{n+1}^d = \sum_{k=1}^d (-1)^{d-k} 2^{(d-k)(n-1)}\frac{k}{d}\binom{2d-k-1}{d-1}\pi_n^k \]
  \[  \pi_n^d = \sum_{\frac{d}{2}\ppq k\ppq d} \binom{k}{d-k}2^{(d-k)(n-1)}\pi_{n+1}^k.    \]
\end{coro}

\begin{proof}
  Par construction, si $\alpha = \beta\circ \tau$, alors
  si $\tau^k = \sum_da_{k,d}t^d$, on a $\alpha^d = \sum_k a_{k,d}\beta^k$.

  Les équations ci-dessus se traduisent donc par

  \[ \left( \frac{t}{1+t} \right)^k = \sum_{d\pgq k}(-1)^{d-k}\binom{d-1}{k-1}t^d, \]
  \[  \left( tC(-2^{n-1}t) \right)^k = \sum_{d\pgq k}(-1)^{d-k}2^{(d-k)(n-1)}\frac{k}{d}\binom{2d-k-1}{d-1}t^d, \]
  et
  \[  (t+2^{n-1}t^2)^k = \sum_{k\ppq d\ppq 2k}\binom{k}{d-k}2^{(d-k)(n-1)}t^d. \]
  La première vient directement du développement en série de $(1+t)^{-n}$
  et la troisième du binôme de Newton. On peut se concentrer sur la deuxième,
  qui se réécrit
  \[  C(t)^k = \sum_{d\pgq 0} \frac{k}{d+k}\binom{2d+k-1}{d+k-1}t^d.  \]
  On la montre sans difficulté par récurrence en utilisant l'équation
  fonctionnelle $C(t)=1+tC(t)^2$, qui donne $C(t)^{k+1}=\frac{C(t)^k-C(t)^{k-1}}{t}$.
\end{proof}

\begin{coro}\label{cor_pi_simil}
  Soit $R$ un $\lambda$-anneau. Soient $x\in \Pf_n(R)$ et $a\in R$
  de $\lambda$-dimension 1 tel que $a^2=1$. Alors si $d\pgq 2$ :
  \[ \pi_n^d(ax) = (-1)^d 2^{n(d-1)-1}(1-a)x. \]
\end{coro}

\begin{proof}
  On a $\lambda_t(ax) = (1+at)\times_G \lambda_t(x) = 1+x\pi_n^{-1}(at)$,
  donc en faisant agir $\pi_n$ à gauche $(\pi_n)_t(ax) = 1+x\pi_n^{-1}(a\pi_n(t))$. On doit donc montrer
  \[ \pi_n^{-1}(a\pi_n(t)) = at + \sum_{d\pgq 2}(-1)^d2^{n(d-1)-1}(1-a)t^d.  \]
  On montre en utilisant $(1-a)^m=2^{m-1}(1-a)$ et $a(1-a)^n=-(1-a)^n$ que le membre
  de droite donne $\frac{at}{1+2^{n-1}(1-a)t}$.

  Pour le membre de gauche : on a
  \begin{align*}
    \frac{1+a\pi_n(t)}{1-a\pi_n(t)} &=  \frac{  1+a\frac{  (1+2^nt)^{\frac{1}{2^{n-1}}}-1  }{(1+2^nt)^{\frac{1}{2^{n-1}}}+1}  }
                                      { 1-a\frac{  (1+2^nt)^{\frac{1}{2^{n-1}}}-1  }{(1+2^nt)^{\frac{1}{2^{n-1}}}+1}} \\
                                    &= \frac{(1+2^nt)^{\frac{1}{2^{n-1}}}(1+a) + (1-a)}{(1+2^nt)^{\frac{1}{2^{n-1}}}(1-a)+(1+a)}.
  \end{align*}
  Or comme $(1+a)(1-a)=0$, on a $(y(1+a)+(1-a))^m = 2^{m-1}(y^m(1+a)+(1-a))$
  pour tous $y\in R$ et $m\in \N^*$, et de même en échangeant les $1+a$ et $1-a$.
  De là :
  \begin{align*}
    \pi_n^{-1}(a\pi_n(t)) &= \frac{1}{2^n}\left( \frac{(1+2^nt)(1+a) + (1-a)}{(1+2^nt)(1-a)+(1+a)}-1\right) \\
                          &= \frac{1}{2^n}\left( \frac{2+2^n(1+a)t}{2+2^n(1-a)t}-1\right) \\
                          &= \frac{at}{1+2^{n-1}(1-a)t}
  \end{align*}
  comme souhaité.
\end{proof}

\section{Invariants de $I^n$}

Dans cette partie on étudie les invariants de $I^n$ à valeurs dans
l'anneau de Witt $W(K)$ (invariants de Witt), et dans $H^*(K,\mu_2)$
(invariants cohomologiques).

Il s'avère que les résultats sont extrêmement proches dans les deux
cas, mais qu'il est délicat de passer de l'un à l'autre directement.
On choisit donc d'opter pour une approche hybride, qui englobe les
deux types d'invariants, en étudiant des invariants à valeurs dans
un anneau $A(K)$ qui vérifie un certain nombre de propriétés formelles
qui sont exactement celles qui font fonctionner les résultats. La
généralisation est un peu artificielle dans la mesure où il n'est pas
évident que des exemples intéressants autres que l'anneau de Witt
et l'anneau de cohomologie soient couvert par ces définitions, mais
l'artifice a le mérite d'une certaine simplicité et lisibilité, et
évite à tout le moins de répéter des arguments presque identiques
pour les deux cas qui nous intéressent.

\subsection{Contexte}\label{sec_contexte}

\label{par_a}On se fixe donc dans cette toute partie un foncteur
$A: \mathbf{Field}_{/k}\To \mathbf{FRing}$, où $\mathbf{FRing}$ est
la catégorie des anneaux filtrés. On a donc une filtration
$A(K)=A^0(K)\supset A^1(K)\supset \dots$, naturelle en $K$,
qu'on suppose par ailleurs séparée, ie $\bigcap_n A^n(K)=0$.

\label{par_f}On suppose également qu'on dispose d'une application naturelle
$f_1:I^1(K)\To A^1(K)$ vérifiant les conditions suivantes :
\begin{enumerate}[label=(\roman*)]
\item l'association
  \[ \pfis{a_1,\dots,a_n}\mapsto \{a_1,\dots,a_n\} := f_1(\pfis{a_1})\cdots f_1(\pfis{a_n}) \]
  définit une fonction injective $\Pf_n(K)\To A^n(K)$ ;
\item pour tout $n\in \N^*$, cette fonction se prolonge en un morphisme de groupes
  $f_n:I^n(K)\To A^n(K)$ ;
\item \label{cond_3} si $x\in A(K)$ vérifie $\{a\}\cdot x\in A^{d+1}(L)$ pour tout $a\in L^*$,
  pour toute extension $L/K$, alors $x\in A^d(K)$ ;
\item \label{cond_4} $\Inv(\Pf_n,A) = A(K)\oplus A(K)\cdot f_n$ où on considère
  les invariants définis sur $K$.
\end{enumerate}

Notons que la condition \ref{cond_3} est équivalente à :
\begin{equation}
   \forall n\in \N^*,\,  \left( \forall L/K,\forall \phi\in \Pf_n(K), f_n(\phi)\cdot x\in A^{d+n}(L) \right) \impl x\in A^d(K)
\end{equation}
et en particulier par séparation implique
\begin{equation}\label{eq_f_gen}
 \forall n\in \N^*,\,  \left( \forall L/K,\forall \phi\in \Pf_n(K), f_n(\phi)\cdot x = 0 \right) \impl x=0. 
\end{equation}

On peut également remarquer que par construction, pour tous
$x\in I^n(K)$, $y\in I^m(K)$, on a :
\begin{equation}\label{eq_f_prod}
  f_n(x)\cdot f_m(y) = f_{n+m}(xy)
\end{equation}

\begin{ex}
  Le premier exemple, tautologique, est $A(K)=W(K)$ (on notera $A=W$), muni de la filtration
  fondamentale $A^n(K)= I^n(K)$, et de $f_1=\Id$ (donc $f_n=\Id$). Les deux
  premières propriétés sont évidentes, la troisième se montre en prenant
  pour $a$ un élément générique, et la quatrième est un résultat de Serre
  (\cite[ex 27.17]{GMS}).
  Dans cet exemple, $\{a_1,\dots,a_n\} = \pfis{a_1,\dots,a_n}$.
\end{ex}

\begin{ex}
  Le deuxième exemple est $A(K)=H^*(K,\mu_2)$ (on notera $A=H$), muni de la filtration
  $A^n(K)=\bigoplus_{d\pgq n}H^n(K,\mu_2)$, et de $f_1=e_1$ (et donc
  $f_n = e_n$). Les deux premières propriétés découlent de la
  conjecture de Milnor, la troisième se démontre encore en prenant
  un élément générique, et la quatrième est encore un résultat de Serre
  (\cite[thm 18.1]{GMS}).
  Dans cet exemple, $\{a_1,\dots,a_n\} = (a_1,\dots,a_n)$.
\end{ex}

\begin{rem}
  En revanche, $A(K)=GW(K)$ ne convient pas : la condition \ref{cond_3}
  n'est pas remplie. En effet, si $x\in GW(K)$ vérifie $x\pfis{a}=0$
  pour tout $a\in L^*$, on peut seulement conclure que
  $x$ est hyperbolique, et non $x=0$, ce qui est en contradiction
  avec (\ref{eq_f_gen}).
\end{rem}

\subsection{Les invariants $f_n^d$}

On met en application le théorème \ref{thm_pi}, dans le cas du
$\lambda$-anneau $GW(K)$. On dispose donc pour tous $n\in \N^*$
et $d\in \N$ d'applications $\pi_n^d:GW(K)\To GW(K)$, naturelles en $K$.

\begin{defi}\label{def_f}
  Pour tous $n,d\in \N^*$, on définit $f_n^d\in \Inv^{nd}(I^n,A)$
  par $f_n^d = f_{nd}\circ \pi_n^d$. On pose aussi $f_n^0=1$.
\end{defi}

\begin{ex}\label{ex_pi_u}
  Si $A(K)=W(K)$, alors $f_n^d=\pi_n^d$ (restreint à $I^n(K)$). Si
  $A(K)=H^*(K,\mu_2)$, alors on notera $u_{nd}^{(n)}=f_n^d=e_{nd} \circ \pi_n^d$.

  En ce qui concerne la notation, on aurait pu choisir d'utiliser
  $u_n^d$ pour être plus cohérent avec les autres notations, mais on
  a choisi de suivre la tradition de noter le degré d'un invariant
  cohomologique en indice. L'exposant sert alors à distinguer les
  invariants de même degré, par exemple $u^{(2)}_6$ et $u^{(3)}_6$,
  définis respectivement sur $I^2$ et $I^3$ (et dont l'un n'est
  absolument pas la restriction de l'autre, par ailleurs).
\end{ex}

Les propriétés de base de ces invariants sont les suivantes :

\begin{prop}\label{prop_f}
  Pour tout $n\in \N^*$, on a :
  \begin{enumerate}[label=(\roman*)]
  \item \label{cond_f_ini} $f_n^0 = 1$ et $f_n^1=f_n$ ;
  \item \label{cond_f_som} pour tous $q,q'\in I^n(K)$ :
    \[ f_n^d(q+q') = \sum_{k=0}^d f_n^k(q)\cdot f_n^{d-k}(q') ; \]
  \item \label{cond_f_pfis}  pour tous $\phi\in Pf_n(K)$ et $d\pgq 2$,
    $f_n^d(\phi)=0$.
  \end{enumerate}
  En particulier, $f_n^d(I^n(K))\subset A^{nd}(K)$.
\end{prop}

\begin{proof}
  Les trois propriétés sont une reformulation du fait que
  $\pi_n$ est une opération grecque telle que les $n$-formes de Pfister
  sont de dimension 1 (voir \ref{thm_pi}), en utilisant la formule (\ref{eq_f_prod}).

  Pour la dernière assertion, par la formule \ref{cond_f_som}, il
  suffit de montrer que $f_n^d(x)\in A^{nd}(K)$ si $x$ est une
  $n$-forme de Pfister ou l'opposé d'une telle forme. Or pour une
  forme de Pfister cela découle directement des formules ci-dessus,
  et pour l'opposé d'une $n$-forme de Pfister on peut utiliser
  le corollaire \ref{cor_pi_simil} avec $a=-1$.
\end{proof}

\begin{prop}\label{prop_f_simil_pfis}
  Si $\phi\in Pf_n(K)$, $a\in K^*$, $d\pgq 2$, alors
  \[ f_n^d(\fdiag{a}\phi) = (-1)^d \{-1\}^{n(d-1)-1}\{a\}f_n(\phi). \]
\end{prop}

\begin{proof}
  C'est une application directe du corollaire \ref{cor_pi_simil}.
\end{proof}

\begin{prop}\label{prop_somme_pfis}
  Si $q = \sum_{i=1}^r \phi_i \in I^n(K)$ où $\phi_i$ est une $n$-forme de Pfister, 
  alors pour tout $d\pgq1$ :
  \[ f_n^d(q) = \sum_{1\ppq i_1<\dots <i_d\ppq r} f_n(\phi_{i_1})\cdots f_n(\phi_{i_d}), \]
  et en particulier $f_n^d(q)=0$ si $q$ est une somme de $k$ $n$-formes de
  Pfister avec $k<d$.
\end{prop}

\begin{proof}
  C'est une conséquence directe de la proposition \ref{prop_somme_dim_1}.
\end{proof}

\begin{coro}\label{cor_moinsun}
  Si $-1$ est un carré dans $K$, alors pour tout $q\in I^n(K)$,
  $f_n^d(q)=0$ pour $d$ assez grand.
\end{coro}

\begin{proof}
  Quand $-1$ est un carré, tout $q\in I^n$ est une somme de $n$-formes de Pfister, 
  donc ça découle de la proposition.
\end{proof}

\begin{rem}
  On peut montrer que, pour $K/k$ fixé, $(f_n^d)$ est la seule famille
  d'applications $I^n(K)\To A(K)$ à vérifier les conditions de la proposition \ref{prop_f}.
  En effet, toute famille $(a_n^d)$ qui vérifie \ref{cond_f_ini} et \ref{cond_f_som}
  doit satisfaire les conclusions de la proposition \ref{prop_somme_pfis},
  donc ses valeurs sont fixées sur les sommes de formes de Pfister. En
  développant $a_n^d(q-q)$ pour un tel $q$ avec \ref{cond_f_som} on trouve
  par récurrence que $a_n^d(-q)$ est aussi déterminé, et donc par somme
  les valeurs sont déterminées sur un élément quelconque.
\end{rem}

On présente une application de la proposition \ref{prop_somme_pfis},
qui utilise le cas $A(K)=W(K)$ :

\begin{prop}
  Si toute forme dans $I^n(K)$ peut s'écrire comme somme d'au plus $r$
  $n$-formes de Pfister pour un certain $r\in \N^*$,
  alors $I^d(K)=0$ pour tout $d\pgq n(r+1)$.
\end{prop}

\begin{proof}
  Tout d'abord il suffit de montrer que $I^{n(r+1)}(K)=0$, puisque
  de façon générale si $I^d(K)=0$ alors $I^{d+1}(K)=0$.
  
  Il suffit donc de montrer que toute $n(r+1)$-forme de Pfister
  est nulle. Or une telle forme s'écrit $\phi=\phi_1\cdots \phi_{r+1}$
  où les $\phi_i$ sont des $n$-formes de Pfister.

  De là, $\phi = \pi_n^{r+1}(q)$ avec $q = \phi_1+\cdots +\phi_{r+1}$. Or
  par hypothèse $q$ peut s'écrire comme une somme d'au plus $r$
  $n$-formes de Pfister, donc d'après la proposition \ref{prop_somme_pfis}
  $\pi_n^{r+1}(q)=0$.
\end{proof}

\begin{rem}
  Cette proposition peut être démontrée indépendamment en
  utilisant le Hauptsatz de Pfister.
\end{rem}

\subsection{L'opérateur de décalage}

\label{par_m}Pour alléger les notations, on se fixe un $n\in \N^*$, et
on pose $M = \Inv(I^n, A)$ ainsi que $M^d = \Inv(I^n, A^d)$.

\begin{propdef}\label{prop_def_delta}
  Soit $\eps=\pm 1$. Il existe un unique morphisme de $A(k)$-modules filtrés,
  de degré $-n$,
\[ \begin{foncdef}{\Phi^\eps}{M}{M}{\alpha}{\alpha^\eps,} \end{foncdef} \]
tel que
\begin{equation}\label{eq_alpha_pm1}
\alpha(q + \eps\phi) = \alpha(q) + \eps f_n(\phi)\cdot \alpha^\eps(q)
\end{equation}
pour tous $\alpha\in M$, $q\in I^n(K)$ et $\phi\in \Pf_n(K)$.
\end{propdef}

\begin{proof}
  Soient $\alpha\in M$ et $q\in I^n(K)$. Pour toute
  extension $L/K$ et tout $\phi\in \Pf_n(L)$, on pose
  \[ \beta_q(\phi) = \alpha(q + \eps\phi). \]

  Alors $\beta_q\in \Inv(\Pf_n,A)$, défini sur $K$. D'après
  la condition \ref{cond_4} sur $A$, il existe d'uniques $x_q,y_q\in A(K)$
  tels que $\beta_q = x_q + y_q\cdot f_n$.

  En prenant $\phi=0$ on voit que $x_q = \alpha(q)$, et on pose
  alors $\alpha^\eps(q) = \eps y_q$, ce qui nous donne bien la formule voulue,
  ainsi que l'unicité.

  Ainsi défini, $\Phi^\eps$ est clairement un morphisme de $A(k)$-modules,
  et il est bien de degré $-n$ car si $\alpha\in M^d$, alors pour tout
  $q\in I^n(K)$, $f_n(\phi)\cdot \alpha^\eps(q)\in A^d(L)$ pour tout $\phi\in Pf_n(L)$
  et toute extension $L/K$, donc $\alpha^\eps(q)\in A^{d-n}(K)$ par la condition
  \ref{cond_3} sur $A$.
\end{proof}

\begin{rem}
  Les opérateurs $\Phi^\eps$ dépendent de $n$, mais on le
  passe sous silence dans la notation, aucune confusion
  n'étant susceptible d'advenir.
\end{rem}

On a des liens naturels entre $\Phi^+$ et $\Phi^-$, traduisant le fait
que l'addition et la soustraction commutent, et qu'elles
sont inverses l'une de l'autre :

\begin{prop}\label{prop_phi_pm}
  Les opérateurs $\Phi^+$ et $\Phi^-$ commutent, et de plus
  pour tout $\alpha\in \hat{M}$ on a
  \[ \alpha^+ - \alpha^- = \{-1\}^n \alpha^{+-} = \{-1\}^n \alpha^{-+}. \]
\end{prop}

\begin{proof}
  Soient $q\in I^n(K)$, et $\phi,\psi\in \Pf_n(L)$.
  On a
  \begin{align*}
    \alpha(q+\phi-\psi) &= \alpha(q+\phi) - f_n(\psi) \alpha^-(q+\phi) \\
                        &= \alpha(q) + f_n(\phi)\alpha^+(q) -f_n(\psi)\alpha^-(q)
                          -f_n(\phi)f_n(\psi)\alpha^{-+}(q)
  \end{align*}
  mais aussi
  \begin{align*}
    \alpha(q+\phi-\psi) &= \alpha(q-\psi) + f_n(\phi) \alpha^+(q-\psi) \\
                        &= \alpha(q) - f_n(\psi)\alpha^-(q) +f_n(\phi)\alpha^+(q)
                          -f_n(\phi)f_n(\psi)\alpha^{+-}(q)
  \end{align*}
  donc $f_n(\phi)f_n(\psi)\alpha^{-+}(q) = f_n(\phi)f_n(\psi)\alpha^{+-}(q)$,
  et comme c'est vrai pour tous $\phi$, $\psi$ sur toute extension,
  par la condition (\ref{eq_f_gen}) on a $\alpha^{+-}=\alpha^{-+}$.

  Si maintenant on prend $\phi=\psi$, la formule ci-dessus donne
  \[ f_n(\phi)\alpha^+(q) - f_n(\phi)\alpha^-(q) = f_n(\phi)f_n(\phi)\alpha^{+-}(q) \]
  ce qui permet de conclure en utilisant $f_n(\phi)f_n(\phi) = \{-1\}^nf_n(\phi)$
  et encore la condition (\ref{eq_f_gen}).
\end{proof}

Comme les opérateurs commutent, pour tous $s,t\in \N$ et tout $\alpha\in M$
on définit
\begin{equation}\label{eq_alpha_s_t}
  \alpha^{s+,t-}=(\Phi^+)^s\circ (\Phi^-)^t(\alpha),
\end{equation}
qui ne dépend en réalité pas de l'ordre d'application des opérateurs.

On appelle $\Phi=\Phi^+$ l'\emph{opérateur de décalage}, comme le justifie
le résultat élémentaire suivant :

\begin{prop}\label{prop_phi_pi}
Pour tout $d\in \N$, $\Phi(f_n^{d+1})=f_n^d$ (et $\Phi(f_n^0)=0$).
\end{prop}

\begin{proof}
  Il faut montrer que $f_n^{d+1}(q+\phi)=f_n^{d+1}(q) + f_n(\phi)\cdot f_n^d(q)$,
  ce qui est une conséquence immédiate de la proposition \ref{prop_f}.
\end{proof}

L'action de $\Phi^-$ sur les $f_n^d$ est plus compliquée, reflétant
ainsi le fait que les $f_n^d$ se comportent bien mieux vis-à-vis
des sommes des formes de Pfister que des soustractions de telles
formes.

\begin{prop}\label{prop_fn_phi_moins}
  Soit $d\in \N^*$. Alors
  \[  (f_n^d)^- = \sum_{k=0}^{d-1} (-1)^{d-k-1} \{-1\}^{n(d-k-1)} f_n^k. \]
\end{prop}

\begin{proof}
  Soient $q\in I^n(K)$ et $\phi\in \Pf_n(K)$. Alors
  \begin{align*}
    f_n^d(q-\phi) &= \sum_{k=0}^df_n^k(q)f_n^{d-k}(-\phi) \\
                  &= f_n^d(q) + \sum_{k=0}^{d-1} (-1)^{d-k}\{-1\}^{n(d-k-1)}f_n(\phi)f_n^k(q)
  \end{align*}
  en utilisant la proposition \ref{prop_f_simil_pfis} avec $a=-1$.
\end{proof}

Outre son comportement vis-à-vis des invariants $f_n^d$, la principale
propriété de $\Phi$ est la suivante :

\begin{prop}\label{prop_ker_phi}
  Le morphisme $\Phi^\eps$ induit un morphisme $M/M^{d+n}\to M/M^d$ pour tout
  $d\in \N$, qui donne la suite exacte suivante :
  \[ 0\To A(k)/A^{d+n}(k) \To M/M^{d+n}\To M/M^d. \]
  En particulier, le noyau de $\Phi^\eps$ est le sous-module des invariants
  constants.
\end{prop}

\begin{proof}
  Si $\alpha,\beta\in M$ sont congrus modulo $M^{d+n}$, alors comme
  \[ \Phi^\eps(M^{d+n})\subset M^d, \] $\Phi^\eps(\alpha)$ et $\Phi^\eps(\beta)$ sont
  congrus modulo $M^d$, ce qui montre la première affirmation.

  Soit $\alpha\in M$ tel que $\alpha^\eps\in M^d$. Alors pour tout $q\in I^n(K)$
  et tout $\phi\in \Pf_n(K)$, on a $\alpha(q+\eps\phi)\equiv \alpha(q)$ modulo
  $A^{n+d}(K)$.
  Donc en écrivant $q=q_1-q_2$ où les $q_i$ sont des sommes de $n$-formes
  de Pfister, on voit par une simple récurrence sur le nombre de termes
  dans chaque somme que $\alpha \equiv \alpha(0)$ modulo $M^{d+n}$
  (où $\alpha(0)$ est vu comme invariant constant).

  Pour la dernière assertion, en prenant $d$ assez grand,
  et comme la filtration sur $A(K)$ est séparée,
  on voit que si $\alpha^\eps=0$ alors $\alpha = \alpha(0)$.
\end{proof}

\begin{rem}
  On peut montrer directement que $\Ker(\Phi)=A(k)$ avec le
  même schéma de preuve, mais la version filtrée, plus fine,
  nous sera utile dans la suite.
\end{rem}

\begin{coro}\label{cor_phi_exact}
  Soit $M'$ le sous-module de $M$ engendré par les $f_n^d$
  pour $d\in \N$. Alors $\Phi^\eps$ induit une suite exacte
  \[ 0 \to A(k) \To M' \To M'[-n] \To 0. \]
\end{coro}

\begin{proof}
  Seule la surjectivité reste à montrer, mais on le déduit facilement
  des propositions \ref{prop_phi_pi} pour $\Phi^+$, et \ref{prop_fn_phi_moins}
  pour $\Phi^-$.
\end{proof}

\begin{rem}\label{rem_shift}
  Tout ceci implique que $\Phi$ peut être vu comme une sorte d'opérateur différentiel :
  si on connaît $\alpha^+$ pour un certain invariant $\alpha$, on peut \og intégrer\fg{}
  pour trouver $\alpha$, avec une certaine constante d'intégration.
  Précisément, on montrera dans la partie suivante qu'on peut toujours écrire
  $\alpha^+ = \sum a_d f_n^d$ ; alors $\alpha = \alpha(0)+\sum a_d f_n^{d+1}$.

  On utilisera très souvent cette méthode pour calculer $\alpha$ par
  récurrence en passant par $\alpha^+$, ce qu'on qualifiera
  de \og{}récurrence sur le décalage\fg{}.
\end{rem}

On peut tirer une conséquence élémentaire de ce qui précède :

\begin{coro}\label{cor_additif}
  Soit $\alpha\in M$ additif (c'est-à-dire que $\alpha(q+q')=\alpha(q)+\alpha(q')$
  pour tous $q,q'$). Alors $\alpha = af_n$ pour un unique $a\in A(k)$.
\end{coro}

\begin{proof}
  Pour tous $q\in I^n(K)$, $\phi\in \Pf_n(K)$, on a
  $\alpha(q+\phi)=\alpha(q)+\alpha(\phi)$ donc
  $\alpha(\phi)=f_n(\phi)\alpha^+(q)$. En particulier,
  $f_n(\phi)\alpha^+(q)=f_n(\phi)\alpha^+(q')$ pour tout $q'$,
  et comme c'est vrai pour tout $\phi$ sur toute extension,
  $\alpha^+$ est un invariant constant, disons $a\in A(k)$.
  Alors $\alpha^+ = (af_n)^+$, et comme par additivité $\alpha(0)=0$,
  on a bien $\alpha=af_n$. Comme $a=\alpha^+(0)$, il est uniquement
  déterminé.
\end{proof}

\subsection{Classification des invariants}

On rappelle qu'on s'est donné $n\in \N^*$, et qu'on a posé
$M = \Inv(I^n,A)$ et $M^d = \Inv(I^n,A^d)$.

On souhaite montrer que tout élément de $M$ peut s'écrire de façon
unique comme une combinaison $\sum a_d f_n^d$, ce qui est suggéré
par :

\begin{prop}\label{prop_pi_gen}
  Le $A(k)/A^d(k)$-module $M/M^d$ est engendré  
  par les $f_n^k$ avec $nk<d$.
\end{prop}

\begin{proof}
  On raisonne par récurrence sur $d$. Pour $d=0$, c'est trivial
  puisque $M^0=M$. Supposons que la propriété tienne jusqu'à
  $d-1$, et soit $\alpha\in M$ ; on note $\bar{\alpha}\in M/M^d$.
  Par hypothèse de récurrence, $\Phi(\bar{\alpha}) = \sum a_k f_n^k$
  avec $nk<d-n$, donc si on pose $\beta = \alpha -\sum a_k f_n^{k+1}$
  on a $\Phi(\bar{\beta})=0$. De là, d'après la proposition \ref{prop_ker_phi},
  $\beta$ est congru modulo $M^d$
  à un invariant constant $a_{-1}$, d'où $\bar{\alpha} = \sum a_{k-1} f_n^k$
  avec $nk<d$.
\end{proof}

\begin{rem}
  Cela montre au passage que la suite exacte de la proposition
  \ref{prop_ker_phi} peut être complétée à droite par un 0, puisque
  tous les $f_n^d$ sont dans l'image des $\Phi^\eps$ (même argument
  que pour le corollaire \ref{cor_phi_exact}).
\end{rem}

Le problème est que pour exprimer un invariant en fonction des $f_n^d$
il est en général nécessaire d'utiliser une combinaison infinie, comme
le souligne l'exemple suivant.

\begin{ex}\label{ex_disc}
  On se place dans le cas $A=W$. Soit $\alpha(q) = \fdiag{\disc(q)}$ ;
  c'est un invariant de Witt de $I$. Alors $\alpha^+ = -\alpha$ ; en effet:
  \[ \fdiag{\disc(q + \pfis{a})} = \fdiag{\disc(q)a} = \fdiag{\disc(q)} -\pfis{a}\fdiag{\disc(q)}. \]
  De là, $\alpha$ ne peut s'écrire comme une combinaison finie
  des $f_1^d$, puisque la longueur d'une telle combinaison diminue
  strictement quand on applique $\Phi^+$. En revanche, on peut imaginer
  écrire, au moins formellement pour l'instant (mais les résultats
  de cette partie permettent de montrer que c'est effectivement valide)
  \[ \alpha = \sum_{d\in \N}(-1)^d f_1^d. \] 
\end{ex}

Cependant, une telle combinaison infinie pourrait ne pas toujours
être bien définie. Comme $f_n^d$ est à valeur dans des rangs strictement
croissants de la filtration de $A$, toute somme $\sum_{d\in \N}a_d f_n^d$
définit bien un invariant à valeurs dans la complétion de $A$, mais
a priori pas dans $A$ lui-même.

\begin{ex}\label{ex_moins_pfis}
  Si $k$ est formellement réel, alors $\sum_d f_1^d$
  envoie $-\pfis{-1}$ sur $\sum_{d\in \N} (-1)^d \{-1\}^d$,
  qui n'est pas un élément de $A(k)$ pour $A=W$ ou $A=H$.
\end{ex}

Comme cet exemple le souligne, le problème est le mauvais
comportement des $f_n^d$ avec les opposés de formes de Pfister
(en effet, toute somme $\sum_{d\in \N}a_d f_n^d$ est en revanche
à valeur dans $A$ si on la restreint aux \emph{sommes} de
formes de Pfiser).
Pour obtenir une description satisfaisante de $M$, on va
introduire une nouvelle famille, dont la définition équilibre
les sommes et les différences de formes de Pfister, et on
montrera que toute combinaison, même infinie, des éléments
de cette famille est cette fois à valeur dans $A$.

\begin{defi}\label{def_gn}
  Pour tout $d\in \N$, on définit  $g_n^d\in \Inv(I^n,A^{nd})$ par :
  \begin{itemize}
    \bitem $g_n^0 = 1$ ;
    \bitem si $d\in \N$ est pair, $(g_n^{d+1})^-=g_n^d$
    et $g_n^{d+1}(0)=0$ ;
    \bitem si $d\in \N$ est impair, $(g_n^{d+1})^+=g_n^d$
    et $g_n^{d+1}(0)=0$.
  \end{itemize}
\end{defi}

Le corollaire \ref{cor_phi_exact} montre que ces conditions
définissent bien les $g_n^d$ de façon unique.
Notons par ailleurs que $g_n^1=f_n$.

\begin{rem}
  On observe que si $-1$ est un carré dans $k$, alors
  $g_n^d = f_n^d$. C'est cohérent avec le fait que si $-1$
  est un carré, tout élément de $I^n(K)$ est une somme de
  formes de Pfister, donc les problèmes soulevés précédemment
  n'apparaissent pas.
\end{rem}

\begin{ex}\label{ex_def_v}
  Dans le cas où $A(K)=H^*(K,\mu_2)$, on notera $g_n^d = v^{(n)}_{nd}$
  pour garder une cohérence avec notre façon de noter les invariants
  cohomologiques.
\end{ex}

Cette définition, qui équilibre $\Phi$ et $\Phi^-$, donne
un comportement raisonnable pour les deux opérateurs à la fois :

\begin{prop}\label{prop_g_pm}
  Soit $d\in \N$. Alors :
  \[ (g_n^{d+2})^{+-} = (g_n^{d+2})^{-+} = g_n^d ; \]
  \[ (g_n^{d+1})^+ = \left\{ \begin{array}{lc} g_n^d & \text{si $d$ impair} \\
             g_n^d + \{-1\}^n g_n^{d-1} & \text{si $d$ pair ;} \end{array} \right. \]
  \[ (g_n^{d+1})^- = \left\{ \begin{array}{lc} g_n^d & \text{si $d$ pair} \\
             g_n^d - \{-1\}^n g_n^{d-1} & \text{si $d$ impair.} \end{array} \right. \]
\end{prop}

\begin{proof}
  D'après la définition \ref{def_gn}, si $d$ est impair, alors $(g_n^{d+2})^-=g_n^{d+1}$ et $(g_n^{d+1})^+=g_n^d$,
  et si $d$ est pair, $(g_n^{d+2})^+=g_n^{d+1}$ et $(g_n^{d+1})^-=g_n^d$.
  Dans tous les cas on vérifie la première formule, puisque les deux
  opérateurs commutent.

  Pour les deux suivantes, on utilise $(g_n^{d+1})^+ - (g_n^{d+1})^- = \{-1\}^n g_n^{d-1}$
  qui vient de la proposition \ref{prop_phi_pm}. On conclut en distinguant
  selon la parité de $d$.
\end{proof}

Le lien entre les $f_n^d$ et $g_n^d$ est donné par :

\begin{prop}\label{prop_f_g}
  Pour tout $d\in \N^*$, on a :
  \begin{align*}
    g_n^d &= \sum_{k=\lfloor \frac{d}{2}\rfloor +1}^d \binom{\lfloor \frac{d-1}{2} \rfloor}{k-\lfloor \frac{d}{2}\rfloor -1} \{-1\}^{n(d-k)} f_n^k \\
    f_n^d &= \sum_{k=1}^d (-1)^{d-k}\binom{d-\lfloor \frac{k+1}{2}\rfloor -1}{\lfloor \frac{k}{2}\rfloor -1}\{-1\}^{n(d-k)} g_n^k.
  \end{align*}
\end{prop}

\begin{proof}
  Notons $\alpha_d$ les invariants définis par la partie droite
  de la première formule.
  Si $d=2m$, la formule s'écrit
  \[ \alpha_d = \sum_{k=m+1}^{2m} \binom{m-1}{k-m-1} \{-1\}^{n(2m-k)} f_n^k\]
  ce qui donne
  \[ \alpha_d^+ = \sum_{k=m+1}^{2m} \binom{m-1}{k-m-1} \{-1\}^{n(2m-k)} f_n^{k-1},\]
  et si $d=2m+1$ alors la formule devient
  \[ \alpha_d = \sum_{k=m+1}^{2m+1} \binom{m}{k-m-1} \{-1\}^{n(2m+1-k)} f_n^k\]
  d'où
  \[ \alpha_d^+ = \sum_{k=m+1}^{2m+1} \binom{m}{k-m-1} \{-1\}^{2m+1-k} f_n^{k-1}.\]
  Il faut donc vérifier que dans les deux cas on trouve pour $\alpha_{d+1}^+$
  la récurrence définissant les $g_n^d$.
  C'est immédiat si $d=2m+1$, et si $d=2m$ on doit comparer
  \[ \sum_{k=m}^{2m}\binom{m}{k-m}\{-1\}^{n(2m-1)} f_n^k \]
  et
  \[ \sum_{k=m+1}^{2m}\binom{m-1}{k-m-1}\{-1\}^{n(2m-k)}f_n^k +  \sum_{k=m}^{2m-1}\binom{m-1}{k-m-2}\{-1\}^{n(2m-k)}f_n^k \]
  dont on voit facilement qu'ils sont égaux.

  Pour montrer la deuxième formule, on peut soit inverser la première,
  soit procéder de la même façon. On pose $\beta_d$ défini par la
  partie droite de la deuxième formule. On a alors :
  \begin{align*}
    \beta_d^+ &= (-1)^d\sum_m \binom{d-m-1}{m-1}\{-1\}^{n(d-2m)}(g_n^{2m})^+ \\
              &+ (-1)^{d+1}\sum_m \binom{d-m-2}{m-1}\{-1\}^{n(d-2m-1)}(g_n^{2m+1})^+ \\
              &= (-1)^d\sum_m \binom{d-m-1}{m-1}\{-1\}^{n(d-2m)}g_n^{2m-1} \\
              &+ (-1)^{d+1}\sum_m \binom{d-m-2}{m-1}\{-1\}^{n(d-2m-1)}(g_n^{2m}+\{-1\}^ng_n^{2m-1}) \\
              &= (-1)^{d+1}\sum_m \binom{d-m-2}{m-1}\{-1\}^{n(d-2m-1)}g_n^{2m} \\
              &+ (-1)^d\sum_m \left( \binom{d-m-1}{m-1} - \binom{d-m-2}{m-1}\right) \{-1\}^{n(d-2m)}g_n^{2m-1} \\
              &= (-1)^{d-1}\sum_m \binom{d-1-m-1}{m-1}\{-1\}^{n(d-1-2m)}g_n^{2m} \\
              &+ (-1)^{d-1+1}\sum_m \binom{d-m-1}{m-2} \{-1\}^{n(d-2m)}g_n^{2m-1}
  \end{align*}
  ce qui donne bien $\beta_{d-1}$.
\end{proof}

\begin{coro}\label{cor_f_g_gen}
  En particulier, $(f_n^i)_{i\ppq d}$ et $(g_n^i)_{i\ppq d}$ engendrent le même
  sous-module de $M$.
\end{coro}

\begin{proof}
  On voit que la matrice de transition est triangulaire avec des $1$ sur
  la diagonale.
\end{proof}

La conséquence cruciale de l'équilibre des signes dans la définition
des $g_n^d$ est donnée par :

\begin{prop}\label{prop_g_borne}
  Soit $q\in I^n(K)$, qui s'écrit $q=\sum_{i=1}^s \phi_i - \sum_{i=1}^t\psi_i$,
  où $\phi_i,\psi_i\in \Pf_n(K)$. Alors pour tout $d>2\max(s,t)$, $g_n^d(q)=0$.
\end{prop}

\begin{proof}
  On peut ajouter des formes hyperboliques dans une des sommes de sorte que $s=t$.
  On prouve alors l'énoncé par récurrence sur $s$ : si $s=0$ alors $q=0$,
  donc pour $d>0$ on a en effet $g_n^d(q)=0$ par construction.
  
  Si le resultat tient jusqu'à $s-1$ pour un certain $s\in \N^*$, alors on
  écrit $q' = q - \phi_s$ et $q'' = q' + \psi_s$. On a 
  \begin{align*}
    g_n^d(q) &= g_n^d(q') + f_n(\phi_s)(g_n^d)^+(q') \\
             &= g_n^d(q'') -  f_n(\psi_s)(g_n^d)^-(q'') + f_n(\phi_s)(g_n^d)^+(q'') \\
             &- f_n(\phi_s)f_n(\psi_s)(g_n^d)^{+-}(q'').
  \end{align*}
  Or d'après la proposition \ref{prop_g_pm}, $(g_n^d)^-$, $(g_n^d)^+$
  et $(g_n^d)^{+-}$ peuvent tous s'exprimer comme combinaisons des
  $g_n^k$ avec $k\pgq d-2$, donc on peut appliquer l'hypothèse de récurrence
  à $q''$.
\end{proof}

\begin{coro}\label{cor_g_fixed_dim}
  Si $q\in I(K)$ est la classe de Witt d'une forme de dimension $2m$,
  alors $g_1^d(q)=0$ pour tout $d>2m$.
\end{coro}

\begin{proof}
  Si $q=\fdiag{a_1,b_1,\dots,a_m,b_m}$, alors
  $q=\sum_{i=1}^m \pfis{-a_i}-\pfis{b_i}$, ce qui conclut d'après la proposition.
\end{proof}

On peut à présent montrer le théorème central :

\begin{thm}\label{thm_g}
  Soit $N=A(k)^{\N}$, qui est un $A(k)$-module filtré pour la
  filtration $N^m = \ens{(a_d)_{d\in \N}}{a_d\in A^{m-nd}}$.
  
  Les applications suivantes sont des isomorphismes mutuellement réciproques
  de $A(k)$-modules filtrés :
  \[ \isomdef{F}{N}{M}{(a_d)}{\sum_d a_d g_n^d,} \]
  \[ \isomdef{G}{M}{N}{\alpha}{(\alpha^{[d]}(0))_d.} \]
  où on note $\alpha^{[d]} = \alpha^{s+,s-}$ si $d=2s$, et
  $\alpha^{[d]} = \alpha^{(s+1)+,s-}$ si $d=2s+1$.
\end{thm}

\begin{proof}
  Déjà, l'application $F$ est bien définie, puisque d'après la
  proposition \ref{prop_g_borne}, pour tout $q\in I^n(K)$ fixé
  on a $g_n^d(q)=0$ pour $d$ assez grand. Alors $F$ et $G$
  sont clairement des morphismes de modules, et le fait qu'ils
  respectent les filtrations est juste une reformulation du fait
  que $g_n^d$ est à valeurs dans $A^{nd}$, et que $\Phi^\eps$
  est de degré $-n$.

  En utilisant la proposition \ref{prop_g_pm}, on voit facilement
  que si $\alpha = \sum_d a_d g_n^d$, alors $a_{2s}=\alpha^{s+,s-}(0)$
  et $a_{2s+1}=\alpha^{(s+1)+,s-}(0)$. Donc $G\circ F = \Id$.

  On prouve à présent que $G$ est injectif, ce qui achève la preuve
  du théorème. Soient $\alpha\in \Ker(G)$ et $d\in \N$.
  D'après la proposition \ref{prop_pi_gen} et le corollaire
  \ref{cor_f_g_gen}, on voit que $\alpha$ est congru à une
  certaine combinaison $\sum_{nk< d}a_k g_n^k$ modulo $M^d$. De là,
  la suite exacte de la proposition \ref{prop_ker_phi} permet de
  voir que $a_k \equiv \alpha^{[k]}(0)$ modulo $A^{d-nk}(k)$, donc,
  puisque $\alpha^{[k]}(0)=0$, $a_k\in A^{d-nk}(k)$. Cela implique
  que $\sum_{nk\ppq d}a_k g_n^k\in M^d$, et donc $\alpha\in M^d$.
  Comme c'est vrai pour tout $d\in \N$, on peut conclure que
  $\alpha=0$.
\end{proof}

On en déduit aisément le résultat suivant, important
pour les invariants cohomologiques :

\begin{coro}\label{cor_relev}
  Tout invariant de $I^n$ à valeur dans $H^d(K,\mu_2)$
  se relève en un invariant à valeurs dans $I^d(K)$.
\end{coro}

\begin{proof}
  Cela vient simplement du fait que les $u_{nd}^{(n)}$,
  qui sont les $f_n^d$ dans le cas $A=H$, se relèvent
  en $\pi_n^d$.
\end{proof}

\begin{coro}\label{cor_ecr_unique}
  Tout $\alpha\in M$ peut s'écrire de façon unique comme
  $\sum a_d f_n^d$ avec $a_d\in A(k)$, en prenant $a_d=\alpha^{d+}(0)$.
  En particulier, $\alpha$ est déterminé par ses valeurs sur
  les sommes de formes de Pfister.
\end{coro}

\begin{proof}
  En utilisant les formules de la proposition \ref{prop_f_g},
  on voit qu'une combinaison à coefficients dans $A(k)$
  des $g_n^d$ donne encore une telle combinaison des $f_n^d$.
  La formule pour $a_d$ est claire, et comme les
  $\Phi^d(\alpha)(0)$ sont entièrement déterminés par les
  valeurs de $\alpha$ sur les sommes de formes de Pfister,
  on en déduit le résultat.
\end{proof}

\begin{rem}
  Comme on l'a déjà mentionné, si on prend une combinaison
  infinie des $f_n^d$, on trouve en général un invariant
  à valeurs dans $\hat{A}$, la complétion de $A$ vis-à-vis
  de sa filtration. On peut adapter les preuves ci-dessus
  pour montrer que les $\Inv(I^n,\hat{A})$ sont exactement
  les combinaisons $\sum a_d f_n^d$ avec $a_d\in \hat{A}(k)$,
  ou alternativement les combinaisons $\sum a_d g_n^d$, cela
  ne fait plus de différence.

  Une façon de voir : si on observe les formules
  de la proposition \ref{prop_f_g}, on constate que si on
  exprime $\sum a_d g_n^d$ avec $a_d\in A(k)$ en fonction
  des $f_n^d$, on trouve $\sum b_d f_n^d$ avec $b_d\in A(k)$,
  tandis que si on part de $\sum a_d f_n^d$ avec $a_d\in A(k)$,
  on trouve en général $\sum b_d g_n^d$ avec $b_d\in \hat{A}(k)$
  (parce que l'expression de $f_n^d$ en fonction des $g_n^k$ fait
  intervenir des $k$ arbitrairement petits, ce qui n'est pas
  le cas dans l'autre sens).
\end{rem}

\begin{rem}
  Si $k$ n'est pas formellement réel, alors $\{-1\}^d=0$ pour $d$
  assez grand, et donc par la proposition \ref{prop_f_simil_pfis}
  $f_n^d(-\phi)=0$ pour tout $\phi\in \Pf_n(K)$. On voit donc
  que dans ce cas, pour tout $q\in I^n(K)$ on a $f_n^d(q)=0$
  pour $d$ assez grand (pour les mêmes raisons que dans la proposition
  \ref{prop_somme_pfis}), et donc on peut utiliser les $f_n^d$ au lieu
  des $g_n^d$ dans le théorème (avec $G(\alpha) = (\alpha^{d+}(0))_d$).
  Dans le cas extrême où $-1$ est un carré dans $k$, on a même $f_n^d=g_n^d$,
  comme le montre la proposition \ref{prop_f_g}. De façon générale,
  quand $k$ n'est pas formellement réel, cette proposition montre
  qu'une somme $\sum a_d f_n^d$ s'écrit cette fois $\sum b_d g_n^d$ avec $b_d\in A(k)$
  contrairement à ce qu'on avait noté dans la remarque précédente.

  D'un autre côté, l'exemple \ref{ex_moins_pfis} montre qu'on ne peut
  pas utiliser les $f_n^d$ si $k$ est formellement réel, donc on peut
  établir une caractérisation exacte : toute combinaison des $f_n^d$
  est à valeurs dans $A$ si et seulement si $k$ n'est pas formellement
  réel.
\end{rem}

\begin{rem}
  On voit qu'on peut construire des invariants cohomologiques $\alpha$
  (par exemple $\alpha = \sum_d v_{nd}^{(n)}$)
  tels que, bien que le degré cohomologique de $\alpha(q)$ soit
  borné pour tout $q$ fixé, il n'est pas uniformément borné quand $q$
  varie (donc $\alpha$ est bien à valeurs dans $H^*(K,\mu_2)$ mais
  dans aucun $H^d(K,\mu_2)$). Cela reflète en un certain sens la nature
  \og infinie\fg{} de $I^n$, et c'est un comportement qui n'apparaît
  jamais pour les invariants de groupes algébriques.

  Le sous-module $M'$ des invariants cohomologiques de degré
  uniformément borné est précisément le sous-module engendré (finiment)
  par les $f_n^d$ (ou les $g_n^d$).
\end{rem}

\subsection{Structure d'algèbre}

Comme $\Inv(I^n,A)$ n'est pas seulement un $A(k)$-module, mais également
une algèbre, on souhaite comprendre comment le produit s'exprime en
fonction des éléments de base $f_n^d$.

On rappelle que le coefficient multinomial $\binom{n}{a_1,\dots,a_r}$ (où $a_1+\cdots+a_r=n$)
est défini comme $\frac{n!}{a_1!\cdots a_r!}$ (et en particulier, le coefficient
binomial usuel $\binom{n}{m}$ est $\binom{n}{m,n-m}$).
Quand $m<0$ ou $m>n$, on s'autorise quand même à écrire $\binom{n}{m}$,
qui vaut alors 0.

\begin{prop}\label{prop_pi_prod}
Soient $s,t\in \N$. On a 
\[ f_n^s \cdot f_n^t = \sum_{d=\max(s,t)}^{s+t}\binom{d}{s+t-d,d-s,d-t}\{-1\}^{n(s+t-d)} f_n^d. \]
\end{prop}

\begin{proof}
Soient $q\in I^n(K)$ et $\phi\in \Pf_n(K)$. Alors on a
\[ (f_n^s\cdot f_n^t)(q+\phi) = (f_n^s(q) + f_n(\phi) f_n^{s-1}(q))\cdot (f_n^t(q) + f_n(\phi) f_n^{t-1}(q))   \]
d'où
\begin{equation}\label{eq_prod}
  \Phi\left( f_n^s \cdot f_n^t\right) = f_n^s\cdot f_n^{t-1} + f_n^{s-1}\cdot f_n^t + \{-1\}^n f_n^{s-1}\cdot f_n^{t-1}.  
\end{equation}

On procède par récurrence, disons sur $(s,t)$ avec l'ordre lexicographique.
Déjà le résultat est clair si $s=0$ ou $t=0$. Par symétrie on peut supposer $s>t$
(il reste alors le cas $s=t$ qui se traite de façon similaire).

Alors par récurrence on peut remplacer chaque terme dans (\ref{eq_prod}),
et réarranger les termes pour trouver
\begin{align*}
  \Phi\left( f_n^s\cdot f_n^t\right) &= \binom{s}{t}\cdot \{-1\}^{nt} f_n^{s-1} + \binom{s+t}{t} f_n^{s+t-1}  \\
                                     &+  \sum_{d=s}^{s+t-2}\binom{d+1}{s+t-d-1,d-s+1,d-t+1}\{-1\}^{n(s+t-d-1)} f_n^d
\end{align*}
où pour le coefficient devant $f_n^{s-1}$ on utilise $\binom{s-1}{t} + \binom{s-1}{t-1} = \binom{s}{t}$,
pour celui de $f_n^{s+t-1}$ on se sert de 
$\binom{s+t-1}{t} + \binom{s+t-1}{t-1} = \binom{s+t}{t}$, et pour les autres termes on emploie
$\binom{d}{s+t-1-d,d-s+1,d-t} + \binom{d}{s+t-1-d,d-s,d-t+1}  +  \binom{d}{s+t-d-2,d-s+1,d-t+1} = \binom{d+1}{s+t-d-1,d-s+1,d-t+1}$.

On exploite alors l'idée dans la remarque \ref{rem_shift} pour trouver
la formule attendue pour $f_n^s\cdot f_n^t$.
\end{proof}

On peut noter quelques cas particuliers simples, qui
s'appliquent notamment à la cohomologie. On introduit quelques notations :
si $s,t\in \N$, on notera $s\lor t$
et $s\land t$ les entiers obtenus en appliquant respectivement un \emph{OU}
et un \emph{ET} logiques sur les représentations binaires de $s$ et $t$.
En particulier, $s\lor t +s\land t = s+t$. Si $s\land t=0$, on dit que
$s$ et $t$ ont des représentations binaires \emph{disjointes}.

\begin{coro}\label{cor_prod_f}
  Si $A(k)$ est de caractéristique 2, alors :
  \[ f_n^s \cdot f_n^t = \{-1\}^{s\land t} f_n^{s\lor t}. \]

  Si $-1$ est un carré dans $k$, alors $f_n^s\cdot f_n^t$
  vaut $f_n^{s+t}$ si $s$ et $t$ ont des représentations binaires
  disjointes, et $0$ sinon.
\end{coro}

\begin{proof}
  Pour le premier point, il faut appliquer le lemme combinatoire
  \ref{lem_binom}. Pour le deuxième, on montre plus loin (voir \ref{rem_car})
  que si $-1$ est un carré alors $A(k)$ est de caractéristique 2.
\end{proof}

\begin{lem}\label{lem_binom}
  Le coefficient multinomial $\binom{m}{s+t-m,m-s,m-t}$ est impair si
  et seulement si $m=s\lor t$.
\end{lem}

\begin{proof}
  Il est connu que pour tout $a\in \Z$, la valuation $2$-adique de $a!$
  est $a-\nu(a)$ où $\nu(a)$ est le nombre de $1$ dans la représentation
  binaire de $a$. Alors :
  \begin{align*}
    v_2\binom{m}{s+t-m,m-s,m-t} &= (m-\nu(m)) - (s+t-m -\nu(s+t-m)) \\
                                &   - (m-s - \nu(m-s)) - (m-t - \nu(m-t)) \\
                                &=  \nu(s+t-m) + \nu(m-s) + \nu(m-t) - \nu(m).
  \end{align*}

  Mais il est facile de voir que pour tous $a,b\in \Z$, $\nu(a+b)\ppq \nu(a)+\nu(b)$,
  avec égalité ssi $a\land b =0$. Donc $\binom{m}{s+t-m,m-s,m-t}$ est impair
  si et seulement si $s+t-m$, $m-s$ et $m-t$ ont tous des représentations
  binaires disjointes deux à deux.

  On affirme que c'est équivalent à $m=s \lor t$. En effet, si $m=s\lor t$
  c'est évident, et si $m\neq s\lor t$,
  considérons le bit le plus faible où $m$ et $s\lor t$ diffèrent ; on a plusieurs
  possibilités pour les bits de $s$, $t$ et $m$ à cet emplacement :
  $s$ a un 1 et $m$ un 0, $t$ a un 1 et $m$ un 0, ou $s$ et $t$ ont un 0
  et $m$ un 1. Dans tous ces cas, au moins deux nombres parmi $m-s$, $m-t$
  et $s+t-m$ ont un 1 à cet emplacement, et leurs représentations binaires
  ne sont pas disjointes.
\end{proof}

\begin{rem}
  Dans le cas où $A(k)$ est de caractéristique 2, on obtient une
  présentation d'algèbre très simple pour $M'$, le sous-module des invariants
  qui est constitué des combinaisons \emph{finies} des $f_n^d$ (il est plus
  délicat de parler de présentation pour $\Inv(I^n,A)$ puisqu'il
  faut utiliser des sommes infinies pour représenter un élément quelconque
  en fonction des $f_n^d$).

  En effet, la $A(k)$-algèbre $M'$ a pour générateurs les $x_i=f_n^{2^i}$,
  et les relations sur ces générateurs sont données par $x_i^2 = \{-1\}^{n2^i}x_i$.
\end{rem}

\subsection{Restriction de $I^n$ à $I^{n+1}$}

\label{par_restr}On dispose bien évidemment d'un morphisme de restriction de $\Inv(I^n,A)$
vers $\Inv(I^{n+1},A)$, et on notera $\alpha_{|I^{n+1}}$ la restriction
de $\alpha\in \Inv(I^n,A)$.

Le comportement de la restriction dépend dans une certaine mesure
de la nature de $A$ :

\begin{prop}\label{prop_exist_delta}
  Il existe un unique $\delta(A)\in A(k)$ tel que
  $(f_1)_{I^2} = \delta(A) f_2$, vérifiant $\{-1\}\delta(A)=2$.

  De plus, pour tous $n\in \N^*$ et $d\in \N$ on a
  $(f_n)_{I^{n+d}} = \delta(A)^df_{n+d}$
\end{prop}

\begin{proof}
  On rappelle que $f_1=f_1^1$ est un morphisme de groupes.
  Notons $\alpha = (f_1)_{I^2}$ ; c'est un invariant de $I^2$
  et d'après le corollaire \ref{cor_additif}, on a bien
  existence et unicité de $\delta(A)$.

  On a $f_1(2q)=2f_1(q)$ et $f_1(2q)=f_1(\pfis{-1}q)=\delta \{-1\}f_1(q)$,
  donc $2f_1=\{-1\}\delta f_1$, d'où $\{-1\}\delta = 2$. 

  D'après le corollaire \ref{cor_ecr_unique}, pour établir $(f_n)_{I^{n+d}} = \delta(A)^df_{n+d}$,
  il suffit de montrer qu'ils ont la même valeur sur les
  $(n+d)$-formes de Pfister, donc que
  $f_n(\pfis{a_1,\dots,a_{n+d}}) = \delta(A)^d\{a_1,\dots,a_{n+d}\}$.
  Or :
  \begin{align*}
    f_n(\pfis{a_1,\dots,a_{n+d}}) &=  f_{n-1}(\pfis{a_1,\dots,a_{n-1}})f_1(\pfis{a_n,\dots a_{n+d}}) \\
                                  &= \{a_1,\dots, a_{n-1}\} f_1(\pfis{a_n,\dots, a_{n+d}})
  \end{align*}
  et on montre facilement par récurrence que
  \[ f_1(\pfis{a_n,\dots a_{n+d}})=\delta(A)^d\{a_n,\dots,a_{n+d}\}. \]
\end{proof}

\begin{ex}
  Si $A(K)=W(K)$, alors $f_1$ et $f_2$ sont l'identité
  donc $\delta(A)=1$. En revanche, si $A(K)=H^*(K,\mu_2)$, $f_1^n=e_1$
  est nul sur $I^2(K)$, donc $\delta(A)=0$.
\end{ex}

\begin{rem}\label{rem_car}
  Si $\delta(A)=0$ ou si $-1$ est un carré dans $K$,
  $A(K)$ est de caractéristique 2.
\end{rem}

Quand il n'y a pas d'ambiguité on écrira $\delta=\delta(A)$.
Une fois ces notations posées, on peut écrire :

\begin{prop}\label{prop_restr_f}
  Pour tout $n\in \N^*$ et tout $d\in \N^*$ :
  \[ (f_n^d)_{I^{n+1}} = \sum_{\frac{d}{2}\ppq k\ppq d} \binom{k}{d-k}\delta(A)^{2k-d}\{-1\}^{(d-k)(n-1)} f_{n+1}^k   \]
\end{prop}

\begin{proof}
  Il suffit d'appliquer $f_{nd}$ à la formule du corollaire \ref{cor_relat_pi}
  exprimant $\pi_n^d$ en fonction des $\pi_{n+1}^k$.
\end{proof}

\begin{coro}\label{cor_restr_cohom}
  Si $\delta=0$ (et en particulier quand $A(K)=H^*(K,\mu_2)$), on a :
  \[   (f_n^d)_{|I^{n+1}} = \left\{ \begin{array}{lc}
\{-1\}^{m(n-1)} f_{n+1}^m & \text{si $d=2m$} \\
 0 & \text{si $d$ impair}
 \end{array} \right.     \]
\end{coro}

\begin{coro}
  Si $-1$ est un carré dans $k$, on a $(f_n^d)_{|I^{n+1}} = \delta^d f_{n+1}^d$
  pour tout $n\pgq 2$, et donc $(f_2^d)_{|I^n} = \delta^{(n-2)d}f_n^d$.

  Si en plus $\delta=0$, alors $(f_n^d)_{|I^{n+1}} = 0$
  si $d\pgq 1$, sauf quand $n=1$ et $d=2m$, auquel cas
  $(f_1^{2m})_{|I^2} = f_2^m$.
\end{coro}

On peut tirer quelques conséquences dans le cas des invariants
cohomologiques.

\begin{rem}\label{rem_restr_cohom}
  En particulier, on obtient la formule $(u_{2d}^{(1)})_{|I^2} = u_{2d}^{(2)}$
  (quel que soit $k$), qui montre que tout invariant de $I^2$
  s'étend à $I$ (pas de façon unique). En revanche, si $n\pgq 3$ alors
  pour $d\pgq 1$ $u_{nd}^{(n)}$ ne s'étend jamais à $I^{n-1}$.
  
  On retrouve notamment comme cas particulier le fait
  bien connu que $e_2$ s'étend à tout $I$ alors que $e_3$ ne s'étend pas
  à $I^2$.
\end{rem}

\begin{rem}
  On peut observer que si $(-1)\in H^1(k,\mu_2)$ n'est pas un diviseur
  de zéro dans $H^*(k,\mu_2)$ alors le noyau de la restriction des
  invariants de $I^n$ à $I^{n+1}$ est
  exactement l'idéal engendré par $e_n$, mais en général il est plus
  gros (quand $-1$ est un carré, il est même constitué de tous les invariants
  normalisés). C'est notable dans la mesure où on aurait pu s'attendre
  à ce que les invariants nuls sur les formes vérifiant $\alpha(q)=0$
  soient exactement les multiples de $\alpha$.
\end{rem}

\subsection{Similitudes}

Dans cette partie on étudie le comportement des invariants
vis-à-vis des similitudes.

\begin{propdef}\label{prop_def_psi}
  Il existe un unique morphisme de $A(k)$-modules filtrés, de degré $-1$,
\[ \begin{foncdef}{\Psi}{M}{M}{\alpha}{\tld{\alpha}}\end{foncdef}  \]
tel que
\begin{equation}
\alpha(\fdiag{\lambda}q) = \alpha(q) + \{\lambda\}\tld{\alpha}(q)
\end{equation}
pour tous $\alpha\in M$, $q\in I^n(K)$ et $\lambda\in K^*$.
\end{propdef}

\begin{proof}
Soient $\alpha\in \Inv(I^n,A)$ et $q\in I^n(K)$. Pour tout $\lambda\in L^*$, 
avec $L/K$ extension de corps, on pose $\beta_q(\lambda) = \alpha(\fdiag{\lambda} q)$.

Alors $\beta_q$ est un invariant sur $K$ des classes de carrés à valeurs dans $A$.
Or le foncteur des classes de carrés est isomorphe à $\Pf_1$, donc on peut
appliquer la condition \ref{cond_4} sur $A$ : il existe d'uniques $x_q,y_q\in A(K)$
tels que $\beta_q(\lambda) = x_q + \{\lambda\}\cdot y_q$.
En prenant $\lambda = 1$ on voit que $x_q = \alpha(q)$, et on pose $\tilde{\alpha}(q) = y_q$.

L'unicité de $y_q$ permet de voir que $\tilde{\alpha}\in \Inv(I^n,A)$.
\end{proof}

\begin{rem}\label{rem_norm}
  On peut remarquer que $\tilde{\alpha}$ est toujours un invariant normalisé.
\end{rem}

La composition des similitudes donne le résultat intéressant suivant :

\begin{prop}\label{prop_psi_carre}
  On a $\Psi^2=-\delta(A)\Psi$.
\end{prop}

\begin{proof}
  On note que
  \begin{equation}\label{eq_delta_sum}
    \{\lambda\mu\}=\{\lambda\}+\{\mu\} - \delta(A)\{\lambda,\mu\}.
  \end{equation}
  De là, 
  \begin{align*}
    \alpha(\fdiag{\lambda\mu}q) &= \alpha(\fdiag{\lambda}q)
                                  + \{\mu\} \tld{\alpha}(\fdiag{\lambda}q) \\
    &= \alpha(q) + \{\lambda\} \tld{\alpha}(q) + \{\mu\} \tld{\alpha}(q)
      + \{\lambda,\mu\} \tld{\tld{\alpha}}(q) \\
    &= \alpha(q) + \{\lambda\mu\}\tld{\alpha}(q) + \delta(A)\{\lambda,\mu\}\tld{\alpha}(q) + \{\lambda,\mu\} \tld{\tld{\alpha}}(q)
  \end{align*}
  et également
  \[ \alpha(\fdiag{\lambda\mu}q) = \alpha(q) + \{\lambda\mu\} \tld{\alpha}(q) \]
  donc $\{\lambda,\mu\}(\delta(A)\tld{\alpha}(q)+ \tld{\tld{\alpha}}(q))=0$.
  Comme c'est vrai pour tous $\lambda$, $\mu$ sur toute extension
  on peut conclure que $\tld{\tld{\alpha}}=-\delta(A)\tld{\alpha}$.
\end{proof}

\begin{rem}\label{rem_inv_simil}
  Par définition, $\tld{\alpha}=0$ si et seulement si $\alpha(\fdiag{\lambda}q)=\alpha(q)$,
  ce qu'on qualifiera d'\emph{invariance par similitude}. Mais la
  proposition précédente suggère qu'il est intéressant d'étudier aussi
  la propriété $\tld{\alpha} = -\delta(A)\alpha$, qui est notamment
  vérifiée par tout élément de la forme $\tld{\beta}$. Dans le cas des
  invariants cohomologiques, on retrouve la même propriété puisque $\delta=0$, mais
  dans le cas des invariants de Witt elle est équivalente à
  $\alpha(\fdiag{\lambda}q)=\fdiag{\lambda}\alpha(q)$ ; on dira alors
  que $\alpha$ est \emph{compatible avec les similitudes}.

  Lorsque $\delta\neq 0$ (et donc en particulier pour les invariants
  de Witt), la proposition montre alors que tout $\alpha$ peut être décomposé
  de façon unique en $\alpha = \beta + \gamma$ où $\beta$ est compatible
  avec les similitudes, et $\gamma$ est invariant par similitudes ;
  précisément, $\beta = -\tld{\alpha}$ et $\gamma = \alpha + \tld{\alpha}$.

  De façon moins intrinsèque, tout invariant de Witt $\alpha$ est une combinaison
  des $\lambda^d$ (puisque c'est le cas des $\pi_n^d$), et alors $\beta$
  correspond aux termes avec $d$ impair, et $\gamma$ aux termes avec
  $d$ pair (pour cette raison on parlera parfois de \emph{partie paire}
  et \emph{partie impaire} de $\alpha$).
\end{rem}

On veut maintenant étudier l'action de $\Psi$ sur les $g_n^d$,
qui s'avèrent avoir un bien meilleur comportement que les $f_n^d$
relativement aux similitudes (ce qui n'est pas surprenant si on
considère la cas de la mutliplication par $\fdiag{-1}$).

\begin{prop}\label{prop_simil}
  Soient $n\in \N^*$, $d\in \N$. Alors
  \[ \tld{g_n^d} = \left\{ \begin{array}{lc}
                             -\delta(A) g_n^d & \text{si $d$ est impair} \\
                             \{-1\}^{n-1}g_n^{d-1} & \text{si $d$ est pair.}
                           \end{array} \right. \]
\end{prop}

\begin{proof}
  On procède par récurrence sur $d$, les cas $d=0,1$ étant triviaux.
  Si $q\in I^n(K)$, $\phi\in \Pf_n(K)$ et $\lambda\in K^*$, alors
  \begin{align*}
    g_n^d(\fdiag{\lambda}(q+\phi)) &= g_n^d(q+\phi) + \{\lambda\}\tld{g_n^d}(q+\phi) \\
                                   &= g_n^d(q) + f_n(\phi)(g_n^d)^+(q) + \{\lambda\} \tld{g_n^d}(q) +
                                     \{\lambda\}f_n(\phi) \tld{g_n^d}^+(q).
  \end{align*}
  D'un autre côté, si on écrit $\phi = \pfis{a}\psi$ :
  \begin{align*}
    g_n^d(\fdiag{\lambda}(q+\phi)) &= g_n^d(\fdiag{\lambda}q + \pfis{\lambda a}\psi - \pfis{\lambda}\psi) \\
                                   &= g_n^d(\fdiag{\lambda}q - \pfis{\lambda}\psi) + \{\lambda a\}f_{n-1}(\psi) (g_n^d)^+(\fdiag{\lambda}q - \pfis{\lambda}\psi) \\
                                   &= g_n^d(\fdiag{\lambda}q) - \{\lambda\}f_{n-1}(\psi)(g_n^d)^-(\fdiag{\lambda}q) \\
                                   &+ \{\lambda a\}f_{n-1}(\psi)(g_n^d)^+(\fdiag{\lambda}q) - \{\lambda,\lambda a\}\{-1\}^{n-1}f_{n-1}(\psi)(g_n^d)^{+-}(\fdiag{\lambda}q) \\
                                   &= g_n^d(q) + \{\lambda\}\tld{g_n^d}(q) - \{\lambda\}f_{n-1}(\psi)(g_n^d)^-(q) - \{\lambda,\lambda\}f_{n-1}(\psi)\tld{(g_n^d)^-}(q) \\
                                   &+ \{\lambda a\}f_{n-1}(\psi)(g_n^d)^+(q) - \{\lambda,\lambda a\}\{-1\}^{n-1}f_{n-1}(\psi)(g_n^d)^{+-}(q)  \\
                                   &+ \{\lambda, \lambda a\}f_{n-1}(\psi)\tld{(g_n^d)^+}(q) - \{\lambda, \lambda,\lambda a\}\{-1\}^{n-1}f_{n-1}(\psi)\tld{(g_n^d)^{+-}}(q).
  \end{align*}
  En prenant $\lambda$, $a$ et $\psi$ génériques et en considérant les résidus,
  on trouve :
  \[ \tld{g_n^d}^+ = -\delta (g_n^d)^+  - \tld{(g_n^d)^+} + \{-1\}^{n-1}(g_n^d)^{+-} + \{-1\}^n\tld{(g_n^d)^{+-}} \]
  en utilisant plusieurs fois l'équation (\ref{eq_delta_sum}) et le fait que
  $\delta \{-1\} = 2$.

  Si $d$ est pair, alors $(g_n^d)^+=g_n^{d-1}$, donc $\tld{(g_n^d)^+} = -\delta g_n^{d-1}$,
  et $(g_n^d)^{+-}=g_n^{d-2}$ donc $\tld{(g_n^d)^{+-}} = \{-1\}^{n-1}g_n^{d-3}$.
  On obtient donc :
  \begin{align*}
    \tld{g_n^d}^+ &= -\delta g_n^{d-1} + \delta g_n^{d-1} + \{-1\}^{n-1}(g_n^{d-2}+\{-1\}^ng_n^{d-3}) \\
                  &= \{-1\}^{n-1}(g_n^{d-1})^+
  \end{align*}
  ce qui permet de conclure puisque $\tld{g_n^d}$ est normalisé (voir
  la remarque (\ref{rem_norm})).

  De même, si $d$ est impair, $(g_n^d)^+=g_n^{d-1} + \{-1\}^ng_n^{d-2}$,
  donc $\tld{(g_n^d)^+} = \{-1\}^{n-1}g_n^{d-2} - \delta\{-1\}^ng_n^{d-2} = -\{-1\}^{n-1}g_n^{d-2}$,
  et $(g_n^d)^{+-}=g_n^{d-2}$ so $\tld{(g_n^d)^{+-}} = -\delta g_n^{d-2}$.
  Alors
  \begin{align*}
    \tld{g_n^d}^+ &= -\delta (g_n^d)^+ + \{-1\}^{n-1} g_n^{d-2}  + \{-1\}^{n-1}g_n^{d-2} -\delta \{-1\}^n g_n^{d-2} \\
                  &= -\delta (g_n^d)^+
  \end{align*}
  ce qui permet encore de conclure.
\end{proof}

Ce comportement très simple permet quelques observations :

\begin{rem}
  On retrouve facilement en utilisant les $g_n^d$ le fait que
  tout invariant $\alpha$ de la forme $\tld{\beta}$ est compatible
  avec les similitudes (au sens de la remarque \ref{rem_inv_simil}).
  On peut également répondre à la question légitime de savoir si
  ces deux conditions sont équivalentes. Si $\alpha = \sum a_d g_n^d$,
  alors $\alpha$ est invariant par similitude si et seulement si
  pour tout $d$ pair, $\delta(A)a_d=0$ et $\{-1\}^{n-1}a_d=0$ (si $d\pgq 2$) ;
  en revanche $\alpha$ est de la forme $\tld{\beta}$ pour un certain $\beta$
  si et seulement si $a_d=0$ pour $d$ pair, et $a_d=\{-1\}^{n-1}x_d-\delta(A)y_d$
  pour certains $x_d,y_d$ si $d$ est impair. En particulier, les conditions
  sont équivalentes si $\delta(A)=1$ (donc pour les invariants de Witt),
  et dans ce cas elles correspondent au fait d'être une combinaison des
  $g_n^d$ avec $d$ impair, mais elles ne sont jamais équivalentes si
  $\delta(A)=0$ (donc pour les invariants cohomologiques) et $n>1$.
\end{rem}

\begin{rem}
  On a défini deux opérateurs, $\Phi$ et $\Psi$, qui contrôlent respectivement
  le comportement des invariants vis-à-vis des sommes et des similitudes.
  Il est tentant d'essayer d'établir un lien entre $\Phi\circ \Psi$ et
  $\Psi\circ \Phi$, mais en fait il est relativement facile de vérifier
  en utilisant la base des $g_n^d$ que lorsque $\delta(A)=0$ ces deux compositions sont
  complètement décorrélées, au sens suivant : si $\beta$ est n'importe
  quel invariant dans l'image de $\Phi\circ \Psi$, et $\gamma$ dans l'image
  de $\Psi\circ \Phi$, alors il existe $\alpha$ tel que $\beta = (\tld{\alpha})^+$
  et $\gamma = \tld{(\alpha^+)}$.
\end{rem}

\begin{rem}
  On peut également établir des formules pour l'action de $\Psi$
  sur les $f_n^d$, soit directement par des méthodes similaires,
  soit en utilisant les formules de passage de la proposition
  \ref{prop_f_g}. On ne détaillera pas le calcul ici, mais
  on peut donner ces formules à titre indicatif :
  \[ \tld{f_n^d} = (-1)^d\sum_{k=1}^{d-1}\binom{d-1}{k-1} \{-1\}^{n(d-k)-1} f_n^k
     + \left\{ \begin{array}{cl} 0 & \text{si $d$ pair} \\ -\delta(A) f_n^d & \text{si $d$ impair.}\end{array} \right. \]
\end{rem}

\subsection{Ramification des invariants}

Dans cette courte partie on étudie le comportement des invariants
par rapport aux résidus de valuations discrètes. Soit donc $(K,v)$
un corps valué, où $v$ est une $k$-valuation discrète de rang 1.

Pour éviter d'avoir à introduire une notion abstraite de ramification
sur un $A$ général vérifiant nos conditions, on se restreint ici
explicitement à $A(K)=W(K)$ ou $A(K)=H^*(K,\mu_2)$.

\begin{prop}\label{prop_ram_witt}
Soient $n\in \N^*$ et $q\in I^n(K)$. Si $q$ est non ramifiée,
alors $\alpha(q)$ est non ramifiée pour tout $\alpha\in \Inv(I^n, A)$.
\end{prop}

\begin{proof}
  Comme $v$ est une $k$-valuation, il suffit de montrer le résultat
  pour $\alpha = f_n^d$. Soit $q\in I^n(K)$ non ramifiée.
  Comme on peut écrire $q = \fdiag{a_1,\dots,a_r}$
  avec $a_i\in \mathcal{O}_K$, $\lambda^k(q)$ est également non ramifiée
  pour tout $k\in \N$, et donc $\pi_n^d(q)$ aussi. Comme $e_{nd}$
  envoie une forme non ramifiée sur une classe de cohomologie
  non ramifiée, $u_{nd}^{(n)}(q)$ est également non ramifié.
\end{proof}

\subsection{Invariants en dimension fixée}\label{sec_dim_fix}

Dans cette partie, on décrit les restrictions des invariants
de $I^n$ en dimension fixée pour $n=1$ et $2$.

\subsubsection*{Invariants de $I$}

Dans \cite[thm 17.3, thm 28.5]{GMS}, Serre décrit les invariants cohomologiques ainsi
que les invariants de Witt de $\Quad_{2r}$ (les formes quadratiques
en dimension fixée $2r$). On trouve dans les deux cas un $A(k)$-module
libre de rang $n+1$, respectivement sur les classes de Stiefel-Whitney $w_k$
et sur les $\lambda^k$, pour $0\ppq k\ppq n$ (on rappelle que
$w_k(\fdiag{a_1,\dots,a_n}) = \sum_{i_1<\dots <i_k}(a_{i_1},\dots , a_{i_k})$).
Comme on peut restreindre les invariants de $I$ à $\Quad_{2r}$,
on peut se demander comment ils se décomposent
en fonction des classes de Stiefel-Whitney et des $\lambda$-opérations.

Il faut ici faire attention aux définitions que l'on utilise : quand on écrit
(par exemple) que $\pi_1^2(q)=\lambda^2(q)-\lambda^1(q)$, on travaille au niveau
des classes de Witt, ramenées dans $GW(K)$ par l'isomorphisme canonique
$\hat{I}(K)\simeq I(K)$. Si $q$ est de dimension $n=2r$, on applique donc en
réalité les $\lambda^k$ à $q-r\pfis{1}\in \hat{I}(K)$. Pour dissiper les
ambiguïtés de notations, on utilisera dans cette partie $\hat{\lambda}^k$
pour désigner cette opération, par opposition à $\lambda^k$ qui désignera
directement l'opération sur $GW(K)$ (en particulier, $\lambda^k$ n'est
pas bien défini sur les classes de Witt, par opposition à $\hat{\lambda}^k$).
On peut passer de l'un à l'autre de la façon suivante :

\begin{lem}\label{lem_hat_lambda}
  Soit $q$ une forme quadratique de dimension $2r$ sur $K$. Alors
  pour tout $d\in \N$, on a dans $W(K)$: 
  \[ \lambda^d(q) = \sum_{\substack{i=0 \\ i\equiv d \text{mod} 2}}^d \binom{r}{\frac{d-i}{2}} (-1)^{\frac{d-i}{2}}\hat{\lambda}^i(q) \]
  et
  \[ \hat{\lambda}^d(q) = \sum_{\substack{i=0 \\ i\equiv d \text{mod} 2}}^d \binom{r+\frac{d-i}{2}-1}{\frac{d-i}{2}} \lambda^i(q). \]
\end{lem}

\begin{proof}
  On note, similairement à la notation familière $\lambda_t(q)\in W(K)[[t]]$,
  $\hat{\lambda}_t(q) = \sum_{d\in \N}\hat{\lambda}^d(q)t^d$. On a alors
  dans $W(K)$:
  \begin{align*}
    \lambda_t(\pfis{1}) &= \lambda_t(\fdiag{1})\lambda_t(\fdiag{-1}) \\
                        &= (1+t)(1-t) \\
                        &= 1-t^2
  \end{align*}
  donc $\hat{\lambda}_t(q) = \lambda_t(q-r\pfis{1}) = (1-t^2)^{-r}\lambda_t(q)$,
  et réciproquement $\lambda_t(q) = (1-t^2)^r\hat{\lambda}_t(q)$,
  ce qui donne les formules voulues en développant
  $(1-t^2)^r$ et $(1-t^2)^{-r}$.
\end{proof}

On peut exprimer les $g_1^d$ à partir des opérations $\lambda$ :

\begin{prop}\label{prop_pi1_g1}
  Soit $q$ une forme quadratique de dimension $2r$ sur $K$.
  Alors pour tout $m\in \N^*$, on a
  \[ g_1^{2m}(q) = \sum_{i=0}^{2m}(-1)^i\binom{r-\lceil \frac{i}{2}\rceil}{m-\lceil\frac{i}{2}\rceil} \lambda^i(q) \]
  et pour tout $m\in \N$ :
  \[ g_1^{2m+1}(q) = \sum_{i=0}^m \binom{r-1-i}{m-i} \lambda^{2i+1}(q). \]
  
\end{prop}


\begin{proof}
  On pose $\alpha_d(q)$ défini par le terme de droite des équations
  de l'énoncé. A priori ce n'est pas un invariant de classes de  Witt
  mais simplement un invariant de classes d'isométrie défini en toute
  dimension paire $2r$. On va montrer que $\alpha_{2m+2}(q+\pfis{a})
  =\alpha_{2m}(q)+\pfis{a}\alpha_{2m+1}(q)$  et $\alpha_{2m+1}(q+\pfis{a})
  =\alpha_{2m+1}(q)+\pfis{a}(\alpha_{2m}(q)+ \pfis{-1}\alpha_{2m-1}(q)$,
  ce qui montre à la fois que $\alpha_d$ est bien défini sur les classes
  de Witt en prenant $a=1$, et qu'ils vérifient la même récurrence que
  les $g_1^d$.

  On voit facilement que
  \[ \lambda^i(q+\pfis{a}) = \lambda^i(q)-\lambda^{i-2}(q) +\pfis{a}(\lambda^{i-1}(q)+\lambda^{i-2}(q). \]
  De là:
  \begin{align*}
   & \alpha_{2m+2}(q+\pfis{a}) \\
   &= \sum_j\binom{r+1-j}{m+1-j}\lambda^{2j}(q+\pfis{a}) - \sum_j \binom{r-j}{m-j}\lambda^{2j+1}(q+\pfis{a}) \\
   &= \sum_j\binom{r+1-j}{m+1-j}(\lambda^{2j}(q)-\lambda^{2j-2}(q)) - \sum_j \binom{r-j}{m-j}(\lambda^{2j+1}(q)-\lambda^{2j-1}(q)) \\
   &+ \pfis{a}\left( \sum_j\binom{r+1-j}{m+1-j}(\lambda^{2j-1}(q)+\lambda^{2j-2}(q)) - \sum_j \binom{r-j}{m-j}(\lambda^{2j}(q)+\lambda^{2j-1}(q)) \right) \\
   &= \sum_j \left( \binom{r+1-j}{m+1-j}-\binom{r-j}{m-j}\right)\lambda^{2j}(q) - \sum_j\left( \binom{r-j}{m-j}-\binom{r-j-1}{m-j-1}\right)\lambda^{2j+1}(q) \\
   &+ \pfis{a}\sum_j\left( \binom{r+1-j}{m+1-j}-\binom{r-j}{m-j}\right)\lambda^{2j-1}(q) \\
   &= \sum_j \binom{r-1}{m+1-j}\lambda^{2j}(q) - \sum_j \binom{r-j-1}{m-j}\lambda^{2j+1}(q) + \pfis{a}\sum_j\binom{r-j}{m+1-j}\lambda^{2j-1}(q)\\
   &= \alpha_{2m+2}(q)+\pfis{a}\alpha_{2m+1}(q).
  \end{align*}
  De même:
  \begin{align*}
    & \alpha_{2m+1}(q+\pfis{a})\\
    &= \sum_i \binom{r-i}{m-i}\lambda^{2i+1}(q+\pfis{a}) \\
    &= \sum_i \binom{r-i}{m-i}(\lambda^{2i+1}(q)-\lambda^{2i-1}(q)) + \pfis{a}\sum_i\binom{r-i}{m-i}(\lambda^{2i}(q) + \lambda^{2i-1}(q)) \\
    &= \sum_i \left( \binom{r-i}{m-i}-\binom{r-1-i}{m-1-i}\right) \lambda^{2i+1}(q) \\
    &+ \pfis{a}\left( \sum_i \binom{r-i}{m-i}\lambda^{2i}(q)-\sum_i\binom{r-1-i}{m-1-i}\lambda^{2i+1}(q) +2\sum_i \binom{r-1-i}{m-1-i}\lambda^{2i+1}(q)\right) \\
    &= \alpha_{2m+1}(q) +\pfis{a}(\alpha_{2m}(q)+2\alpha_{2m-1}(q)).
  \end{align*}
\end{proof}

\begin{rem}
  En particulier, on retrouve le fait que à $r$ fixé,
  $g_1^d$ est nul pour $d>2r$.
\end{rem}

\begin{rem}
  On rappelle (voir corollaire \ref{cor_relat_pi}) que
  \[ \pi_1^d = \sum_{k=1}^d (-1)^{d-k}\binom{d-1}{k-1}\hat{\lambda}^k. \]
  En revanche, il est assez pénible d'exprimer les $\pi_n^d$
  en fonction des $\lambda^k$, ce qui est cohérent avec le fait
  que les $\pi_n^d$ se comportent moins bien en dimension fixée.
\end{rem}

On veut également exprimer les $u_d^{(1)}$ et $v_d^{(1)}$ en
fonction des invariants de Stiefel-Whitney, et pour cela
on propose un formalisme qui s'applique également aux invariants
de Witt, donnant ainsi des formules valables dans tous les cas.

\begin{defi}\label{def_pd}
  Pour tout $d\in \N$, on définit $P^d: \Quad_n\To W$ par :
  \[ P^d(q) = \sum_{i=0}^d (-1)^i \binom{n-i}{d-i} \lambda^i(q). \]
\end{defi}

Comme la matrice de passage de $(\lambda^d)$ à $(P^d)$ est triangulaire
avec des 1 sur la diagonale, $(P^d)_{0\ppq d\ppq n}$ est
aussi une $W(k)$-base de $\Inv(\Quad_n, W)$. Sa définition
est motivée par la formule suivante :

\begin{prop}
  Pour tous $a_i\in K^*$, $1\ppq i\ppq n$, on a
  \[ P^d(\fdiag{a_1,\dots,a_n}) = \sum_{i_1<\dots <i_d} \pfis{a_{i_1},\dots,a_{i_d}} \]
  dans $W(K)$.
  En particulier, $P^d\in \Inv(\Quad_n,I^d)$, et $w_d = e_d\circ P^d$.
\end{prop}

\begin{proof}
  Si on développe $\pfis{a_{i_1},\dots,a_{i_d}}$ on trouve
  \[ \sum_{I\subset \{i_1,\dots,i_d\}} (-1)^{|I|}\prod_{i\in I}\fdiag{a_i} \]
  et pour $I$ fixé le terme $\prod_{i\in I}\fdiag{a_i}$ apparaît
  $\binom{n-i}{d-i}$ fois dans $\sum_{i_1<\dots <i_d} \pfis{a_{i_1},\dots,a_{i_d}}$.
\end{proof}

\label{par_h}On peut donc otenir des énoncés unifiés pour $A=W$ et $A=H$, en
posant $h^d=P^d$ dans le premier cas, et $h^d=w_d$ dans le second.
On peut alors exprimer nos invariants en fonction des $h^d$ :

\begin{prop}\label{prop_fixed_dim}
  Soient $d\in \N$ et $q\in \Quad_n(K)$. Alors (si $A=W$ ou $A=H$) :
  \[ f_1^d(q) = \sum_{i=0}^d (-1)^i \binom{r-i}{d-i}\{-1\}^{d-i} h^i(q);  \]
  \[ g_1^d(q) = \sum_{i=0}^d (-1)^i\binom{r-i-1 + \lfloor \frac{d+1}{2}\rfloor}{d-i} \{-1\}^{d-i} h^i(q). \]
\end{prop}

\begin{proof}
  On prouve l'énoncé concernant $f_1^d$ ; le cas des $g_n^d$
  peut s'en déduire en utilisant la proposition \ref{prop_f_g},
  ou directement en utilisant la même méthode.

  On note $\alpha_d$ pour l'invariant défini par le terme de droite de
  l'équation. A priori ce ne sont pas des invariants de classes de Witt
  mais simplement des invariants de classes d'isométrie en toute
  dimension paire.
  
  Il est clair par définition que $\alpha_0 = 1 = f_1^0$. On doit alors montrer
  que pour tout $d\in \N^*$, $\alpha_d(q)=0$ si $q$ est hyperbolique, et que
  \begin{equation}\label{eq_alpha_add}
    \alpha_d(q+\pfis{a}) = \alpha_d(q)+\{a\} \alpha_{d-1}(q)
  \end{equation}
  pour tout $a\in K^*$, ce qui montre à la fois que $\alpha_d$ est un invariant
  de classes de Witt (en prenant $a=1$), et que $\alpha_d^+ = \alpha_{d-1}$,
  ce qui permet de conclure par récurrence sur le décalage.

  On obtient
  \begin{equation}\label{eq_sw_hyp}
    h^i(r\pfis{1}) = \binom{r}{i}\{-1\}^i
  \end{equation}
  et
  \begin{align}\label{eq_sw_add}
    h^i(q+\pfis{a}) &= h^i(q)+\{-a\} h^{i-1}(q) \nonumber \\
                    &= \left(h^i(q)+\{-1\}h^{i-1}(q)\right) -\{a\}h^{i-1}(q). 
  \end{align}

  Soit $q$ hyperbolique en dimension $n=2r$. D'après l'équation (\ref{eq_sw_hyp}),
  \[\alpha_d(q) =  \left( \sum_{i=0}^d(-1)^i \binom{d-r}{d-i}\binom{r}{i}\right)\{-1\}^d = 0\]
  en utilisant une formule combinatoire simple.

  Alors pour (\ref{eq_alpha_add}), on pose $q$ de dimension $2r$,
  et on utilise l'équation (\ref{eq_sw_add}) pour trouver :
  \begin{align*}
    \alpha_d(q+\pfis{a}) &= \sum_{i=0}^d (-1)^i\binom{r+1-i}{d-i}\{-1\}^{d-i}\left(h^i(q)+\{-1\}h^{i-1}(q)\right) \\
                         &- \{a\} \sum_{i=0}^d (-1)^i\binom{r+1-i}{d-i}\{-1\}^{d-i}h^{i-.1}(q) \\
                         &= \sum_{i=0}^{d-1}\left( (-1)^i\binom{r-i+1}{d-i} + (-1)^{i+1}\binom{r-i}{d-i-1} \right) \{-1\}^{d-i} h^i(q)  \\
                         &+ (-1)^d h^d(q) - \{a\} \sum_{i=0}^{d-1}(-1)^{i+1}\binom{r-i}{d-i-1} \{-1\}^{d-i-1}h^i(q) \\
                         &= \sum_{i=0}^d (-1)^i\binom{r-i}{d-i}\{-1\}^{d-i}h^i(q) \\
                         &+ \{a\} \sum_{i=0}^{d-1}(-1)^i\binom{r-i}{d-1-i} \{-1\}^{d-1-i}h^i(q)
  \end{align*}
  ce qui donne la formule attendue.
\end{proof}

\begin{rem}
  Si $-1$ est un carré dans $k$, alors $f_1^d=g_1^d=h^d$.
\end{rem}

\begin{coro}
  Pour tout $m\in \N$, $\Inv(\Quad_{2m},\mu_2)$ est le $H^*(k,\mu_2)$-module
  libre sur les $u_d^{(1)}$ pour $0\ppq d\ppq 2m$.

  En particulier, tout invariant de $\Quad_{2m}$ peut être
  étendu en un invariant de $I$.
\end{coro}

\begin{proof}
  On voit que la matrice de passage entre les $u^{(1)}_d$ et les $w_d$
  est triangulaire avec des 1 sur la diagonale.
\end{proof}

\subsubsection*{Invariants de $I^2$}

\label{par_quad}Dans \cite[thm 20.6]{GMS}, Serre décrit également les
invariants cohomologiques de $\Quad'_{2r}$, le foncteur des formes
de dimension $2r$ et de discriminant trivial. Spécifiquement,
$\Inv(\Quad'_{2r},\mu_2)$ est la somme directe de l'image de
$\Inv(\Quad_{2r},\mu_2)$ par restriction à $\Quad'_{2r}$, et
de $I_{(-1)^r}$ où $I_\delta = \ens{x\in H^*(k,\mu_2)}{(\delta)\cup x=0}$.
L'invariant correspondant à $x\in I_{(-1)^r}$ est $b_x$
qui envoie $\fdiag{a_1,\dots,a_{2r}}$ sur $(a_1,\dots,a_{2r-1})\cup x$.

En particulier, sauf dans des cas très exceptionnels (quand
$I_{(-1)^r}=0$), il existe des invariants de $\Quad'_{2r}$
qui ne sont pas la restriction d'un invariant de $\Quad_{2r}$.
Or, d'après la remarque \ref{rem_restr_cohom}, tout invariant
de $I^2$ s'étend à $I$, donc il existe des invariants de
$\Quad'_{2r}=\Quad_{2r}\cap I^2$ qui ne s'étendent pas à $I^2$
(et l'obstruction est précisément $I_{(-1)^r}$).


\subsection{Opérations sur la cohomologie}\label{sec_operations}

Dans cette partie on s'intéresse spécifiquement au cas des invariants
cohomologiques. Il a été observé par Serre qu'on pouvait définir des sortes de carrés
divisés sur la cohomologie modulo 2 :

\[ \anonfoncdef{H^n(K,\mu_2)}{H^{2n}(K,\mu_2)/(-1)^{n-1}\cup H^{n+1}(K,\mu_2)}
{\sum_i \alpha_i}{\sum_{i<j} \alpha_i\cup \alpha_j.} \]

Le quotient à droite est nécessaire pour que l'application soit bien
définie. De même, on peut définir des puissances divisées de degré
supérieur :

\[ \anonfoncdef{H^n(K,\mu_2)}{H^{dn}(K,\mu_2)/(-1)^{n-1}\cup H^{(d-1)n+1}(K,\mu_2)}
{\sum_i \alpha_i}{\sum_{i_1<\dots <i_d} \alpha_{i_1}\cup\cdots \cup \alpha_{i_d}.} \]

D'un autre côté, Vial (\cite{Via}) caractérise les opérations
\[ K^M_n(K)/2 \To K^M_*(K)/2 \]
en fonction de certaines puissances divisées. En utilisant
la conjecture de Milnor, cela se traduit par la classification
des opérations
\[ H^n(K,\mu_2) \To H^*(K,\mu_2). \]

De même que pour les opérations de Serre il était nécessaire de
considérer un quotient à droite, il faut ici s'attendre à une certaine
restriction sur l'ensemble de départ. De fait, l'énoncé de Vial
est le suivant :

\begin{prop}[\cite{Via}, Thm 2]
  Si $n\in \N^*$, le module des opérations $H^n(K,\mu_2)\to H^*(K,\mu_2)$ est
  le $H^*(k,\mu_2)$-module
  \[ H^*(k,\mu_2)\cdot 1 \oplus H^*(k,\mu_2)\cdot \Id \oplus  \bigoplus_{d\in \N} \Ker(\tau_n)\cdot \gamma_d \]
  où $\tau_n: H^*(k,\mu_2)\To H^*(k,\mu_2)$ est défini par $\tau_n(x) = (-1)^{n-1}\cup x$
  et si $a\in \Ker(\tau_n)$, alors
  \[ a\cdot \gamma_d \left(\sum_{1\ppq i\ppq r} x_i\right) = a\cdot \sum_{i_1<\cdots <i_d}x_{i_1}\cup \cdots \cup x_{i_d} \]
  où les $x_i$ sont des symboles.
\end{prop}

Dans les deux cas, il est nécessaire pour que les opérations soient bien définies
d'annuler certaines puissances de $(-1)$, que ce soit par un noyau à gauche ou
un conoyau à droite.

Les invariants $u^{(n)}_{nd}$ peuvent être vus comme des relevés de ces
opérations de puissances divisées au niveau de $I^n$. Le phénomène
remarquable qui se produit alors est que ces opérations y sont définies
sans aucune restriction.

De plus, on peut sans difficulté retrouver la classification de Vial
en utilisant notre analyse des invariants de $I^n$ : les opérations
sur $H^n(K,\mu_2)$ ne sont rien d'autre que les invariants dans $\Inv(I^n,\mu_2)$
qui vérifient la condition
\begin{equation}\label{eq_vial}
\alpha(q+\phi)=\alpha(q) \quad \forall q\in I^n(K), \phi\in \Pf_{n+1}(K).
\end{equation}

On peut maintenant utiliser le fait que, d'après le corollaire \ref{cor_restr_cohom},
$u_{nk}^{(n)}(\phi)$ vaut $1$ si $k=0$, $(-1)^{n-1}\cup e_{n+1}(\phi)$ si $k=2$,
et $0$ autrement ; couplé à la formule d'addition, on obtient
\[ u_{nd}^{(n)}(q+\phi)=u_{nd}^{(n)}(q) + (-1)^{n-1}\cup e_{n+1}(\phi)\cup u_{n(d-2)}^{(n)}(q). \]
En écrivant $\alpha$ comme combinaison des $u_{nd}^{(n)}$, on obtient
\[ \alpha(q+\phi) = \alpha(q) + (-1)^{n-1}\cup e_{n+1}(\phi)\cup \alpha^{++}(q). \]

De là, $\alpha$ vérifie la condition (\ref{eq_vial}) si et seulement si 
$(-1)^{n-1}\cup \alpha^{++}=0$, ce qui revient précisément à dire que le coefficient
devant $u^{(n)}_{nd}$ pour $d\pgq 2$ est dans le noyau de $\tau_n$, et on retrouve
exactement l'énoncé de Vial.

\subsection{Invariants de formes semi-factorisées}

Dans \cite[def 20.8]{Gar}, Garibaldi définit un invariant cohomologique sur les formes de
rang 12 dans $I^3$ de la façon suivante : toute telle forme peut
s'écrire $q=\pfis{c}q'$ où $q'\in I^2(K)$, et on pose $a_5(q)=e_5(\pfis{c} \pi_2^2(q))$
(en utilisant notre notation). Bien entendu, l'ingrédient non trivial
est que $\pfis{c} \pi_2^2(q)$ est indépendant de la décomposition de $q$.

On souhaite, notamment en vue d'application aux algèbres à involution,
couvrir ce type de construction. Notamment, cet invariant $a_5$
peut se définir sur tout élément de $I^3$ qui s'écrit comme le produit
d'une 1-forme de Pfister et d'un élément de $I^2$ ; il est donc
naturel de travailler au niveau des classes de Witt.

De façon plus générale, on pose :

\begin{defi}\label{def_in_r}
  On pose $I^{n,r}(K) = \ens{\phi\cdot q}{\phi\in \Pf_r(K),\, q\in I^{n-r}(K)}$,
  pour tous $n\in \N^*$ et $0\ppq r < n$.
  
  En particulier, $I^{n,0} = I^n$.
\end{defi}

\begin{propdef}\label{prop_delta_inv}
  Il existe un unique morphisme de $A(k)$-modules filtrés
  \[  \foncdef{\Delta^{n,r}}{\Inv_{norm}(I^{n,r},A)}{\Inv_{norm}(I^{n-r},A)}{\alpha}{\alpha^{(r)},} \]
  injectif et de degré $-r$, tel que
  \[ \alpha(\phi\cdot q) = f_r(\phi)\cdot \alpha^{(r)}(q) \]
  pour tous $\alpha\in \Inv_{norm}(I^{n,r},A)$, $\phi\in \Pf_r(K)$ et $q\in I^{n-r}(K)$.
\end{propdef}

\begin{proof}
  Soient $\alpha\in \Inv(I^{n,r},A)$ et $q\in I^{n-r}(K)$. On peut alors définir
  un invariant sur $K$ des $r$-formes de Pfister par $\phi\mapsto \alpha(\phi\cdot q)$.
  D'après la condition \ref{cond_4} sur $A$, il y a
  d'uniques $x(q),y(q)\in A(K)$ tels que
  \[  \alpha(\phi\cdot q) = x(q) + f_n(\phi)\cdot y(q)  \]
  et par unicité ce sont des invariants de Witt de $I^{n-r}$,
  avec $x = \alpha(0)$ constant. On pose $\alpha^{(r)} := y$.
\end{proof}

Comme $I^{n,r}\subset I^n$, on peut se poser la question de la
restriction des invariants de $I^n$ à $I^{n,r}$. On a le résultat
élémentaire suivant :

\begin{prop}
Soient $n\in \N^*$ et $0\ppq r< n$. Alors pour tout $d\in \N^*$, 
\[  \Delta^{n,r}(f_n^d) = \{-1\}^{r(d-1)}f_{n-r}^d.       \]
\end{prop}

Cette proposition découle immédiatement, par récurrence sur le décalage, de
la proposition suivante, qui est indépendamment intéressante.

\begin{prop}
  Soit $\alpha\in \Inv_{norm}(I^n,A)$. Alors
  \[ (\alpha^{(r)})^+ = \{-1\}^r(\alpha^+)^{(r)}. \]
\end{prop}

\begin{proof}
  Soit $\alpha\in \Inv_{norm}(I^n,A)$. Si $\phi\in \Pf_r(K)$,
  $q\in I^{n-r}(K)$ et $\psi\in \Pf_{n-r}(K)$, on a
  \[ \alpha(\phi(q+\psi)) = \alpha(\phi q) + f_n(\phi\psi) \alpha^+(\phi q) = f_r(\phi)\alpha^{(r)}(q) + \{-1\}^rf_n(\phi\psi)(\alpha^+)^{(r)}(q) \]
  ainsi que
  \[ \alpha(\phi(q+\psi)) = f_r(\phi)\alpha^{(r)}(q+\psi) = f_r(\phi)\alpha^{(r)}(q) + f_r(\phi)f_{n-r}(\psi)(\alpha^{(r)})^+(q) \]
  d'où la formule.
\end{proof}

La question qu'on peut alors se poser est : l'application $\Delta^{n,r}$
est-elle surjective ? Autrement dit, tout invariant de $I^{n-r}$
permet-il de définir un invariant de $I^{n,r}$ comme on l'a illustré
avec $a_5$ ?

On ne dispose pas de réponse en toute généralité, mais on peut
quand même traiter un cas particulier :

\begin{prop}\label{prop_pi_delta}
Pour tout $n\pgq 2$, $\Delta^{n,1}$ est un isomorphisme.
\end{prop}

\begin{proof}
  Il suffit de montrer que $f_{n-1}^d$ est dans l'image pour tout $d\pgq 1$.
  On doit donc montrer que $\pfis{a}q\mapsto \{a\}f_{n-1}^d(q)$ est
  un invariant bien défini, autrement dit que si $q,q'\in I^{n-1}(K)$
  et $a,b\in K^*$, alors $\pfis{a}q = \pfis{b}q'$ implique que
  $\{a\}f_{n-1}^d(q) = \{b\}f_{n-1}^d(q')$.

  Supposons d'abord que $a=b$. Alors d'après \cite[thm 41.3]{EKM}, 
  \[ q - q' = \sum_{i\in J} \pfis{c_i}\phi_i \]
  où $c_i$ est représenté par $\pfis{a}$. On peut alors raisonner
  par récurrence sur $|J|$, et on est donc réduit au cas où
  $q' = q +\pfis{c}\phi$ avec $c$ représenté par $\pfis{a}$.

  Or pour tout $k\in \N^*$,  $f_{n-1}^k(\pfis{c}\phi)$ est
  un multiple de $\{c\}$, et donc $\{a\}f_{n-1}^k(\pfis{c}\phi) = 0$.
  De là :
  \[ \{a\}f_{n-1}^d(q') = \{a\}\sum_{k=0}^d f_{n-1}^k(q)f_{n-1}^{d-k}(\pfis{c}\phi) = \{a\}f_{n-1}^d(q). \]

  Supposons maintenant $a\neq b$. Alors d'après l'appendice B de \cite{Gar}, on a
  \[ \pfis{a}q = \pfis{a}\phi = \pfis{b}\phi = \pfis{b}q' \]
  avec $\phi = \sum_{i\in J} \fdiag{\lambda_i}\pfis{c_i}$, et $c_i$
  représenté par $\pfis{ab}$.

  La discussion précédente montre que $\{a\}f_{n-1}^d(q) = \{a\}f_{n-1}^d(\phi)$
  et \linebreak $\{b\}f_{n-1}^d(q) = \{b\}f_{n-1}^d(\phi)$, donc il suffit de montrer
  que $\{a\}f_{n-1}^d(\phi) = \{b\}f_{n-1}^d(\phi)$ pour tout $\phi$
  ayant une décomposition comme ci-dessus, ce qui constitue
  précisément l'énoncé du lemme ci-après.
\end{proof}

\begin{lem}
  Soient $a,b\in K^*$, et soit $q\in I^n(K)$ de la forme
  \[ q = \sum_{i=1}^r \fdiag{\lambda_i}\pfis{c_i} \]
  où $c_i$ est représenté par $\pfis{ab}$. Alors pour tout
  $d\pgq 1$, 
  \[ \{a\}f_n^d(q) = \{b\}f_n^d(q). \]
\end{lem}

\begin{proof}
  On se ramène au cas où $A=W$, puisque si
  \[ \pfis{a}\pi_n^d(q) = \pfis{b}\pi_n^d(q) \]
  alors on peut appliquer $f_{d+1}$ pour obtenir le résultat voulu.
  Dans le cas $A=W$, on montre une propriété plus forte : on
  suppose seulement $q\in I(K)$ (ce qui a un sens puisqu'on peut
  appliquer $\pi_n^d$ à n'importe quel élément de $GW(K)$).

  On raisonne par récurrence sur $r$. Si $r=0$, le résultat est trivial
  puisque $q=0$. Pour le cas $r=1$ : on écrit $\pi_n^d=\sum_{k=1}^d x_k\pi_1^k$
  avec $x_k\in \Z$, et $q=\fdiag{\lambda}\pfis{c}$ ; alors par
  hypothèse $\pfis{a,c}=\pfis{b,c}$ donc le résultat vaut clairement
  pour $k=1$, et si $k\pgq 2$ :
  \begin{align*}
    \pfis{a}\pi_1^k(\fdiag{\lambda}\pfis{c}) &= (-1)^k\pfis{a}\pfis{-1}^{k-2}\pfis{\lambda,c} \\
                                             &= (-1)^k\pfis{b}\pfis{-1}^{k-2}\pfis{\lambda,c} \\
                                             &= \pfis{b}\pi_1^k(\fdiag{\lambda}\pfis{c})
  \end{align*}
  et donc on conclut par combinaison linéaire.
 
  Si le résultat vaut jusqu'à $r\pgq 1$, on écrit
  $q=q_0+\fdiag{\lambda}\pfis{c}$, où $q_0$ correspond à une
  écriture de longueur $r$ et $c$ est représenté par $\pfis{ab}$.
  Alors par hypothèse de récurrence on a
  \[ \pfis{a}\pi_n^k(q_0) = \pfis{b}\pi_n^k(q_0),\quad \pfis{a}\pi_n^k(\fdiag{\lambda}\pfis{c})
    = \pfis{b}\pi_n^k(\fdiag{\lambda}\pfis{c})  \]
  donc
  \begin{align*}
    \pfis{a}\pi_n^d(q) &= \sum_{k=0}^d\pfis{a}\pi_n^k(q_0)\pi_n^{d-k}(\fdiag{\lambda}\pfis{c}) \\
                       &= \sum_{k=0}^d\pfis{b}\pi_n^k(q_0)\pi_n^{d-k}(\fdiag{\lambda}\pfis{c}) \\
                       &= \pfis{b}\pi_n^d(q).
  \end{align*}
\end{proof}

\section{Tours d'invariants cohomologiques}

On peut réinterpréter les invariants cohomologiques qu'on a construit dans 
le cadre suivant : on part de $I$, et on a un certain nombre d'invariants,
dont $e_1$ (qui peut d'ailleurs être caractérisé par le fait que c'est l'invariant
normalisé non nul de plus petit degré). On considère alors l'ensemble des formes
sur lesquelles cet invariant s'annule, à savoir $I^2$. On a également classifié
les invariants sur $I^2$, parmi lesquels $e_2$ (qui peut encore être caractérisé
par sa minimalité parmi les invariants normalisés), et on regarde encore les formes
sur lesquelles il s'annule, soit $I^3$. Et ainsi de suite.

On peut donc voir nos constructions comme définissant des \emph{tours d'invariants} :

\begin{defi}
Soit $F$ et $H$ deux foncteurs sur une catégorie $\mathbf{C}$ à valeurs
dans $\mathbf{Set}$ et $\mathbf{Set_*}$ respectivement (on note $0$ le point base
des ensembles pointés). 
Une tour d'invariant sur $F$ à valeurs dans $H$ est
une suite finie $(\alpha_1,\dots,\alpha_r)$ telle que $\alpha_i$ est un
invariant de $F_i$ à valeurs dans $H$, où $F_1=F$, et 
$F_{i+1}(X) = \ens{x\in F_i(X)}{\alpha_i(x)=0}$.
\end{defi}

Dans notre cas, un invariant cohomologique de $I^n$ peut être vu
comme s'insérant dans une tour d'invariants $(e_1,e_2,\dots,e_{n-1},\alpha)$
sur $I$.
\\

On peut alors se poser la question des tours d'invariants cohomologiques
générales sur $I$. Notamment, si $\alpha\in \Inv(I^n,\mu_2)$, quels sont
les invariants définis sur $I_\alpha(K)  = \ens{q\in I^n(K)}{\alpha(q)=0}$ ?
On a déjà la restriction des invariants définis sur $I^n$, mais en général
(comme nous le suggère l'exemple $\alpha = e_n$) il y en a bien d'autres.

On utilise l'idée suivante : supposons que $\alpha$ soit homogène de degré $d$. 
On peut choisir un invariant $\hat{\alpha}\in \Inv(I^n,I^d)$ tel que 
$\alpha = e_n\circ \hat{\alpha}$ (il est bien défini modulo $\Inv(I^n,I^{d+1})$).
Alors pour $q\in I^n(K)$, $\alpha(q)=0$ si et seulement si $\hat{\alpha}(q)\in I^{d+1}(K)$.
On peut alors définir $\beta(\hat{\alpha}(q))$ pour tout $\beta\in \Inv(I^{d+1},\mu_2)$,
ce qui définit un invariant de $I_\alpha$. Bien évidemment, cet invariant dépend
du choix de $\hat{\alpha}$. En revanche, on montre que l'ensemble des invariants
qu'on obtient ainsi en faisant varier $\beta$ ne dépend pas de ce choix.

Si on veut itérer la construction, on se heurte à une difficulté : si $\alpha_2$
est un invariant de $I_\alpha$, on ne sait pas garantir que $\alpha_2$ se relève
à un invariant de Witt. À défaut, on va se restreindre à ce type d'invariant :

\begin{defi}
Soit $F$ un foncteur de $\mathbf{Field}_{/k}$. On dit qu'un invariant 
$\alpha\in \Inv^d(F,\mu_2)$ est \emph{spécial} s'il existe 
$\tld{\alpha}\in \Inv(F,I^d)$ tel que $\alpha = e_d \circ \tld{\alpha}$.

Un invariant $\alpha\in \Inv(F,\mu_2)$ est spécial si toutes ses
composantes homogènes le sont. On note $\Inv_s(F,\mu_2)$ le module
des invariants spéciaux.

Une tour d'invariant est spéciale si tous les invariants de la tour le sont.
\end{defi}

\begin{rem}
  D'après le corollaire \ref{cor_relev}, tout invariant cohomologique
  de $I^n$ est spécial.
\end{rem}

On peut alors préciser notre construction :

\begin{prop}
Soit $F$ un foncteur de $\mathbf{Field}_{/k}$ vers $\mathbf{Set}$, 
et soit $\alpha\in \Inv^d_s(F,\mu_2)$
un invariant spécial. On pose $F_\alpha(K) = \ens{x\in F(K)}{\alpha(x)=0}$.
On choisit $\hat{\alpha}\in \Inv(F,I^d)$ tel que 
$\alpha = e_d \circ \hat{\alpha}$, et on pose 
\[ \foncdef{\Phi_{\hat{\alpha}}}{\Inv(I^{d+1},\mu_2)}{\Inv_s(F_\alpha,\mu_2)}{\beta}{\beta\circ \hat{\alpha}.} \]
Alors la sous-algèbre de $\Inv_s(F_\alpha,\mu_2)$ engendrée par $\Inv_s(F,\mu_2)$ et
par l'image de $\Phi_{\hat{\alpha}}$ ne dépend pas du choix de $\hat{\alpha}$.
\end{prop}

\begin{proof}
  Soit $\hat{\alpha}'$ un autre relevé de $\alpha$. Il suffit alors de montrer
  que $\Phi_{\hat{\alpha}'}(u_{m(d+1)}^{(d+1)})$ est dans l'algèbre engendrée
  par $\Inv(F,\mu_2)$ et l'image de $\Phi_{\hat{\alpha}}$.
  
  Soit $\gamma\in \Inv(F,I^{d+1})$ tel que $\hat{\alpha}'=\hat{\alpha}+\gamma$.
  On a pour tout $q\in F_\alpha(K)$ :
  \begin{align*}
    \Phi_{\hat{\alpha}'}(u_{m(d+1)}^{(d+1)}) &= u_{m(d+1)}^{(d+1)}\circ (\hat{\alpha}+\gamma) \\
                                             &= \sum_{k=0}^m\left( u_{(m-k)(d+1)}^{(d+1)}\circ \gamma\right) \cup \left( u_{k(d+1)}^{(d+1)}\circ \hat{\alpha} \right) \\
                                             &= \sum_{k=0}^m \gamma_{m-k} \cup \Phi_{\hat{\alpha}}(u_{k(d+1)}^{(d+1)})
  \end{align*}
  où $\gamma_k = u_{k(d+1)}^{(d+1)}\circ \gamma$ est un invariant spécial de $F$.  
\end{proof}

Comme cette construction donne des invariants qui sont eux-mêmes
spéciaux par définition, on peut itérer la construction pour définir
des tours d'invariants (spéciaux).

Dans le cas qui nous intéresse, on part de $F=I$. On part de n'importe quel
invariant homogène $\alpha$ de $I$, qui est forcément spécial. Le procédé
ci-dessus donne alors toute une algèbre d'invariants spéciaux de $I_\alpha$
(qui est canoniquement déterminée par $\alpha$). Si on en choisit un,
disons $\beta$, on peut encore utilise ce procédé pour construire des invariants
(spéciaux) de $I_{\alpha,\beta}$, et ainsi de suite.

Il se pose alors naturellement un certain nombre de questions : obtient-on
ainsi toutes les tours d'invariants de $I$ ? Toutes les tours d'invariants spéciales ?

\chapter{Anneaux de Witt mixtes}

L'objectif de ce chapitre est de définir des anneaux commutatifs
$\tld{GW}(A,\sigma)$ et $\tld{W}(A,\sigma)$ pour tout algèbre à
involution $(A,\sigma)$ de première espèce, qui contiennent
les classes d'isométrie des formes $\eps$-hermitiennes sur $(A,\sigma)$.
Précisément, on note $GW^\eps(A,\sigma)$ le groupe de Grothendieck-Witt
des formes $\eps$-hermitiennes sur $(A,\sigma)$, et
$GW^\pm(A,\sigma)=GW(A,\sigma)\oplus GW^-(A,\sigma)$, et alors
\[ \tld{GW}(A,\sigma) = GW^\pm(K)\oplus GW^\pm(A,\sigma). \]
L'idée est qu'on peut définir à travers une équivalence de Morita
canonique (on développe au préalable un formalisme de théorie de Morita hermitienne
dans la première partie) le produit de deux formes $\eps$-hermitiennes, mais que
l'élément résultant est alors une forme bilinéaire sur $K$ ; on
peut donc définir une structure stable par produit en rassemblant
ces deux types d'éléments dans un seul ensemble. On montre
dans le théorème \ref{thm_gw} qu'on obtient effectivement une
structure d'anneau commutatif en procédant de la sorte. On
définit également un analogue de l'anneau de Witt $\tld{W}(A,\sigma)$
en quotientant par les formes hyperboliques.

On définit également sur $\tld{GW}(A,\sigma)$ une structure de
pré-$\lambda$-anneau (ou d'anneau grec selon la terminologie du
premier chapitre), voir la proposition \ref{prop_lambda_mixte}.
Un point important est que toutes nos constructions sont naturellement
compatibles avec les équivalences de Morita hermitiennes (voir
la proposition \ref{prop_lambda_morita}). On étudie en particulier
le cas des algèbres de quaternions munies de leur involution
canonique, ainsi que le cas des algèbres déployées. Il est
à noter que ce cas des algèbres de quaternions avait en réalité
déjà été traité par Lewis dans \cite{Lew}, bien qu'il ne donne
pas de preuve de l'associativité de l'anneau (qui n'a rien d'une évidence)
et ne construise pas les $\lambda$-opérations.
Cependant les arguments de Lewis sont basés sur l'existence de la
forme norme des algèbres de quaternions, qui contrairement à la forme
trace d'involution ne se généralise pas au cas d'une algèbre à involution
quelconque.

Enfin, la dernière partie met en place une filtration \og fondamentale\fg{}
$\tld{I}^n(A,\sigma)$ de $\tld{W}(A,\sigma)$, par analogie avec le cas
des formes quadratiques (voir la définition \ref{def_filtr}), et définit
par analogie avec la conjecture de Minor une \og cohomologie mixte\fg{}
$\tld{H}^n(A,\sigma)$ comme le gradué associé. Cette filtration et ce gradué
sont alors étudiés, avec une dichotomie importante entre les cas déployé
et non déployé.

\section{Théorie de Morita hermitienne}\label{sec_morita}

La théorie de Morita classique admet un analogue hermitien
bien connu (voir par exemple \cite[§9]{Knu}). On peut encoder
cette théorie dans une structure assez compacte, qu'on
appelera le \emph{2-groupe de Brauer hermitien} de $K$.
Il s'agit d'une catégorie monoïdale symétrique, qui a de plus de
bonnes propriétés d'inversibilité. On n'utilisera pas la
notion de 2-groupe dans la suite, et on ne rentrera donc pas dans
les détails de la définition, mais on se servira en revanche
librement du formalisme des catégories monoïdales,
qui permet une unification et une simplification des notations.

\subsection{Le 2-groupe de Brauer}\label{sec_brauer}

En guise de motivation et d'introduction, on commence par
décrire la version non hermitienne, qui
correspond à la théorie usuelle des algèbres simples
centrales et du groupe de Brauer. On définit donc une
certaine catégorie $\CBr$ qu'on appelle le
2-groupe de Brauer de $K$ :

\begin{itemize}
\item Les objets sont les algèbres
  simples centrales sur $K$. Le produit monoïdal de $\CBr$
  est simplement le produit tensoriel des algèbres sur $K$, l'objet
  unité étant l'algèbre triviale $K$.

\item Si $A$ et $B$ sont deux algèbres simples centrales
  sur $K$, un morphisme de $B$ vers $A$ dans $\CBr$
  est une classe d'isomorphisme de $B$-$A$-bimodule $U$ sur $K$,
  (donc une classe d'isomorphisme de
  $B\otimes_K A^{op}$-module) tel que
   $B\simeq \End_A(U)$, ce qui est équivalent à
  $A\simeq \End_B(U)$ (où comme on en a l'habitude, si $U$ est
  un $B$-module à gauche $\End_B(U)$ désigne l'algèbre opposée
  à l'algèbre d'endomorphismes \og{} naïve\fg{}, de sorte à ce
  que $\End_B(U)$ agisse naturellement à \emph{droite} sur $U$).

  La composition est donnée
  par le produit tensoriel des modules : la composition de
  $C\xrightarrow U B$ et $B\xrightarrow V A$
  est $C\xrightarrow {U\otimes_B V} A$. L'identité de $A$ est
  $A$ lui-même vu canoniquement comme un $A$-$A$-bimodule.
  L'action du produit monoïdal sur les morphismes est
  celle qu'on attend : si on dispose de $B\xrightarrow U A$
  et $B'\xrightarrow V A'$ alors on obtient
  $B\otimes_K B'\xrightarrow {U\otimes_K V} A\otimes_K A'$.
\end{itemize}

On voit aisément le lien avec la théorie de Morita usuelle :
il existe un morphisme entre $A$ et $B$ si et seulement si
$A$ et $B$ sont Brauer-équivalentes, ce qu'on notera $A\sim B$,
et la donnée d'un morphisme spécifique correspond à une donnée de Morita.
Les propriétés d'associativité et de symétrie requises correspondent aux
propriétés usuelles des produits tensoriels. 
\\

\begin{rem}\label{rem_cat_mor}
  On peut observer que la définition s'étend manifestement
  pour tout anneau commutatif $R$ à une catégorie $\mathbf{Mor}(R)$
  dont les objets sont des $R$-algèbres quelconques, et les
  morphismes des bimodules quelconques.
\end{rem}

Le fait de restreindre la définition comme on l'a fait correspond à considérer des éléments
\og inversibles\fg{}, ce qui justifie la terminologie de \og groupe\fg{}.
La question de l'inversibilité des objets dans les catégories monoïdales
est délicate (voir par exemple \cite{BL}); on évitera donc de se perdre dans
des considérations générales, et on donne simplement les notations qui
nous seront utiles.

Si $U$ est un $B$-$A$-bimodule, alors $\Hom_A(U,A)$ et $\Hom_B(U,B)$
sont naturellement munis d'une structure de $A$-$B$-bimodule, définie
respectivement par $(a\phi b)(x) = a\phi(bx)$ et $(a\phi b)(x) = \phi(xa)b$.
Alors on a un isomorphisme naturel
\begin{equation}\label{eq_isom_dual}
 \anonisomdef{\Hom_B(U,B)}{\Hom_A(U,A)}{\phi}{\left( u \mapsto (v\mapsto \phi(v)u)\right)}
\end{equation}
où on utilise l'identification $A\simeq \End_B(U)$. On considèrera
cet isomorphisme comme une identification, et on utilise $U^*$
pour dénoter l'un ou l'autre de ces bimodules, indifféremment.
On dispose donc d'un $B$-$B$-bimodule $U\otimes_A U^*$, et d'un $A$-$A$-bimodule
$U^*\otimes_B U$. Un résultat classique de théorie de Morita affirme
alors que si $U$ est un bimodule simple, alors $U^*$ aussi et on a des
isomorphismes naturels de bimodules $U\otimes_A U^*\simeq B$ et
$U^*\otimes_B U\simeq A$. On peut donc dire que $U^*$ est un inverse
de $U$ dans $\CBr$, puisque la composition de $U$
et $U^*$ (vus comme morphismes) dans un sens comme dans l'autre
est isomorphe (donc égale) à l'identité. 

Si $A$ est un objet de $\CBr$, alors l'algèbre opposée
$A^{op}$ est un inverse faible de $A$ pour la structure monoïdale,
au sens où on dispose d'un isomorphisme naturel $A\otimes_K A^{op}\simeq K$
dans $\CBr$, donnée par le bimodule $A$ muni à gauche de l'action
\og sandwich\fg{} de $A\otimes_K A^{op}$, et à droite de l'action
tautologique de $K$.


\subsection{Le 2-groupe de Brauer hermitien}\label{sec_brauer_herm}

On présente maintenant la version hermitienne de la construction
précédente, qu'on note $\CBrh$.

\begin{itemize}
\item Les objets sont les algèbres
  à involution $(A,\sigma)$ de première espèce sur $K$.
  Le produit monoïdal de $\CBrh$
  est encore le produit tensoriel sur $K$, à savoir
  $(A,\sigma)\otimes (B,\tau) = (A\otimes_KB, \sigma\otimes \tau)$,
  et l'objet unité est $(K,\Id)$.

\item Un morphisme de $(B,\tau)$ vers $(A,\sigma)$ est une classe d'isométrie de
  $B$-$A$-bimodule simple $\eps$-\emph{hermitien}, à savoir un couple
  $(U,h)$ où $U$ est un $B$-$A$-bimodule simple, et $h: U\times U\To A$
  est une forme $\eps$-hermitienne (où $\eps=\pm 1$) relativement à $\sigma$ telle que
  à travers l'identification $B = \End_A(U)$ l'involution adjointe
  de $h$ soit $\tau$.

  La composition de $(C,\theta)\xrightarrow {(U,h)} (B,\tau)$
  et $(B,\tau)\xrightarrow {(V,g)} (A,\sigma)$
  est \[ (C,\theta) \xrightarrow {(U\otimes_B V,f)} (A,\sigma) \]
  où
  \begin{equation}\label{eq_compo_herm}
    f(u\otimes v, u'\otimes v') = g(v, h(u,u')v').
  \end{equation}
  Si $h$ est $\eps$-hermitienne et $g$ $\eps'$-hermitienne, alors
  $f$ est $\eps\eps'$-hermitienne.
  L'identité de $(A,\sigma)$ est le $A$-$A$-bimodule $A$, muni de
  \begin{equation}\label{eq_ha}
    \foncdef{h_A}{A\times A}{A}{(x,y)}{\sigma(x)y}.
  \end{equation}
  L'action du produit monoïdal sur les morphismes est donnée
  naturellement par le produit tensoriel.
\end{itemize}

On a de façon évidente un foncteur canonique (\og d'oubli\fg{})
de $\CBrh$ vers $\CBr$, qui est monoïdal.

Le fait que la composition des morphismes dans $\CBrh$ soit bien définie
et associative est une simple reformulation de la théorie
de Morita hermitienne telle qu'elle est présentée dans \cite{Knu}.

\begin{rem}
  Il existe un (iso)morphisme entre $(A,\sigma)$ et $(B,\tau)$ si
  et seulement si $A\sim B$ (voir par exemple \cite[4.2]{BOI}) ;
  si $\sigma$ et $\tau$ sont de même type, le bimodule
  donnant le morphisme sera hermitien, et anti-hermitien si $\sigma$
  et $\tau$ sont de type différent. De plus, tandis que dans
  $\CBr$ on a au plus un morphisme entre deux objets, ici
  si $A\sim B$ deux morphismes entre $A$ et $B$ correspondent
  à deux modules hermitiens semblables, et en général
  non isomorphes.
\end{rem}

Les propriétés d'inversibilité dans $\CBrh$ sont en un certain sens
plus fortes que dans $\CBr$, reflétant un phénomène d'auto-dualité
caractéristique des théories hermitiennes et des involutions.

En effet, comme inverse faible de $(A,\sigma)$, on peut choisir  $(A,\sigma)$
lui-même : on dispose d'une équivalence $(A\otimes_K A,\sigma\otimes \sigma) \simeq (K,\Id)$
donnée par le bimodule $A$, où l'action de $K$ à droite est tautologique, et
$A\otimes_K A$ agit à gauche par l'action sandwich tordue par $\sigma$ :
\begin{equation}
  (a\otimes b)\cdot x = ax\sigma(b).
\end{equation}
La structure hermitienne, correspondant en l'occurence à une forme bilinéaire
symétrique, est donnée par la forme trace d'involution (voir \cite[11.1]{BOI})
\begin{equation}\label{eq_t_sigma}
  \foncdef{T_\sigma}{A\times A}{K}{(x,y)}{\Trd_A(\sigma(x)y)}.
\end{equation}

Quant aux inverses de morphismes, si $(B,\tau)\xrightarrow {(U,h)} (A,\sigma)$
est un morphisme dans $\CBrh$, il admet comme inverse $(\mbox{}^\sigma U^\tau,h^*)$,
où $\mbox{}^\sigma U^\tau$ est le $A$-$B$-bimodule donné par $U$ muni de l'action tordue
\begin{equation}\label{eq_u_tordu}
a\cdot u\cdot b= \tau(b)\cdot u\cdot \sigma(a),
\end{equation}
et
\[ h^*:\mbox{}^\sigma U^\tau\times \mbox{}^\sigma U^\tau\To B \]
est définie par
\begin{equation}\label{eq_h_star}
 h^*(x,y)u = xh(y,u). 
\end{equation}

\begin{rem}\label{rem_cat_mor_herm}
  Dans la lignée de la remarque \ref{rem_cat_mor},
  On peut étendre cette définition à une catégorie $\mathbf{Mor}_h(R,\iota)$
  où $(R,\iota)$ est un anneau commutatif muni d'une involution, auquel
  cas les objets sont les $R$-algèbres $A$ munies d'une involution $\sigma$
  telles que $\sigma_{|R}=\iota$. On perd alors les propriétés d'inversibilité,
  mais on obtient toujours une catégorie monoïdale. Notons que dans
  ce contexte une forme peut être $\eps$-hermitienne sur $(A,\sigma)$
  pour tout $\eps\in Z(A)^*$ tel que $\eps\sigma(\eps)=1$.
\end{rem}

Si on dispose de deux morphismes $f$ et $f'$ vers $(A,\sigma)$ correspondant
à des modules $\eps$-hermitiens, on peut former la somme directe orthogonale
$f\oplus f'$, qui est encore un morphisme vers $(A,\sigma)$ correspondant
à un module $\eps$-hermitien.

\begin{lem}\label{lem_somme_compo}
  Soient $f,f'$ morphismes dans $\CBrh$ vers $(B,\tau)$ de même signe,
  et soit $g:(B,\tau)\To (A,\sigma)$. Alors
  \[ (f\oplus f')\circ g = (f\circ g)\oplus (f'\circ g). \]
\end{lem}

\begin{proof}
  Cela découle directement de la définition à partir du produit tensoriel.
\end{proof}

\section{L'anneau de Witt mixte}

\subsection{Anneau de Grothendieck-Witt mixte}

\subsubsection*{Définition}

On définit d'abord l'analogue de l'anneau de Grothendieck-Witt, qui
est en un sens plus fondamental que l'anneau de Witt, et sur lequel on définira
plus tard les opération $\lambda^d$. Comme dans le cas de l'anneau
de Grothendieck-Witt usuel, on le définit à partir d'un certain
semi-anneau, qui dans notre cas est gradué. Ce semi-anneau n'est
pas le plus intéressant en soi, mais il est pratique de l'utiliser
pour les preuves car tous les éléments sont représentés par des modules
hermitiens, puis de transférer les résultats à l'anneau.

\begin{defi}\label{def_gw_mixte}
  Soit $(A,\sigma)$ une algèbre à involution de première espèce
  sur $K$. On définit les monoïdes additifs commutatifs suivants,
  tous munis de la somme orthogonale :
  $SGW(K)$ est constitué des classes d'isométrie de modules bilinéaires symétriques non
  dégénérés sur $K$, $SGW^-(K)$ est constitué des classes d'isométrie de modules bilinéaires
  alternés, $SGW^+(A,\sigma)$ des classes d'isométrie de modules hermitiens (à droite) sur
  $(A,\sigma)$, et $SGW^-(A,\sigma)$ des classes d'isométrie de modules anti-hermitiens.

  On pose alors
  \[ \tld{SGW}(A,\sigma) = SGW(K)\oplus SGW^-(K) \oplus SGW^+(A,\sigma) \oplus SGW^-(A,\sigma) \]
  ainsi que son groupe de Grothendieck, qu'on appelle le groupe de Grothendieck-Witt \emph{mixte} de $(A,\sigma)$ :
  \[ \tld{GW}(A,\sigma) = GW(K)\oplus GW^-(K) \oplus GW^+(A,\sigma) \oplus GW^-(A,\sigma). \]
\end{defi}

Pour l'instant $\tld{SGW}(A,\sigma)$ est seulement un monoïde additif,
et $\tld{GW}(A,\sigma)$ un $GW(K)$-module. Notons que comme les
modules $\eps$-hermitiens sur une algèbre à involution vérifient
le théorème de décomposition de Witt, $\tld{SGW}(A,\sigma)$ est
un monoïde à simplification, et donc le morphisme naturel
$\tld{SGW}(A,\sigma) \To \tld{GW}(A,\sigma)$ est injectif.

\subsubsection*{Graduation}

On définit le groupe
\begin{equation}\label{eq_gamma}
  \Gamma = \{+, -, o, s\},
\end{equation}
isomorphe au groupe de Klein $(\Z/2\Z)^2$ de sorte que $+$ soit l'élément
neutre. Les symboles sont à interpréter respectivement comme \emph{symétrique},
\emph{alterné}, \emph{orthogonal} et \emph{symplectique}, correspondant
aux différents types de formes bilinéaires et d'involutions.

\label{par_grad}On munit alors $\tld{SGW}(A,\sigma)$ d'une $\Gamma$-graduation :
la composante \linebreak $\tld{SGW}(A,\sigma)_+$ est $SGW(K)$, la composante
$\tld{SGW}(A,\sigma)_-$ est $SGW^-(K)$. La composante $\tld{SGW}(A,\sigma)_o$
est $SGW^+(A,\sigma)$ si $\sigma$ est orthogonale et $SGW^-(A,\sigma)$
si $\sigma$ est symplectique, de sorte que si un élément de
$\tld{SGW}(A,\sigma)_o$ représente un certain $A$-module $\eps$-hermitien $(U,h)$,
alors l'involution adjointe $\sigma_h$ est orthogonale. Inversement,
l'involution adjointe d'un élément de $\tld{SGW}(A,\sigma)_s$ est symplectique.

Cette graduation s'étend de façon unique à $\tld{GW}(A,\sigma)$ (voir
lemme \ref{lem_semi_grad}).

\subsubsection*{Produit}

On veut maintenant donner une structure de semi-anneau gradué à \linebreak $\tld{SGW}(A,\sigma)$.
Il suffit de définir le produit et de vérifier distributivité et associativité
composante par composante vis-à-vis de la $\Gamma$-graduation : on veut
des applications $\tld{SGW}(A,\sigma)_{\gamma_1}\times \tld{SGW}(A,\sigma)_{\gamma_2} \To
\tld{SGW}(A,\sigma)_{\gamma_1 + \gamma_2}$ pour tous $\gamma_1,\gamma_2\in \Gamma$.

Comme les éléments de $\tld{SGW}(A,\sigma)_{\gamma_i}$ correspondent à des
modules $\eps_i$-hermitiens sur une certaine algèbre à involution $(A_i,\sigma_i)$
(qui vaut $(A,\sigma)$ si $\gamma_i$ est $o$ ou $s$, et $(K,\Id)$ si $\gamma_i$
est $+$ ou $-$), donc à des morphismes vers $(A_i,\sigma_i)$ dans $\CBrh$,
on peut déjà utiliser la structure monoïdale de $\CBrh$ pour définir un
morphisme vers $(A_1\otimes_K A_2,\sigma_1\otimes \sigma_2)$. Si un des $A_i$
est égal à $K$, on obtient alors bien un élément de $\tld{SGW}(A,\sigma)$.

Il reste donc à traiter le cas $\gamma_1,\gamma_2\in \{o,s\}$. Dans ce cas
on obtient un morphisme vers $(A\otimes_K A,\sigma\otimes \sigma)$ ; or on
dispose d'un isomorphisme canonique dans $\CBrh$ de cet objet vers
$(K,\Id)$, à savoir l'espace bilinéaire symétrique $(A,T_\sigma)$
(voir (\ref{eq_t_sigma})), qui nous permet de transporter par composition des
morphismes vers $(A\otimes_K A,\sigma\otimes \sigma)$ sur des morphismes
vers $(K,\Id)$, et donc des éléments de $\tld{SGW}(A,\sigma)$. On explicite
cette définition :

\begin{defi}
  Soient $(U,h)$ et $(V,g)$ des modules respectivement $\eps$-hermitien
  et $\eps'$-hermitien à droite sur $(A,\sigma)$. Alors $(U,h)\cdot (V,g)$,
  qu'on notera souvent $h\cdot g$, est le $K$-espace bilinéaire donné par
  \begin{equation}
    \foncdef{h\cdot g}{W \times W}{K}
    {\left( (u\otimes v)\otimes a, (u'\otimes v')\otimes a'\right)}
    {\Trd_A\left( \sigma(a) h(u,u') a' \sigma(g(v,v')) \right)}
  \end{equation}
  où $W = (U\otimes_K V)\otimes_{A\otimes_K A} A$.
\end{defi}

La définition vient de la formule (\ref{eq_compo_herm}), sachant que 
\[   h(u,u')\otimes g(v,v')\cdot a' =  h(u,u') a' \sigma(g(v,v')).  \]

\begin{lem}\label{lem_biadd}
  Le produit ainsi défini est commutatif et biadditif, et s'étend de
  façon unique en un produit commutatif et distributif sur $\tld{SGW}(A,\sigma)$
  et $\tld{GW}(A,\sigma)$.
\end{lem}

\begin{proof}
  Soient $(V_i,h_i)$ modules $\eps_i$-hermitiens sur $(A,\sigma)$, pour
  $i=1,2,3$. Alors on définit une $K$-isométrie de $(W,f)=(V_1,h_1)\cdot (V_2,h_2)$
  vers $(W',f')=(V_2,h_2)\cdot (V_1,h_1)$ par
  \[ \anonfoncdef{W}{W'}{(v_1\otimes v_2)\otimes a}{(v_2\otimes v_1)\otimes \sigma(a)} \]
  ce qui montre la commutativité (on omet de montrer que cela définit
  bien une isométrie, ce qui résulte d'un déroulage immédiat des définitions).
  
  On a également
  \[ \left( (V_1,h_1)\oplus (V_2,h_2) \right) \otimes_K (V_3,h_3) \simeq (V_1\otimes_KV_3,h_1\otimes h_3)
    \oplus (V_2\otimes_KV_3,h_1\otimes h_3)  \]
  donc on obtient la biadditivité par application du lemme \ref{lem_somme_compo}.
  On a finalement bien construit un produit sur toutes les composantes
  de $\tld{SGW}(A,\sigma)$, et la commutativité et la distributivité
  sont vérifiées composantes par composantes.
  L'assertion sur $\tld{GW}(A,\sigma)$ s'en déduit par propriété
  universelle du groupe de Grothendieck.
\end{proof}

\begin{rem}
  Par précomposition avec $T_A$ dans $\CBrh$, la commutativité démontrée
  ci-dessus implique qu'on a directement une isométrie
  \[ (V_1\otimes_K V_2,h_1\otimes h_2) \simeq  (V_2\otimes_K V_1,h_2\otimes h_1) \]
  en tant que $(A\otimes_K A,\sigma\otimes\sigma)$-modules hermitiens, mais
  la démonstration est sensiblement plus pénible, et exhiber une isométrie
  explicite nécessite de recourir à l'élément de Goldman de $A$. Il est à
  noter qu'en général si $A$ n'est pas munie d'une involution, et si $V$ et
  $W$ sont des $A$-modules, on n'a pas de raison
  d'avoir d'isomorphisme naturel de $A\otimes_K A$-modules entre $V\otimes_K W$
  et $W\otimes V$ (mais ce qu'on a montré est que dès qu'on dispose d'une involution,
  alors non seulement on dispose de cet isomorphisme, mais on peut rajouter gratuitement
  des formes hermitiennes et obtenir une isométrie).
\end{rem}

\subsubsection*{Associativité}

Par utilisation successive de l'isomorphisme canonique entre
$(A\otimes_K A,\linebreak \sigma\otimes \sigma)$ et $(K,\Id)$, on obtient
pour tout $d\in \N$ un isomorphisme dans $\CBrh$ de
$(A,\sigma)^{\otimes d}$ vers $(A,\sigma)$ si $d$ est impair,
et vers $(K,\Id)$ si $d$ est pair. A priori, cet isomorphisme
dépend de la façon dont on regroupe les termes, et il faut
vérifier une certaine condition d'associativité.

On aura besoin par la suite d'étudier le rôle de l'élément
de Goldman (voir \cite[3.5]{BOI}) dans notre contexte.

\begin{lem}\label{lem_action_goldman}
  On considère l'action tordue à gauche de $A\otimes_K A$ sur
  le $K$-espace vectoriel $A$, donnée par $(a\otimes b)\cdot x = ax\sigma(b)$.
  Alors l'élément de Goldman $g\in A\otimes_K A$ vérifie, pour tout $x\in A$,
  $(\Id\otimes \sigma)(g)\cdot x=\Trd_A(x)$ et $g\cdot x = \eps \sigma(x)$
  où $\eps=1$ si $\sigma$ est orthogonale et $-1$ si elle est symplectique.
  Par ailleurs, $(\sigma\otimes \sigma)(g)=g$, et donc
  $(\sigma\otimes \Id)(g)\cdot x=\Trd_A(x)$ et
  $(\sigma\otimes \sigma)(g)\cdot x = \eps \sigma(x)$.
\end{lem}

\begin{proof}
  La première formule est essentiellement la définition de $g$,
  puisqu'il est défini pour agir comme la trace pour l'action sandwich
  non tordue. Pour la deuxième, on se ramène au cas où $A$ est déployée.
  Alors $A=\End_K(V)$ pour un certain espace vectoriel $V$, et $\sigma=\sigma_b$
  où $b$ est $\eps$-symétrique. On dispose alors de l'isomorphisme
  canonique $A\simeq V\otimes V$, tel que
  $(v\otimes w)(v'\otimes w') = b(w,v')v\otimes w'$ et
  $\sigma(v\otimes w)=\eps w\otimes v$. Soit $(e_i)$ une base de $V$,
  et $(e_i^*)$ sa base duale pour $b$ (si $\sigma$ est orthogonale,
  on peut prendre pour $(e_i)$ une base orthogonale, auquel cas
  $e_i=e_i^*$). Alors $g=\sum_{i,j}e_i\otimes e_j^*\otimes e_j\otimes e_i^*$
  (voir la preuve de \cite[3.6]{BOI}), donc
  \begin{align*}
    g\cdot (v\otimes w) &= \sum_{i,j} (e_i\otimes e_j^*)\cdot (v\otimes w)\cdot \eps(e_i^*\otimes e_j) \\
                        &= \eps\sum_{i,j}b(e_j^*,v)b(w,e_i^*)e_j\otimes e_j \\
                        &= \left( \sum_i b(e_i^*,w)e_i\right)\otimes \left( \sum_j b(e_j^*,v)e_j\right) \\
                        &= w\otimes v.
  \end{align*}
  Pour la dernière affirmation, on renvoie à \cite[10.19]{BOI}.
\end{proof}

\begin{prop}\label{prop_prod_red}
  L'isomorphime $\phi_{(A,\sigma)}^{(d)}$ de $(A,\sigma)^{\otimes d}$ vers $(A,\sigma)$
  ou $(K,\Id)$ dans $\CBrh$ est canonique et ne dépend pas de l'ordre de
  regroupement des termes.
\end{prop}

\begin{proof}
  Clairement, comme pour toute forme d'associativité il suffit de le vérifier
  pour trois termes.
  Il faut donc montrer qu'on a un carré commutatif dans $\CBrh$
  \[ \begin{tikzcd}
      (A\otimes_K A\otimes_K A, \sigma\otimes\sigma\otimes\sigma) \rar \dar &
      (K\otimes A, \Id\otimes \sigma) \dar \\
      (A\otimes K, \sigma\otimes \Id) \rar & (A,\sigma)
    \end{tikzcd} \]
  ce qui revient à construire
  une isométrie entre les bimodules hermitiens $({A\otimes_K A},h_1)$
  et $(A\otimes_K A,h_2)$ où l'action de $A\otimes_KA\otimes_K A$ à gauche
  et de  $A$ à droite est dans le premier cas
  \[ (a\otimes b\otimes c) \cdot_1 (x\otimes y)\cdot_1 d = (ax\sigma(b))\otimes (cyd)  \]
  et dans le deuxième cas
  \[ (a\otimes b\otimes c) \cdot_2 (x\otimes y)\cdot_2 d =  (axd)\otimes (by\sigma(c)) \]
  et où les formes hermitiennes sont
  \[ h_1(x\otimes y, x'\otimes y') = \Trd_A(\sigma(x)x')\sigma(y)y' \]
  et
  \[ h_2(x\otimes y, x'\otimes y') = \Trd_A(\sigma(y)y')\sigma(x)x'.  \]

  Soit $g\in A\otimes_K A$ l'élément de Goldman ;
  si $g = \sum_i a_i\otimes b_i$, on définit notre isométrie de $h_1$ vers $h_2$ par
  \[ \foncdef{\phi}{A\otimes_K A}{A\otimes_K A}{x\otimes y}
    {\sum_i (xa_iy)\otimes \sigma(b_i).} \]
  Notons que $1\otimes 1$ est envoyé sur $g'= (\Id\otimes \sigma)(g)$. Comme
  le bimodule $A\otimes_K A$ est engendré par l'élément $1\otimes 1$
  pour les deux action définies ci-dessus, cela suffit à caractériser $\phi$,
  mais il faut vérifier que $g'$ est un élément admissible pour l'image
  de $1\otimes 1$, ce qui revient à vérifier qu'on a bien défini un
  morphisme de bimodules. On a
  \begin{align*}
    \phi\left( (a\otimes b\otimes c)\cdot_1(x\otimes y)\cdot_1 d \right)
    &= \phi\left( (ax\sigma(b))\otimes (cyd) \right) \\
    &= \sum_i (ax\sigma(b)a_icyd)\otimes \sigma(b_i)
  \end{align*}
  et
  \begin{align*}
    (a\otimes b\otimes c)\cdot_2 \phi(x\otimes y)\cdot_2 d
    &= \sum_i  (a\otimes b\otimes c)\cdot_2 (xa_iy)\otimes \sigma(b_i) \cdot_2 d \\
    &= \sum_i (axa_iyd)\otimes (b\sigma(b_i)\sigma(c)).
  \end{align*}
  On est donc amené à montrer que
  \[ \sum_i \sigma(b)a_ic \otimes \sigma(b_i) = \sum_i a_i\otimes b\sigma(b_i)\sigma(c), \]
  ce qui revient à
  \[ (\Id\otimes \sigma)((\sigma(b)\otimes 1)g(c\otimes 1))
    = (\Id\otimes \sigma)((1\otimes c)g(1\otimes \sigma(b))) \]
  qui découle des propriétés de l'élément de Goldman (voir \cite[3.6]{BOI}, ou
  le lemme \ref{lem_goldman_commut}). Donc $\phi$ définit bien un morphisme
  de bimodules, et comme $g'$ est inversible c'est bien un isomorphisme. 


  Il reste à vérifier que $\phi$ réalise une isométrie de $h_1$ vers $h_2$.
  On a
  \begin{align*}
    h_2(\phi(x\otimes y),\phi(x'\otimes y'))
    &= \sum_{i,j}h_2((xa_iy)\otimes \sigma(b_i), (x'a_jy')\otimes \sigma(b_j)) \\
    &= \sum_{i,j}\Trd_A(b_i\sigma(b_j))\sigma(y)\sigma(a_i)\sigma(x)x'a_jy' \\
    &= \sum_{i,j,k}\sigma(y)\sigma(a_i)\sigma(x)x'a_jy'a_kb_i\sigma(b_j)b_k \\
    &= \eps \sum_{i,k}\sigma(y)\sigma(a_i)\sigma(x)x'\sigma(b_i)\sigma(a_k)\sigma(y')b_k \\
    &= \eps \sigma(y)\Trd_A(\sigma(x)x')\eps y' \\
    &= h_1(x\otimes y,x'\otimes y')
  \end{align*}
  en utilisant les relations du lemme \ref{lem_action_goldman}.

\end{proof}

En conséquence :

\begin{thm}\label{thm_gw}
  L'application $\tld{SGW}(A,\sigma)\times \tld{SGW}(A,\sigma)\To \tld{SGW}(A,\sigma)$
  définie dans la partie précédente  munit $\tld{SGW}(A,\sigma)$ d'une structure de
  semi-anneau commutatif gradué, et induit sur $\tld{GW}(A,\sigma)$ une structure de
  $GW(K)$-algèbre commutative graduée.
\end{thm}

\begin{proof}
  Par propriété universelle des groupes de Grothendieck (voir notamment le lemme
  \ref{lem_semi_grad}), il suffit de montrer
  les affirmations sur \linebreak $\tld{SGW}(A,\sigma)$. Tout ce qu'il reste à démontrer
  pour ça est l'associativité, qui se vérifie composante par composante
  pour la $\Gamma$-graduation.
  
  On pose donc pour $1\ppq i\ppq 3$ $(U_i,h_i)$ module $\eps_i$-hermitien
  à droite sur $(A_i,\sigma_i)$, où $(A_i,\sigma_i)$ est soit $(A,\sigma)$
  soit $(K,\Id)$. Par construction du produit, $(h_1\cdot h_2)\cdot h_3$
  et $h_1\cdot (h_2\cdot h_3)$ sont chacun obtenus à partir d'un module
  $\eps_1\eps_2\eps_3$-hermitien sur $(A_1\otimes A_2\otimes A_3,
  \sigma_1\otimes\sigma_2\otimes\sigma_3)$, transporté par composition
  à travers certains isomorphismes dans $\CBrh$. De plus ces deux modules
  hermitiens sont isométriques par associativité du produit tensoriel.
  Il faut donc montrer qu'on a un carré commutatif dans $\CBrh$
  \[ \begin{tikzcd}
      (A_1\otimes A_2\otimes A_3, \sigma_1\otimes\sigma_2\otimes\sigma_3) \rar \dar &
      (B_1\otimes A_3, \tau_1\otimes \sigma_3) \dar \\
      (A_1\otimes B_2, \sigma_1\otimes \tau_2) \rar & (C,\theta)
  \end{tikzcd} \]
  où les $(B_i,\tau_i)$ et $(C,\theta)$ sont soit $(A,\sigma)$ soit $(K,\Id)$
  selon la nature des $(A_i,\sigma_i)$, et les morphismes sont les identifications
  canoniques de la forme $(A,\sigma)\otimes (K,\Id)\simeq (A,\sigma)$,
  $(K,\Id)\otimes (K,\Id)\simeq (K,\Id)$ ou $(A,\sigma)\otimes (A,\sigma)\simeq (K,\Id)$.

  Le cas où un des $A_i$ vaut $K$ est simplement une conséquence du fait
  que $(K,\Id)$ est une identité pour la structure monoïdale de $\CBrh$,
  et le cas où tous les $A_i$ sont $A$ se ramène exactement à la proposition
  \ref{prop_prod_red}.
\end{proof}

Par construction, pour toute extension $L/K$ on a un morphisme de $GW(K)$-algèbres
graduées $\tld{GW}(A,\sigma)\To \tld{GW}(A_L,\sigma_L)$.

\begin{ex}\label{ex_prod_diag}
  Dans $\tld{GW}(A,\sigma)$, on a certains éléments privilégiés qui
  sont les formes diagonales $\fdiag{a}$ où $a\in A$ est un élément
  inversible symétrique ou antisymétrique. On notera parfois $\fdiag{a}_\sigma$
  ou $\fdiag{a}_A$ s'il y a un risque de confusion.
  Si $A$ est à division, ces
  formes engendent additivement $GW^{\pm}(A,\sigma)$. Alors
  par construction du produit :
  \[ \fdiag{a}\cdot \fdiag{b} = \left( (x,y)\mapsto \Trd_A(\sigma(x)ay\sigma(b)) \right). \]
  En particulier, on dispose de $\fdiag{1}_\sigma\in GW^+(A,\sigma)$, qui est
  l'identité de $A$ dans $\CBrh$. Alors
  \[ \fdiag{a}\cdot \fdiag{1} = T_{\sigma,a} \]
  où $T_{\sigma,a}$ est la forme bilinéaire définie dans \cite[§11]{BOI} par
  \begin{equation}\label{eq_t_sigma_a}
    T_{\sigma,a}(x,y) = \Trd_A(\sigma(x)ay). 
  \end{equation}
  Notamment, $\fdiag{1}_\sigma^2=T_\sigma$, la forme trace d'involution.
\end{ex}

\subsection{Anneau de Witt mixte}

On peut également former l'équivalent de l'anneau de Witt, qui est
généralement plus facilement manipulable, bien qu'il ait moins de structure.

\begin{defi}\label{def_w_mixte}
  Soit $(A,\sigma)$ une algèbre à involution de première espèce sur $K$. On
  définit le $W(K)$-module
  \[ \tld{W}(A,\sigma) = W(K) \oplus W^+(A,\sigma) \oplus W^-(A,\sigma). \]
\end{defi}

\label{par_grad_w}On munit $\tld{W}(A,\sigma)$ d'une $\Gamma$-graduation analogue à celle
de $\tld{GW}(A,\sigma)$, avec une composante nulle pour l'élément $-\in \Gamma$ (puisque
$W^-(K)=0$).

On dispose manifestement d'un morphisme surjectif canonique de $GW(K)$-modules gradués 
$\tld{GW}(A,\sigma)\to \tld{W}(A,\sigma)$ donné par la projection naturelle
sur chaque composante.

\begin{prop}
  Le noyau de l'application canonique $\tld{GW}(A,\sigma)\to \tld{W}(A,\sigma)$
  est l'idéal $\Hyp(A,\sigma)$ engendré par les formes hyperboliques dans chaque composante.
  En particulier, $\tld{W}(A,\sigma)$ est muni d'une unique structure d'anneau
  telle que l'application précédente soit un morphisme d'anneau.
\end{prop}

\begin{proof}
  Par définition des divers groupes de Witt concernés, le noyau est constitué
  des éléments dont toutes les composantes sont hyperboliques. Il s'agit donc
  de montrer que ce sous-module est un idéal, ce qui revient à montrer que le
  produit d'une forme hyperbolique avec n'importe quelle forme est hyperbolique.
  À part pour les composantes de la forme $GW^-(K)$, le résultat est clair puisque
  toute forme hyperbolique est de la forme $h\perp \fdiag{-1}h$, donc le produit
  avec n'importe quel élément $x$ est de la forme $(xh)\perp \fdiag{-1}(xh)$,
  donc est hyperbolique. Pour le cas des formes
  bilinéaires alternées (qui sont toutes hyperboliques) on peut se ramener au cas
  d'un plan, qui a pour matrice $\begin{pmatrix} 0 & 1 \\ -1 & 0 \end{pmatrix}$,
  et donc son produit avec une forme $h$ est de la forme $\begin{pmatrix} 0 & h\\
    \fdiag{-1}h & 0 \end{pmatrix}$ ce qui donne une forme hyperbolique.
\end{proof}

Le morphisme d'extension des scalaires est manifestement compatible avec
la projection, de sorte que pour toute extension $L/K$ on a un morphisme
de $W(K)$-algèbres graduées $\tld{W}(A,\sigma)\To \tld{W}(A_L,\sigma_L)$.

\subsection{Équivalences de Morita}

Les structures d'anneau sur $\tld{GW}(A,\sigma)$ et $\tld{W}(A,\sigma)$
sont définies à partir du produit tensoriel grâce aux isomorphismes
$\phi_{(A,\sigma)}^{(d)}$ décrits par la proposition \ref{prop_prod_red}.
On va établir la compatibilité de ces isomorphismes avec les équivalences
de Morita hermitiennes.

\begin{prop}\label{prop_morita}
  Soit $f: (B,\tau)\To (A,\sigma)$ un morphisme dans $\CBrh$ (correspondant
  à une équivalence de Morita hermitienne). Pour tout $d\in \N$ on a alors
  un carré commutatif dans $\CBrh$
  \[ \begin{tikzcd}
      (B^{\otimes d},\tau^{\otimes d}) \dar{f^{\otimes d}}\rar{\phi_{(B,\tau)}^{(d)}} & (B,\tau) \dar{f} \\
      (A^{\otimes d},\sigma^{\otimes d}) \rar{\phi_{(A,\sigma)}^{(d)}} & (A,\sigma) 
    \end{tikzcd} \]
  si $d$ est impair, et
  \[ \begin{tikzcd}
      (B^{\otimes d},\tau^{\otimes d}) \dar{f^{\otimes d}}\rar{\phi_{(B,\tau)}^{(d)}} & (K,\Id)\dar{\Id} \\
      (A^{\otimes d},\sigma^{\otimes d}) \rar{\phi_{(A,\sigma)}^{(d)}} & (K,\Id)
    \end{tikzcd} \]
  si $d$ est pair.
\end{prop}

\begin{proof}
  Si $d=0$, le carré est entièrement constitué d'égalités, et si $d=1$
  les flèches verticales sont identiques et les flèches horizontales sont
  des égalités. En utilisant la proposition \ref{prop_prod_red}, on voit
  qu'on peut se ramener par récurrence au cas $d=2$, et il faut donc construire
  une isométrie entre les modules correspondant à $\phi_{(A,\sigma)}^{(2)}\circ f^{\otimes 2}$
  et $\phi_{(A,\sigma)}^{(2)}$.

  On définit alors le morphisme de $(B\otimes_KB)$-$K$-bimodule suivant,
    où $f$ correspond au bimodule $\eps$-hermitien $(V,h)$ :
  \[ \foncdef{\psi}{(V\otimes_K V)\otimes_{A\otimes_K A}A}{B = \End_A(V)}
    {(v\otimes w)\otimes a}{\phi_h(va\otimes w).} \]
  où $\phi_h: V\otimes_K V\to B$ correspond l'identification canonique
  $V\otimes_A \mbox{}^\tau V\to B$ (voir \cite[5.1]{BOI}), et est donnée par
  \[ \phi_h(v\otimes w)(x) = vh(w,x). \]

  Pour montrer que c'est une isométrie, on doit établir une égalité entre
  d'une part
  \begin{align*}
    & \Trd_B\left(\tau \left( \psi((v\otimes w)\otimes a)\right)\cdot \psi(((v'\otimes w')\otimes b))\right) \\
    &= \Trd_B\left( \tau(\phi_h(va\otimes w))\cdot \phi_h(v'b\otimes w')\right)
  \end{align*}
  et d'autre part
  \begin{align*}
    & \Trd_A \left(\sigma(a)(h\otimes h)(v\otimes w,v'\otimes w')\cdot b\right) \\
    &= \eps \Trd_A(\sigma(a)h(v,v')bh(w',w)).
  \end{align*}
  Or en appliquant successivement les formules du théorème \cite[5.1]{BOI},
  on obtient :
  \begin{align*}
    & \Trd_B\left( \tau(\phi_h(va\otimes w)) \cdot \phi_h(v'b\otimes w')\right) \\
    &= \eps \Trd_B\left(\phi_h(w\otimes va) \cdot \phi_h(v'b\otimes w')\right) \\
    &= \eps \Trd_B\left(\phi_h( wh(va,v'b) \otimes w') \right) \\
    &= \eps \Trd_A\left(h(w',wh(va,v'b))\right) \\
    &= \eps \Trd_A\left(h(w',w)\sigma(a)h(v,v')b\right).
  \end{align*}
\end{proof}

Cette compatibilité entraîne un comportement remarquable pour les
algèbres $\tld{GW}$ et $\tld{W}$ : soient $(A,\sigma)$ et $(B,\tau)$
munies d'un morphisme $f: (B,\tau)\to (A,\sigma)$ dans $\CBrh$ ; on peut
alors définir une application
\[ f_* : \tld{SGW}(B,\tau) \To \tld{SGW}(A,\sigma) \]
comme l'identité sur les composantes $SGW^{\pm}(K)$, et la composition
avec $f$ sur les composantes $SGW^\pm(B,\tau)$, en interprétant les
modules $\eps$-hermitiens comme des morphismes vers $(B,\tau)$ dans $\CBrh$.

\begin{thm}\label{thm_morita}
  L'association $(A,\sigma)\mapsto \tld{SGW}(A,\sigma)$ et $f\mapsto f_*$
  définit un foncteur de $\CBrh$ vers la catégorie des semi-anneaux
  $\Gamma$-gradués. En particulier, $\CBrh$ étant un groupoïde,
  $f_*$ est un isomorphisme de semi-anneaux.

  De plus, ce foncteur induit naturellement des foncteurs
  $(A,\sigma)\mapsto \tld{GW}(A,\sigma)$ et $(A,\sigma)\mapsto \tld{W}(A,\sigma)$
  de $\CBrh$ vers respectivement les $GW(K)$-algèbres graduées
  et les $W(K)$-algèbres graduées.
\end{thm}

\begin{proof}
  Les énoncés concernant $\tld{GW}$ et $\tld{W}$ découlent directement de
  celui à propos de $\tld{SGW}$, puisqu'on vérifie facilement que $f_*$
  préserve les sommes orthogonales (voir \ref{lem_somme_compo})
  et les formes hyperboliques.
  
  D'un point de vue ensembliste, il est clair par définition que
  $(g\circ f)_* = g_*\circ f_*$. Le fait que $f_*$ préserve la graduation
  est évident pour les composantes $+$ et $-$, et pour les composantes
  $o$ et $s$ cela vient de notre choix de graduation : si $f$ correspond
  à un module hermitien, alors $\sigma$ et $\tau$ sont de même type, et
  la composition avec $f$ préserve le signe des modules, alors que si $f$
  correspond à un module anti-hermitien, alors $\sigma$ et $\tau$ sont
  de type différent, et la composition avec $f$ inverse le signe des modules.
  Dans tous les cas la graduation est préservée. 

  Il faut donc montrer que $f_*$ préserve le produit. On se ramène à des
  modules $\eps_i$-hermitiens sur $(B_i,\tau_i)$, qui vaut soit $(B,\tau)$
  soit $(K,\Id)$, et on doit montrer que le diagramme naturel suivant commute :
  \[ \begin{tikzcd}[column sep=small, scale=0.1]
      GW^{\eps_1}(B_1,\tau_1) \times GW^{\eps_2}(B_2,\tau_2) \rar \dar &
      GW^{\eps\eps_1}(A_1,\sigma_1) \times GW^{\eps\eps_2}(A_2,\sigma_2) \dar  \\
      GW^{\eps_1\eps_2}(B_1\otimes_K B_2,\tau_1\otimes \tau_2) \dar \rar &
      GW^{\eps_1\eps_2}(A_1\otimes_K A_2,\sigma_1\otimes \sigma_2) \dar \\
      GW^{\eps_1\eps_2}(B_3,\tau_3) \rar &
      GW^{\eps_1\eps_2}(A_3,\sigma_3).
    \end{tikzcd} \]
  Le carré du haut commute toujours par propriété du produit tensoriel,
  et celui du bas ne pose pas de problème si un des $B_i$ vaut $K$.
  On doit donc montrer que le carré du bas commute dans la configuration
  suivante :
  \[ \begin{tikzcd}
      GW^{\eps_1\eps_2}(B\otimes_K B,\tau\otimes \tau) \dar \rar &
      GW^{\eps_1\eps_2}(K) \dar \\
      GW^{\eps_1\eps_2}(A\otimes_K A,\sigma\otimes \sigma) \rar &
      GW^{\eps_1\eps_2}(K)
    \end{tikzcd} \]
  ce qui est une conséquence directe de la proposition \ref{prop_morita}.
\end{proof}

Ce théorème implique notamment qu'à isomorphisme près, $\tld{GW}(A,\sigma)$
et $\tld{W}(A,\sigma)$ ne dépendent que de la classe de Brauer de $A$. En revanche,
comme on prend en compte les formes hermitiennes et anti-hermitiennes, ils ne dépendent
pas du type de l'involution.
En particulier, on peut toujours se ramener au cas d'une algèbre à division,
munie d'une involution du type que l'on veut.

\begin{ex}\label{ex_morita}
  Dans la situation du théorème, si $f: (B,\tau)\To (A,\sigma)$ est représenté
  par $(V,h)$, alors $f_*(\fdiag{1}_\tau) = h$ (où $\fdiag{1}_\tau$ est
  la forme diagonale, voir exemple \ref{ex_prod_diag}). Par conséquent,
  si $h\in GW^\eps(A,\sigma)$ représente un module $\eps$-hermitien
  non nul, alors il existe un isomorphisme $\tld{GW}(B,\tau) \Isom \tld{GW}(A,\sigma)$
  qui envoie $\fdiag{1}_\tau$ sur $h$. On a le même résultat pour $h\in W^\eps(A,\sigma)$
  non nul quelconque (tous les éléments de $W^\eps(A,\sigma)$ sont représentés
  par des modules $\eps$-hermitiens). On peut donc dans une certaine
  mesure toujours se ramener à étudier des formes du type $\fdiag{1}_\tau$.
\end{ex}

\begin{ex}\label{ex_prod_morita}
  Soient $h$ et $h'$ formes respectivement $\eps$-hermitienne
  et $\eps'$-hermitienne sur $(A,\sigma)$, de même dimension.
  D'après l'exemple précédent on peut se ramener par équivalence
  de Morita hermitienne à $h=\fdiag{1}_\tau$, sur $(B,\tau)$. Alors par
  hypothèse sur la dimension $h'$ est de la forme $\fdiag{b}_\tau$ avec
  $b\in B$, et donc $h\cdot h' = T_{\tau,b}$ (voir exemple \ref{ex_prod_diag}).
  Ces formes traces d'involution \og tordues\fg{} étudiées dans
  \cite[§11]{BOI} sont donc l'exemple universel de produit de deux formes
  de même dimension.
\end{ex}

\subsection{Algèbres déployées}\label{sec_witt_depl}

Dans le cas où $A$ est déployée, le théorème \ref{thm_morita} permet
de se ramener (mais de façon non canonique) à $(A,\sigma)=(K,\Id)$. Or dans ce cas, on a par
définition
\begin{equation}\label{eq_gw_depl}
   \tld{GW}(K) := \tld{GW}(K,\Id) = GW(K) \oplus GW^-(K) \oplus GW(K) \oplus GW^-(K) 
\end{equation}
donc les différentes composantes ne sont plus toutes distinguables.
On a en réalité
\[ \tld{GW}(K) = GW^\pm(K)[\Zd] \]
où $GW^\pm(K) = GW(K)\oplus GW^-(K)$ est l'anneau habituel des formes
bilinéaires, et où $R[G]$ désigne l'anneau de groupe de $G$ si $R$ est un
anneau commutatif. De même,
\begin{equation}\label{eq_w_depl}
  \tld{W}(K) := \tld{W}(K,\Id) = W(K)[\Zd]. 
\end{equation}

Pour $R=GW^\pm(K)$ ou $W(K)$ (ou de façon générale pour un anneau commutatif),
on dispose d'un isomorphisme tautologique
\begin{equation}\label{eq_def_phi_ann}
  \Phi : R\oplus R \Isom R[\Zd]
\end{equation}
d'inverse $\Phi^{-1} = (\pi_0,\pi_1)$
où $\pi_0,\pi_1$ sont les projections sur les composantes homogènes.
Cependant, on va plutôt utiliser un autre isomorphisme qui sera
plus naturellement compatible avec les opérations qu'on souhaite
exploiter. Notons $\Phi(x,0) = x_{(0)}$ et $\Phi(0,x) = x_{(1)}$.
On pose
\begin{equation}\label{eq_def_delta_ann}
  \foncdef{\Delta}{R}{R[\Zd]}{x}{x_{(0)}-x_{(1)}} 
\end{equation}
qu'on appelle \emph{plongement diagonal} de $R$ dans $R[\Zd]$.
On définit alors
\begin{equation}\label{eq_def_psi_ann}
   \isomdef{\Psi}{R\oplus R}{R[\Zd]}{(x,y)}{x_{(0)}+\Delta(y),} 
\end{equation}
qui vérifie $\Psi^{-1} = (\pi_0+\pi_1,-\pi_1)$. Notons que $\Phi(x,0)=\Psi(x,0)=x_{(0)}$.

On a alors de façon immédiate :
\begin{equation}
  \Phi(x_1,y_1)\cdot \Phi(x_2,y_2) = \Phi(x_1x_2+y_1y_2, x_1y_2+x_2y_1) 
\end{equation}
et (en utilisant $\Delta(x)\cdot \Delta(y)=2\Delta(xy)$)
\begin{equation}\label{eq_psi_prod}
  \Psi(x_1,y_1)\cdot \Psi(x_2,y_2) = \Psi(x_1x_2, x_1y_2+x_2y_1+2y_1y_2). 
\end{equation}

\subsection{Algèbres de quaternions}

On peut également donner une description concrète de l'anneau de Witt
pour une algèbre de quaternion. Soit donc $(Q,\can)$ une algèbre
de quaternions sur $K$, munie de son involution canonique. Il suffit
pour décrire $\tld{GW}(Q,\can)$ de donner les produits
des formes diagonales élémentaires $\fdiag{z}_Q$ où $z$ est soit
un scalaire soit un quaternion pur (voir exemple \ref{ex_prod_diag}).

On commence par définir une certaine forme de Pfister qui sera
importante dans la suite :

\begin{propdef}\label{prop_phi_quater}
  Soient $z_1, \dots, z_r$ des quaternions purs dans $Q^*$, avec
  $r\pgq 2$. Il existe alors une unique $r$-forme de Pfister
  $\phi_{z_1,\dots,z_r}$ telle que dans $W(K)$ :
  \[ \phi_{z_1,\dots,z_r} = \pfis{z_1^2,\dots,z_r^2} - \pfis{-1}^{r-2}n_Q. \]

  Si deux des $z_i$ anti-commutent, $\phi_{z_1,\dots,z_r}$ est
  hyperbolique.
  De plus, si $z_{r+1},\dots,z_{r+s}$ sont des quaternions purs, alors
  \[ \phi_{z_1,\dots,z_r}\pfis{z_{r+1}^2,\dots,z_{r+s}^2} = \phi_{z_1,\dots,z_{r+s}}. \]
\end{propdef}

\begin{proof}
  On a $\pfis{z_{r+1}^2,\dots,z_{r+s}^2}\cdot \pfis{-1}^{r-2}n_Q = \pfis{-1}^{r+s-2}n_Q$,
  d'où la dernière affirmation.
  
  De là, pour montrer que $\phi_{z_1,\dots,z_r}$ est une $r$-forme de Pfister,
  il suffit de montrer que $\pfis{z_1^2,z_2^2} - n_Q$ est une 2-forme
  de Pfister. Or si $z_0$ est un quaternion pur qui anti-commute
  avec $z_1$, on a dans $W(K)$ :
  \begin{align*}
    \pfis{z_1^2,z_2^2} - n_Q &= \pfis{z_1^2,z_2^2} + n_Q - \pfis{z_2^2}n_Q \\
                             &= \pfis{z_1^2,z_2^2} + \pfis{z_1^2,z_0^2} - \pfis{z_1^2,z_0^2,z_2^2} \\
                             &= \pfis{z_1^2,z_2^2z_0^2}.
  \end{align*}
  
  Si $z_i$ et $z_j$ anti-commutent, alors $\pfis{z_i^2,z_j^2} = n_Q$,
  donc $\pfis{z_1^2,\dots,z_r^2}=\pfis{-1,\dots,-1}n_Q$.
\end{proof}

\begin{rem}
  Dans cet énoncé, on parle de forme de Pfister au sens usuel,
  c'est-à-dire que $\phi_{z_1,\dots,z_r}$ est bien une forme
  quadratique de dimension $2^r$, et non un élément de
  $\hat{I}^r(K)$.
\end{rem}

\begin{prop}\label{prop_prod_quater}
  Soient $z_1,z_2\in Q$ des quaternions purs. Alors on a dans
  $\tld{GW}(Q,\can)$ :
  \[  \fdiag{z_1}\cdot \fdiag{z_2} = \fdiag{-\Trd_Q(z_1z_2)}\phi_{z_1,z_2} \]
  où si $z_1$ et $z_2$ anti-commutent on sous-entend que la
  forme est hyperbolique.

  Si $a,b\in K^*$, alors :
  \[ \fdiag{a}_Q\cdot \fdiag{b}_Q = \fdiag{2ab}n_Q. \]
\end{prop}

\begin{proof}
  D'après l'exemple \ref{ex_prod_diag}, $\fdiag{z_1}\cdot \fdiag{z_2}=b$
  où $b$ est la forme bilinéaire symétrique sur le $K$-espace vectoriel
  $Q$ telle que $b(x,y)=\Trd_Q(xz_1\overline{y}\overline{z_2})$.

  Soit $z_0$ un quaternion pur qui anti-commute avec $z_1$ et $z_2$
  (ce qui existe toujours, et si $z_1$ et $z_2$ ne commutent pas on peut
  prendre la partie pure de $z_1z_2$). On voit facilement que
  $\phi_{z_1,z_2} = \pfis{z_2^2z_0^2,z_1^2}$ puisque ce sont des 2-formes
  de Pfister avec le même invariant $e_2$ (ou voir la preuve de la
  proposition précédente).
  Dans la base $(1,z_0,z_1,z_1z_0)$ la matrice de $b$ est :
  \[ \begin{pmatrix}
      -Tr_Q(z_1z_2) & Tr_Q(z_1z_2z_0) & 0 & 0 \\
      Tr_Q(z_1z_2z_0) & -z_0^2Tr_Q(z_1z_2) & 0 & 0 \\
      0 & 0 & z_1^2Tr_Q(z_1z_2) & -z_1^2Tr_Q(z_1z_2z_0) \\
      0 & 0 & -z_1^2Tr_Q(z_1z_2z_0) & z_1^2z_0^2Tr_Q(z_1z_2)
    \end{pmatrix}. \]
  Soit $\Delta$ le déterminant du bloc supérieur gauche. Alors $b$
  est équivalente à 
  \[ \fdiag{-Tr_Q(z_1z_2)}\fdiag{1,\Delta,-z_1^2,-z_1^2\Delta} = \fdiag{-Tr_Q(z_1z_2)}\pfis{-\Delta,z_1 ^2} \]
  au sens où $b$ est hyperbolique si $Tr_Q(z_1z_2)= 0$,
  et sinon $b$ est isométrique à la forme ci-dessus. On va donc
  montrer que $\pfis{-\Delta,z_1^2} = \phi_{z_1,z_2}$.

  On a
  \[ \Delta = z_0^2Tr_Q(z_1z_2)^2 - Tr_Q(z_1z_2z_0)^2 \]
  donc on peut distinguer deux cas : si $z_1$ et $z_2$ commutent
  (ce qui revient à dire qu'ils sont colinéaires), alors $Tr_Q(z_1z_2z_0) = 0$
  et $\Delta = z_0^2Tr_Q(z_1z_2) = 4z_0^2z_1^2z_2^2$ ; si maintenant $z_1$ 
  et $z_2$ ne commutent pas, alors on peut choisir
  $z_0 = (z_1z_2)_0$, donc $z_1z_2 = \lambda + z_0$ pour un certain
  $\lambda\in k^*$. Alors:
  \begin{align*}
    \Delta &= z_0^2\Trd_Q(\lambda + z_0)^2 - \Trd_Q((\lambda + z_0)z_0)^2 \\
           &= 4z_0^2\lambda^2 - \Trd_Q(z_0^2)^2 \\
           &= 4z_0^2(\lambda^2 - z_0^2) \\
           &= 4z_0^2\Nrd_Q(z_1z_2) \\
           &= 4z_0^2z_1^2z_2^2
  \end{align*}
  donc dans tous les cas on trouve $\pfis{-\Delta,z_1^2} = \pfis{-z_1^2z_2^2z_0^2,z_1^2}$
  ce qui donne bien $\phi_{z_1,z_2}$.

  Pour le calcul de $\fdiag{a}_Q\cdot \fdiag{b}_Q$ avec $a,b\in K^*$
  on se ramène immédatement au cas $a=b=1$, et d'après l'exemple
  \ref{ex_prod_diag} on trouve la forme trace d'involution, qui
  a pour diagonalisation $\fdiag{2,-2u,-2v,2uv}=\fdiag{2}\pfis{u,v}$
  si $Q=(u,v)$, et donc qui est isométrique à $\fdiag{2}n_Q$.
\end{proof}

\section{Opérations $\lambda$}

L'ingrédient clé des constructions d'invariants de classes de Witt
dans le chapitre précédent était la structure de $\lambda$-anneau
de $GW(K)$. On veut définir une structure semblable sur $\tld{GW}(A,\sigma)$.

\subsection{Puissances alternées d'un module}

\subsubsection*{Groupes symétriques et mélanges}

\label{par_young}On commence par un petit rappel sur les mélanges. Soit $d\in \N$
et soient $p,q\in \N$ tels que $p+q=d$. On dispose alors du
sous-groupe $\mathfrak{S}_{p,q}\subset \mathfrak{S}_d$ constitué
des permutations qui préservent les $p$ premiers (et donc les $q$ derniers)
éléments. On a naturellement $\mathfrak{S}_{p,q}\simeq \mathfrak{S}_p\times \mathfrak{S}_q$,
et la signature à gauche correspond au produit des signatures à droite.

\label{par_shuffle}On définit aussi l'ensemble des $(p,q)$-mélanges $Sh(p,q)\subset \mathfrak{S}_d$,
qui sont les permutations (strictement) croissantes sur les $p$ premiers
et sur les $q$ derniers éléments.

\begin{lem}\label{lem_shuffle}
  Tout élément de $\mathfrak{S}_d$ s'écrit de façon unique
  comme $\pi\sigma$ avec $\pi\in Sh(p,q)$ et $\sigma\in \mathfrak{S}_{p,q}$.
\end{lem}

\begin{proof}
  Soit $\tau\in \mathfrak{S}_d$. Soient $\sigma_1\in \mathfrak{S}_p$
  et $\sigma_2\in \mathfrak{S}_q$ définies par $\tau(\sigma_1^{-1}(1))<\dots<\tau(\sigma_1^{-1}(p))$
  et $\tau(\sigma_2^{-1}(p+1))<\dots<\tau(\sigma_2^{-1}(p+q))$ ;
  autrement dit, $\sigma_1$ est obtenue en ordonnant $\tau(1),\dots,\tau(p)$
  dans l'ordre croissant, et de même pour $\sigma_2$ avec $\tau(p+1),\dots,
  \tau(p+q)$. On pose $\sigma\in \mathfrak{S}_{p,q}$ correspondant
  à $(\sigma_1,\sigma_2)$, et $\pi = \tau\sigma^{-1}$. Alors par
  construction $\pi(1)<\dots<\pi(p)$ et $\pi(p+1)<\dots<\pi(p+q)$,
  donc $\pi\in Sh(p,q)$. Si on dispose d'une autre décomposition
  $\tau=\sigma'\pi'$, alors comme $\pi'\in Sh(p,q)$ on doit
  avoir $\tau((\sigma'_1)^{-1}(1))<\dots<\tau((\sigma'_1)^{-1}(p))$
  et $\tau((\sigma'_2)^{-1}(p+1))<\dots<\tau((\sigma')_2^{-1}(p+q))$,
  et donc $\sigma'=\sigma$ (et $\pi'=\pi$).
\end{proof}

\subsubsection*{Puissances alternées}

Si $V$ est un $K$-espace vectoriel, le fait que $K$ soit de caractéristique
différente de 2 nous permet de voir la puissance extérieure $\Lambda^d V$
de deux façons : soit comme un quotient (ce qui est la construction canonique),
soit comme un sous-espace de $V^{\otimes d}$. Précisément,
pour tout $\pi\in \mathfrak{S}_d$ on pose
\[ \foncdef{g_\pi}{V^{\otimes d}}{V^{\otimes d}}{v_1\otimes \cdots \otimes v_d}
  { v_{\pi^{-1}(1)}\otimes \cdots \otimes v_{\pi^{-1}(d)},} \]
et on définit l'application d'antisymétrisation
\[ s_d = \sum_{\pi\in \mathfrak{S}_d} (-1)^\pi g_\pi \]
où $(-1)^\pi$ désigne la signature de la permutation $\pi$.
Alors l'application $s_d$ est alternée, donc on obtient
par propriété universelle une application induite $\Lambda^d V\To \Alt^d(V)$,
où $\Alt^d(V)\subset V^{\otimes d}$ est l'image de $s_d$, et un
résultat classique d'algèbre linéaire affirme que c'est un isomorphisme,
qu'on peut expliciter comme $v_1\wedge \cdots \wedge v_d\mapsto
s_d(v_1\otimes \cdots \otimes v_d)$.
\\

\label{par_goldman}On va adapter cette contruction dans le cadre des modules sur des
algèbres simples centrales. Soit donc $A$ une algèbre simple centrale
sur $K$. Suivant \cite[10.1]{BOI}, on définit pour tout $d\in \N$ un morphisme
de groupes $\mathfrak{S}_d\To (A^{\otimes d})^*$, qu'on note $\pi\mapsto g_A(\pi)$
(on écrira $g(\pi)$ s'il n'y a pas de confusion possible), caractérisé
par $g_A((i\; i+1)) = 1\otimes 1\otimes \cdots \otimes g_A\otimes 1\otimes \cdots \otimes 1$ où
$g_A\in A^{\otimes 2}$ est l'élément de Goldman. Pour $d=0,1$, il s'agit
du morphisme trivial. On énonce quelques propriétés de base de ces éléments :

\begin{lem}\label{lem_goldman_commut}
  Soit $V$ un $B$-$A$-bimodule simple. Alors si $v_1,\dots,v_d\in V$
  et $\pi\in \mathfrak{S}_d$, on a
  \[ g_B(\pi)\cdot (v_1\otimes\cdots \otimes v_d) =
    (v_{\pi^{-1}(1)}\otimes \cdots \otimes v_{\pi^{-1}(d)}) \cdot g_A(\pi). \]
\end{lem}

\begin{proof}
  On se ramène facilement au cas où $A$ et $B$ sont déployées et $d=2$,
  avec $g_B(\pi)=g_B$ et $g_A(\pi)=g_A$. On a alors $A\simeq \End_K(U)$,
  $B\simeq \End_K(W)$, et $V\simeq \Hom_K(U,W)$ avec les actions évidentes,
  et:
  \begin{align*}
    (g_B\cdot f\otimes g)(u\otimes u') &= g_B(f(u)\otimes g(u')) \\
                                       &= g(u')\otimes f(u) \\
                                       &= (g\otimes f)(g_A(u\otimes u')) \\
                                       &= (g\otimes f \cdot g_A)(u\otimes u').
  \end{align*}
\end{proof}

\begin{lem}\label{lem_goldman_young}
  Soient $V_1$ et $V_2$ deux $A$-modules à droite, et soient $B_i = \End_A(V_i)$
  et $B = \End_A(V_1\oplus V_2)$. Soit $\pi\in \mathfrak{S}_{p,q}$, correspondant
  au produit de $\pi_1\in \mathfrak{S}_p$ et $\pi_2\in \mathfrak{S}_q$.
  Soient $x\in V_1^{\otimes p}$ et $y\in V_2^{\otimes q}$. Alors :
  \[ g_B(\pi)\cdot (x\otimes y) = (g_{B_1}(\pi_1)\cdot x)\otimes (g_{B_2}(\pi_2)\cdot y). \]
\end{lem}

\begin{proof}
  On se ramène au cas où $A$ est déployée, auquel cas le résultat est
  clair par construction.
\end{proof}

On définit alors comme dans \cite[§10.A]{BOI} $s_{d,A} = s_d\in A^{\otimes d}$ par :
\begin{equation}\label{eq_def_sd}
 s_d = \sum_{\pi \in \mathfrak{S}_d} (-1)^\pi g(\pi). 
\end{equation}

\begin{defi}\label{def_alt}
  Soit $V$ un $A$-module à droite, avec $B = \End_A(V)$. On pose
  \[ \Alt^d(V) = s_{B,d}V^{\otimes d}\subset V^{\otimes d} \]
  $A^{\otimes d}$-module à droite, avec en particulier $\Alt^0(V)=K$
  et $\Alt^1(V)=V$
\end{defi}

Notons que si $A=K$, alors on retrouve la construction
ci-dessus pour les espaces vectoriels.

\begin{rem}\label{rem_lambda_alt}
  Dans le cas où $V=A$ muni de l'action tautologique, on
  peut caractériser l'algèbre $\lambda^d(A)$ définie dans
  \cite[10.4]{BOI} par le fait que $\Alt^d(A)$ est un isomorphisme
  dans $\CBr$ de $\lambda^d(A)$ vers $A^{\otimes d}$.
\end{rem}

\subsubsection*{Produit de mélange}

Si $p+q=d$, on définit un morphisme de $A^{\otimes d}$-modules
$V^{\otimes p}\otimes_K V^{\otimes q}\to V^{\otimes d}$,
qu'on appelle \emph{produit de mélange} et qu'on note $x \# y$,
par
\begin{equation}\label{eq_prod_shuffle}
  x \# y = sh_{p,q}(x\otimes y)
\end{equation}
où on définit l'application de mélange
\begin{equation}\label{eq_app_shuffle}
   \foncdef{sh_{p,q}}{V^{\otimes d}}{V^{\otimes d}}{x}{\sum_{\pi\in Sh(p,q)}(-1)^\pi g_B(\pi)\cdot x.}
 \end{equation}
On voit facilement par définition que le produit de mélange est associatif
et alterné, et en particulier anti-symétrique.

\begin{prop}
  Le produit de mélange induit un diagramme commutatif
  \[ \begin{tikzcd}
      V^{ \otimes p}\otimes_K V^{ \otimes q} \dar \rar{\otimes} & V^{ \otimes d} \dar \\
      \Alt^p(V) \otimes_K \Alt^q(V) \rar{\#} & \Alt^d(V).
    \end{tikzcd} \]
\end{prop}

\begin{proof}
  On doit vérifier que pour tout $x\in V^{\otimes p}$ et tout
  $y\in V^{\otimes q}$ on a $(s_p x)\#(s_qy) = s_d(x\otimes y)$.
  Or:
  \begin{align*}
    (s_p x)\#(s_qy) &= sh_{p,q}(s_px\otimes s_qy) \\
                       &= \sum_{\pi\in Sh(p,q)}(-1)^\pi g_B(\pi)(s_px\otimes s_qy) \\
                       &= \sum_{\pi\in Sh(p,q)}(-1)^\pi g_B(\pi)\left( \sum_{\sigma\in \mathfrak{S}_{p,q}}(-1)^\sigma g_B(\sigma)(x\otimes y)\right) \\
                       &= \sum_{\pi\in Sh(p,q), \sigma\in \mathfrak{S}_{p,q}}(-1)^{\pi\sigma}g_B(\pi\sigma)(x\otimes y) \\
                       &= s_d(x\otimes y)
  \end{align*}
  en utilisant les lemmes \ref{lem_goldman_young} et \ref{lem_shuffle}.
\end{proof}

On montre maintenant l'analogue de la formule d'addition bien
connue pour les puissances extérieures d'espaces vectoriels :

\begin{prop}\label{prop_add_mod_alt}
  Soient $U$ et $V$ deux $A$-modules à droite. Alors pour tout
  $d\in \N$ le produit de mélange induit un isomorphisme de
  $A^{\otimes d}$-modules :
  \[ \bigoplus_{k=0}^d \Alt^k(U)\otimes_K \Alt^{d-k}(V) \Isom \Alt^d(U\oplus V). \]
\end{prop}

\begin{proof}
  En utilisant la proposition précédente, on établit facilement
  que $\Alt^d(U\oplus V)$ est linéairement engendré par les éléments
  de la forme $x_1 \# \cdots \# x_d$, et par bilinéarité du produit
  de mélange, on peut clairement prendre
  $x_i\in U$ ou $V$. Alors comme le produit de mélange est anti-symétrique,
  on peut permuter les $x_i$ jusqu'à se ramener à $x_1,\dots,x_k\in U$
  et $x_{k+1},\dots,x_d\in V$. Or tout élément de cette forme
  est clairement dans l'image de l'application de l'énoncé, de sorte
  qu'elle surjective. On peut alors conclure qu'elle est bijective en
  comparant les dimensions sur $K$.
\end{proof}

\subsection{Puissances alternées d'une forme $\eps$-hermitienne}

On suppose maintenant que $A$ est munie d'une involution de première
espèce $\sigma$, et que $V$ est munie d'une forme $\eps$-hermitienne $h$
relativement à $\sigma$. On note $\tau$ l'involution adjointe sur $B$.

\begin{lem}
  L'élément $s_d\in B^{\otimes d}$ est symétrique pour l'involution
  $\tau^{\otimes d}$, et pour tous $x,y\in V$ on a
  \[ h^{\otimes d}(x,s_dy) = h^{\otimes d}(s_dx,y). \]
\end{lem}

\begin{proof}
  La première affirmation se ramène par construction au fait que
  l'élément de Goldman est symétrique pour $\tau\otimes \tau$, ce
  qu'on a déjà observé dans le lemme \ref{lem_action_goldman} (voir
  aussi directement \cite[10.19]{BOI}), et la deuxième découle directement
  de ce fait puisque $\tau^{\otimes d}$ est l'involution adjointe
  à $h^{\otimes d}$.
\end{proof}

\begin{defi}\label{def_alt_h}
  Le lemme précédent permet alors de poser:
  \[ \foncdef{\Alt^d(h)}{\Alt^d(V)\times \Alt^d(V)}{A^{\otimes d}}
    {(s_dx, s_dy)}{h^{\otimes d}(x,s_dy) = h^{\otimes d}(s_dx,y).} \]
\end{defi}

\begin{prop}
  L'application $\Alt^d(h)$ est une forme $\eps^d$-hermitienne
  sur $(A^{\otimes d},\sigma^{\otimes d})$.
\end{prop}

\begin{proof}
  On a
  \begin{align*}
    \Alt^d(h)(s_dx\cdot a,s_dy\cdot b) &= h^{\otimes d}(xa,s_dyb) \\
                                       &= \sigma^{\otimes d}(a)h^{\otimes d}(x,s_dy)b \\
                                       &= \sigma^{\otimes d}(a)\Alt^d(h)(s_dx,s_dy)b
  \end{align*}
  et
  \begin{align*}
    \Alt^d(h)(s_dy,s_dx) &= h^{\otimes d}(y,s_dx) \\
                         &= \eps^d\sigma^{\otimes d}(h^{\otimes d}(s_dx,y)) \\
                         &= \eps^d\sigma^{\otimes d}(\Alt^d(h)(s_dx,s_dy)).
  \end{align*}
\end{proof}

\begin{rem}\label{rem_lambda_alt_sigma}
  Dans la continuité de la remarque \ref{rem_lambda_alt}, on peut
  observer que si $V=A$ et $h=\fdiag{1}_\sigma$, alors l'involution
  adjointe de $\Alt^d(h)$ est l'involution $\sigma^{\wedge d}$ telle
  que définie dans \cite[10.20]{BOI}, de sorte que $(\Alt^d(A),\Alt^d(h))$
  réalise un isomorphisme dans $\CBrh$ de $(\lambda^d(A),\sigma^{\wedge d})$
  vers $(A^{\otimes d},\sigma^{\otimes d})$.
\end{rem}

\begin{rem}\label{rem_alt_restr}
  On a construit $\Alt^d(V)$ comme un sous-module de $V^{\otimes d}$, il
  est donc muni de la restriction de la forme hermitienne $h^{\otimes d}$.
  On observe facilement que
  \[ h^{\otimes d}_{| \Alt^d(V)} = \fdiag{d!}\Alt^d(h) \]
  puisque $h^{\otimes d}(s_dx,s_dy)=h^{\otimes d}(x,s_d^2y)$ et $s_d^2=d!s_d$.
  En particulier, en caractéristique non nulle on ne peut pas
  définir simplement $\Alt^d(h)$ à partir de $h^{\otimes d}$.
\end{rem}

On montre alors la compatibilité avec la formule d'addition :

\begin{prop}\label{prop_add_alt}
  Soient $(U,h)$ et $(V,h')$ deux modules $\eps$-hermitiens sur
  $(A,\sigma)$. L'isomorphisme de modules de la proposition \ref{prop_add_mod_alt}
  induit une isométrie de
  \[ \bigoplus_{k=0}^d (\Alt^k(U),\Alt^k(h))\otimes_K (\Alt^{d-k}(V),\Alt^{d-k}(h')) \]
  vers
  \[ (\Alt^d(U\oplus V), \Alt^d(h\perp h')). \]
\end{prop}

\begin{proof}
  Soient $u,u'\in U^{\otimes k}$ et $v,v'\in V^{\otimes d-k}$. Alors
  \begin{align*}
    & \Alt^d(h\perp h')(s_ku \# s_{d-k}v, s_ku \# s_{d-k}v) \\
    &= \Alt^d(h\perp h')(s_d(u\otimes v), s_d(u'\otimes v')) \\
    &= (h\perp h')^{\otimes d}(s_d(u\otimes v),u'\otimes v') \\
    &= \sum_{\pi\in \mathfrak{S}_d}(-1)^\pi (h\perp h')^{\otimes d}(g(\pi)(u\otimes v),u'\otimes v').
  \end{align*}
  On veut montrer que $(h\perp h')^{\otimes d}(g(\pi)(u\otimes v),u'\otimes v')$
  est nul si $\pi\not\in \mathfrak{S}_{k,d-k}$. Or si $u = x_1\otimes\cdots\otimes x_k$,
  $u' = y_1\otimes\cdots\otimes y_k$, et $v= x_{k+1}\otimes\cdots\otimes x_d$,
  $v'= y_{k+1}\otimes\cdots\otimes y_d$, alors
  \begin{align*}
    & (h\perp h')^{\otimes d}(g(\pi)(u\otimes v),u'\otimes v') \\
    &= (h\perp h')^{\otimes d}((x_{\pi^{-1}(1)}\otimes \cdots \otimes x_{\pi^{-1}(d)}) \cdot g_A(\pi),
      (y_1\otimes \cdots \otimes y_d)) \\
    &= \sigma^{\otimes d}(g_A(\pi))\cdot (h\perp h')(x_{\pi^{-1}(1)}, y_1)\otimes \cdots \otimes
      (h\perp h')(x_{\pi^{-1}(d)},y_d)
  \end{align*}
  ce qui est en effet nul si $\pi\not\in \mathfrak{S}_{k,d-k}$. De là :
  \begin{align*}
    & \Alt^d(h\perp h')(s_ku \# s_{d-k}v, s_ku \# s_{d-k}v) \\
    &= \sum_{\pi\in \mathfrak{S}_{k,d-k}}(-1)^\pi (h\perp h')^{\otimes d}(g(\pi)(u\otimes v),u'\otimes v') \\
    &= \sum_{\pi_1\in \mathfrak{S}_k}\sum_{\pi_2\in \mathfrak{S}_{d-k}}(-1)^{\pi_1\pi_2} (h\perp h')^{\otimes d}(g(\pi_1)u\otimes g(\pi_2)v),u'\otimes v') \\
    &= h(s_ku,u')\otimes h'(s_{d-k}v,v').
  \end{align*}
\end{proof}

\subsection{Puissances extérieures d'un module $\eps$-hermitien}

Dans le cas des espaces vectoriels, $(\Alt^d(V), \Alt^d(h))$ fournissait
une définition appropriée pour une opération $\lambda^d: GW(K)\To GW(K)$,
mais si $A$ n'est pas déployée, on obtient seulement un module $\eps$-hermitien
sur $(A^{\otimes d}, \sigma^{\otimes d})$. On va alors utiliser l'isomorphisme
$\phi_{(A,\sigma)}^{(d)}$ dans $\CBrh$.

\begin{defi}\label{def_lambda_v}
  Soit $(V,h)$ un module $\eps$-hermitien sur $(A,\sigma)$. Pour tout $d\in \N$,
  on définit $(\Lambda^d(V), \lambda^d(h))$ comme la composition (comme morphismes
  dans $\CBrh$) de $(\Alt^d(V),\Alt^d(h))$ et $\phi_{(A,\sigma)}^{(d)}$.
  Il s'agit donc d'un espace bilinéaire symétrique sur $K$ si $d$ est pair,
  et d'un module $\eps$-hermitien sur $(A,\sigma)$ si $d$ est impair.
\end{defi}

\begin{rem}
  Noter que $\Lambda^d(V)$ dépend de $\sigma$, bien que le module $V$ soit
  défini indépendamment. En revanche $\Lambda^d(V)$ ne dépend pas de $h$.
  Voir l'exemple ci-dessous.
\end{rem}

\begin{ex}\label{ex_lambda2}
  Considérons la forme hermitienne $h=\fdiag{1}_\sigma$ (voir exemple \ref{ex_prod_diag}),
  de module sous-jacent $V=A$.
  Alors si $\sigma$ est orthogonale $\Lambda^2(A)$ est naturellement identifié
  à l'espace vectoriel $Skew(A,\sigma)$ des éléments antisymétriques de $A$,
  et si $\sigma$ est symplectique $\Lambda^2(A)$ est naturellement $Sym(A,\sigma)$,
  l'espace des éléments symétriques. En effet, par construction
  $\Lambda^2(A)$ est obtenu à partir du $A\otimes_K A$-module à droite
  $(1-g)(A\otimes_K A)$ (où $g$ est l'élément de Goldman) par l'équivalence
  de Morita entre $A\otimes_K A$ et $K$. De là, $\Lambda^2(A)$ est le
  sous-espace vectoriel $(1-g)\cdot A\subset A$ où l'action à gauche de $A\otimes_K A$
  est l'action tordue par $\sigma$. Or d'après le lemme \ref{lem_action_goldman},
  on a $g\cdot a=\eps\sigma(a)$ où $\eps=1$ si $\sigma$ est orthogonale
  et $\eps=-1$ si $\sigma$ est symplectique, donc $\Lambda^2(A)$ est
  l'espace des éléments symétrisés ou anti-symétrisés selon le type de $\sigma$
  (et comme la caractéristique est différente de 2, c'est bien équivalent
  aux éléments symétriques ou anti-symétriques).

  De plus, si $\sigma$ est orthogonale $\lambda^2(h)$ est isométrique à
  $\fdiag{\frac{1}{2}}T_\sigma^-$, et si $\sigma$ est symplectique alors
  $\lambda^2(h)$ est isométrique à $\fdiag{\frac{1}{2}}T_\sigma^+$,
  où $T_\sigma^\pm$ est la restriction de la forme trace
  d'involution aux éléments symétriques ou antisymétriques. En effet,
  $T_\sigma$ correspond par l'équivalence de Morita ci-dessus à
  $h\otimes h$, tandis que $\lambda^2(h)$ correspond à $\Alt^2(h)$,
  et on sait que la restriction de $h\otimes h$ à $\Alt^2(A)$ est
  $\fdiag{2}\Alt^2(h)$ (voir la remarque \ref{rem_alt_restr}).
\end{ex}

\begin{ex}\label{ex_disc}
  On prend encore le cas de $h=\fdiag{1}_\sigma$ pour simplifier,
  où on suppose $\sigma$ orthogonale ;
  alors si $d=\deg(A)$, on a $\lambda^d(h)=\fdiag{\det(\sigma)}$
  (discriminant non signé).
  En effet, on le vérifie facilement après déploiement générique,
  où c'est une conséquence de la propriété correspondante pour les
  formes quadratiques, mais le morphismes $K^*/(K^*)^2\to F^*/(F^*)^2$,
  où $F$ est un corps de déploiement générique de $A$, est injectif.
\end{ex}

\begin{rem}
  En utilisant la remarque \ref{rem_lambda_alt_sigma}, on constate
  qu'en prenant encore $V=A$ et $h=\fdiag{1}_\sigma$, $(\lambda^d(V),\lambda^d(h))$
  est un isomorphisme dans $\CBrh$ de $(\lambda^d(A),\sigma^{\wedge d})$
  vers $(A,\sigma)$ ou $(K,\Id)$ selon la parité de $d$. 
\end{rem}

On va montrer que ces opérations vérifient bien les propriétés attendues.
On renvoie au premier chapitre pour les définitions nécessaires.

\begin{prop}\label{prop_lambda_mixte}
  L'association $(V,h)\mapsto (\Lambda^d(V), \lambda^d(h))$ pour $(V,h)$
  module $\eps$-hermitien sur $(A,\sigma)$ ou $(K,\Id)$ définit une
  opération grecque homogène sur $\tld{SGW}(A,\sigma)$, qui induit une
  structure d'anneau grec gradué sur $\tld{GW}(A,\sigma)$.
\end{prop}

\begin{proof}
  Pour tout $\gamma\in \Gamma$ et tout $d\in \N$ on a défini
  $\lambda^d: \tld{SGW}(A,\sigma)_\gamma\to \tld{SGW}(A,\sigma)_{d\gamma}$,
  et la proposition \ref{prop_add_alt} montre que l'application
  induite $\lambda_t: \tld{SGW}(A,\sigma)_\gamma\to G(\tld{SGW}(A,\sigma))$
  est un morphisme de monoïdes, donc par définition d'une somme directe
  on obtient de façon naturelle un morphisme
  $\lambda_t: \tld{SGW}(A,\sigma)\to G(\tld{SGW}(A,\sigma))$
  qui définit une opération grecque graduée.

  L'affirmation concernant $\tld{GW}(A,\sigma)$ est une simple
  application du lemme \ref{lem_semi_grec}.
\end{proof}

Le théorème \ref{thm_morita} s'étend lui aussi à la structure d'anneau grec :

\begin{prop}\label{prop_lambda_morita}
  Le foncteur $(A,\sigma)\mapsto \tld{GW}(A,\sigma)$ définit un foncteur
  de $\CBrh$ dans la catégorie des algèbres \emph{grecques} graduées.
\end{prop}

\begin{proof}
  Il s'agit donc de vérifier que les opérations $\lambda$ sont respectées
  par les isomorphismes induits par les équivalences de Morita. Soient
  donc $(B,\tau)$ et $(A,\sigma)$, et $(U,g)$ qui définit un morphisme
  $f: (B,\tau)\To (A,\sigma)$ dans $\CBrh$, et soit $(V,h)$ un
  module $\eps$-hermitien sur $(B,\tau)$.

  Comme $(\Lambda^d(V),\lambda^d(h))$ est obtenu par composition avec
  $\phi_{(B,\tau)}^{(d)}$ à partir de $(\Alt^d(V),\Alt^d(h))$ et que la proposition
  \ref{prop_morita} montre la compatibilité de $\phi_{(B,\tau)}^{(d)}$ avec
  les équivalences de Morita, il suffit de vérifier que $f^{\otimes d}$
  induit un isomorphisme entre $\Alt^d(h)$ et $\Alt^d(h')$ où $h'= f_*(h)$.

  On note $W = V\otimes_B U$, sur lequel est défini $h'$. Alors on a canoniquement
  $\End_A(W)\simeq \End_B(V)$ (qu'on notera $C$), donc
  \begin{align*}
    \Alt^d(V)\otimes_{B^{\otimes d}} U^{\otimes d}
    &= s_{d,C}\cdot V^{\otimes d}\otimes_{B^{\otimes d}} U^{\otimes d} \\
    &= s_{d,C}\cdot (V\otimes_B U)^{\otimes d} \\
    &= s_{d,C}\cdot W^{\otimes d} \\
    &= \Alt^d(W).
  \end{align*}
  Si $x,y\in V^{\otimes d}$ et $u,u'\in U^{\otimes d}$, où
  $x=x_1\otimes\cdots\otimes x_d$ et de même pour les autres :
  \begin{align*}
    & f^{\otimes d}_*(\Alt^d(h)) (s_{d,C}x\otimes u, s_{d,C}y\otimes u') \\
    &= g^{\otimes d}(u, \Alt^d(h)(s_{d,C}x,s_{d,C}y)\cdot u') \\
    &= g^{\otimes d}(u, h^{\otimes d}(x,s_{d,C}y)\cdot u') \\
    &= \sum_{\pi \in \mathfrak{S}_d}(-1)^\pi g(u_1,h(x_1,y_{\pi^{-1}(1)})u'_{\pi^{-1}(1)})\otimes \cdots\otimes g(u_d,h(x_d,y_{\pi^{-1}(d)})u'_{\pi^{-1}(d)})  g_{A}(\pi) \\
    &= \sum_{\pi \in \mathfrak{S}_d}(-1)^\pi h'(x_1\otimes u_1, y_{\pi^{-1}(1)}\otimes u'_{\pi^{-1}(1)}) \otimes \cdots \otimes h'(x_d\otimes u_d, y_{\pi^{-1}(d)}\otimes u'_{\pi^{-1}(d)})  g_{A}(\pi) \\
    &= \sum_{\pi \in \mathfrak{S}_d}(-1)^\pi (h')^{\otimes d}(x\otimes u, g_C(\pi)y\otimes u') \\
    &= h'(x\otimes u, s_{d,C}y\otimes u') \\
    &= \Alt^d(h')(s_{d,C}x\otimes u, s_{d,C}y\otimes u').
  \end{align*}
\end{proof}

\begin{ex}
  Si $h$ est une forme $\eps$-hermitienne sur $(A,\sigma)$,
  alors en combinant les exemples \ref{ex_morita} et \ref{ex_lambda2},
  la proposition montre que $\lambda^2(h) = T_{\sigma_h}^+$ si
  $\sigma_h$ est symplectique, et $\lambda^2(h) = T_{\sigma_h}^-$
  si $\sigma_h$ est orthogonale.
\end{ex}

\section{Idéal fondamental et filtration}

Outre la structure d'anneau grec, l'ingrédient principal pour étendre
les méthodes du chapitre précédent aux algèbres à involution est l'existence
de la filtration fondamentale de l'anneau de Grothendieck-Witt et son lien
avec la cohomologie galoisienne.

\subsection{Dimension et idéaux}

Comme pour l'idéal fondamental de $GW(K)$, l'idéal fondamental de $\tld{GW}(A,\sigma)$
est défini par rapport à une application \og dimension\fg{}.

\begin{defi}\label{def_dim}
  Soit $V$ un module à droite sur une algèbre simple centrale $A$. On appelle
  \emph{dimension réduite} de $V$, et on note $\dim(V)$, le degré de l'algèbre $\End_A(V)$
  (voir \cite[def 1.9]{BOI}).

  Si $(V,h)$ est un module $\eps$-hermitien sur $(A,\sigma)$, on notera
  aussi $\dim(h) = \dim(V)$.
\end{defi}

\begin{prop}\label{prop_diag_dim}
  L'application $(V,h)\mapsto \dim(V)$ définit un morphisme de
  semi-anneaux grecs gradués $\tld{SGW}(A,\sigma)\To \N[\Gamma]$,
  qui se prolonge de façon unique en
  un morphisme d'anneaux grecs gradués $\tld{GW}(A,\sigma)\To \Z[\Gamma]$.
  On a un diagramme commutatif :
  \[ \begin{tikzcd}\label{eq_diag_dim}
      GW^\pm(K) \rar \dar & \tld{GW}(A,\sigma) \dar \rar{\dim} & \Z[\Gamma] \rar{\eps} \dar & \Z \dar \\
      W(K) \rar & \tld{W}(A,\sigma) \rar{\overline{\dim}} & \Z/2\Z[\Gamma] \rar{\eps} & \Z/2\Z
    \end{tikzcd} \]
  où $\eps$ est la flèche d'augmentation. Ce diagramme est naturel en
  $(A,\sigma)$ (par rapport à la catégorie $\CBrh$), et naturel en $K$.
\end{prop}

\begin{proof}
  La première assertion dit simplement que $\dim(h\cdot g)= \linebreak \dim(h)\dim(g)$
  et $\dim(\lambda^d(h)) = \binom{\dim(h)}{d}$, ce qui découle des constructions.
  Le diagramme commutatif est clair, de même que la naturalité en $K$,
  et la naturalité en $(A,\sigma)$ dit essentiellement que la dimension
  est préservée par équivalence de Morita, ce qui résulte du fait d'avoir
  choisi la dimension réduite.
\end{proof}

\label{par_gi}On définit $\tld{GI}(A,\sigma)$ comme le noyau du morphisme $\tld{GW}(A,\sigma)\to \Z$
et $\tld{I}(A,\sigma)$ comme le noyau de $\tld{W}(A,\sigma)\to \Zd$ (ou de
façon équivalente comme la projection de $\tld{GI}(A,\sigma)$ dans $\tld{W}(A,\sigma)$).
On observe ici une dichotomie importante selon si $A$ est déployée ou non :
si $A$ n'est pas déployée, tous les éléments de $GW^\pm(A,\sigma)$ sont
de dimension paire (et les éléments de $GW^-(K)$ sont de dimension paire
dans tous les cas), donc $\tld{I}(A,\sigma) = I(K)\oplus W^\pm(A,\sigma)$
est un idéal homogène.

Si $\gamma=o,s\in \Gamma$, on notera $\tld{GI}(A,\sigma)_\gamma$
(resp. $\tld{I}(A,\sigma)_\gamma$) l'intersection de $\tld{GI}(A,\sigma)$
(resp. $\tld{I}(A,\sigma)$) avec $\tld{GW}(A,\sigma)_\gamma\subset \tld{GW}(A,\sigma)$
(resp. $\tld{W}(A,\sigma)_\gamma\subset \tld{W}(A,\sigma)$).

\begin{defi}\label{def_filtr}
  La \emph{filtration fondamentale} de $\tld{GW}(A,\sigma)$
  (resp. $\tld{W}(A,\sigma)$) est $(\tld{GI}^n(A,\sigma))_{n\in \N}$
  (resp. $(\tld{I}^n(A,\sigma))_{n\in \N}$). On note
  \[ \tld{H}^n(A,\sigma) = \tld{I}^n(A,\sigma)/\tld{I}^{n+1}(A,\sigma) \]
  le gradué associé, qu'on qualifiera de \emph{cohomologie mixte}
  sur $(A,\sigma)$.
\end{defi}

Par définition, $\tld{H}^0(A,\sigma) = \Zd$. Notons que par naturalité
du diagramme (\ref{eq_diag_dim}), $\tld{H}^n(A,\sigma)$ ne dépend à
isomorphisme (non canonique) près que de la classe de Brauer de $A$.

\subsection{Le cas déployé}

Comme on vient d'en faire l'observation, le cas où $A$ est déployée
est sensiblement différent du cas général. On le traite à part, en
se ramenant par équivalence de Morita (non canonique a priori, donc)
au cas $(A,\sigma)=(K,\Id)$, et on utilise les notations de la
partie \ref{sec_witt_depl}.
\label{par_gi_depl}On utilise un léger abus de notation en écrivant $\tld{GI}(K)$ et
$\tld{I}(K)$ au lieu de $\tld{GI}(K,\Id)$ et $\tld{I}(K,\Id)$.
Pour décrire ces idéaux, on pose $GI(K) = \Ker(GW^\pm(K)\to \Z)$
et $GJ(K) = \Ker(GW^\pm(K)\to \Zd)$. On pose également
$GJ^{(n)}(K)=GI^{n-1}(K)GJ(K)$ si $n\pgq 1$ (et $GJ^{(0)}(K)=GW^\pm(K)$).

\begin{rem}
  On peut montrer que si $n\pgq 2$, alors $GJ^{(n)}(K)= GJ^n(K)\cap GI(K)$,
  et que l'inclusion $GI^n(K)\subset GJ^{(n)}(K)$ est en général stricte.
\end{rem}

Le fait que les idéaux qui nous intéressent ne soient pas homogènes
signifie qu'ils ne se décomposent pas sous la forme $I_1\oplus I_2\subset R\oplus R$
à travers l'application $\Phi$. En revanche, on a :

\begin{prop}
  Soit $n\pgq 1$. Alors (toujours avec les notations de la partie \ref{sec_witt_depl}):
  \begin{align*}
    \tld{GI}^n(K) &= \Psi(GI^n(K)\oplus GJ^{(n-1)}(K)), \\
    \tld{I}^n(K) &= \Psi(I^n(K)\oplus I^{n-1}(K)).
  \end{align*}
\end{prop}

\begin{proof}
  On procède naturellement par récurrence sur $n$. Il est
  clair par définition de $\Psi$ que $\dim(\Psi(x,y))=\dim(x)$,
  donc $\tld{GI}(K) = \Psi(GI(K)\oplus GW^\pm(K))$, et donc
  $\tld{I}(K) = \Psi(I(K)\oplus W(K))$.

  Supposons que la propriété tienne jusqu'à $n\in \N^*$. Soient alors
  $x\in \tld{GI}(K)$ et $y\in \tld{GI}^n(K)$, qu'on écrit $x=\Psi(x_1,x_2)$
  et $y=\Psi(y_1,y_2)$ avec $x_1\in GI(K)$, $y_1\in GI^n(K)$ et $y_2\in GJ^{(n-1)}(K)$.
  On a d'après la formule (\ref{eq_psi_prod}) $xy = \Psi(z_1,z_2)$
  avec $z_1 = x_1y_1\in GI^{n+1}(K)$ et $z_2 = x_1y_2+x_2y_1+2x_2y_2\in GJ^{(n)}(K)$
  (noter que c'est le terme $2x_2y_2$ dans le cas $n=1$ qui nous oblige
  à prendre $GJ^{(n)}$ et non $GI^n$),
  donc $\tld{GI}^{n+1}(K) \subset \Psi(GI^{n+1}(K)\oplus GJ^{(n)}(K))$.

  Réciproquement, si $x\in GI^{n+1}(K)$ est de la forme $x=x_1x_2$
  avec $x_1\in GI(K)$ et $x_2\in GI^n(K)$, alors $\Psi(x,0)=\Psi(x_1,0)\Psi(x_2,0)$
  donc $\Psi(x,0)\in \tld{GI}^{n+1}(K)$, et si $y\in GJ^{(n)}(K)$ est de la forme
  $y=y_1y_2$ avec $y_1\in GI(K)$ et $y_2\in GJ^{(n-1)}(K)$, alors
  $\Psi(0,y) = \Psi(y_1,0)\Psi(0,y_2)$, donc $\Psi(0,y)\in \tld{GI}^{n+1}(K)$.
  Finalement, on a bien $\tld{GI}^{n+1}(K) \supset \Psi(GI^{n+1}(K)\oplus GJ^{(n)}(K))$.

  L'assertion équivalente dans $\tld{W}(K)$ en découle facilement. 
\end{proof}

\begin{coro}\label{cor_h_tilde_depl}
  En particulier, $\tld{H}^n(K) \simeq H^n(K,\mu_2)\oplus H^{n-1}(K,\mu_2)$
  à travers $\Psi$.
\end{coro}

\begin{rem}
  Si $A$ est déployée, alors dans l'isomorphisme
  $\tld{H}^n(A,\sigma)\simeq H^n(K,\mu_2)\oplus H^{n-1}(K,\mu_2)$,
  si on note $(x_1,x_2)$ l'image de $x\in\tld{H}^n(A,\sigma)$, alors
  $x_2$ ne dépend pas du choix de l'équivalence de Morita entre $(A,\sigma)$
  et $(K,\Id)$, mais $x_1$ n'est défini qu'à $x_2\cup H^1(K,\mu_2)$ près. 
\end{rem}

\subsection{Le cas non déployé}

Dans le cas où $A$ n'est pas déployée, on a déjà une bonne décomposition
homogène $I(A,\sigma) = I(K)\oplus W^\pm(A,\sigma)$.

\begin{prop}\label{prop_decomp_cohom}
  Pour tout $n\pgq 1$, la décomposition homogène de $\tld{I}^n(A,\sigma)$
  est donnée par
  \[ \tld{I}^n(A,\sigma) = I^n(K) \oplus I^n(A,\sigma)_o \oplus I^n(A,\sigma)_s \]
  avec $I^n(A,\sigma)_\gamma = I^{n-1}(K)W(A,\sigma)_\gamma$ pour $\gamma =o,s\in \Gamma$.
\end{prop}

\begin{proof}
  L'inclusion de la droite vers la gauche est claire. Pour l'autre
  sens on raisonne par récurrence sur $n$, le cas $n=1$ étant
  trivial. On suppose que l'égalité tient jusqu'à $n\pgq 1$.

  Soient $x\in \tld{I}(A,\sigma)$ et $y\in \tld{I}^n(A,\sigma)$. On
  peut se ramener à $x$ et $y$ homogènes, et on voit que pour montrer
  $xy\in I^n(K)$, $I^n(A,\sigma)_o$ ou $I^n(A,\sigma)_s$, il
  suffit de montrer que $W^\eps(A,\sigma)\cdot W^{\eps'}(A,\sigma)\subset I^2(K)$.
  Or c'est vrai après déploiement générique puisque tout élément
  de $W^\eps(A,\sigma)$ est de dimension paire ; mais le fait d'être dans $I^2(K)$
  est détecté par le discriminant, qui est préservé par déploiement générique.
\end{proof}

On pose pour $\gamma=o,s\in \Gamma$ :
\begin{equation}\label{eq_h_gamma}
  H^n(A,\sigma)_\gamma = I^n(A,\sigma)_\gamma/I^{n+1}(A,\sigma)_\gamma. 
\end{equation}

\begin{coro}
  On a pour tout $n\pgq 1$
  \[ \tld{H}^n(A,\sigma) = H^n(K,\mu_2)\oplus H^n(A,\sigma)_o \oplus H^n(A,\sigma)_s, \]
\end{coro}

Notons que d'après la proposition \ref{prop_decomp_cohom}, on peut
définir une application naturelle biadditive et surjective
\[ H^{n-1}(K,\mu_2)\times W(A,\sigma)_\gamma \To H^n(A,\sigma)\gamma, \]
dont il serait intéressant de comprendre le noyau. On note
\[ (a_1,\dots,a_{n-1}; a)_\sigma \]
l'image de $\left( (a_1,\dots,a_{n-1}), \fdiag{a}_\sigma \right)$ par cette
application ($a_i\in K^*$, $a\in A^*$ symétrique ou anti-symétrique).

\begin{rem}\label{rem_cohom_depl}
  A priori, si $\gamma=o,s\in \Gamma$, on peut définir $\tld{I}^n(A,\sigma)_\gamma$
  soit comme la $n$-ème puissance de $\tld{I}(A,\sigma)_\gamma$, soit comme 
  l'intersection de $\tld{I}^n(A,\sigma)$ avec $\tld{W}(A,\sigma)_\gamma$.
  La proposition \ref{prop_decomp_cohom} permet alors de voir que ces deux définitions
  sont équivalentes quand $A$ n'est pas déployée, donc la notation n'est pas ambiguë.

  En revanche, si $A$ est déployée, $\tld{I}(A,\sigma)_s=I(K)$, donc sa $n$-ième
  puissance est $I^n(K)$, tandis que l'intersection de $\tld{I}(A,\sigma)^n$
  avec $\tld{W}(A,\sigma)_s=W(K)$ est $I^{n-1}(K)$. C'est alors la première
  définition qu'on emploiera.
\end{rem}







\subsection{Réduction générique d'indice}

Un des outils standard dans l'étude des algèbres à involution est
le déploiement générique des algèbres en question, ou, de façon
un peu plus générale, la réduction générique de l'indice de l'algèbre.
Il va de soi qu'on dispose pour toute extension de corps $L/K$ d'un
morphisme de restriction $r_{L/K}: \tld{H}^n(A,\sigma)\to \tld{H}^n(A_L,\sigma_L)$,
et dans la mesure où on a une bonne description de $\tld{H}^n$ dans le
cas déployé, ce type de méthode s'impose de façon assez naturelle
dans notre contexte.

Supposons que $A$ ne soit pas déployée. Si $A_L$ n'est pas déployée
non plus, alors le morphisme $\tld{H}^n(A,\sigma)\to \tld{H}^n(A_L,\sigma_L)$
est simplement défini sur chaque composante : $H^n(K,\mu_2)\to H^n(L,\mu_2)$,
et $H^n(A,\sigma)_\gamma\to H^n(A_L,\sigma_L)_\gamma$ pour $\gamma=o,s$.
Si maintenant $A_L$ est déployée :

\begin{prop}
  Soit $(A,\sigma)$ non déployée sur $K$, et soit $L/K$ une extension
  qui déploie $A$. Le morphisme de restriction
  \[ H^n(K,\mu_2)\oplus H^n(A,\sigma)_o\oplus H^n(A,\sigma)_s \To H^n(L,\mu_2)\oplus H^{n-1}(L,\mu_2)\]
  est indépendant du choix d'une équivalence de Morita hermitienne
  entre $(A_L,\sigma_L)$ et $(L,\Id)$. Il est de la forme
  \[ (x,y,z)\mapsto (x_L+y_L,0) \]
  où $y\mapsto y_L$ est un morphisme $H^n(A,\sigma)_o\to H^n(L,\mu_2)$ également
  indépendant de tout choix d'équivalence de Morita.
\end{prop}

\begin{proof}
  On se fixe un choix d'équivalence de Morita hermitienne entre
  $(A_L,\sigma_L)$ et $(L,\Id)$. On en déduit un morphisme
  $W(A,\sigma)_o\to I(L)$ (puisque tout élément de $W(A,\sigma)_o$
  est de dimension paire), qu'on note $h\mapsto h_L$, et un choix différent
  d'équivalence donne un morphisme différent d'une similitude.
  Soit $(x,y,z)\in \tld{H}^n(A,\sigma)$ se relevant en
  $(q,q'h,q''h')\in \tld{I}^n(A,\sigma)$. Son image par
  le morphisme de restriction
  \[ I^n(K)\oplus I^{n-1}(K)W(A,\sigma)_o\oplus I^{n-1}(K)W(A,\sigma)_s \To I^n(L)\oplus I^{n-1}(L) \]
  est donnée par $(q_L+q'_Lh_L,-q'_Lh_L)$, et donc l'image de $(x,y,z)$
  par la restriction est
  \[ (e_n(q_L)+ e_n(q'_Lh_L),e_{n-1}(q'_Lh_L)) = (x_L+y_L,0) \]
  où $y_L=e_n(q'_Lh_L)$ ne dépend que de $y$.
\end{proof}

\label{par_red_gen}Pour tout $r\in \N$, on pose $K_r(A)$ un corps de réduction générique
à l'indice $2^r$ d'une algèbre simple centrale $A$ ; notamment $K_0(A)$
est un corps de déploiement générique de $A$, et si $2^r\pgq \ind(A)$
alors $K_r(A)=K$.

\begin{defi}\label{def_gi_red}
  Pour tous $n,r\in \N$, on pose
  \[ \tld{GI}_r^n(A,\sigma) = r_{K_r(A)/K}^{-1}(\tld{GI}^n(A_{K_r(A)},\sigma_{K_r(A)})) \]
  où $r_{K_r(A)/K}: \tld{GW}(A,\sigma) \To \tld{GW}(A_{K_r(A)},\sigma_{K_r(A)})$
  est le morphisme de restriction.

  On pose de même
  \[ \tld{I}_r^n(A,\sigma) = r_{K_r(A)/K}^{-1}(\tld{I}^n(A_{K_r(A)},\sigma_{K_r(A)})). \]
\end{defi}

Notons que $\tld{GI}^1_r(A,\sigma)=\tld{GI}(A,\sigma)$ pour tout $r\in \N$
puisque la condition sur la dimension est inchangée par extension des
scalaires ; de même, $\tld{I}^1_r(A,\sigma)=\tld{I}(A,\sigma)$.
Notons également que la composante symplectique de $\tld{GW}(A,\sigma)$
est contenue dans tous les $\tld{GI}^n_0(A,\sigma)$ puisqu'après
déploiement générique une involution symplectique devient hyperbolique.

On a clairement une croissance par rapport à $n$ et à $r$ :
\[ \tld{GI}^n_r(A,\sigma) \subset \tld{GI}^{n+1}_r(A,\sigma) ,\quad
  \tld{GI}^n_{r+1}(A,\sigma) \subset \tld{GI}^n_r(A,\sigma). \]

Cette hiérarchie permettra de donner des énoncés plus fins concernant
les invariants cohomologiques dans le chapitre suivant.

\subsection{Algèbres de quaternions}\label{sec_quater}

On s'attarde ici sur le cas des algèbres de quaternions munies de leur
involution canonique, et de leur déploiement générique. Soit donc
$(Q,\can)$ une algèbre de quaternions sur $K$, munie de son
involution canonique.

\label{par_valu}Dans \cite{QT17}, Quéguiner et Tignol décrivent l'effet du déploiement
générique de $Q$ sur $W^\eps(Q,\can)$. Soit $C$ la variété de Severi-Brauer
de $Q$, qui est une conique, et soit $F=K(C)$ son corps de fonction,
qui est donc un corps de déploiement générique de $Q$. Pour tout
point fermé $\mathfrak{p}$ de $C$, on a une valuation $v_{\mathfrak{p}}$ discrète de rang
1 sur $F$, et donc on a les deux application résidus
usuelles $\partial_{1,\mathfrak{p}},\partial_{2,\mathfrak{p}}: W(F)\To W(K(\mathfrak{p}))$
(où $\partial_{2,\mathfrak{p}}$ dépend du choix d'une uniformisante locale
en $\mathfrak{p}$). Le choix d'un point fermé $a\in C$ de degré 2,
qu'on notera en général $a=\infty$,
donne en particulier une équivalence de Morita hermitienne entre $(Q_F,\can)$
et $(F,\Id)$ et donc un isomorphisme $W^-(Q_F,\can)\simeq W(F)$,
ce qui donne une application d'extension des scalaires
\begin{equation}\label{eq_ext_a}
  \ext_a : W^-(Q,\can) \To W(F). 
\end{equation}

Si on choisit un autre point $b\in C$ comme point à l'infini,
on dispose alors d'un scalaire $\lambda_{a,b}\in F^*$ tel que
le diagamme suivant commute :
\[ \begin{tikzcd}
    & W(F) \arrow{dd}{\cdot \fdiag{\lambda_{a,b}}} \\
    W^-(Q,\can) \urar{\ext_a} \drar{\ext_b}  & \\
    & W(F).
  \end{tikzcd} \]

On pose alors :
\begin{equation}\label{eq_w1}
  W_{nr,1,\infty}(F) = \ens{q\in W(F)}{\partial_{1,\infty}(q)=0,\,
    \forall \mathfrak{p}\neq \infty, \partial_{2,\mathfrak{p}}(q)=0} 
\end{equation}
et
\begin{equation}\label{eq_w2}
  W_{nr,2}(F) = \ens{q\in W(F)}{\forall \mathfrak{p}, \partial_{2,\mathfrak{p}}(q)=0}. 
\end{equation}
De plus on peut montrer que $\lambda_{a,b}$ est non ramifié sur tout
$c\in C$ sauf en $a$ et $b$, de sorte que la multiplication par $\lambda_{a,b}$
envoie $W_{nr,1,a}(F)$ sur $W_{nr,1,b}(F)$.
Dans \cite{QT17} il est démontré que $\ext_\infty$ induit :
\begin{equation}
  W^-(Q,\can)\Isom W_{nr,1,\infty}(F)
\end{equation}
et
\begin{equation}
  0\To n_QW(K) \To W(K) \To W_{nr,2}(F) \To 0.
\end{equation}

En particulier, le noyau de l'application de restriction
$\tld{W}(Q,\can)\to \tld{W}(Q_F,\can)$ est $n_QW(K)\oplus W^+(Q,\can)$.

\begin{rem}
  On pourrait tenter d'utiliser ces résultats pour définir les opérations
  $\lambda^d$ pour des formes quaternioniques : si $d$ est impair,
  on voit facilement que $\lambda^d$ sur $W(F)$ préserve $W_{nr,1,\infty}(F)$,
  donc si $h\in W^-(Q,\can)$, on peut définir $\lambda^d(h)$ comme
  $\ext_\infty^{-1}(\lambda^d(\ext_\infty(h)))$, ce qui ne dépend pas
  du choix de $\infty$, puisque $\lambda^d$ commute avec la multiplication
  par $\fdiag{\lambda_{a,b}}$ où $a$ et $b$ sont deux choix différents
  pour $\infty$. En revanche, si $d$ est pair, on ne peut définir
  par cette méthode que $\lambda^d(h)$ modulo $n_Q$ : en effet $\lambda^d$
  préserve cette fois $W_{nr,2}(F)$, mais le morphisme de restriction
  $W(F) \To W_{nr,2}(F)$ n'est pas injectif.
\end{rem}

\label{par_valu_cohom}En ce qui concerne la cohomologie, on a également des applications
de résidu $\partial_{\mathfrak{p}}: H^d(F,\mu_2)\To H^{d-1}(K(\mathfrak{p}),\mu_2)$,
et on pose
\begin{equation}\label{eq_h_nr}
  H^d_{nr}(F,\mu_2) = \Ker\left(\bigoplus_{\mathfrak{p}} \partial_{\mathfrak{p}}\right).
\end{equation}
On peut alors noter que si $q\in I^d(F)\cap W_{nr,i}(F)$ pour $i=1,2$, alors
$e_d(q)\in H^d_{nr}(F,\mu_2)$.

Cela nous permet de définir un morphisme
\begin{equation}\label{eq_ed_tilde}
   \tld{e}_d : \tld{I}_0^d(Q,\can) \To  H^{d-1}_{nr}(F,\mu_2) 
\end{equation}
qui est indépendant du choix de $\infty$. En effet, 
après déploiement générique, et équivalence de Morita
définie par le choix de $\infty$, un élément de $\tld{I}_0^d(Q,\can)$
est envoyé sur un élément de $\tld{I}^d(F)$ tel que la partie
paire soit dans $W_{nr,2}(F)$ et la partie impaire dans $W_{nr,1,\infty}(F)$,
donc on obtient un élément de $I^{d-1}(F)\cap W_{nr,2}(F)$ (indépendant de $\infty$)
et un élément de $I^{d-1}(F)\cap W_{nr,1,\infty}(F)$ (qui ne dépend de $\infty$
qu'à similitude près), dont la somme est dans $I^d(F)$,
et donc l'invariant $e_{d-1}$ de chacun de ces éléments donne la
même classe dans $H^{d-1}_{nr}(F,\mu_2)$ (qui ne dépend donc pas de $\infty$).

\begin{rem}
  On constate qu'on peut définir par ce déploiement générique
  une classe dans $H^d(F)$ (autre façon de le voir: on peut naturellement
  définir une classe dans $\tld{H}^d(F) = H^d(F)\oplus H^{d-1}(F)$),
  mais cette classe est en général ramifiée en $\infty$, et on ne pourra
  donc pas la redescendre sur $K$. 
\end{rem}

\label{par_m_n}On trouve dans \cite[prop A.1]{KRS} une description
du noyau et du conoyau de $H^d(K,\Q/\Z(d-1))\To H^d(F,\Q/\Z(d-1)$.
En ce qui nous concerne, pour toute algèbre simple
centrale $A$, on définit $M_A^d(K)$ comme étant le sous-groupe
de 2-torsion de
\[  H^d(K,\mu_4^{\otimes d-1})/[A]\cdot H^{d-2}(K,\mu_2) \]
(qui ne dépend que de la classe de Brauer de $A$), ainsi que
$N_A^d(K)\subset M_A^d(K)$ l'image de l'application naturelle
$H^d(K,\mu_2)\To M_A(K)$.
Le résultat de \cite{KRS} donne alors un isomorphisme induit par la restriction :
\begin{equation}\label{eq_depl_ind_2}
  M_Q^d(K) \simeq H_{nr}^d(F,\mu_2).
\end{equation}

La construction précédente nous donne donc un morphisme
\[ \tld{e}_d : \tld{I}_0^d(Q,\can) \To  M_Q^{d-1}(K). \].
Il s'avère qu'on peut faire un peu mieux. 

\begin{prop}\label{prop_ed_quater}
  Pour tout $d\pgq 1$, le morphisme $\tld{e}_d$ ci-dessus est
  à valeurs dans $N_Q^{d-1}(K)$.
\end{prop}

Pour cela on démontre d'abord le lemme suivant :

\begin{lem}
  Soient $m\in \N$ et $q\in W(K)$ telle que $q_F\in I^m(F)$.
  Alors il existe $q'\in W(K)$ telle que $q-n_Qq'\in I^m(K)$.
\end{lem}

\begin{proof}[Démonstration du lemme]
  Le résultat est clair si $m\ppq 2$ car alors on a directement
  $q\in I^m(K)$ (donc on peut prendre $q'=0$).
  Pour le cas général on procède par récurrence : supposons
  que le résultat vale jusqu'à $m\pgq 2$. Si $q_F\in I^{m+1}(F)$,
  alors par hypothèse on a $q_1\in W(K)$ tel que $q-n_Qq_1\in I^m(K)$.
  On a alors $e_m(q-n_Qq_1)_F=e_m(q_F)=0$, donc $e_m(q-n_Qq_1)\in [Q]H^{m-2}(K,\mu_2)$.
  On pose $q_2\in I^{m-2}(K)$ tel que $e_m(q-n_Qq_1)=[Q]e_{m-2}(q_2)$,
  et on définit $q'=q_1+q_2$. On a bien $e_m(q-n_Qq')=0$ donc
  $q-n_Qq'\in I^{m+1}(K)$.
\end{proof}

\begin{proof}[Démonstration de la proposition]
  Les cas $d=1,2$ sont clairs puisque $N_Q^0(K)=M_Q^0(K)=H^0(K,\mu_2)$ et
  $N_Q^1(K)=M_Q^1(K)=H^1(K,\mu_2)$.
  Si $d\pgq 3$, on prend $q\in W(K)$ la
  composante quadratique de $x\in \tld{I}_0^d(Q,\can)$, et
  on choisit $q'$ comme dans le lemme pour que $q-n_Qq'\in I^{d-1}(K)$.
  Alors $e_{d-1}(q-n_Qq')_F=e_{d-1}(q_F)$ donc
  $e_{d-1}(q-n_Qq')$ est bien une descente de $\tld{e}_d(x)$
  dans $H^{d-1}(K,\mu_2)$.
\end{proof}

\begin{rem}
  Il y a une dimension \og{}concrète\fg{} à cette proposition,
  contrairement à la définition donnée ci-dessus de $\tld{e}_d$,
  qui utilise le résultat de \cite{KRS} comme boîte noire
  (lui-même utilisant des résultats difficiles comme la conjecture
  de Milnor) et ne permet pas d'avoir des éléments explicites de
  $M_Q^d(K)$. En effet, il suffit en théorie de suivre le fil de la
  preuve du lemme, et de trouver à chaque étape une factorisation
  de $e_m(q-n_Qq_1)\in H^m(K,\mu_2)$ (en reprenant les notations de la preuve) par
  $[Q]$. On utilise ici seulement le résultat de \cite{KRS} pour garantir
  qu'une telle factorisation existe.
\end{rem}

\chapter{Invariants d'algèbres à involution}

Dans ce chapitre on met en action les différentes techniques
développées précédemment pour construire des invariants
cohomologiques d'algèbres à involution.

Dans la première partie, on exploite le fait que les invariants
$\Inv(I,I^d)$ construits dans le premier chapitre peuvent en réalité
s'appliquer dans n'importe quel anneau grec, puisque ce sont des combinaisons
d'opérations $\lambda$. On peut donc les appliquer dans l'anneau
$\tld{GW}(A,\sigma)$ défini au deuxième chapitre. L'espoir initial
qu'on obtienne ainsi des applications de $\tld{GI}(A,\sigma)$ (voir
\ref{def_filtr}) vers $\tld{H}^d(A,\sigma)$ est mis en défaut, mais
on obtient tout de même des invariants de degré arbitrairement grands
en tenant compte de l'obstruction donnée par l'indice de l'algèbre
(voir le corollaire \ref{cor_inv_herm} et la discussion qui suit).
On accorde un soin particulier à l'étude des algèbres d'indice 2, pour lesquelles
on dispose de formules plus précises (voir la proposition
\ref{prop_g_quater}), ainsi que de méthodes de déploiement générique.

La deuxième partie est consacrée à des calculs explicites autour
de l'équivalence exceptionnelle entre algèbres de type $A_3$ et $D_3$
(voir les théorèmes \ref{thm_d3_a3} et \ref{thm_a3_d3}). Ces calculs
sont mobilisés dans la partie suivante, qui donne quelques exemples
d'applications de nos méthodes (et notamment des méthodes de déploiement
générique en indice 2, voir la proposition \ref{prop_ber} adaptée de \cite{Ber})
pour généraliser certains invariants précédemment connus. Notamment,
on étend la définition d'un invariant de degré 4 construit dans \cite{RST}
(voir la proposition \ref{prop_a3_a4}) pour les algèbres de type $D_3$,
et d'un invariant de degré 5 construit dans \cite{Gar} pour les
formes quadratiques de degré 12 dans $I^3$ (voir la partie \ref{sec_a5}).

\section{Invariants de formes hermitiennes}\label{sec_inv_herm}

\subsection{Opérations $\pi_1^d$}

Soit $(A,\sigma)$ une algèbre à involution de première espèce sur $K$.
Comme $\tld{GW}(A,\sigma)$ est un anneau grec, on peut y appliquer
les opérations $\pi_n^d$ définies dans le premier chapitre. On dispose
de la filtration $\tld{GI}^n(A,\sigma)$, qui est l'analogue de la filtration
fondamentale de $GW(K)$, et on peut donc espérer que
$\pi_n^d(\tld{GI}^n(A,\sigma))\subset \tld{GI}^{nd}(A,\sigma)$. Ce
résultat est en général faux, mais on obtient quand même (où on
note $\lceil x \rceil$ la partie entière \emph{supérieure} d'un
réel $x\in \R$) :

\begin{prop}\label{prop_inv_herm}
  Soit $m=\max(2,\ind(A))$. Alors pour tout $d\in \N$, l'opération
  $\pi_1^d$ définit une application $\tld{GI}(A,\sigma)\to
  \tld{I}^{\lceil d/m \rceil}(A,\sigma)$.

  De plus, en prenant $p=\ind(A)$, $\pi_1^d$ définit une application
  $\tld{GI}(A,\sigma)_o\to \tld{I}^{\lceil d/p \rceil}(A,\sigma)_o$.
\end{prop}

\begin{proof}
  On commence par traiter le cas où $A$ n'est pas déployée ;
  le deuxième énoncé découle alors du premier, et on a $m=\ind(A)$.
  On peut conclure si on arrive à montrer que $\tld{GI}(A,\sigma)$
  est additivement engendré par des éléments $x$ tels que la classe
  de Witt de $\pi_1^d(x)$ est dans $\tld{I}^{\lceil d/m \rceil}(A,\sigma)$
  pour tout $d\in \N$. En effet, si cette propriété vaut pour $x$
  elle vaut aussi pour $-x$ : on a $0 = \pi_1^d(x - x)$ donc
  $\pi_1^d(-x) = -\sum_{k=0}^{d-1}\pi_1^{d-k}(x)\pi_1^k(-x)$, et on
  peut procéder par récurrence sur $d$, le cas $d=0$ étant trivial.
  Il suffit alors de voir que pour tout $0\ppq k\ppq d-1$,
  $\lceil (d-k)/m \rceil + \lceil k/m \rceil\pgq \lceil d/m \rceil$.
  De même, si $x$ et $y$ vérifient la propriété, on montre que c'est aussi
  le cas de $x+y$ puisque $\pi_1^d(x+y)=\sum_{k=0}^d\pi_1^k(x)\pi_1^{d-k}(y)$
  et on utilise la même inégalité. On peut alors conclure par récurrence
  sur le nombre de générateurs additifs intervenant dans l'écriture
  d'un élément.
  
  On rappelle que $\pi_1^d = \sum_{k=1}^d (-1)^{d-k}\binom{d-1}{k-1}\lambda^k$.
  Or si $x\in \tld{GW}(A,\sigma)$, dans $\tld{W}(A,\sigma)$ on a
  $(-1)^k\lambda^k(x) = \fdiag{-1}^k\lambda^k(x)=\lambda^k(\fdiag{-1}x)$,
  donc :
  \begin{align*}
    \pi_1^d(x) &= (-1)^d\sum_{k=0}^d \lambda^{d-k}(d-1)(-1)^k\lambda^k(x) \\
               &= (-1)^d\sum_{k=0}^d\lambda^{d-k}(d-1)\lambda^k(\fdiag{-1}x)\\
               &= (-1)^d\lambda^d(\fdiag{-1}x+(d-1)\fdiag{1}) \\
               &= \lambda^d(x+(d-1)\fdiag{-1}).
  \end{align*}

  De là, on utilise le fait que $\tld{GI}(A,\sigma)$ est additivement
  engendré par : les éléments de $\hat{I}(K)$ (pour lesquels la propriété
  voulue est une conséquence des résultats du premier chapitre),
  et les éléments de la forme
  $x=y-r\fdiag{-1}$ où $y$ est un élément de dimension minimale $r$
  dans $GW^-(K)$ (auquel cas $r=2$) ou $SGW^\eps(A,\sigma)$ (auquel cas
  $r=m$). Pour un tel élément $x$, on a $\pi_1^d(x) = \lambda^d(y+(d-1-r)\fdiag{-1})$,
  donc si $d>r$ (et donc en particulier si $d>m$) on a $\pi_1^d(x)=0$
  (en effet pour un élément $x$ de $\tld{SGW}(A,\sigma)$, si $d$ est strictement
  plus grand que la dimension de $x$, $\lambda^d(x)=0$).

  Supposons maintenant que $A$ soit déployée. Pour le premier énoncé,
  on a $m=2$. On peut alors reprendre la même preuve que dans le cas
  non déployé, en observant que les générateurs de dimension minimale
  dans les composantes alternée et symplectique sont de dimension $m=2$
  bien que l'indice soit $1$. Pour le deuxième énoncé, on a $p=1$,
  mais les générateurs de dimension minimale sont bien cette fois de
  dimension $p=\ind(A)=1$, puisqu'on n'a plus l'obstruction des plans
  hyperboliques alternés.
\end{proof}

En utilisant les notations de réducion d'indice de la définition
\ref{def_gi_red}, on obtient :

\begin{coro}
  Soit $r\pgq 1$. Alors pour tout $d\in \N$, $\pi_1^d$ induit
  des applications $\tld{GI}_r(A,\sigma)\to \tld{I}_r^{\lceil d/2^r \rceil}(A,\sigma)$,
  $\tld{GI}_0(A,\sigma)\to \tld{I}_0^{\lceil d/2 \rceil}(A,\sigma)$, et
  $\tld{GI}_0(A,\sigma)_{o}\to \tld{I}_0^d(A,\sigma)_o$.
\end{coro}

\begin{proof}
  Il suffit d'appliquer les propositions précédentes après
  une extension des scalaires à $K_r(A)$, le corps de réduction
  générique à l'indice $2^r$.
\end{proof}

\subsection{Invariants cohomologiques}

On peut utiliser les résultats précédents sur les opérations
$\pi_1^d$ pour étendre tous les
invariants de Witt de $I$ étudiés dans le premier chapitre.

\begin{coro}\label{cor_inv_herm}
  Soient $d\in \N$ et $\alpha\in \Inv(I,I^d)$. Pour tout
  $r\pgq 1$, l'opération $\alpha$ sur $\tld{GW}(A,\sigma)$
  induit des applications $\tld{GI}_r(A,\sigma)\to \tld{I}_r^{\lceil d/2^r \rceil}(A,\sigma)$,
  $\tld{GI}_0(A,\sigma)\to \tld{I}_0^{\lceil d/2 \rceil}(A,\sigma)$, et
  $\tld{GI}_0(A,\sigma)_{o}\to \tld{I}_0^d(A,\sigma)_o$.
\end{coro}

\begin{proof}
  On écrit $\alpha = \sum_{n=0}^da_n\pi_1^n$ avec $a_n\in I^{d-n}(k)$.
  Alors l'image de $a_n\pi_1^n(\tld{GI}(A,\sigma))$ est dans
  $I^{d-n}(k)\tld{I}_r^{\lceil n/2^r \rceil}(A,\sigma)$, et
  $d-n+\lceil n/2^r \rceil \pgq \lceil d/2^r \rceil$.

  On procède de même pour les cas où $r=0$.
\end{proof}

On peut alors obtenir des invariants cohomologiques :
si $\alpha\in \Inv(I,I^d)$, on obtient une application
naturelle
\[ \bar{\alpha}: \tld{GI}(A,\sigma)\To \tld{H}^{\lceil d/m\rceil}(A,\sigma) \]
où $m=\max(2,\ind(A))$. Sachant que, par la proposition
\ref{prop_decomp_cohom}, $\tld{H}^n(A,\sigma)$ a une composante
identifiée à $H^n(K,\mu_2)$, on obtient donc également une
application naturelle
\[ \alpha': \tld{GI}(A,\sigma)\To H^{\lceil d/m\rceil}(K,\mu_2) \]
donc un authentique invariant cohomologique. On obtient
par ailleurs des invariants à valeurs dans les autres composantes
de $\tld{H}^n(A,\sigma)$, à savoir $H_o^n(A,\sigma)$ et $H_s^n(A,\sigma)$,
mais la nature précise de ces groupes reste à explorer, donc nous
nous limiterons ici à la projection sur la composante cohomologique.

Pour obtenir un invariant \og{}de $(A,\sigma)$\fg{}, on peut
appliquer une de ces constructions à un élément $h\in \tld{GI}(A,\sigma)$
qui \og{}représente\fg{} $(A,\sigma)$ : par exemple
$h=\fdiag{1}_\sigma-\frac{\deg(A)}{2}\mathcal{H}$ (où $\mathcal{H}\in GW(K)$
est le plan hyperbolique), ou $h=\fdiag{1}_\sigma-\deg(A)\fdiag{1}$, ou
encore si $A$ est de coindice $s$ pair $h=\fdiag{1}_\sigma-\frac{s}{2}\mathcal{H}(A,\sigma)$
où $\mathcal{H}(A,\sigma)\in GW(A,\sigma)$ est l'espace hyperbolique
hermitien. Tous ces choix ne sont pas fondamentalement différents
et correspondent simplement à différentes façon de normaliser
l'élément canonique $\fdiag{1}_\sigma$ pour obtenir un élément
de $\tld{GI}(A,\sigma)$ ; précisément, à $\alpha$ fixé ces
choix conduisent à des invariants différents, mais en faisant
varier $\alpha$ l'ensemble des invariants obtenus ne change pas.

Il est à noter que l'élément de $H^{\lceil d/m\rceil}(K,\mu_2)$ ainsi
obtenu ne constitue pas toujours un invariant intéressant
en soi : en effet, en général il est nul après déploiement
générique, puisque par construction l'élément de $I^{\lceil d/m\rceil}(K)$
dont il est la réduction est en réalité dans $I^d$ après
déploiement générique. Il est tout à fait possible que la classe
obtenue soit de la forme $x\cup [A]$ où $x$ ne dépend pas de $\sigma$
(voir l'exemple \ref{ex_discr_sympl} ainsi que la remarque
\ref{rem_gd_ind_2}), et
en particulier les différents $\alpha$ ne donnent a priori pas toujours
des invariants distincts dans les cas non déployés (bien qu'ils
soient distincts en tant qu'invariants globaux puisqu'ils
le sont dans le cas déployé). On peut se faire la représentation
imagée suivante : l'élément $\alpha(h)$ (ou disons sa composante
dans $W(K)$) possède effectivement de l'information \og{}intéressante\fg{}
au niveau de $I^d(K)$ comme dans le cas déployé, mais on n'y a
pas accès directement parce qu'il y a une obstruction au niveau
de $I^{\lceil d/m\rceil}(K)$ (qui disparaît après déploiement),
et possiblement à d'autres $I^n$ pour $n<d$.

Une façon possible de remédier à cela dans certains cas favorables
est de construire des invariants relatifs : au lieu de considérer
la classe de cohomologie de $\alpha(h)$, on peut étudier l'élément
$\alpha(h)-\alpha(h')$, qui peut potentiellement se trouver
dans une puissance supérieure de $I$ si les \og{}obstructions\fg{}
évoquées précédemment se compensent. On illustre cette idée dans la
partie \ref{sec_ind_2} en indice 2, ainsi que dans l'exemple
suivant.

\begin{ex}\label{ex_discr_sympl}
  Le discriminant des involutions symplectiques, qui est un
  invariant de degré 3, peut être défini de la façon suivante
  (voir \cite[thm 4]{BMT}) : soient $\sigma$ et $\sigma'$ deux involutions
  symplectiques sur $A$; alors $T_\sigma^+-T_{\sigma'}^+\in I^3(K)$
  et le discriminant (relatif) $d_\sigma(\sigma')$ est l'invariant
  $e_3$ de cette forme. Or dans $\tld{GW}(A,\sigma)$, si $h\in GW(A,\sigma)$
  a pour adjoint $\sigma'$, on a $T_\sigma^+ = \fdiag{2}\lambda^2(\fdiag{1}_\sigma)$ et
  $T_{\sigma'}^+ = \fdiag{2}\lambda^2(h)$ (voir exemple \ref{ex_lambda2}),
  donc cela correspond à l'idée présentée
  précédemment, avec $\alpha=\lambda^2$. Dans \cite{BMT}, le fait que
  la différence soit dans $I^3$ est justifiée par le fait qu'un groupe
  symplectique n'a pas d'invariant cohomologique non constant en degré 1 et 2,
  mais on peut le justifier par un calcul explicite, qui permet
  de faire certaines remarques.

  On pose $q_\sigma = \fdiag{2}T_\sigma^+ + r\fdiag{-1}$ où $\deg(A)=2r$ ;
  c'est une forme quadratique de dimension $2r^2$.
  Dans \cite{Q97}, Quéguiner montre que $w_1(T_\sigma^+)=w_1(T_\sigma^-)=(2^r)$,
  ce qui implique que $e_1(q_\sigma)=0$ (voir la proposition \ref{prop_fixed_dim}
  pour le lien entre les invariants de Stiefel-Whitney $w_1$ et $w_2$
  et les invariants $e_1$ et $e_2$, en se rappelant que $e_1=u_1^{(1)}$
  et $e_2$ est la restriction à $I^2$ de $u_2^{(1)}$). Elle montre également que
  $w_2(T_\sigma^+) = \binom{r}{2}[A] + \binom{r}{2}(-1,-1)$, ce dont
  on déduit que $w_2(q_\sigma)=\binom{r}{2}[A]$, puis que
  \begin{align*}
    e_2(q_\sigma) &= w_2(q_\sigma) + (r^2-1)(-1)\cup w_1(q_\sigma) + \binom{r^2-2}{2}(-1,-1) \\
                  &= \binom{r}{2}[A] + \binom{r^2-2}{2}(-1,-1).
  \end{align*}
  On voit donc que, comme $T_\sigma^+-T_{\sigma'}^+$ est semblable
  à $q_\sigma-q_{\sigma'}$, $e_2(T_\sigma^+-T_{\sigma'}^+)=0$, ce qui
  justifie la définition de l'invariant (on voit ici un exemple
  où l'invariant absolu n'est que de degré 2 a priori, mais la
  partie de degré 2 ne dépend pas de $\sigma$ donc elle disparaît
  quand on considère la version relative).

  On peut aller plus loin : lorsque $r$ est divisible par 4,
  $e_2(q_\sigma)=(-1,-1)$, donc on peut poser $\phi_\sigma = q_\sigma-\pfis{-1,-1}$,
  et définir un invariant absolu par $e_3(\phi_\sigma)$. Quel que soit $r$,
  lorsque l'indice est au plus 2, on peut poser
  $\phi_\sigma = q_\sigma - \binom{r}{2}n_Q + \binom{r^2-2}{2}\pfis{-1,-1}$,
  et encore définir un invariant absolu par $e_3(\phi_\sigma)$.
  Ces cas couvrent tous ceux où le discriminant symplectique
  peut être défini de façon absolue (voir \cite{GPT}, qui utilise une
  preuve très différente). Il ne reste en réalité que le cas
  où $A$ est d'indice exactement 4 et de coindice impair, car dans ce cas
  on a $e_2(q_\sigma) = [A] + (-1,-1)$, et on ne dispose pas
  d'un relevé canonique de $[A]$ dans $I^2(K)$, ce qui nous
  empêche de faire notre construction.

  On constate en tout cas qu'on peut gérer les obstructions
  évoquées précédemment dans certains cas, soit en considérant
  des invariants relatifs, soit en éliminant les obstructions
  quand elles ne dépendent que de la dimension, soit en les éliminant
  en indice 2 quand elles ne dépendent que de la dimension et de $[Q]$,
  en utilisant la forme norme comme relevé canonique de cette classe
  de Brauer.
\end{ex}

\subsection{Invariants en indice 2}\label{sec_ind_2}

\subsubsection*{Déploiement générique}

On présente ici une méthode due à Berhuy dans \cite{Ber},
basée sur l'isomorphisme décrit au chapitre précédent (voir
(\ref{eq_depl_ind_2})) entre $M_Q^d(K)$ et $H_{nr}^d(K(Q),\mu_2)$ pour
une algèbre de quaternions $Q$ sur $K$. On présente ici
une variation de l'énoncé de Berhuy.

\label{par_fq}Soit $F$ un sous-foncteur de $I$ ou de $\Quad$, le foncteur
de $\mathbf{Field}_{/k}$ vers $\mathbf{Set}$ donné par
les classes d'isométries de formes quadratiques. On suppose
que $F$ est stable par similitude : pour tout $K/k$,
si $q\in F(K)$ alors $\fdiag{\lambda}q\in F(K)$ pour tout
$\lambda\in K^*$ ; c'est notamment le cas si $F=I^n$.
Soit $Q$ une algèbre de quaternions
sur $k$ ; à partir de $F$ et de $Q$,
on définit le foncteur $F_Q$ sur $\mathbf{Field}_{/k}$ tel que $F_Q(K)$
soit l'ensemble des classes d'isométrie
de formes anti-hermitiennes $h$ relativement à $(Q_K,\can)$
pour lesquelles, quelle que soit l'extension $L/K$ déployant $Q$,
il existe $q\in F(L)$ tel que $h_L \simeq q$ pour une certaine
équivalence de Morita (et par hypothèse sur $F$ s'il existe un
tel $q$, alors tous les autres $q$ tels que $h_L\simeq q$ sont aussi
dans $F(L)$).

\label{par_fq_alpha}Soit maintenant $\alpha\in \Inv^d(F,\mu_2)$. On définit
le sous-foncteur $F_{\alpha,Q}\subset F_Q$ par la condition
que $h\in F_Q(K)$ est dans $F_{\alpha,Q}(K)$ si et seulement si
pour tout $q\in F(L)$ tel que $h_L\simeq q$ où $L$ déploie $Q$,
pour tout $\lambda\in L^*$ on a $\alpha(\fdiag{\lambda}q)=\alpha(q)$.
Notamment, si $\alpha$ est invariant par similitude, alors
$F_{\alpha,Q}=F_Q$.

\begin{rem}\label{rem_f_invol}
  Par construction, si $(A,\sigma)$ est une algèbre à involution
  orthogonale d'indice 2 sur $K$, et
  si $Q_K$ est l'algèbre de quaternions Brauer-équivalente à
  $A$, alors cela a un sens de dire que $\sigma\in F_Q(K)$
  ou que $\sigma\in F_{\alpha,Q}(K)$, puisque $F_Q$ et $F_{\alpha,Q}$
  sont stables par similitude.
\end{rem}

Enfin, on dit que $\alpha$ est non ramifié si pour tout
corps $K$ muni d'une $k$-valuation discrète de rang 1,
si $q\in F(K)$ est non ramifiée alors $\alpha(q)\in H^d(K,\mu_2)$
est une classe non ramifiée. On arrive alors à l'énoncé
adapté de celui de Berhuy :

\begin{prop}\label{prop_ber}
  Si $\alpha$ est non-ramifié, alors il existe un unique
  $\hat{\alpha}\in \Inv(F_{\alpha,Q}, M_Q^d)$ tel que pour
  tout $K/k$ déployant $Q$, tout $h\in F_{\alpha,Q}(K)$
  et tout $q\in F(K)$ tel que $\sigma_h\simeq \sigma_q$
  on ait $\hat{\alpha}(h)=\alpha(q)$.
\end{prop}

On notera souvent $\hat{\alpha}$ simplement $\alpha$
puisque par construction ils coïncident dans les
cas où ils ont tous les deux un sens.

On peut reprendre la preuve donnée dans \cite[prop 9]{Ber}, à ceci près
que dans l'énoncé original on n'a pas de condition sur
$\alpha$ vis-à-vis des similitudes. Or il est indispensable
de rajouter une condition de ce type (éventuellement
une variante) pour avoir une compatibilité avec
les équivalences de Morita, qui ne permettent de définir
une forme quadratique à partir d'une forme anti-hermitienne
quaternionique qu'à similitude près. On donnera également
une autre justification plus loin.

\begin{rem}
  Suivant la démarche de la remarque \ref{rem_f_invol},
  on peut aussi bien définir $\hat{\alpha}$ sur des involutions
  plutôt que des formes anti-hermitiennes : si $(A,\sigma)\in F_{\alpha,Q}(K)$,
  on peut définir $\hat{\alpha}(A,\sigma)\in M_Q^d(K)$ (simplement
  en appliquant $\hat{\alpha}$ à n'importe quelle forme anti-hermitienne
  induisant $\sigma$).
\end{rem}

L'exemple naturel est $F = I^n$, qui est bien stable par similitude.
On sait alors (voir la proposition \ref{prop_ram_witt}) que tout $\alpha\in \Inv(I^n,\mu_2)$
est non ramifié. Si $\tld{\alpha}=0$, on peut alors définir
$\alpha(A,\sigma)$ pour toute algèbre à involution orthogonale
d'indice 2 \og{}génériquement dans $I^n$\fg{}.

On peut même obtenir un peu mieux : si $\alpha$ est quelconque,
alors on peut toujours définir $\tld{\alpha}(A,\sigma)$
puisque $\tld{\alpha}$ est invariant par similitude (voir la remarque
\ref{rem_inv_simil}), et
alors si $\tld{\alpha}(A,\sigma)=0$, on peut définir
$\alpha(A,\sigma)$. On peut donc interpréter $\tld{\alpha}(A,\sigma)$
comme une obstruction au fait que $\alpha(A,\sigma)$
soit défini. En effet, si $\tld{\alpha}(A,\sigma)=0$,
cela signifie que pour toute extension de déploiement
$L$ et tout $q\in I^n(L)$ tel que $\sigma_L\simeq \sigma_q$,
on a $\tld{\alpha}(q)=0$, donc $\alpha(q)$ ne dépend que
de la classe de similitude de $q$, autrement dit $\sigma\in F_{\alpha,Q}(K)$.

On donne une autre façon de comprendre ce phénomène. On
part de $\alpha\in \Inv^d(I^n,\mu_2)$, qu'on relève en un
invariant de Witt $\beta\in \Inv(I^n,I^d)$. Soit $h\in GW^-(Q,\can)$
de dimension $2r$, tel que $\sigma_h=\sigma$.
On étend les scalaires de façon à considérer
$h_{K(Q)}\in GW^-(Q_{K(Q)},\can)$, et on choisit une équivalence de Morita
entre $(Q_{K(Q)},\can)$ et $(K(Q),\Id)$ qui fait correspondre $h$ à un
certain $q\in GW(K(Q))$ (qui est donc bien défini seulement
à similitude près), dont par hypothèse la classe de Witt est
dans $I^n(K(Q))$. On pose $q'=q-r\mathcal{H}\in \hat{I}^n(K(Q))$,
et alors $\beta(q')\in \hat{I}^d(K(Q))$. 
On rappelle (voir la partie \ref{sec_quater}) que si l'équivalence de Morita choisie
correspond à un certain point fermé $\infty\in C$ où $C$
est la variété de Severi-Brauer de $Q$, alors $q'\in W_{nr,1,\infty}(K(Q))$.
De plus il est facile de voir que pour tout $i\in \N^*$,
$\lambda^i(q')$ est dans $W_{nr,1,\infty}(K(Q))$ si $i$ est impair,
et dans $W_{nr,1,\infty}(K(Q))$ si $i$ est pair. En particulier,
comme $-\tld{\beta}$ est la partie impaire de $\beta$, et
$\beta+\tld{\beta}$ est sa partie paire (voir la remarque
\ref{rem_inv_simil}), on a $\beta(q')$
somme de $-\tld{\beta}(q')\in I^{d-1}(K(Q))\cap  W_{nr,1,\infty}(K(Q))$
(qui est bien définie à similitude près)
et de $(\beta+\tld{\beta})(q')\in I^{d-1}(K(Q))\cap  W_{nr,2,\infty}(K(Q))$
(qui est bien définie).
En général, on voit que la somme des deux n'est ni dans
$W_{nr,1,\infty}(K(Q))$ ni dans $W_{nr,2,\infty}(K(Q))$, et donc
$\alpha(q)=e_d(\beta(q'))$ n'a aucune raison d'être non ramifié.
En revanche, lorsque $\tld{\beta}(q')\in I^d(K(Q))$, ce qui
correspond précisément à l'hypothèse $\tld{\alpha}(A,\sigma)=0$,
alors on peut écrire $\alpha(q) = e_d(-\tld{\beta}(q')) + e_d((\beta+\tld{\beta})(q'))$,
et chacune de ces deux classes de cohomologie est non ramifiée
qui est bien définie indépendamment de tout choix. On peut
alors définir $\alpha(A,\sigma)$ comme l'élément de $M_Q^d(K)$
correspondant.

\subsubsection*{Méthodes rationnelles}

Au lieu d'utiliser le déploiement générique pour se ramener
au cas de formes quadratiques, on peut mobiliser les méthodes
développées dans la section \ref{sec_inv_herm}, combinées au traitement
spécifique qu'on a accordé aux algèbres d'indice 2 dans le chapitre
précédent, pour directement travailler au niveau des formes
anti-hermitiennes.

Si on reprend les considérations de la partie précédente, et qu'on
suppose que $\deg(A)$ est divisible par 4, donc que $r$ est pair,
alors $q'$ provient par extension des scalaires de $h'=h-h_0\in \tld{GW}(Q,\can)$
où $h_0\in GW^-(Q,\can)$ est hyperbolique de dimension $2r$.
Il faut alors voir $q'$ comme étant dans la composante impaire
de $\tld{GW}(K(Q))=GW^\pm(K(Q))[\Zd]$, et alors si on écrit
$\beta(h')=x+y$ avec $x\in GW(K)$ et $y\in GW^-(Q,\can)$,
on obtient que $-\tld{\beta}(q') = x_{K(Q)}$, et $(\beta+\tld{\beta})(q')$
correspond à travers l'équivalence de Morita choisie à $y_{K(Q)}$.
En général, on sait par le corollaire \ref{cor_inv_herm} que la classe de Witt
de $\beta(h')$ est dans $\tld{I}_0^d(Q,\can)$, donc par la proposition
\ref{prop_ed_quater} on obtient un invariant à valeurs dans $N_Q^{d-1}(K)=
H^{d-1}(K,\mu_2)/[Q]H^{d-3}(K,\mu_2)$, qui est exactement $\tld{\alpha}(A,\sigma)$,
mais en restant au niveau du corps de base on a gardé suffisamment
de contrôle pour garantir que la classe est dans $N_Q^{d-1}(K)$
et non simplement dans $M_Q^{d-1}(K)$. De plus, si cette classe
est nulle alors à nouveau on peut définir $\alpha(A,\sigma)$
à valeurs dans $N_Q^d(K)$ et non simplement dans $M_Q^d(K)$.

Si $r$ n'est pas nécessairement pair, on ne peut pas normaliser
$h$ de la sorte, et on prend plutôt $h'=h-r\mathcal{H}\in \tld{GI}(Q,\can)$,
où $\mathcal{H}$ désigne le plan hyperbolique dans $GW(K)$.
On a alors encore $\beta(h')\in \tld{I}_0^d(Q,\can)$, donc on
en retire un invariant à valeurs dans $N_Q^{d-1}(K)$ (et s'il est nul
on peut en déduire un invariant dans $N_Q^d(K)$). En revanche ces
invariants ne coïncident plus tout à fait avec ceux obtenus
par la méthode de Berhuy à cause du choix de normalisation différent.
\\

On présente un exemple un peu détaillé, le cas de $\beta = g_1^d\in \Inv(I,I^d)$.
On part donc de $h=\fdiag{z_1,\dots,z_r}\in GW^-(Q,\can)$ et on
veut calculer la composante dans $GW(K)$ de $g_1^d(h')$
avec $h'=h-r\mathcal{H}$, ou disons son image dans $W(K)$.
Déjà on peut se limiter à $d$ pair, puisque si $d$ est impair
$g_1^d(h')\in GW^-(Q,\can)$. On utilise alors les formules de
la proposition \ref{prop_pi1_g1}, et on est donc amené à calculer
\[ q = \sum_{i=0}^d \binom{r-i}{d-i}\lambda^{2i}(h), \]
dont on sait a priori qu'il est dans $I^d(K)$, et qu'après
extension des scalaires à $K(Q)$ il donne un élément de
$I^{2d-1}(K(Q))$.

\begin{prop}\label{prop_g_quater}
  La composante symétrique de $g_1^{2d}(h-r\mathcal{H})\in \tld{I}^d(Q,\can)$
  est donnée par
  \[ \sum_{s=0}^d \sum_{i_1<\dots <i_{2s}} \fdiag{z_{i_1}}\cdots \fdiag{z_{i_{2s}}}
    \left( \sum_{t=d-2s}^{d-s}\binom{s}{d-s-t}\sum_{\substack{j_1<\dots <j_t \\ j_i,\dots,j_t\neq i_1,\dots i_{2s}}}\pfis{z_{j_1}^2,\dots,z_{j_t}^2} \right) \]
  (où la somme entre parenthèses vaut $\binom{s}{d-s}$ si $t=0$, et est nulle si $t<0$).
\end{prop}

\begin{proof}
  On a dans $W(K)$:
  \begin{align*}
    & \sum_{i=0}^d \binom{r-i}{d-i}\lambda^i(h) \\
    &= \sum_{i=0}^d\binom{r-i}{d-i}\sum_{s=0}^i\sum_{i_1<\dots <i_{2s}}\sum_{\substack{j_1<\dots <j_{i-s} \\ j_i,\dots,j_{i-s}\neq i_1,\dots i_{2s}}}
      \fdiag{(-z_{j_1}^2)\cdots (-z_{j_{i-s}}^2)}\fdiag{z_{i_1}}\cdots \fdiag{z_{i_{2s}}} \\
    &= \sum_{s=0}^d\sum_{i_1<\dots <i_{2s}}\fdiag{z_{i_1}}\cdots \fdiag{z_{i_{2s}}}\left(
      \sum_{t=0}^{d-s}(-1)^t\binom{r-s-t}{d-s-t}\lambda^t(\psi_{i_1,\dots,i_{2s}}) \right)
  \end{align*}
  où $\psi_{i_1,\dots,i_{2s}} = \fdiag{z_1^2,\dots,z_r^2}$ dans laquelle ne figurent
  pas les $z_{i_1}^2,\dots,z_{i_{2s}}^2$ (donc $\psi_{i_1,\dots,i_{2s}}$
  est de dimension $r-2s$). On doit donc montrer que pour une
  forme de dimension $r-2s$ on a
  \[ \sum_{t=0}^{d-s}(-1)^t\binom{r-s-t}{d-s-t}\lambda^t = \sum_{t=d-2s}^{d-s}\binom{s}{d-s-t}P^t.\]
  Or (voir \ref{def_pd}) en dimension $r-2s$
  \[ P^t = \sum_{i=0}^t (-1)^i\binom{r-2s-i}{t-i}\lambda^i, \]
  donc
  \begin{align*}
    \sum_{t=d-2s}^{d-s}\binom{s}{d-s-t}P^t &= \sum_{t=d-2s}^{d-s}\sum_{i=0}^t(-1)^i\binom{r-2s-i}{t-i}\binom{s}{d-s-t} \lambda^i \\
                                           &= \sum_{i=0}^{d-s}(-1)^i\left( \sum_{t=i}^{d-s}\binom{r-2s-i}{t-i}\binom{s}{d-s-t}\right) \lambda^i.
  \end{align*}
  Or $\sum_t\binom{r-2s-i}{t-i}\binom{s}{d-s-t}=\binom{r-s-i}{d-s-i}$
  par la formule de Vandermonde, d'où le résultat.
\end{proof}

\begin{rem}
  On retrouve sur cette formule le fait que cet élément est
  dans $I^d(K)$, et aussi qu'il est nul sur $r<2d$.
\end{rem}

On peut également écrire (voir \ref{prop_phi_quater} pour la
définition des $\phi_{z_{a_1},\dots,z_{a_p}}$) :

\begin{coro}
  Si $d \pgq2$, la forme ci-dessus est aussi égale à
  \[  \binom{r}{d}\pfis{-1}^{d-2}n_Q + \sum_{p=d}^{2d}\sum_{a_1<\dots <a_p}\phi_{z_{a_1},\dots,z_{a_p}}
    \left( \sum_{s=0}^{\lfloor p/2\rfloor}\binom{s}{d+s-p} \omega_s(z_{a_1},\dots,z_{a_p}) \right) \]
  où $\omega_s(z_1,\dots,z_p) = \sum_{i_1<\dots<i_{2s}}\fdiag{(-1)^s\Trd_Q(z_{i_1}z_{i_2})\cdots \Trd_Q(z_{i_{2s-1}}z_{i_{2s}})}$.
\end{coro}

\begin{proof}
  On part de la formule de la proposition, en écrivant
  \[ \fdiag{z_{i_1}}\cdots \fdiag{z_{i_{2s}}}=\fdiag{(-1)^s\Trd_Q(z_{i_1}z_{i_2})\cdots \Trd_Q(z_{i_{2s-1}}z_{i_{2s}})}\phi_{z_{i_1},\dots,z_{i_{2s}}}. \]
  On pose $p=2s+t$, et
  $\{a_1,\dots,a_p\}=\{i_1,\dots, i_{2s}\}\cup \{j_1,\dots, j_t\}$.
  Alors si $s>0$ on a
  \[ \pfis{z_{j_1}^2,\dots,z_{j_t}^2}\phi_{z_{i_1},\dots,z_{i_{2s}}} = \phi_{z_{a_1},\dots,z_{a_p}}, \]
  tandis que si $s=0$ (auquel cas $t=p=d$),
  \[ \pfis{z_{j_1}^2,\dots,z_{j_d}^2} = \pfis{-1}^{d-2}n_Q+\phi_{z_{j_1},\dots,z_{j_d}}. \]
  Le cas $s=0$ correspond à $\binom{r}{d}$ termes (donnés par chaque
  choix des $j_1,\dots,j_d$), ce qui donne le $\binom{r}{d}\pfis{-1}^{d-2}n_Q$.
  Les autres termes correspondent simplement au regroupement
  donné en sommant d'abord sur $p$ puis sur $s$.
\end{proof}

\begin{rem}\label{rem_gd_ind_2}
  On constate que l'invariant $e_d$ de la forme est
  \[ \binom{r}{d}(-1,\dots,-1)\cup[Q]. \]
  En effet, il correspond au terme $\binom{r}{d}\pfis{-1}^{d-2}n_Q$,
  puisque les termes pour $p>d$ sont dans $I^p$ donc ne contribuent pas,
  et devant chaque $\phi_{z_{a_1},\dots,z_{a_d}}\in I^d$ on a $\sum_{s=0}^{\lfloor d/2\rfloor}\omega_s(z_{a_1},\dots,z_{a_d})$ qui est de dimension $\sum _{s=0}^{\lfloor d/2\rfloor}\binom{d}{2s}=2^{d-1}$,
  donc est au moins dans $I$.

  On en déduit que la version relative de cet invariant est au moins
  à valeur dans $I^{d+1}$, donc on peut en déduire un invariant relatif
  de degré au moins $d+1$ ; de plus, on peut soustraire $\binom{r}{d}\pfis{-1}^{d-2}n_Q$
  à cette forme, ce qui nous donne un invariant absolu directement
  à valeurs dans $I^{d+1}$.
\end{rem}

\section{Autour de l'équivalence entre $A_3$ et $D_3$}\label{sec_a3_d3}

Dans cette partie on souhaite calculer de façon explicite l'équivalence
entre algèbres de type $A_3$ et $D_3$. On commence pour cela par expliquer
comment on peut définir une involution sur un produit croisé.

\subsection{Produits croisés à involution}

On veut donc décrire à l'aide d'un produit croisé
une algèbre à involution. Pour couvrir les cas qui nous intéressent,
on doit inclure les involutions unitaires, et donc potentiellement
des algèbres qui ne sont pas simples dans le cas où l'extension
quadratique centrale est déployée. De plus, on veut autoriser
la sous-algèbre galoisienne donnant le produit croisé à être
seulement étale et non un corps.

\label{par_etale}On notera pour cette partie $k$ notre corps de base, et $K$
sera simplement une algèbre étale (généralement quadratique)
sur $k$, et non plus une extension de corps comme dans le
reste de la thèse. On rappelle que si $K$ est une $k$-algèbre étale,
on a $K = \prod_{\p\in \Spec(K)}K_\p$, où $K_\p$ est une
extension finie séparable de $k$.

\subsubsection*{Algèbres galoisiennes}

On commence par quelques rappels sur les algèbres galoisiennes quand l'anneau de base 
est une algèbre étale sur un corps.

\begin{defi}
Soit $K$ une $k$-algèbre étale. Soient $L$ une $K$-algèbre étale 
et $G$ un sous-groupe de $\Aut_K(L)$. On dit que $L$ est $G$-galoisienne sur $K$
si pour tout $\p \in \Spec(K)$, $L_\p/K_\p$ est étale de degré $|G|$ et l'action de $G$ sur
$L_\p$ est fidèle.
\end{defi}

La proposition suivante résume les propriétés élémentaires des
algèbres galoisiennes dans ce contexte (voir \cite{FP} pour des
preuves).

\begin{prop}
  $L$ est $G$-galoisienne sur $K$ si et seulement s'il existe $K'_\p/K_\p$
  extension galoisienne de corps de groupe $H_\p\subset G$ telle que
  $L_\p = \Ind_{H_\p}^G(K'_\p)$, \ie{} $L_\p=(K'_\p)^{[G:H_\p]}$ en tant que $K_\p$-algèbre, 
  et $G$ agit sur $L_\p$ en permutant les copies de $K'_\p$ selon l'action
  régulière sur $G/H_\p$ et en agissant comme $H_\p$ sur les copies de $K'_\p$.

  L'application $H\mapsto L^H$ donne une correspondance bijective
  entre les sous-groupes de $G$ et les sous-algèbres $M$ de $L$ telles
  que $M_\p/K_\p$ est séparable et chaque élément de $G$
  agit soit trivialement sur $M$, soit non-trivialement sur chaque $M_\p$ 
  (sans forcément les stabiliser).

  De plus, si $H\subset G$, $L$ est $H$-galoisienne sur $L^H$.
\end{prop}

Si pour tout $\p\in \Spec(K)$, $L_\p/K_\p$ est une extension de corps,
alors $L$ est $G$-galoisienne sur $K$ si et seulement si pour tout $\p\in \Spec(K)$,
$L_\p/K_\p$ est galoisienne au sens usuel des extensions de corps,
de groupe de Galois $G$, et $G$ agit diagonalement sur $L$.
En particulier dans ce cas $G$ est entièrement déterminé par $L/K$ ; le fait d'être
galoisienne est juste une propriété de l'algèbre.

Mais dans le cas général, $L/K$ peut être galoisienne pour différents $G$ non isomorphes, y compris
lorsque $K$ est un corps. En effet, si $K'/K$ est une extension galoisienne de corps, 
alors le théorème montre que pour tout $G$ tel que $H$ soit un sous-groupe d'indice $r$, 
l'algèbre $L=(K')^r$ est $G$-galoisienne avec l'action de permutation de $G$ sur $G/H$ 
(l'idée étant que $\Aut_K(L)$ est gros~: il contient au moins le groupe symétrique $S_r$).

\subsubsection*{Une version de Skolem-Noether}

L'ingrédient essentiel de la construction des produits croisés
est le théorème de Skolem-Noether, qui permet à partir d'un sous-corps
galoisien d'une algèbre simple d'obtenir des éléments qui agissent
par conjugaison comme le groupe de Galois. Or on sait qu'en général
on n'a pas d'équivalent du théorème de Skolem-Noether pour des sous-algèbres
semi-simples d'algèbres d'Azumaya (même sur un corps).

\begin{ex}\label{ex_plong_conjug}
Les deux plongements  $(a,b)\mapsto diag(a,a,b)$ et $(a,b)\mapsto diag(a,b,b)$
de $k^2$ dans $M_3(k)$ ne sont pas conjugués.
\end{ex}

On va donc donner un énoncé un peu plus fin pour travailler avec des
sous-algèbres étales (plus fin que ce dont on a strictement besoin, par ailleurs).

\begin{defi}
  Soit $K$ une $k$-algèbre étale, et soit $E$ une algèbre
  d'Azymaya sur $K$. Soit $M$ une sous-$K$-algèbre séparable de $E$, avec
  $L=Z(M)$. Si $\p \in \Spec(L)$, on note $\mathfrak{q} \in \Spec(K)$ l'idéal
  sous $\p$, et on pose $C_\p=C_{E_\mathfrak{q}}(M_\p)$ (pour clarifier les
  notations, qui peuvent s'avérer encombrantes, $E_\mathfrak{q}$ est une algèbre simple centrale sur le corps
  $K_\mathfrak{q}$, et $M_\p$, algèbre simple centrale sur $L_\p$,
  est une des composantes directes de $M_\mathfrak{q}$, qui est une algèbre d'Azumaya
  sur $L_\mathfrak{q}$).
  
  On définit alors la \emph{signature} du plongement $M\to E$ comme
  la fonction $\Spec(L)\to \N$ donnée par $\p\mapsto \deg(C_\p)$.
\end{defi}

La signature dépend en général du plongement particulier de $M$
dans $E$, mais on voit facilement que deux plongements
conjugués ont la même signature, ce qui justifie la validité de l'exemple
\ref{ex_plong_conjug} (les signatures des deux plongements
étant respectivement de la forme $(2,1)$ et $(1,2)$).
Réciproquement :

\begin{prop}
  Soit $K$ une $k$-algèbre étale, et soit $E$ une algèbre
  d'Azumaya sur $K$. Soit $M$ une $K$-algèbre séparable, et
  soient $j_1,j_2:M \To E$ deux $K$-plongements. S'ils ont
  la même signature, alors $j_1$ et $j_2$ sont conjugués
  dans $E$.
\end{prop}

\begin{proof}
  On va imiter la preuve de la version classique (voir par exemple
  \cite[thm 4.9]{Jac}).
  On pose $A= M\otimes_K E^{op}$, qui est une algèbre
  d'Azumaya sur $Z(M)$. On définit alors deux actions de $A$
  sur $E$ : $(a\otimes b)\cdot_i x = j_i(a)xb$ pour $i=1,2$ ;
  on note $E_i$ pour désigner $E$ muni de la structure de $A$-module
  correspondante. Alors $\End_A(E_i)=C_E(j_i(M))$ de façon naturelle,
  et par hypothèse sur les signatures ces deux algèbres sont isomorphes. Or pour une
  algèbre séparable sur un corps (ici $k$), deux modules ayant des algèbres
  d'endomorphismes isomorphes sont isomorphes. On a donc un
  isomorphisme $K$-linéaire $\phi: E\to E$ telle que
  $\phi(j_1(a)xb)=j_2(a)\phi(x)b$ pour tous $a\in M$, $x,b\in E$.
  En prenant $a=1$, on voit que $\phi(xb)=\phi(x)b$ donc
  $\phi(x) = ex$ pour $e=\phi(1)\in E^*$,
  et en prenant $x=b=1$ on obtient $ej_1(a)=j_2(a)e$, donc $j_1$ et $j_2$
  sont conjugués par $e$.
\end{proof}

Dans la situation de la proposition, si $M=L$ est commutative
(donc étale sur $K$), on dit qu'elle est \emph{strictement maximale}
si pour tout $\p\in \Spec(K)$, $\dim_{K_\p}(L_\p)=\deg(E_\p)$ (autrement
dit elle est de dimension maximale possible étant donné le
degré local de $E$).

\begin{coro}\label{cor_sko}
  Soit $K$ une $k$-algèbre étale, et soit $E$ une algèbre
  d'Azumaya sur $K$. Soit $L\subset E$ une sous-$K$-algèbre
  étale strictement maximale. Alors tout $K$-automorphisme
  de $L$ s'étend en un automorphisme intérieur de $E$, et
  l'élément réalisant la conjugaison est unique à multiplication
  par $L^*$ près.
\end{coro}

\begin{proof}
  Un automorphisme de $L$ définit un deuxième plongement
  de $L$ dans $E$ en le composant avec l'inclusion naturelle.
  Or comme $L$ est strictement maximale, on doit avoir $C_E(L)=L$
  quel que soit le plongement, donc la signature est nécessairement
  la même. Ces deux plongements sont donc conjugués dans $E$,
  ce qui revient exactement à dire que l'automorphisme de $L$
  se prolonge en un automorphisme intérieur. L'élément réalisant
  la conjugaison est déterminé à un facteur dans $C_E(L)$ près,
  mais $C_E(L)=L$.
\end{proof}

\subsubsection*{Produits croisés}

On se donne $K$ une algèbre étale sur $k$, $E$ une algèbre
d'Azumaya sur $K$, et $L\subset E$ une sous-$K$-algèbre
$G$-galoisienne strictement maximale. On dit alors que
$E$ est un produit croisé de $L$ par $G$.

\label{par_ug}On peut alors appliquer le corollaire \ref{cor_sko} pour
imiter la situation classique et choisir pour tout $g\in G$
un élément $u_g\in E^*$ qui induit par conjugaison
l'automorphisme $g$ sur $L$. On vérifie en regardant la
dimension que $E = \bigoplus_{g\in G}L\cdot u_g$.

On a $u_g$ bien défini à multiplication par un élément de $L^*$ près,
donc en particulier, comme la conjugaison par $u_gu_h$ induit
l'automorphisme $gh$, on a $\alpha(g,h)\in L^*$ tel que
\begin{equation}\label{eq_alpha_gh}
   u_gu_h = \alpha(g,h)u_{gh} 
\end{equation}
et on vérifie comme dans le cas usuel que l'associativité
de $E$ est équivalente au fait que $\alpha$ soit un 2-cocycle
de $G$ à valeurs dans $L^*$, et qu'un choix différent de $u_g$
donne un cocycle cohomologue de la façon naturelle,
et ainsi on obtient avec exactement
la même preuve la bijection usuelle entre $H^2(G,L^*)$
et les classes d'isomorphismes de produits croisés de $L$
par $G$ (le seul ingrédient supplémentaire est la version
du théorème de Skolem-Noether énoncée précédemment).

Notons qu'on prendra toujours $u_1=1$, ce qui correspond
au fait que le cocycle soit normalisé.

\subsubsection*{Descente}

En gardant les mêmes notations, soit $V$ un $E$-module
à droite, et soit $A=\End_E(V)$. Alors $V_L$ est
un $A_L$-$E_L$-bimodule simple ; de plus, $E\otimes_K L$
agit à gauche sur $E$ par $(a\otimes \lambda)\cdot x=ax\lambda$,
et $A\otimes_K L$ agit à gauche sur $V$ par
$(a\otimes \lambda)\cdot v =av\lambda$. Cela donne un triangle
commutatif dans $\mathbf{Mor}(K)$ (voir la remarque
\ref{rem_cat_mor}) :

\[\begin{tikzcd}
  & E_L \dlar[swap]{V_L} \drar{E} & \\
A_L \arrow{rr}{V} & & L. 
\end{tikzcd}\]

On a donc une identification canonique $A_L = \End_L(V)$,
et on doit pouvoir en déduire une action naturelle de
$G$ sur $\End_L(V)$. On pose pour tout $g\in G$
\begin{equation}\label{eq_def_tg}
   \foncdef{T_g}{V}{V}{v}{vu_g^{-1}} 
\end{equation}
et
\begin{equation}\label{eq_def_sg}
   \foncdef{S_g}{\End_L(V)}{\End_L(V)}{f}{T_g\circ f\circ T_g^{-1}}. 
\end{equation}

\begin{prop}\label{prop_desc}
  La fonction $g\mapsto S_g$ définit une action semi-linéaire de $G$
  sur $\End_L(V)$ qui coïncide avec l'action galoisienne naturelle
  sur $A_L$. En particulier, la sous-algèbre
  des éléments invariants est $A = \End_E(V)\subset \End_L(V)$.
\end{prop}

\begin{proof}
  La composition $S_gS_h$ est la conjugaison par $T_gT_h$, qui
  est la multiplication à droite par $u_h^{-1}u_g^{-1}$, qui à
  un élément de $L^*$ près est $u_{gh}^{-1}$, donc par $L$-linéarité
  des éléments de $\End_L(V)$, $S_gS_h=S_{gh}$.
  De plus, $S_g(f\lambda)(v) = (f(vu_g)\lambda)u_g^{-1} = f(vu_g)u_g^{-1}g(\lambda)$,
  donc $S_g$ est bien semi-linéaire, d'action $g$ sur $L$.

  Pour montrer que c'est l'action galoisienne naturelle sur $A_L$,
  il suffit de montrer que $A$ est bien l'ensemble des points fixes.
  Or $S_g(f)=f$ est équivalent à $f(vu_g)=f(v)u_g$ pour tout $v\in V$,
  donc les points fixes sous l'action des $S_g$ sont exactement
  les $f$ qui commutent avec l'action des $u_g$, ce qui est équivalent
  à commuter avec l'action de $E$ puisque $E$ est engendré par les $u_g$
  et $L$.
\end{proof}

\subsubsection*{Involutions}

On décrit maintenant comment définir une involution sur un
produit croisé, en conservant les notations précédentes.

\begin{defi}
  On dit qu'une involution $\theta$ sur $E$ est $G$-centrée
  sur $L$ si $\theta(L)=L$ et $\theta_{|L} G \theta_{|L} = G$.
\end{defi}

\begin{rem}
  La dernière hypothèse est automatiquement vérifiée si $L$
  est un corps, puisque dans ce cas $G=\Aut_K(L)$, et la restriction
  de $\theta$ à $K$ est un automorphisme d'ordre au plus 2.
\end{rem}

On se munit pour la suite d'une involution $\theta$ $G$-centrée
sur $L$.  Pour tout $g\in G$, on notera
\begin{equation}\label{eq_g_bar}
  \overline{g} = \theta_{|L} g^{-1}\theta_{|L},
\end{equation}
et
\begin{equation}\label{eq_theta_g}
  \theta_g = g\theta_{|L},\quad \theta_g'=g^{-1}\theta_{|L}.
\end{equation}

\begin{rem}\label{rem_theta_g}
  Les automorphismes $\theta_g$ et $\theta_g'$ sont chacun involutifs
  si et seulement si $g=\overline{g}$, ce qui est par exemple le
  cas pour $g=1$, auquel cas $\theta_g = \theta_g' = \theta$.
  Si $\theta_{|L}\in G$ (ce qui arrive
  notamment lorsque $L$ est un corps et $\theta$ est de première espèce),
  alors on peut prendre $g=\theta_{|L}$ et dans ce cas
  $\theta_g = \theta_g' = \Id$.
\end{rem}

\begin{prop}\label{prop_mu}
  Pour tout $g\in G$, il existe un unique $\mu_g\in L^*$
  tel que
  \begin{equation}\label{eq_mu}
    \theta(u_g) = \mu_g u_{\overline{g}},
  \end{equation}
  qui vérifie
  \begin{equation}\label{eq_structure_mu}
    \mu_{\overline{g}}\theta_g(\mu_g) = 1
  \end{equation}
  et
  \begin{equation}\label{eq_mu_alpha}
    \mu_h\overline{g}(\mu_g)\alpha(\overline{h},\overline{g}) = \mu_{gh}\theta_{\overline{gh}}(\alpha(g,h)).
  \end{equation}
  De plus, si on remplace les $u_g$ par $c_gu_g$ avec
  $c_g\in L^*$, on trouve $\mu'_g = \mu_g\frac{\theta_{\overline{g}}(c_g)}{c_{\overline{g}}}$.

  Réciproquement, si $\theta\in \Aut_k(L)$ est un automorphisme d'ordre
  au plus 2 tel que $\theta G \theta = G$, et si on se donne $(\mu_g)_{g\in G}$
  vérifiant les équations (\ref{eq_structure_mu}) et (\ref{eq_mu_alpha}),
  alors $\theta$ s'étend de façon unique en une involution de $E$
  vérifiant (\ref{eq_mu}).
\end{prop}

\begin{proof}
  Soit $\lambda\in L$. Alors $\theta(\theta(\lambda)u_g) = \theta(u_g)\lambda$
  mais aussi
  \begin{align*}
    \theta(\theta(\lambda)u_g) &= \theta(u_g(g^{-1}\theta)(\lambda)) \\
                               &= (\theta g^{-1} \theta)(\lambda)\theta(u_g)
  \end{align*}
  donc $\theta(u_g)$ agit comme $\overline{g}$ sur $L$,
  et donc il existe un unique $\mu_g\in L^*$ tel que
  $\theta(u_g) = \mu_g u_{\overline{g}}$.

  On a
  \begin{align*}
    \theta(\theta(u_g)) &= \theta(\mu_g u_{\overline{g}}) \\
                        &= \mu_{\overline{g}}u_g\theta(\mu_g) \\
                        &= \mu_{\overline{g}}(g\theta)(\mu_g)u_g
  \end{align*}
  d'où la formule (\ref{eq_structure_mu}).
  De plus,
  \begin{align*}
    \theta(u_g u_h)  &=  \theta(\alpha(g,h)u_{gh}) \\
    \mu_hu_{\overline{h}}\mu_gu_{\overline{g}}  &= \mu_{gh}u_{\overline{gh}}\theta(\alpha(g,h)) \\
    \mu_h\overline{g}(\mu_g)\alpha(\overline{h},\overline{g})u_{\overline{gh}}   
                     &=  \mu_{gh}\theta_{\overline{gh}}(\alpha(g,h))u_{\overline{gh}}
  \end{align*}
  d'où la formule (\ref{eq_mu_alpha}). Si on remplace
  $u_g$ par $u'_g=c_gu_g$, alors
  \begin{align*}
    \theta(u'_g) &= \theta(u_g)\theta(c_g) \\
                 &= \mu_g (\overline{g}\theta)(c_g)u_{\overline{g}} \\
                 &= \mu_g \frac{\theta_{\overline{g}}(c_g)}{c_{\overline{g}}} u'_{\overline{g}}
  \end{align*}
  d'où la formule de l'énoncé.

  Réciproquement, on peut remonter toutes les égalités ci-dessus
  pour montrer que si on définit $\theta$ par les $\mu_g$ vérifiant
  les équations, on obtient bien $\theta(\theta(u_g))=u_g$
  et $\theta(u_g u_h)=\theta(u_h)\theta(u_g)$, et $\theta$ définit
  bien une involution.
\end{proof}

\subsubsection*{Modules hermitiens et déploiement}

Soit $(V,h)$ module $\eps$-hermitien à droite sur $(E,\theta)$,
et soit $A = \End_E(V)$ munie de $\sigma=\sigma_h$. On est
intéressé par la question suivante : comme $\theta$ est une involution
sur $L$, on peut considérer l'involution $\sigma\otimes \theta$
sur $A_L$ ; sachant que $A_L$ est déployée, peut-on trouver
une forme hermitienne naturelle sur $(L,\theta)$ dont c'est l'involution
adjointe ?

On va en fait considérer un cadre un peu plus large.
On écrit
\begin{equation}\label{eq_hg}
  h(x,y) = \sum_{g\in G} u_gh_g(x,y)
\end{equation}
avec $h_g:V\times V\To L$, uniquement définie par $h_g = \pi_h\circ h$
où
\begin{equation}\label{eq_pi_g}
  \pi_g\left(\sum_t u_tc_t\right) = c_g
\end{equation}
(attention, on prend ici la convention
d'agir à droite). On pose également $f_g: E\times E\To L$ définie
par
\begin{equation}\label{eq_fg}
  f_g(x,y) = \pi_g(\theta(x)y).
\end{equation}
On décrit le comportement général de ces applications :

\begin{lem}\label{lem_relat_pig}
  Soient $g,t\in G$, $x\in E$ et $a,b\in L$. Alors :
  \begin{align*}
    \pi_g(axb) &= g^{-1}(a)\pi_g(x)b, \\
    \pi_g(\theta(x)) &= g^{-1}(\mu_{\overline{g}})\theta'_g(\pi_{\overline{g}}(x)), \\
    \pi_g(u_tx) &= g^{-1}(\alpha(t,t^{-1}g))\pi_{t^{-1}g}(x), \\
    \pi_g(x u_t) &= g^{-1}(\alpha(gt^{-1},t))t^{-1}(\pi_{gt^{-1}}(x)).
  \end{align*}
\end{lem}

\begin{proof}
  On écrit $x=\sum_g u_gx_g$. Alors $axb = \sum_g au_gx_gb = \sum_g u_gg^{-1}(a)x_gb$,
  d'où la première formule. Pour la deuxième :
  \begin{align*}
    \theta(x) &= \sum_g \theta(x_g)\theta(u_g) \\
              &= \sum_g u_{\overline{g}}\overline{g}^{-1}(\mu_g \theta(x_g)) \\
              &= \sum_g u_g g^{-1}(\mu_{\overline{g}} \theta(x_{\overline{g}})).
  \end{align*}
  Pour la troisième formule :
  \begin{align*}
    u_t x &= \sum_g u_tu_g x_g \\
          &= \sum_g \alpha(t,g)u_{tg}x_g \\
          &= \sum_g \alpha(t,t^{-1}g)u_gx_{t^{-1}g}
  \end{align*}
  et la quatrième se démontre de façon identique en calculant
  $xu_t$.
\end{proof}

\begin{lem}\label{lem_relat_fg}
  Soient $g\in G$, $x,y\in E$ et $a,b\in L$. Alors :
  \begin{align*}
    f_g(xa,yb) &= \theta'_g(a)f_g(x,y)b, \\
    f_g(y,x) &= g^{-1}(\mu_{\overline{g}})\theta'_g(f_{\overline{g}}(x,y))
  \end{align*}
\end{lem}

\begin{proof}
  En utilisant la première formule du lemme précédent :
  $f_g(xa,yb) = \pi_g(\theta(a)\theta(x)yb) = (g^{-1}\theta)(a)f_g(x,y)b$.
  Puis on applique la deuxième formule du même lemme
  à $f_g(y,x) = \pi_g(\theta(\theta(x)y))$.
\end{proof}

\begin{lem}\label{lem_relat_hg}
  Soient $g,t\in G$, $x,y\in V$, $a,b\in L$. Alors :
  \begin{align*}
    h_g(xa,yb) &= \theta'_g(a)h_g(x,y)b, \\
    h_g(y,x) &= \eps g^{-1}(\mu_{\overline{g}}) \theta'_g(h_{\overline{g}}(x,y)), \\
    h_g(xu_t,y) &= g^{-1}\left(\mu_t \alpha(\overline{t},\overline{t}^{-1}g)\right)h_{\overline{t}^{-1}g}(x,y), \\
   h_g(x,yu_t) &= g^{-1}\left( \alpha(gt^{-1},t)\right) t^{-1}\left( h_{gt^{-1}}(x,y)\right).
  \end{align*}
\end{lem}

\begin{proof}
  Sachant que $h_g(x,y)=\pi_g(h(x,y))$,
  on applique successivement les quatre formules du lemme \ref{lem_relat_pig}
  aux relations $h(xa,yb) = \theta(a)h(x,y)b$, $h(y,x) = \eps\theta(h(x,y))$,
  $h(xu_t,y)=\theta(u_t)h(x,y)=\mu_tu_{\overline{t}}h(x,y)$ et
  $h(x,yu_t)=h(x,y)u_t$.
\end{proof}

On en déduit :

\begin{prop}\label{prop_morita_crois}
  Si $g=\overline{g}$, alors $h_g$ et $f_g$ sont respectivement
  $\eps\eps_g$-hermitienne et $\eps_g$-hermitienne sur $(L,\theta'_g)$,
  avec $\eps_g = g^{-1}(\mu_g)$. De plus, on a un triangle commutatif
  dans $\mathbf{Mor}_h(L,\theta'_g)$ (voir la remarque \ref{rem_cat_mor_herm}) :
  \[ \begin{tikzcd}
      & (E_L, \theta\otimes \theta'_g)  \drar{f_g} & \\
      (A_L, \sigma\otimes \theta'_g) \urar{h_L} \arrow{rr}{h_g} & & (L,\theta'_g)
    \end{tikzcd} \]
\end{prop}

\begin{rem}
  Jusqu'ici on a toujours implicitement considéré que les
  formes étaient $\eps$-hermitiennes avec $\eps=\pm 1$
  puisqu'on travaillait avec des involutions de première
  espèce. Ici si $g\neq \theta$ on a en général $\theta'_g$
  non triviale sur $L$, donc une forme peut être $\eps$-hermitienne
  pour tout $\eps$ tel que $\eps\theta'_g(\eps)=1$.
\end{rem}

\begin{proof}
  Si $g=\overline{g}$ alors déjà $\theta'_g$ est une involution
  (voir remarque \ref{rem_theta_g}), et  la deuxième formule du lemme
  \ref{lem_relat_hg} devient
  \[ h_g(y,x) = \eps g^{-1}(\mu_g) \theta'_g(h_g(x,y))  \]
  ce qui combiné avec la première formule de ce lemme montre bien que
  $h_g$ est $\eps\eps_g$-hermitienne relativement à $\theta'_g$.

  De même, les deux formules du lemme \ref{lem_relat_fg} montrent
  que quand $g=\overline{g}$, $f_g$ est $\eps_g$-hermitienne
  relativement à $\theta'_g$.

  On doit alors montrer la relation entre les trois formes hermitiennes.
  Si on identifie naturellement $V$ à $V_L\otimes_{E_L} E$ par
  $\Phi: (v\otimes \lambda)\otimes x\mapsto vx\lambda$, on doit
  alors vérifier que
  \[ h_g(\Phi((v\otimes \lambda)\otimes x),\Phi((v'\otimes \lambda')\otimes x')) 
    = f_g(x,h_L(v\otimes \lambda,v'\otimes\lambda')x'). \]
  Or
  \begin{align*}
    f_g(x,h_L(v\otimes \lambda,v'\otimes\lambda')x')
    &= \pi_g(\theta(x)h_L(v\otimes\lambda,v'\otimes\lambda')x') \\
    &= \pi_g(\theta(x)h(v,v')x'(\theta'_g)(\lambda)\lambda') 
  \end{align*}
  et de l'autre côté
  \begin{align*}
    h_g(\Phi((v\otimes \lambda)\otimes x),\Phi((v'\otimes \lambda')\otimes x'))
    &= \pi_g(h(vx\lambda,v'x'\lambda'))\\
    &= \pi_g(\theta(\lambda)\theta(x)h(v,v')x'\lambda') \\
    &= \theta'_g(\lambda) \pi_g(\theta(x)h(v,v')x'\lambda') \\
    &=  \pi_g(\theta(x)h(v,v')x'\lambda'(\theta'_g)(\lambda)).
\end{align*}
\end{proof}

\subsection{Algèbre de Clifford en indice 2}

Soit $(A,\sigma)$ algèbre à involution orthogonale sur $k$, qu'on
suppose d'indice au plus 2. Soit donc $Q$ l'algèbre de quaternions Brauer-équivalente
à $A$, et on écrit $A = \End_Q(V)$ où $V$ est un $Q$-module à droite.
On a alors $\sigma = \sigma_h$ où $h$ est une forme anti-hermitienne
sur $(Q,\can)$. On choisit une base orthogonale $(e_i)$, qui donne
une diagonalisation $h=\fdiag{z_1,\dots,z_r}$ avec $z_i\in Q$ quaternion
pur inversible. On supposera dans la suite que $r\pgq 2$.

On souhaite décrire explicitement, par générateurs, l'algèbre
de Clifford $B=Cl(A,\sigma)$ de $(A,\sigma)$, munie de son involution
canonique $\tau$. La méthode utilisée est tirée de
\cite{SH}, notre apport se situant principalement dans le calcul
explicite des relations entre les générateurs. On commence par énoncer
le résultat auquel on veut arriver.

\begin{thm}\label{thm_d3_a3}
  On suppose que les $z_i$ donnés par la diagonalisation de $h$ sont
  $k$-linéairement indépendants deux à deux. Alors il existe $L\subset B$ sous-$k$-algèbre
  étale strictement maximale de la forme $L=k(\xi_1,\dots,\xi_r)$
  avec $\xi_i^2=z_i^2$, de sorte que $Z(B)=k(\xi)$ où $\xi=\prod_i \xi_i$.
  L'algèbre $L$ est naturellement $G$-galoisienne sur $k$ avec
  $G = (\Zd)^r$, l'involution $\tau$ agit comme l'élément $(1,\dots,1)$,
  et $Z(B)=L^H$ où $H\subset G$ est le noyau du morphisme
  $G\to \Zd$ donné par la somme des composantes. Le groupe $H$ est engendré
  par $\tau_{p,q}$ pour tous $p\neq q$, dont les seules composantes non
  nulles sont en indice $p$ et $q$. Alors pour tous $p\neq q$ il existe
  $u_{p,q}\in B$ tels que :
  \begin{enumerate}
    \bitem la conjugaison par $u_{p,q}$ induit $\tau_{p,q}$ sur $L$;
    \bitem $\tau(u_{p,q})=-u_{p,q}$ ;
    \bitem si $\{i,j\}\cap \{p,q\}=\emptyset$, $u_{i,j}$ et $u_{p,q}$ commutent;
    \bitem $u_{p,q}^2 = Tr_Q(z_pz_q)-2\xi_p\xi_q$;
    \bitem si $i,p,q$ sont distincts :
    \[ u_{i,p}u_{i,q} = (\lambda_0 - \xi_i + \lambda_p\xi_p -\lambda_q\xi_q -\lambda_{p,q}\xi_p\xi_q)u_{p,q}  \]
    où $z_i = \lambda_0 + \lambda_pz_p + \lambda_qz_q + \lambda_{p,q}z_pz_q$.
  \end{enumerate}
\end{thm}

Comme les $\tau_{p,q}$ engendrent $H$, les $u_{p,q}$ engendrent $B$
avec $L$, et donc le théorème donne une description complète de la
structure de $(B,\tau)$. Avant de passer à la preuve proprement
dite, on va développer les constructions nécessaires.

\begin{rem}
  L'hypothèse d'indépendance des $z_i$ permet d'assurer que $(1,z_p,z_q,z_pz_q)$
  forme une $k$-base de $Q$ pour tous $p,q$, ce qui légitime
  la décomposition $z_i = \lambda_0 + \lambda_pz_p + \lambda_qz_q + \lambda_{p,q}z_pz_q$
  utilisée dans l'énoncé.
  
  Si $k$ est infini (ce qu'on peut supposer sans conséquences,
  puisque sur un corps fini la théorie n'a pas d'intérêt),
  quitte à changer les $e_i$, on peut toujours se ramener à cette situation,
  et ce sans changer
  la classe de $z_i^2$ dans $k^*/(k^*)^2$. En effet, si $z_p$ et
  $z_q$ sont colinéaires, on remplace $e_q$ par $e_qz$ pour n'importe
  quel $z$ qui ni ne commute ni n'anti-commute avec $z_q$, ce qui revient
  à exclure deux $k$-plans. On peut faire cette
  modification récursivement sur toutes les paires problématiques,
  et comme une union finie de plans ne peut pas être égale à $Q$,
  on peut toujours trouver un $z$ qui convient.

  De plus, quand cette condition est vérifiée, alors $(1,z_p,z_q,z_pz_q)$
  forme une $k$-base de $Q$ pour tous $p\neq q$ donc l'écriture
  de $z_i$ dans cette base utilisée dans l'énoncé existe toujours
  et est unique.
\end{rem}

Pour prouver le théorème, on va utiliser une
méthode de descente. Soit $K\subset Q$ extension quadratique de $k$
(on peut toujours en trouver, sauf si $k$ est quadratiquement clos,
auquel cas rien de tout cela n'a d'intérêt), et soit $j\in Q$
un quaternion pur tel que $j\lambda=\overline{\lambda}j$
pour tout $\lambda\in K$. On pose $j^2 = \pi\in k^*$. On a alors
$Q = K \oplus jK$. On est donc dans la situation d'un produit
croisé à involution (l'algèbre $(Q,\can)$ munie du sous-corps
galoisien $K/k$ de groupe $\Zd$) et d'un module anti-hermitien $(V,h)$ sur
cette algèbre, et on peut appliquer les résultats de la partie
précédente.

D'après la proposition \ref{prop_desc}, on a donc $A = \End_K(V)^{\Gal(K/k)}$
où si $s$ est l'élément non trivial de $\Gal(K/k)$ on a $s\cdot f=T\circ f\circ T^{-1}$
avec $T(v)=vj^{-1}$. De plus, en appliquant la proposition
\ref{prop_morita_crois} à $g=s$ (on est bien ici dans le cas
où l'involution se restreint à un élément du groupe de Galois),
si on note $\sigma_K$ l'involution sur $\End_K(V)$ obtenue
à partir de $\sigma\otimes \Id$ par identification de $A_K$
et $\End_K(V)$, alors $\sigma_K=\sigma_b$ où $b:V\times V\to K$
est caractérisée par $h(x,y)=a(x,y)+jb(x,y)$.

\begin{lem}\label{lem_a_b}
  Pour tous $x,y\in V$, on a
  \[ b(T(x),y) = -\frac{1}{\pi}a(x,y)= -\overline{b(x,T(y))}. \]
\end{lem}

\begin{proof}
  On applique les deux dernières formules du lemme \ref{lem_relat_hg},
  avec ici $h_1=a$, $h_s=b$, et $u_t=u_s=j$. On trouve alors
  $b(xj,y)= -a(x,y)$ et $b(x,yj)= \overline{a(x,y)}$, et
  en combinant ces deux formules on trouve le résultat voulu.
\end{proof}

Ainsi, $(B,\tau)$ est la descente de $Cl^0(V,b)$ pour l'action
galoisienne induite par celle décrite ci-dessus. On décrit
cette action :

\begin{prop}
  L'action naturelle de $\Gal(K/k)$ sur $Cl^0(V,b)$ est donnée pour
  tous $x,y\in V$ par
  \[ s(x\cdot y) = -\pi T(x)\cdot T(y), \]
  et $B = Cl^0(V,b)^s$ pour cette action.
\end{prop}

\begin{proof}
  On considère l'identification canonique
  \[ \foncdef{\phi_b}{V\otimes_K V}{\End_K(V)}{x\otimes y}{v\mapsto xb(y,v)} \]

  Alors $s \cdot \phi_b(x\otimes y) = T\circ \phi_b(x\otimes y)\circ T^{-1}$ donc
  \begin{align*}
    (s\cdot \phi_b(x\otimes y))(v) &= xb(y,vj)j^{-1} \\
                                   &= \pi T(x)\overline{b(y,T(v))} \\
                                   &= -\pi T(x)b(T(y),v) \\
                                   &= -\pi \phi_b(T(x)\otimes T(y))(v).
  \end{align*}
  Donc $s\cdot \phi_b(x\otimes y)) = \phi_b(-\pi T(x)\otimes T(y))$,
  d'où le résultat puisque l'action sur $Cl^0(V,b)$ est caractérisée
  par l'action sur l'image de $\phi_b$.
\end{proof}

On sait que $B$ est engendrée par l'image du morphisme canonique
$\psi: V\otimes_Q ( ^\gamma V)\to Cl(A,\sigma)$. On va décrire
de façon explicite ce morphisme :

\begin{prop}\label{prop_cliff_psi}
  La composition de $\psi$ avec l'injection canonique
  $Cl(A,\sigma)\inj Cl^0(V,b)$ est donné par
  \[ \psi(x\otimes y) = \pi(T(x)\cdot y - x\cdot T(y)).\]
\end{prop}

\begin{proof}
  Il s'agit de vérifier que l'application $f$ définie par
  $f(x\otimes y) = \pi(T(x)\otimes y - x\otimes T(y))$
  fait commuter le diagramme
  \[ \begin{tikzcd}
      V\otimes_Q ^{\gamma} V \rar{f} \dar{\phi_h} & V\otimes_K V \dar{\phi_b} \\
      \End_Q(V) \rar & \End_K(V).
    \end{tikzcd} \]
  Or :
  \begin{align*}
    \phi_b(f(x\otimes y))(v) &= \phi_b(\pi(T(x)\otimes y-x\otimes T(y)))(v) \\
                             &= \pi(T(x)b(y,v)-xb(T(y),v)) \\
                             &= x(jb(y,v)-a(y,v)) \\
                             &= \phi_h(x\otimes y)(v)
  \end{align*}
  où on a utilisé les relations du lemme \ref{lem_a_b}.
\end{proof}

On peut passer à la preuve du théorème :

\begin{proof}
  Pour tous $x,y\in V$, on pose  $u(x,y) = \frac{1}{2}\psi(x\otimes y + y\otimes x)$
  et $u(x)=\frac{1}{2}\psi(x\otimes x)$. On définit alors
  les générateurs de l'énoncé comme :
  \[ \xi_i = u(e_i),\quad u_{p,q}=u(e_p,e_q). \]
  Le choix de $K$ n'intervenant pas dans l'énoncé, on
  peut toujours supposer qu'aucun des $z_i$ n'est dans $K$.
  Pour tout $1\ppq i\ppq r$, on écrit $z_i=x_i+jy_i$
  avec $x_i,y_i\in K$, et on peut donc supposer $y_i\neq 0$.
  Si on pose $f_i = e_iz_i$, alors
  $(e_i,f_i)_i$ est une $K$-base de $V$, orthogonale pour $b$.
  En effet, la seule chose à vérifier est que $b(e_i,f_i)=0$,
  mais $h(e_i,f_i)=h(e_i,e_i)z_i=z_i^2\in k$. Notons que
  $b(e_i,e_i)=y_i$, et $b(f_i,f_i)=-z_i^2y_i$.

  On a
  \[ \xi_i = -\frac{1}{y_i}e_if_i \]
  et
  \[ u_{p,q} = \frac{1}{y_py_q}\left( (x_qy_p-x_py_q)e_pe_q - y_pe_pf_q -y_qe_qf_p \right).  \]
  En effet, on a
  \begin{align*}
    e_ij &= e_i\left[ (x_i+jy_i)\frac{1}{y_i} - \frac{x_i}{y_i}\right] \\
         &= (-e_ix_i + f_i)\frac{1}{y_i}
  \end{align*}
  ce qu'on peut injecter dans $\xi_i = \frac{1}{2}((e_ij)e_i - e_i(e_ij))$
  et $u_{p,q} = \frac{1}{2}((e_pj)e_q - e_p(e_q)j + (e_qj)e_p - e_q(e_pj))$
  en tenant compte des relations d'anticommutation dans $Cl(V,b)$.

  Toutes les relations résultent alors de calculs explicites dans $Cl(V,b)$
  en utilisant la base des $e_i$ et $f_i$. Notamment, on voit
  immédiatement par les relations d'orthogonalité que les $\xi_i$
  commutent, que $u_{p,q}$ commute avec $\xi_i$ si $i\not\in \{p,q\}$,
  qu'ils anti-commutent si $i\in \{p,q\}$,
  et que $u_{i,j}$ et $u_{p,q}$ commutent quand $\{i,j\}\cap \{p,q\}=\emptyset$.
  L'action de l'involution canonique est aussi facile à établir puisque
  par définition $\tau(x\cdot y) = y\cdot x = -x\cdot y$ si $x$ et $y$ sont
  orthogonaux.
  
  On a aussi facilement $\xi_i^2 = -\frac{1}{y_i^2}e_i^2f_i^2 = z_i^2$,
  et
  \begin{align*}
    u_{p,q}^2 
    &= \frac{1}{y_p^2y_q^2}\left(-(x_qy_p-x_py_q)^2e_p^2e_q^2 -y_p^2e_p^2f_q^2 -y_q^2e_q^2f_p^2 - 2y_py_qe_pf_pe_qf_q\right) \\
    &= -\frac{1}{y_py_q}(x_qy_p-x_py_q)^2 + \frac{1}{y_q}z_q^2 + \frac{1}{y_p}z_p^2 - 2\frac{1}{y_py_q}e_pf_pe_qf_q \\
    &= (2x_px_q + \pi y_p\overline{y_q} + \pi y_q \overline{y_p}) -2\xi_p\xi_q \\
    &= \Trd_Q(z_pz_q) -2\xi_p\xi_q.
  \end{align*}

  La dernière relation est la plus difficile à établir. Écrivons
  $z_i = \lambda_0 + \lambda_pz_p + \lambda_qz_q + \lambda_{p,q}z_pz_q$.
  Alors
  \begin{align}\label{eq_xy_lambda}
    x_i &= \lambda_0 + \lambda_p x_p + \lambda_{p,q}x_px_q + \lambda_{p,q}\pi \overline{y_p}y_q \\
    y_i &= \lambda_p y_p + \lambda_q y_q + \lambda_{p,q}y_px_q - \lambda_{p,q}x_py_q.
  \end{align}
  Notons que comme $z_i$ est un quaternion pur, $x_i=-\overline{x_i}$ ce
  qui se traduit par l'équation
  \begin{equation}\label{eq_lambda}
    2\lambda_0 + \lambda_{p,q}(2x_px_q + \pi\overline{y_p}y_q + \pi y_p\overline{y_q}) = 0.
  \end{equation}
  On doit comparer d'une part
  \[ \frac{1}{y_i^2y_py_q}\left( (x_py_i-x_iy_p)e_ie_p -y_ie_if_p-y_pe_pf_i \right)\left( (x_qy_i-x_iy_q)e_ie_q -y_ie_if_q-y_qe_qf_i \right)  \]
  et d'autre part
  \begin{align*}
    & \frac{1}{y_py_q}\left( \lambda_0 + \frac{1}{y_i}e_if_i - \frac{\lambda_p}{y_p}e_pf_p + \frac{\lambda_q}{y_q}e_qf_q - \frac{\lambda_{p,q}}{y_py_q}e_pf_pe_qf_q\right) \\
    & \times \left( (x_qy_p-x_py_q)e_pe_q - y_pe_pf_q - y_qe_qf_p \right).
  \end{align*}
  On peut alors développer et identifier les composantes correspondantes.
  On trouve des termes non nuls pour les composantes selon
  $e_pe_q$, $e_pf_q$, $e_qf_p$, $f_pf_q$, $e_if_ie_pe_q$, $e_if_ie_pf_q$,
  et $e_if_ie_qf_p$. Les coefficients selon ces trois dernières
  composantes sont immédiatement égaux après calcul, et ne font pas intervenir
  les la décomposition de $z_i$ (ils valent respectivement $\frac{x_q}{y_iy_q}-\frac{x_p}{y_iy_p}$,
  $-\frac{1}{y_iy_q}$ et $-\frac{1}{y_iy_p}$). Les quatre
  autres coefficients demandent plus de minutie, on est amené à montrer
  les égalités suivantes :
  \begin{align*}
    \lambda_0(x_qy_p-x_py_q) + \lambda_py_qz_p^2 + \lambda_qy_pz_q^2 &= -x_px_qy_i + (x_py_q+x_qy_p)x_i + \pi y_py_q \overline{y_i} ; \\
    \lambda_0y_p + \lambda_q(x_qy_p-x_py_q)+ \lambda_{p,q}y_qz_p^2 &= x_iy_p - y_ix_p ; \\
    \lambda_0y_q + \lambda_p(x_qy_p-x_py_q)+ \lambda_{p,q}y_pz_q^2 &= y_ix_q - x_iy_q ; \\
    \lambda_py_p + \lambda_qy_q + \lambda_{p,q}(x_qy_p-x_py_q) &= y_i.
  \end{align*}
  Pour la deuxième et la quatrième il suffit de remplacer les occurrences de $x_i$ et $y_i$
  par les formules (\ref{eq_xy_lambda}), et pour les deux autres il faut
  en plus utiliser la relation (\ref{eq_lambda}).
\end{proof}

\begin{rem}
  Bien que le théorème donne une description complète de la structure
  de l'algèbre, il ne donne pas directement de formule pour un cocycle
  associé à la structure de produit croisé associée à $L$. En effet,
  on dispose de générateurs canoniques correspondant aux éléments
  $\tau_{i,j}$ du groupe de Galois, mais il n'y a pas vraiment
  de choix qui s'impose naturellement pour les autres éléments
  du groupe. On pourrait par exemple décomposer un $g\in G$
  en produit de $\tau_{i,j}$ de support disjoint, en faisant
  un choix de normalisation (par exemple décomposer en
  $\tau_{i_1,j_1}\cdots \tau_{i_s,j_s}$ avec $i_1<j_1<i_2<\dots <j_s$),
  et décréter que $u_g=u_{i_1,j_1}\cdots u_{i_s,j_s}$. Cela a l'avantage
  de ne fait intervenir que des générateurs qui commutent entre
  eux, mais le calcul des produits entre différents $u_g$ devient
  très pénible car faisant intervenir de façon très lourde la combinatoire
  de l'ordre sur les coefficients. D'un autre côté, on peut
  aussi par exemple décomposer $g$ en $\tau_{i_1,r}\cdots \tau_{i_s,r}$ en utilisant
  l'indice $r$ comme pivot et définir $u_g=u_{i_1,r}\cdots u_{i_s,r}$.
  On a alors plus de régularité dans le schéma, mais les générateurs
  en question ne commutent pas, et les calculs restent donc pénibles.
  Dans tous les cas, bien qu'on puisse en théorie décrire de façon
  élémentaire un cocycle correspondant à $B$, la tâche est en
  réalité assez pénible dès que $r$ n'est pas très petit.
\end{rem}






\subsection{Algèbre discriminante en degré 4}

On veut maintenant décrire l'autre sens de l'équivalence
entre algèbres de type $A_3$ et $D_3$, c'est-à-dire calculer
explicitement l'algèbre discriminante d'une algèbre à involution
unitaire de degré 4.

Soit donc $(B,\tau)$ une telle algèbre sur $k$, de centre $K = k(\xi)$,
avec $\xi^2 = \delta\in k^*$.
On suppose qu'on dispose de $L\subset B$ $k$-algèbre $G$-galoisienne
avec $G = (\Zd)^3$ telle que $\tau(L)=L$, avec $\tau_{|L}\in G$.
Comme on sait que $(B,\tau)$ est l'algèbre de Clifford d'une
certaine algèbre d'indice 2 (précisément son algèbre discriminante),
la partie précédente montre qu'une telle sous-algèbre existe toujours.

On pose alors $H\subset G$ tel que $K = L^H$, et on écrit
$H = \langle s,t\rangle$. On a $L = K(\xi_1,\xi_2)$ avec
$\xi_i^2 = \delta_i\in K^*$ tels que $s(\xi_1)=-\xi_1$, $s(\xi_2)=\xi_2$, 
$t(\xi_1)\xi_1$ et $t(\xi_2)=-\xi_2$. On choisit
$u_s$, $u_t$ et $u_{st}$ dans $B^*$ qui définissent une structure de produit
croisé. En utilisant les formules de la proposition \ref{prop_mu}
et la théorème de Hilbert 90, on voit que quitte à changer les $u_g$
on peut supposer $\tau(u_g)=u_g$ pour tout $g\in H$ (quand on fait le
calcul en sens inverse on tombe plutôt sur des générateurs vérifiant
$\tau(u_g)=-u_g$ mais cela nous permet d'éviter d'avoir à gérer des
signes, et il suffit de remplacer $u_g$ par $\xi u_g$ pour retrouver
l'autre situation).

On pose un certain nombre de notations : pour tout $g\in H$,
$a_g=\alpha(g,g)=u_g^2\in L^*$, $v=\alpha(s,t)$.
On peut remarquer que $a_g\in (L^*)^{\langle g,\tau\rangle}$,
et on pose $n(a_g)=N_{L^{\langle g,\tau\rangle}/k}(a_g)$
et $tr(a_g)=\Tr_{L^{\langle g,\tau\rangle}/k}(a_g)$
(donc par exemple $n(a_s)=a_st(a_s)$ et $tr(a_s)=a_s+t(a_s)$).
On définit également $N = N_{L/K}(v)$, dont on vérifie sans peine que
c'est un élément de $k^*$. La structure de produit croisé
de $(B,\tau)$ est entièrement caractérisée par $a_s$, $a_t$
et $v$. Pour des raisons techniques, on supposera que
les $a_g$ ne sont pas dans $k^*$, ce qui est toujours
possible quitte à changer les $u_g$.
On arrive alors à l'énoncé suivant :

\begin{thm}\label{thm_a3_d3}
  Soit $(A,\sigma)=D(B,\tau)$. Alors $A\simeq M_3(Q)$ où
  $Q = (\delta\delta_1,n(a_s))$. On note $k'= L^{\langle \tau,t \rangle} =k(\xi\xi_1)\subset Q$,
  et $j\in Q$ tel que $j^2=n(a_s)$ et $Q = k'\oplus jk'$.
  Alors $\sigma\simeq \sigma_h$ où $h = \fdiag{z_1,z_2,z_3}$ avec :
  \begin{align*}
    z_1 &= \frac{\delta tr(a_{st})\xi\xi_1}{Nn(a_{st})} + j\frac{2\delta t(v) \xi\xi_1}{Na_sa_{st}\tau(v)}, \\
    z_2 &= \delta_2 N n(a_{st}) \xi\xi_1, \\
    z_3 &= \frac{-tr(a_s)\xi\xi_1}{n(a_s)} + j\frac{-2\xi\xi_1}{n(a_s)}.
  \end{align*}
\end{thm}

On va passer le reste de cette partie à montrer ce résultat,
et on va procéder par descente. Le plan de la preuve est donc
de caractériser $D(B,\tau)$ par un 1-cocyle $G\To PGL_6(L)$,
noté $g\mapsto [Q_g]$ (où $[Q_g]$ est la classe de $Q_g\in GL_6(L)$),
de sorte que $D(B,\tau)\subset M_6(L)$ soit caractérisé comme
les points fixes de l'action semi-linéaire de $G$ donnée par
$g\ast A = Q_g g(A) Q_g^{-1}$. Pour ça on commence par décrire
$B$ à partir d'un cocycle $H\to PGL_4(L)$, puis on en déduit
un cocycle $H\To PGL_6(L)$ définissant $\Lambda^2B$, et on étend
ce cocycle à $G$ en donnant l'action de $\tau$, comme décrit
dans \cite[§10.E]{BOI}. Il s'agira ensuite de donner un autre 1-cocycle $G\to PGL_6(L)$
(disons $g\mapsto [R_g]$) décrivant $M_3(Q)$, et de montrer que
ces deux cocycles sont cohomologues, ce qui établit l'isomorphisme
de l'énoncé. Cela revient à trouver $Y\in GL_6(L)$ telle que
$[Y^{-1}Q_g g(Y)] = [R_g]$.

En ce qui concerne l'involution $\sigma$, il s'agit d'une descente
de l'involution adjointe de la forme canonique sur
$L^6\simeq \Lambda^2(L^4)$, et on peut utiliser la matrice
explicite $Y$, ainsi qu'une équivalence de Morita, pour
trouver une matrice de la forme hermitienne $h$ à partir
d'une matrice pour la forme canonique.

\subsubsection*{Cocycle définissant $B$}

Soit $E$ un $L$-module libre de dimension 4, de base $(e_i)$,
qu'on utilisera pour identifier $\End_L(E)= M_4(L)$.
On définit alors un $K$-plongement $\phi: B\to \End_L(E)$ par~:
\[\forall \lambda\in L, \phi(\lambda) = \tld{\lambda} = \begin{pmatrix}
\lambda & & & \\
 & s(\lambda) & & \\
 & & t(\lambda) & \\
 & & & st(\lambda)
\end{pmatrix},\]

\[\phi(u_s) = \begin{pmatrix}
0 & 1 & & \\
a_s & 0 & & \\
 & & 0 & u \\
 & & \frac{t(a_s)}{u} & 0
\end{pmatrix},\]

\[\phi(u_t) = \begin{pmatrix}
 & & 1 & 0 \\
 & & 0 & 1 \\
a_t & 0 & & \\
0 & s(a_t) & &  
\end{pmatrix}\]

où on note $u=\frac{v}{st\tau(v)}$, qui vérifie
$u_t u_s = u\cdot u_s u_t$, et $N_{L/K}(u)=1$.

On vérifie qu'on a bien les relations attendues pour que
$\phi$ définisse un plongement. On en déduit un isomorphisme
$\Phi: B_L \Isom \End_L(E)$ defini par
$\Phi(x\otimes \lambda) = \lambda\phi(x)$.
\\

On pose alors

\[P_s = \begin{pmatrix}
0 & 1 & & \\
a_s & 0 & & \\
 & & 0 & 1 \\
 & & a_s & 0
\end{pmatrix}, P_t = \begin{pmatrix}
 & & 1 & 0 \\
 & & 0 & t(u) \\
a_t & 0 & & \\
0 & \frac{a_t}{u} & &  
\end{pmatrix}.\]

On vérifie que si on définit $g\ast A = P_g g(A) P_g^{-1}$ pour $g=s,t$,
alors on obtient l'action semi-linéaire de $H$ sur $\End_L(E)$
venant de $\Phi$ (bien évidemment $P_g$ est défini à un scalaire de $L^*$ près).

Cela se vérifie par le fait que $g\ast \phi(u_h) = \phi(u_h)$
et $g\ast \tld{\lambda}=\tld{g(\lambda)}$ pour
$g,h\in\{s,t\}$ et $\lambda\in L$ (bien sûr il suffit
de définir l'action de $s$ et $t$).

\subsubsection*{Cocycle définissant $\lambda^2 B$}

L'isomorphisme $\Phi$ induit un isomorphisme
$\lambda^2 \Phi: \lambda^2 B_L \Isom \End_L(\Lambda^2 E)$,
qu'on va utiliser pour caractériser $\lambda^2 B$ par descente.
On munit $\Lambda^2 E$ de la base naturelle
$(e_{12}, e_{34}, e_{13}, e_{24}, e_{14}, e_{23})$ où
on écrit $e_{ij}=e_i\wedge e_j$ (l'ordre
est choisi pour faciliter les calculs par la suite), et
on identifiera alors $\End_L(\Lambda^2 E) = M_6(L)$.

Pour tout $f\in \End_L(E)$, on a par fonctorialité un élément
$\Lambda^2(f)\in \End_L(\Lambda^2 E)$ défini par
$(\Lambda^2 f)(x\wedge y) = f(x)\wedge f(y)$. On montre
alors que l'action semi-linéaire de $H$ sur $\End_L(\Lambda^2 E)$
est donnée par $g\ast A = Q_g g(A) Q_g^{-1}$ où
$Q_g = \Lambda^2 P_g$, soit :
\[ Q_s = \begin{pmatrix}
-a_s & 0 & & & & \\
0 & -a_s & & & & \\
 & & 0 & 1 & & \\
 & & a_s^2 & 0 & & \\
 & & & & 0 & a_s \\
 & & & & a_s & 0
\end{pmatrix},\]
\[Q_t = \begin{pmatrix}
0 & \frac{a_t}{s(a_t)} & & & & \\
a_t^2 & 0 & & & & \\
 & & -a_tu & 0 & & \\
 & & 0 & -a_tt(u) & & \\
 & & & & 0 & -a_t \\
 & & & & \frac{-a_t^2}{s(a_t)} & 0
\end{pmatrix}.\]

Il s'agit de montrer que l'image de $\Lambda^2(\Phi)$ est
fixée par cette action. Or si $f$ est dans l'image de $\Phi$
alors $\Lambda^2(f)$ est dans l'image de $\Lambda^2(\Phi)$.
En effet, si on écrit $\lambda^2 B=\End_{B^{\otimes 2}}(B^{\otimes 2}(1-g))$
avec $g$ l'élément de Goldman, alors pour tout $b\in B$
on peut définir $\lambda^2 b\in \lambda^2 B$ par la multiplication
à droite par $b\otimes b$ dans $B^{\otimes 2}(1-g)$ (on a bien
$(b\otimes b)(1-g) = (1-g)(b\otimes b)$ donc cette opération
est bien définie), et cela correspond dans le cas
déployé à la construction $\Lambda^2(f)$, donc par
fonctorialité $(\lambda^2 \Phi)(\lambda^2 x) = \Lambda^2 (\Phi(x))$.
De là, si $A$ est dans l'image de $\Phi$, on a
\[ (\Lambda^2 P_g)\cdot g(\Lambda^2 A)\cdot (\Lambda^2 P_g)^{-1}
  = \Lambda^2(P_g g(A) P_g^{-1}) = \Lambda^2 A \]
et donc l'image de $\Lambda^2(\Phi)$ est bien
fixée par cette action.

\subsubsection*{Cocycle définissant $D(B,\tau)$}

Pour définir $A=D(B,\tau)$, il faut encore définir l'action
de $\tau\in G$ sur $\End_L(\Lambda^2 E)$. On va montrer qu'elle
est donnée par $\tau\ast A = Q_\tau \tau(A) Q_\tau^{-1}$ avec
\[Q_\tau = \begin{pmatrix}
0 & a_s & & & & \\
a_sn(a_t) & 0 & & & & \\
 & & 0 & -a_t & & \\
 & & -a_s^2s(a_t) & 0 & & \\
 & & & & 0 & a_ss(a_t) \\
 & & & & a_sa_t & 0
\end{pmatrix}.\]

D'après \cite[def 10.28]{BOI}, $D(B,\tau)$ est constitué des éléments
fixes par l'action de $\tau$ sur $\lambda^2B$, qui agit par
$\tau^{\wedge 2}\circ \gamma$ où $\gamma$ est l'involution
canonique, donc l'action de $\tau$ sur $\End_L(\Lambda^2 E)$
est donnée par la composition de $(\lambda^2 \Phi)_*(\tau_L^{\wedge 2})$
et $(\lambda^2 \Phi)_*(\gamma_L)$ (qui commutent).

Par fonctorialité, $(\lambda^2 \Phi)_*(\gamma_L)$ est l'adjointe
de la forme canonique
\[ \anonfoncdef{\Lambda^2 E \times \Lambda^2 E}{\Lambda^4 E = L\cdot e_{1234} \simeq L}
  {(x,y)}{x\wedge y}    \]
 qui a pour matrice dans la base choisie de $\Lambda^2 E$:
\[B = \begin{pmatrix}
0 & 1 & & & & \\
1 & 0 & & & & \\
 & & 0 & -1 & & \\
 & & -1 & 0 & & \\
 & & & & 0 & 1 \\
 & & & & 1 & 0
\end{pmatrix}.\]

Quant à $(\lambda^2 \Phi)_*(\tau_L^{\wedge 2})$, elle est l'adjointe
de la forme hermitienne $h^{\wedge 2}$ où $h$ est une forme
hermitienne sur $E$ telle que $\sigma_h = \Phi_*(\tau_L)$.
On montre que $h$ a pour matrice dans la base $(e_i)$
\[ M = \begin{pmatrix}
a_sa_t & & & \\
 & a_t & & \\
 & & a_s & \\
 & & & \frac{a_t}{s(a_t)}
\end{pmatrix}.\]
En effet, on a
\[ \sigma_h(A) = M^{-1} \tau(\transp{A}) M \]
avec $\sigma_h(\tld{\lambda}) = \tld{\tau(\lambda)}$,
donc $M$ commute avec les matrices $\tld{\lambda}$. En prenant
$\lambda$ fixe par aucun élément de $H$, on obtient que $M$ doit
être diagonale. De plus, on doit avoir $\sigma_h(\phi(u_g)) = \phi(u_g)$
pour $g=s,t$, ce qu'on vérifie par calcul.
On a alors comme matrice pour $h^{\wedge 2}$ :
\[ M' = \Lambda^2M = \begin{pmatrix}
n(a_t) & & & & & \\
 & 1 & & & & \\
 & & a_ss(a_t) & & & \\
 & & & \frac{a_t}{a_s} & & \\
 & & & & a_t & \\
 & & & & & s(a_t)
\end{pmatrix}.\]

On trouve donc $\tau\ast A = B^{-1}\transp{((M')^{-1}\tau(\transp{A})M')}B$,
et donc $Q_\tau = B^{-1}\transp{M'}$, ce qui après calcul donne
la matrice annoncée.

\subsubsection*{Cocycle définissant $M_3(Q)$}

On pose donc $Q$ algèbre de quaternions sur $k$, donnée
par la sous-algèbre quadratique $k'=k(\xi\xi_1)$ et le
quaternion pur $j$ tel que $j^2=n(a_s)$ (et $j$ anti-commute
avec $\xi\xi_1$). On a $k'=L^D$ avec $D=\langle t,s\tau \rangle\subset G$.

On choisit alors le plongement
\[ \anonfoncdef{Q}{M_2(k')}{x+jy}{\begin{pmatrix}
x & n(a_s)s(y) \\
y & s(x)
\end{pmatrix}}  \]
où $x,y\in k'$, qui induit un plongement
$M_3(Q)\To M_6(L)$. On vérifie que le 1-cocycle
associé à ce plongement est $g\mapsto [R_g]$
où $R_g=I_6$ si $g\in D$, et $R_g=S$ sinon, où
\[S = \begin{pmatrix}
0 & n(a_s) & & & & \\
1 & 0 & & & & \\
 & & 0 & n(a_s) & & \\
 & & 1 & 0 & & \\
 & & & & 0 & n(a_s) \\
 & & & & 1 & 0
\end{pmatrix}.\]

\subsubsection*{Cocycles cohomologues}

On cherche donc $Y\in GL_6(L)$ telle que
$[Y^{-1}Q_g g(Y)] = [R_g]$ dans $PGL_6(L)$, soit
\[ Y^{-1}Q_gg(Y) = c_g R_g  \]
avec $c_g\in L^*$, et il suffit de le vérifier
pour $g=s,t,\tau$.

On vérifie que 
\[ Y = \begin{pmatrix}
Y_1 & & \\
 & Y_2 & \\
 & & Y_3
\end{pmatrix}\]
convient, où
\[ Y_1 = \xi \begin{pmatrix}
\frac{1}{s(v)} & \frac{-a_s}{st\tau(v)} \\
\frac{a_t}{\tau(v)} & \frac{-a_sa_t}{t(v)}
\end{pmatrix}, Y_2 = \xi_2 \begin{pmatrix}
a_sn(a_t)u & 0 \\
0 & a_s^3n(a_t)us(u)
\end{pmatrix},\]
\[ Y_3 = \begin{pmatrix}
a_tst(u) & a_ss(a_t)u \\
a_t & a_ss(a_t)ut(u)
\end{pmatrix}\]
avec $c_s = \frac{1}{u}$, $c_t = \frac{a_t}{st(u)}$
et $c_\tau = s(a_t)t(u)$.

La preuve est purement calculatoire, on ne reproduira
pas le calcul ici par souci de lisibilité. On indique les
relations à utiliser entre les différents éléments qui
interviennent dans la formule : $\frac{t(a)}{a}=us(u)$,
$\frac{b}{s(b)}=ut(u)$, et $N_{L/K}(u)=1$.
\\

L'existence de $Y$ montre par le mécanisme usuel de la
cohomologie que $D(B,\tau)\simeq M_3(Q)$. Mais on a
quelque chose de plus concret que cela : si on note
$g\ast_1 A =Q_gg(a)Q_g^{-1}$ la première action considérée,
et $g\ast_2 A =R_qg(a)R_g^{-1}$ la deuxième, alors en
posant $\psi(A)=Y^{-1}AY$, qui est un automorphisme
de $\End_L(\Lambda^2 E)$, on a précisément
$g\ast_2(\psi(A))=\psi(g\ast_1 A)$.
Or $D(B,\tau)$ est constituée des points fixes de
l'action $\ast_1$, et $M_3(Q)$ des points fixes de $\ast_2$,
donc on obtient l'isomorphisme concret suivant :
au sein de $\End_L(\Lambda^2 E)$, $\psi$ envoie
$D(B,\tau)$ sur $M_3(Q)$.

\subsubsection*{Descente de la forme quadratique}

On a déjà décrit la forme quadratique canonique sur
$\End_L(\Lambda^2 E)$, de matrice $B$. Elle induit
une certaine involution sur $\End_L(\Lambda^2 E)$,
qui se restreint à $D(B,\tau)$ en son involution
canonique. Pour obtenir l'involution correspondante
sur $M_3(Q)$, on utilise donc le changement de base
donné par $Y$ : l'involution sur $\End_L(\Lambda^2 E)$
correspondant à la forme quadratique de matrice
$B'=\transp{Y}BY$ se restreint en l'involution
cherchée sur $M_3(Q)$.

On est donc dans la situation du triangle commutatif
suivant dans $\CBrh{L}$ :
\[ \begin{tikzcd}
    & (Q_L,\can) \simeq (M_2(L),\can) \drar{\mathcal{H}_{-1}} & \\
    (M_3(Q_L),\sigma_L)\simeq M_6(L) \urar{h_L} \arrow{rr}{b'} & & (L,\Id)
    \end{tikzcd} \]
  où $b'$ est de matrice $B'$ et $\mathcal{H}_{-1}$ est le
  plan (hyperbolique) alterné.
  Par équivalence de Morita on en déduit alors que
  l'involution $\sigma_L$ sur $M_6(L)= M_3(M_2(L))$ est adjointe
  à la forme anti-hermitienne sur $(M_2(L),\can)$
  de matrice $CB'$, où
  \[ C = \begin{pmatrix}
0 & -1 & & & & \\
1 & 0 & & & & \\
 & & 0 & -1 & & \\
 & & 1 & 0 & & \\
 & & & & 0 & -1 \\
 & & & & 1 & 0
\end{pmatrix}.\]

Comme une telle matrice n'est définie qu'à un scalaire
près, on doit avoir $\lambda\in L^*$ tel que $\lambda CB'$
soit de la forme
\[ \begin{pmatrix}
x_1 & n(a)s(y_1) & & & & \\
y_1 & s(x_1) & & & & \\
 & & x_2 & n(a)s(y_2) & & \\
 & & y_2 & s(x_2) & & \\    
 & & & & x_3 & n(a)s(y_3) \\
 & & & & y_3 & s(x_3)
\end{pmatrix}\]
avec $x_i,y_i\in k'$, et alors en posant $z_i = x_i + jy_i\in Q$,
$\fdiag{z_1,z_2,z_3}$ est une diagonalisation de la forme
anti-hermitienne cherchée.

Concrètement, si on choisit $\lambda = \frac{\xi \xi_1}{a_s^2n(a_t)u}$,
on obtient par le calcul les $z_i$ donnés dans l'énoncé du théorème
\ref{thm_a3_d3}, ce qui achève sa démonstration.

\section{Quelques invariants en petit degré}

On donne ici quelques exemples d'invariants d'algèbres à involution,
en utilisant les résultats des parties \ref{sec_inv_herm} et
\ref{sec_a3_d3}.

\subsection{Invariants $a_3$ et $a_4$ en degré 6}

Dans cette partie on se donne une algèbre $(A,\sigma)$ de
type $D_3$ sur $k$: on a donc $A$ de degré 6 et d'indice au plus 2, et $\sigma$
est orthogonale. En particulier, on peut appliquer les
résultats de la partie \ref{sec_ind_2} pour construire
des invariants cohomologiques : d'après la discussion suivant
la proposition \ref{prop_ber}, tout invariant $\alpha\in \Inv^d(I,\mu_2)$
définit un invariant $\tld{\alpha}(A,\sigma)$ à valeur dans $M^{d-1}_Q(K)$,
et si cet invariant est nul on peut définir $\alpha(A,\sigma)\in M^d_Q(K)$.
On va montrer que certains de ces invariants se relèvent en des invariants
cohomologiques (à valeur dans $H^d(K,\mu_2)$ et non
dans $M_Q^d(K)$).

Dans \cite[§7]{RST}, les auteurs considèrent la forme trace d'involution
$T_\tau$ d'une algèbre $(B,\tau)$ de type $A_3$ telle que $Z(B)=k(\sqrt{-1})$,
donc qui correspond à une algèbre $(A,\sigma)$ de type $D_3$ dont
le discriminant est $-1$. Ils prouvent que sous cette condition
$T_\tau = n_Q + q_4$ où $q_4$ est une 4-forme de Pfister (donc
cette écriture est unique), et posent $a_4(A,\sigma)=e_4(q_4)\in H^4(k,\mu_2)$
(on utilise des notations légèrement différentes). On se propose
d'élargir un peu cette construction, en faisant le lien
avec nos méthodes.

Précisément, on va établir un lien avec l'invariant $v_4^{(1)}\in \Inv^4(I,\mu_2)$
(voir \ref{ex_def_v}). Il vérifie $\tld{v_4^{(1)}} = v_3^{(1)}$ (voir
\ref{prop_simil}), donc on peut définir
$v_3^{(1)}(A,\sigma)\in M_Q^3(k)$, et lorsque cette classe est nulle
on peut définir $v_4^{(1)}(A,\sigma)\in M_Q^4(k)$. En particulier,
comme $v_3^{(1)}=(-\delta)\cup u_2^{(1)}$, $v_3^{(1)}$ est toujours
nul en discriminant $-1$, et donc on peut toujours définir
$v_4^{(1)}(A,\sigma)$ dans la situation étudiée dans \cite{RST}.
On va montrer que $a_4(A,\sigma)$ est un relevé de $v_4^{(1)}(A,\sigma)$,
y compris dans une situation plus générale où on ne fait
pas d'hypothèse sur le discriminant.

\subsubsection*{Forme trace d'involution}

Soit $(B,\tau)=Cl(A,\sigma)$. On rappelle que $T_\tau$ est la restriction
de la forme trace de $B$ au sous-$k$-espace des éléments symétriques,
et on rappelle également que si $d$ est pair, $\lambda^d(\sigma)\in GW(k)$
est bien définie (par exemple comme étant $\lambda^d(\fdiag{1}_\sigma)\in GW(k)
\subset \tld{GW}(A,\sigma)$).

\begin{prop}
  On a l'égalité suivante dans $GW(k)$: $T_\tau = 1 + \lambda^4(\sigma)$.
\end{prop}

On aura besoin au cours de la preuve du lemme suivant :

\begin{lem}\label{lem_quater_trace}
  Soit $z\in Q$. Alors
  \[ \Trd_Q(z)^2-4 \Nrd_Q(z) = 4 z_0^2 \]
  où $z_0$ est la partie pure de $z$.
\end{lem}

\begin{proof}
  On a $z=x+z_0$ où $2x=\Trd_Q(z)$, et
  \begin{align*}
    N_Q(z) &= (x+z_0)(x-z_0) \\
           &= x^2 - z_0^2
  \end{align*}
  donc $4x^2-4N_Q(z)=4z_0^2$.
\end{proof}

\begin{proof}
  On écrit $\sigma \simeq \sigma_h$ où $h=\fdiag{z_1,z_2,z_3}$
  avec $z_i\in Q$ quaternion pur, et $Q$ est l'algèbre de quaternions
  Brauer-équivalente à $A$. On supposera que les $z_i$ ne commutent
  pas deux à deux. On utilise alors les notations
  du théorème \ref{thm_d3_a3} pour décrire $(B,\tau)$.
  On note $H = \{1,g_1,g_2,g_3\}$ où $g_i(\xi_i)=\xi_i$,
  donc $g_i=\tau_{p,q}$ où $i\neq p,q$, et $u_i=u_{g_i}=u_{p,q}$.

  Soit $V$ l'ensemble des éléments symétriques de $(B,\tau)$ ; c'est
  un $k$-espace vectoriel de dimension 16.
  Pour tout $\lambda\in L$, et tout $g\in H$, on a $\tau(u_g)=\eps_gu_g$
  avec $\eps_1=1$ et $\eps_g=-1$ si $g\neq 1$, donc 
  $\tau(\lambda u_g) = \tau(u_g)\tau(\lambda)=\eps_g(g\tau)(\lambda)u_g$.
  De là, $\lambda u_g\in V$ si et seulement si
  $(g\tau)(\lambda) = \eps_g\lambda$. On note $L_g$ l'ensemble des
  $\lambda\in L$ vérifiant cette condition, et $V_g=L_gu_g\subset V$.
  Comme $\dim_k(L_g)=4$, on a $\bigoplus_{g\in H}V_g$ de dimension 16
  sur $k$, donc $V=\bigoplus_{g\in H}V_g$. De plus, 
  cette décomposition est orthogonale pour $T_\tau$ puisque
  les $Lu_g$ sont orthogonaux pour la forme trace (en effet,
  $\Trd_B(\lambda u_g)=0$ si $g\neq 1$). On a donc $T_\tau = \sum_g b_g$ où
  \[ \foncdef{b_g}{L_g\times L_g}{k}{(x,y)}{\Tr_{L/K}(xg(y)u_g^2).} \]

  Le $k$-espace $L_1$ a pour base $(1,\xi_1\xi_2,\xi_1\xi_3,\xi_2\xi_3)$,
  qui est une base orthogonale, donc $b_1$ (qui est isométrique à la
  forme trace de l'extension $L^\tau/k$) est isométrique à
  $\fdiag{1,z_1^2z_2^2,z_1^2z_3^2,z_2^2z_3^2}$. Si on note $\{i,p,q\}=\{1,2,3\}$,
  on a comme $k$-base de $L_{g_i}$ : $(\xi,\xi_i,\xi_i\xi_p,\xi_i\xi_q)$.
  Sachant que $u_i^2 = \Trd_Q(z_pz_q)-2\xi_p\xi_q$, on trouve comme
  matrice pour $b_{g_i}$ dans cette base :
  \[ \begin{pmatrix}
      \delta \Trd_Q(z_pz_q) & -2\delta &  & \\
      -2\delta & z_i^2\Trd_Q(z_pz_q) & & \\
      & & -z_i^2z_p^2\Trd_Q(z_pz_q) & 2\delta \\
      & & 2\delta & -z_i^2z_q^2\Trd_Q(z_pz_q)
    \end{pmatrix}. \]
  Les deux blocs diagonaux ont le même déterminant
  \[ \Delta = \delta z_i^2(\Trd_Q(z_pz_q)^2-4z_p^2z_q^2), \]
  qui d'après le lemme \ref{lem_quater_trace} vaut
  \[ \Delta = 4\delta z_i^2(z_pz_q)_0^2 \equiv z_p^2z_q^2(z_pz_q)_0^2 \,\text{mod}\, (k^*)^2. \]
  On en déduit :
  \begin{align*}
    b_{g_i} &= \fdiag{\delta \Trd_Q(z_pz_q)}\fdiag{1,z_p^2z_q^2(z_pz_q)_0^2}
              + \fdiag{-z_i^2z_p^2\Trd_Q(z_pz_q)}\fdiag{1,z_p^2z_q^2(z_pz_q)_0^2} \\
            &= \fdiag{\delta \Trd_Q(z_pz_q)}\fdiag{1,-z_q^2}\fdiag{1,z_p^2z_q^2(z_pz_q)_0^2} \\
            &= \fdiag{\delta \Trd_Q(z_pz_q)}\pfis{z_q^2,-z_p^2z_q^2(z_pz_q)_0^2}.
  \end{align*}
  Or
  \begin{align*}
    (z_q^2,-z_p^2z_q^2(z_pz_q)_0^2) &= (z_q^2,z_p^2) + (z_q^2,(z_pz_q)_0^2) \\
                                    &= (z_p^2,z_q^2)+[Q]
  \end{align*}
  (on rappelle que $z_q$ et $(z_pz_q)_0$ anti-commutent). De là,
  $\pfis{z_q^2,-z_p^2z_q^2(z_pz_q)_0^2} = \phi_{z_p,z_q}$ (d'après la proposition
  \ref{prop_phi_quater}) et finalement $b_{g_i} = \fdiag{-\delta}\fdiag{z_p}\fdiag{z_q}$
  (d'après la proposition \ref{prop_prod_quater}).
  Donc en combinant les calculs de $b_1$ et des $b_{g_i}$ :
  \begin{equation}\label{eq_t_tau}
    T_\tau = \fdiag{1,z_1^2z_2^3,z_1^2z_3^2,z_2^2z_3^2}
    + \sum_{p<q} \fdiag{-\delta}\fdiag{z_p}\fdiag{z_q}.
  \end{equation}
  D'autre part,
  \begin{align*}
    1 + \lambda^4(h) &= 1 + \sum_{p< q}\lambda^2(\fdiag{z_p})\lambda^2(\fdiag{z_q})
                       + \sum_{p<q}\lambda^2(\fdiag{z_i})\fdiag{z_p}\fdiag{z_q} \\
                     &= \fdiag{1,z_1^2z_2^3,z_1^2z_3^2,z_2^2z_3^2}
                       + \sum_{p<q}\fdiag{-z_i^2}\fdiag{z_p}\fdiag{z_q}
  \end{align*}
  et on conclut par le fait que $z_p^2z_q^2$ est représenté par
  $\phi_{z_p,z_q}$  (en effet, $-z_p^2$ et $-z_q^2$ sont tous les deux représentés)
  donc $\fdiag{-\delta}\fdiag{z_p}\fdiag{z_q} = \fdiag{-z_i^2}\fdiag{z_p}\fdiag{z_q}$.
\end{proof}

\subsubsection*{Invariants $a_3$ et $a_4$}

En utilisant la formule (\ref{eq_t_tau}), on voit que
\begin{align*}
  e_2(T_\tau) &= (-z_1^2z_2^2,-z_1^2z_3^2) + \sum_{i<j} [Q] +(z_i^2,z_j^2) \\
              &= (-1,-\delta) + \sum_{i<j} (z_i^2,z_j^2) + [Q] + \sum_{i<j}(z_i^2,z_j^2) \\
              &= (-1,-\delta) + [Q]
\end{align*}
donc on peut poser
\[ q_\tau = T_\tau - \pfis{-1,-\delta} - n_Q \in I^3(k). \]

On définit alors $a_3(A,\sigma) = e_3(q_\tau)$ et si $a_3(A,\sigma)=0$
on pose $a_4(A,\sigma)=e_4(q_\tau)$ (ce qui coïncide bien avec
la définition précédente lorsque $\delta=-1$).

\begin{prop}\label{prop_a3_a4}
  La classe de $a_3(A,\sigma)$ dans $M_Q^3(k)$ est $v_3^{(1)}(A,\sigma)$,
  et si $a_3(A,\sigma)=0$ alors la classe de $a_4(A,\sigma)$ dans
  $M_Q^4(k)$ est $v_4^{(1)}(A,\sigma)$.
\end{prop}

\begin{proof}
  Par construction des invariants dans $M_Q^d$, il suffit de prouver
  le résultat quand $A$ est déployée. Soit donc $q$ de dimension 6
  telle que $\sigma = \sigma_q$. On écrit $q=q_1+q_2$ avec $q_i=\fdiag{a_i,b_i,c_i}$.
  D'après la proposition précédente, $q_\tau = 1 + \lambda^4(q) -\pfis{-1,-\delta}$.
  On a
  \begin{align*}
    1 + \lambda^4(q) &= 1 + \lambda^3(q_1)q_2+\lambda^2(q_1)\lambda^2(q_2)+q_1\lambda^3(q_2) \\
                     &= 1 + \lambda^3(q_1)\lambda^3(q_2)\lambda^2(q_2)+\lambda^2(q_1)\lambda^2(q_2)+\lambda^3(q_1)\lambda^2(q_1)\lambda^3(q_2) \\
                     &= (1 + \fdiag{-\delta}\lambda^2(q_1))(1 + \fdiag{-\delta}\lambda^2(q_2))
  \end{align*}
  car $q_i=\lambda^3(q_i)\lambda^2(q_i)$, et $-\delta=\lambda^3(q_1)\lambda^3(q_2)$.
  Or
  \begin{align*}
    1 + \fdiag{-\delta}\lambda^2(q_i) &= 1 + \fdiag{-\delta}\fdiag{a_ib_i,a_ic_i,b_ic_i} \\
                                      &= (1-\fdiag{-\delta}) + \fdiag{-\delta}\fdiag{1,a_ib_i,a_ic_i,b_ic_i}.
  \end{align*}
  On a donc 
  \begin{align*}
    1 + \lambda^4(q) &= (\pfis{-\delta} + \fdiag{-\delta}\pfis{-a_1b_1,-a_1c_1})(\pfis{-\delta} + \fdiag{-\delta}\pfis{-a_2b_2,-a_2c_2}) \\
                     &= \pfis{-1,-\delta} - \pfis{-\delta}(\pfis{-a_1b_1,-a_1c_1}+\pfis{-a_2b_2,-a_2c_2}) \\
                     &+ \pfis{-a_1b_1,-a_1c_1,-a_2b_2,-a_2c_2}
  \end{align*}
  donc
  \[ q_\tau = - \pfis{-\delta}(\pfis{-a_1b_1,-a_1c_1}+\pfis{-a_2b_2,-a_2c_2}) + \pfis{-a_1b_1,-a_1c_1,-a_2b_2,-a_2c_2}. \]
  En particulier, comme on vérifie que $(-a_ib_i,-a_ic_i) = w_2(q_i)+(-1)\cup w_1(q_i)+(-1,-1)$, on trouve
  \begin{align*}
    a_3(q) &=  w_1(q)\cup ((-1)\cup w_1(q) + w_2(q_1)+w_2(q_2)) \\
           &= (-1,-1)\cup w_1(q) + (w_1(q_1)+w_1(q_2))\cup (w_2(q_1)+w_2(q_2)) \\
           &= (-1,-1)\cup w_1(q) + w_3(q_1) + w_3(q_2) +w_1(q_1)\cup w_2(q_2) + w_1(q_2)\cup w_2(q_1) \\
           &= (-1,-1)\cup w_1(q) + w_3(q).
  \end{align*}
  Or d'après la proposition \ref{prop_fixed_dim} on a justement $v_3^{(1)}(q) = w_3(q) + (-1,-1)\cup w_1(q)$,
  donc on a bien $v_3^{(1)}(q) = a_3(q)$ dans le cas déployé.

  Supposons $a_3(q)=0$. Alors $\pfis{-\delta}(\pfis{-a_1b_1,-a_1c_1}+\pfis{-a_2b_2,-a_2c_2})$
  est dans $I^4(k)$, et donc $\pfis{-\delta,-a_1b_1,-a_1c_1} = \pfis{-\delta,-a_2b_2,-a_2c_2}$,
  ce qui implique que
  \[ q_\tau = - \pfis{-1,-\delta,-a_1b_1,-a_1c_1} + \pfis{-a_1b_1,-a_1c_1,-a_2b_2,-a_2c_2} \]
  et donc
  \begin{align*}
    a_4(q) &= \left( (-1,-1)+(-1)\cup w_1(q_1)+w_2(q_1) \right) \\
    & \cup \left( (-1)\cup w_1(q)+(-1,-1)+(-1)\cup w_1(q_2)+w_2(q_2) \right) \\
           &= (-1)^{\cup 4} + (-1,-1)\cup w_2(q_1)+(-1,-1)\cup w_2(q_2) + w_2(q_1)\cup w_2(q_2) \\
           &+ (-1)^{\cup 3}\cup w_1(q_1) + (-1)\cup w_1(q_1)\cup w_2(q_2) + (-1)\cup w_3(q_1).
  \end{align*}
  D'autre part, d'après la proposition \ref{prop_fixed_dim} on a 
  \begin{align*}
    v_4^{(1)}(q) &= a_3(q) + (-1)^{\cup 4} + (-1,-1)\cup w_2(q)+ w_4(q) \\
                 &= (-1)^{\cup 4} + (-1,-1)\cup w_2(q_1)+(-1,-1)\cup w_2(q_2) + w_2(q_1)\cup w_2(q_2) \\
                 &+ (-1,-1)\cup w_1(q_1)\cup w_1(q_2)  + w_1(q_1)\cup w_3(q_2) + w_3(q_1)\cup w_1(q_2) 
  \end{align*}
  donc
  \begin{align*}
    a_4(q)-v_4^{(1)}(q) &= (-1)^{\cup 3}\cup w_1(q_1) + (-1)\cup w_1(q_1)\cup w_2(q_2) \\
                        &+ (-1)\cup w_3(q_1) + (-1,-1)\cup w_1(q_1)\cup w_1(q_2) \\
                        &+ w_1(q_1)\cup w_3(q_2) + w_3(q_1)\cup w_1(q_2).
  \end{align*}
  On conclut en montrant que cette classe est égale à $w_1(q_1)\cup a_3(q)$,
  qui vaut 0 par hypothèse. En effet, en utilisant le fait que
  $w_1(q_1)\cup w_1(q_1)=(-1)\cup w_1(q_1)$, $w_1(q_1)\cup w_2(q_1) = w_3(q_1)$
  et $w_1(q_1)\cup w_3(q_1)=(-1)\cup w_3(q_1)$, on trouve :
  \begin{align*}
    w_1(q_1)\cup a_3(q) &= w_1(q_1)\cup ((-1,-1)\cup w_1(q) + w_3(q)) \\
                        &= w_1(q_1)\cup((-1,-1)\cup w_1(q_1) +(-1,-1)\cup w_1(q_2) + w_3(q_1) \\
                        &+ w_2(q_1)\cup w_1(q_2) + w_1(q_1)\cup w_2(q_2) + w_3(q_2) \\
                        &= (-1)^{\cup 3}\cup w_1(q_1) + (-1,-1)\cup w_1(q_1)\cup w_1(q_2) \\
                        &+ (-1)\cup w_3(q_1) + w_3(q_1)\cup w_1(q_2) \\
                        &+ (-1)\cup w_1(q_1)\cup w_2(q_2) + w_1(q_1)\cup w_3(q_2).
  \end{align*}
\end{proof}

\subsection{Invariant $a_5$ en degré 12}\label{sec_a5}

Dans le premier chapitre, on a décrit les invariants cohomologiques
des foncteurs $I^n$, mais également des $I^{n,1}$, dont les éléments
sont de la forme $\pfis{c}q$ où $q$ est dans $I^{n-1}$. Ces invariants
sont du type $\Delta^{n,1}(\alpha): \pfis{c}q\mapsto (c)\cup \alpha(q)$ où $\alpha$ est un
invariant quelconque de $I^{n-1}$ (voir la proposition \ref{prop_delta_inv}).
On veut étendre ces invariants à certains types d'algèbres à involution,
et en particulier on étend l'invariant $a_5$ défini dans \cite[def 20.8]{Gar}
(voir aussi \cite[§9]{R99}) aux algèbres d'indice 2.

Précisément, soit $(A,\sigma)$ une algèbre à involution orthogonale de degré 12
telle que $\sigma$ soit \og{}génériquement dans $I^3$\fg{},
c'est-à-dire que sur tout corps de déploiement $L$, $\sigma_L$
est adjointe à une forme $q\in I^3(L)$ (ou encore, c'est équivalent
au fait que $\fdiag{1}_\sigma\in \tld{I}_0^3(A,\sigma)$). Alors,
comme d'après \cite{Pf} toute forme quadratique de degré 12 dans
$I^3$ est en réalité dans $I^{3,1}$, $\sigma$ est génériquement
dans $I^{3,1}$. On suppose maintenant que $A$ est d'indice 2.
Pour tout invariant cohomologique $\alpha\in \Inv^d(I^2,\mu_2)$,
on peut donc appliquer $\Delta^{3,1}(\tld{\alpha})$
à $(A,\sigma)$ pour obtenir un élément de $M_Q^d(K)$, et si cet invariant
est nul on peut définir $\Delta^{3,1}(\alpha)(A,\sigma)\in M_Q^{d+1}(K)$.

\begin{prop}
  Soit $(A,\sigma)$ une algèbre à involution orthogonale de degré 12
  et d'indice au plus 2
  telle que $\sigma$ soit adjointe à une forme dans $I^3(L)$ pour
  tout corps de déploiement $L$ de $A$. Alors si $(-1)\cup e_3(A,\sigma)=0$,
  on peut définir un invariant $a_5(A,\sigma)\in M_A^5(K)$ qui dans le cas
  déployé correspond à l'invariant $a_5$ de \cite{Gar}.
\end{prop}

\begin{proof}
  Par construction, l'invariant $a_5$ est
  égal (avec nos notations) à $\Delta^{3,1}(u_4^{(2)})$, donc la
  discussion du paragraphe précédent s'applique avec
  $\alpha=u_4^{(2)}\in \Inv^4(I^2,\mu_2)$. On a
  alors $\tld{\alpha}=(-1)\cup u_2^{(2)}=(-1)\cup e_2$, donc
  $\Delta^{3,1}(\tld{\alpha})=(-1)\cup e_3$. On définit donc,
  lorsque $(-1)\cup e_3(A,\sigma)=0$, $a_5(A,\sigma)\in M^5_Q(K)$
  comme $\Delta^{3,1}(u_4^{(2)})(A,\sigma)$.
\end{proof}

On peut par ailleurs remarquer que la condition de l'énoncé est
toujours vérifiée lorsque $-1$ est un carré dans $K$.




\chapter*{Index des notations}

\addcontentsline{toc}{chapter}{Index des notations}

\section*{Général}

\begin{longtable}{rll}
  $k$ & corps de caractéristique $\neq 2$ & \\
  $K$ & extension de $k$ & \\
  $L$ & extension de $K$ & \\
  $(K,v)$ & corps $K$ muni d'une valuation discrète de rang 1  & \\
  $\mathbf{Field}_{/k}$ & catégorie des extensions de corps de $k$ & \\
  $\mathbf{Set}$ & catégorie des ensembles & \\
  $\mathbf{Set_*}$ & catégorie des ensembles pointés & \\
  $(a_1,\dots,a_r)$ & symbole galoisien dans $H^n(K,\mu_2)$ & \\
  $\binom{n}{a_1,\dots,a_r}$ & coefficient multinomial & \\
  $\lceil x \rceil$ & partie entière supérieure de $x$ & \\
  $\lfloor x \rfloor$ & partie entière inférieure de $x$ & \\
  $s\lor t$ & \emph{OU} logique appliqué aux bits de $s$ et $t$ & (\ref{cor_prod_f})\\
  $s \land t$ & \emph{ET} logique appliqué aux bits de $s$ et $t$ & (\ref{cor_prod_f}) \\
  $\mathfrak{S}_{p,q}$ & sous-groupe de Young de $\mathfrak{S}_d$ & p.\pageref{par_young} \\
  $Sh(p,q)$ & ensemble des $(p,q)$-mélanges de $\mathfrak{S}_d$ & p.\pageref{par_shuffle} \\
  $R[G]$ & algèbre de groupe de $G$ sur $R$ & \\
  $R_g$ & composante de degré $g\in G$ de $R$, anneau $G$-gradué & \\
  $\Phi$ & isomorphisme tautologique $R\oplus R\to R[\Zd]$ & (\ref{eq_def_phi_ann}) \\
  $x_{(0)}$ & $\Phi(x,0)$ & (\ref{eq_def_phi_ann}) \\
  $x_{(1)}$ & $\Phi(0,x)$ & (\ref{eq_def_phi_ann}) \\
  $\Delta$ & plongement diagonal $R\to R[\Zd]$ & (\ref{eq_def_delta_ann}) \\
  $\Psi$ & isomorphisme modifié $R\oplus R\to R[\Zd]$ & (\ref{eq_def_psi_ann}) \\
\end{longtable}

\section*{Formes quadratiques}

\begin{longtable}{rll}
  $\fdiag{a_1,\dots,a_r}$ & forme quadratique diagonalisée & \\
  $C(q)$ & algèbre de Clifford de $q$ & \\
  $C_0(q)$ & algèbre de Clifford paire de $q$ & \\
  $\disc(q)$ & discriminant (signé) de $q$ & \\
  $\mathcal{H}$ & plan hyperbolique symétrique & \\
  $\mathcal{H}_{-1}$ & plan hyperbolique alterné & \\  
  $SGW(K)$ & semi-anneau des formes quadratiques sur $K$ & (\ref{def_gw_mixte}) \\
  $GW(K)$ & anneau de Grothendieck-Witt de $K$ & \\
  $W(K)$ & anneau de Witt de $K$ & \\
  $SGW^-(K)$ & monoïde des formes alternées sur $K$ & (\ref{def_gw_mixte}) \\
  $GW^-(K)$ & groupe de Grothendieck-Witt des formes alternées sur $K$ & \\
  $SGW^\pm(K)$ & $SGW(K)\oplus SGW^-(K)$ & \\
  $GW^\pm(K)$ & $GW(K)\oplus GW^-(K)$ & \\
  $\hat{I}(K)$ & idéal des formes de dimension nulle de $GW(K)$ & p.\pageref{par_pfister} \\
  $I(K)$ & idéal fondamental de $W(K)$ & p.\pageref{par_pfister} \\
  $\pfis{a}$ & forme de Pfister $\fdiag{1}-\fdiag{a}\in \hat{I}(K)$ & p.\pageref{par_pfister} \\
  $\pfis{a_1,\dots,a_r}$ & $r$-forme de Pfister $\pfis{a_1}\cdots\pfis{a_r}\in \hat{I}^r(K)$ & p.\pageref{par_pfister} \\
  $\Quad_{2r}$ & foncteur des formes quadratiques  de dim $2r$ &  \\
  $\Quad'_{2r}$ & sous-foncteur de $\Quad_{2r}$ des formes de discriminant trivial & p.\pageref{par_quad}\\
  $\hat{\lambda}^d(q)$ & opération $\lambda$ vue à travers $I\simeq \hat{I}$ & (\ref{lem_hat_lambda}) \\
  $P^d(q)$ & opération $\fdiag{a_1,\dots,a_n}\mapsto \sum_{i_1<\dots <i_d} \pfis{a_{i_1},\dots,a_{i_d}}$ & (\ref{def_pd}) \\
  $I^{n,r}(K)$ & formes semi-factorisées dans $I^n(K)$ & (\ref{def_in_r}) \\
\end{longtable}

\section*{Invariants}

\begin{longtable}{rll}
  $\Inv(F,H)$ & invariants de $F$ à valeurs dans $H$ & p.\pageref{par_inv} \\
  $\Inv_{norm}(F,H)$ & invariants normalisés de $F$ à valeurs dans $H$ & p.\pageref{par_inv} \\
  $\Inv(F,\mu_2)$ & inv. cohom. de $F$ à valeurs dans $\mu_2$ & p.\pageref{par_inv} \\
  $\Inv^d(F,\mu_2)$ & inv. cohom. de degré $d$ de $F$ à valeurs dans $\mu_2$ & p.\pageref{par_inv} \\
  $\Inv(F,W)$ & invariants de Witt de $F$ & p.\pageref{par_inv} \\
  $w_i$ & invariant de Stiefel-Whitney & (\ref{eq_stiefel}) \\
  $e_n$ & invariant de la conjecture de Milnor & (\ref{eq_milnor}) \\
  $A(K)$ & généralement $H^*(K,\mu_2)$ ou $W(K)$ & p.\pageref{par_a} \\
  $A^d(K)$ & composante de degré $d$ de $A(K)$ & p.\pageref{par_a} \\
  $\delta(A)=\delta$ & élément de $A(k)$ (généralement 0 ou 1) & (\ref{prop_exist_delta}) \\
  $f_n$ & morphisme $I^n(K)\to A^n(K)$ & p.\pageref{par_f} \\
  $\{a_1,\dots,a_r\}$ & $f_n(\pfis{a_1,\dots,a_r})$ & p.\pageref{par_f} \\
  $M$ & $\Inv(I^n, A)$ & p.\pageref{par_m} \\
  $M^d$ & $\Inv(I^n, A^d)$ & p.\pageref{par_m} \\
  $f_n^d$ & $(f_{nd}\circ \pi_n^d)\in \Inv^{nd}(I^n,A)$ & (\ref{def_f}) \\
  $g_n^d$ & invariant dans $\Inv^{nd}(I^n,A)$ & (\ref{def_gn}) \\
  $u_{nd}^{(n)}$ & $f_n^d$ si $A(K)=H^*(K,\mu_2)$ & (\ref{ex_pi_u}) \\
  $v_{nd}^{(n)}$ & $g_n^d$ quand $A(K) = H^*(K,\mu_2)$ & (\ref{ex_def_v}) \\
  $\Phi^\eps$ & opérateur $M\to M$ ($\eps=\pm 1$) & (\ref{prop_def_delta}) \\
  $\Phi$ & opérateur de décalage $\Phi^+$ & (\ref{prop_def_delta}) \\
  $\alpha^\eps$ & $\Phi^\eps(\alpha)$ ($\eps=\pm 1$)& (\ref{prop_def_delta}) \\
  $\alpha^{s+,t-}$ & $(\Phi^+)^s\circ (\Phi^-)^t(\alpha)$ & (\ref{eq_alpha_s_t}) \\
  $\alpha_{|I^{n+1}}$ & restriction de $\alpha\in \Inv(I^n,A)$ à $I^{n+1}$ & p.\pageref{par_restr} \\
  $\Psi$ & opérateur $M\to M$ & (\ref{prop_def_psi}) \\
  $\tld{\alpha}$ & $\Psi(\alpha)$ & (\ref{prop_def_psi}) \\ 
  $h^d(q)$ & $P^d(q)$ ou $w_d(q)$ & p.\pageref{par_h} \\
  $\Delta^{n,r}$ & opérateur $\Inv_{norm}(I^{n,r},A)\to \Inv_{norm}(I^{n-r},A)$ & (\ref{prop_delta_inv}) \\
  $\alpha^{(r)}$ & $\Delta^{n,r}(\alpha)$ & (\ref{prop_delta_inv}) 
\end{longtable}

\section*{Anneaux grecs}

\begin{longtable}{rll}
  $G(R)$ &  $1 + tR[[t]] \subset R[[t]]$ & (\ref{eq_gr}) \\
  $\alpha_t$ & opération grecque $R\to G(R)$ & (\ref{prop_alpha_t}) \\
  $\alpha^d$ & $d$-ième composante d'une opération grecque & (\ref{defi_ope_grec}) \\
  $\Gamma(R)$ & ensembe des opérations grecques de $R$ & (\ref{defi_ope_grec}) \\
  $+_G$ & addition de $G(R)$ (produit des séries) & (\ref{prop_prod_gr}) \\
  $\times_G$ & produit de $G(R)$ & (\ref{prop_prod_gr}) \\
  $\eps_R$ & morphisme naturel $G(R)\to R$ & (\ref{eq_eps_r}) \\
  $L(R)$ & lettres grecques de $R$ & (\ref{eq_l_r}) \\
  $\Lambda_R$ & application naturelle $\Gamma(R)\To L(R)$ & (\ref{def_lettre}) \\
  $A(R)$ & automorphismes de $G(R)$ commutant avec $\eps_R$ & (\ref{eq_a_r}) \\
  $\pi_n^d$ & $d$ième composante de l'opération grecque $\pi_n$ & (\ref{thm_pi}) \\
\end{longtable}

\section*{Algèbres et involutions}

\begin{longtable}{rll}
  $(A,\sigma)$ & algèbre à involution de première espèce sur $K$ & \\
  $(A_L,\sigma_L)$ & extension de $(A,\sigma)$ au corps $L$ & \\
  $Sym(A,\sigma)$ & éléments symétriques de $(A\sigma)$ & \\
  $Skew(A,\sigma)$ & éléments anti-symétriques de $(A,\sigma)$ & \\
  $C_A(B)$ & centralisateur de $B$ dans $A$ & \\
  $Z(A)$ & centre de $A$ & \\
  $A^{op}$ & algèbre opposée de $A$ & \\
  $\Trd_A$ & trace réduite de $A$ & \\
  $T_\sigma$ & forme trace d'involution de $(A,\sigma)$ & (\ref{eq_t_sigma}) \\
  $T_\sigma^+$ & restriction de $T_\sigma$ à $Sym(A,\sigma)$ & \\
  $T_\sigma^-$ & restriction de $T_\sigma$ à $Skew(A,\sigma)$ & \\
  $T_{\sigma,a}$ & forme trace d'involution tordue par $a$ & (\ref{ex_prod_diag}) \\
  $\Nrd_A$ & norme réduite de $A$ & \\
  $\sigma_b$ & involution adjointe de $b$ & \\
  $\sigma_q$ & involution adjointe de $q$ & \\
  $Cl(A,\sigma)$ & algèbre de Clifford (paire) de $(A,\sigma)$ & \\
  $K(A)$ & corps de déploiement générique de $A$ & \\
  $K_r(A)$ & corps de réduction générique à l'indice $2^r$ de $A$ & p.\pageref{par_red_gen} \\
  $\Br(K)$ & groupe de Brauer de $K$ & \\
  $[A]$ & classe de Brauer de $A$ & \\
  $A\sim B$ & $A$ et $B$ sont Brauer-équivalente & \\
  $\CBr$ & 2-groupe de Brauer de $K$ & partie \ref{sec_brauer} \\
  $\mathbf{Mor}(R)$ & catégorie de Morita d'un anneau commutatif $R$ & (\ref{rem_cat_mor}) \\
  $B\xrightarrow V A$ & $V$ est un $B$-$A$-bimodule simple & partie \ref{sec_brauer} \\
  $\dim(V)$ & dimension réduite d'un $A$-module & (\ref{def_dim}) \\
  $U^*$ & dual du $B$-$A$-bimodule $U$ & (\ref{eq_isom_dual}) \\
  $\mbox{}^\sigma U^\tau$ & bimodule dual vu à travers les involutions $\sigma$ et $\tau$ & (\ref{eq_u_tordu}) \\
  $g_A$ & élément de Goldman de $A$ & p.\pageref{par_goldman} \\
  $g_A(\pi)$ & élément de $A^{\otimes d}$ correspondant à $\pi \in \mathfrak{S}_d$ & p.\pageref{par_goldman} \\
  $s_{d,A}=s_d$ & élément anti-symétrisant de $A^{\otimes d}$ & (\ref{eq_def_sd}) \\
  $\Alt^d(V)$ & puissance alternée d'un $A$-module & (\ref{def_alt}) \\
  $\lambda^d(A)$ & $d$ième puissance extérieure de l'algèbre $A$ & (\ref{rem_lambda_alt}) \\
  $sh_{p,q}$ & application de mélange de $V^{\otimes d}$ & (\ref{eq_app_shuffle}) \\
  $x \# y$ & produit de mélange de $x$ et $y$ & (\ref{eq_prod_shuffle}) \\
  $\sigma^{\wedge d}$ & puissance extérieure d'une involution & (\ref{rem_lambda_alt_sigma}) \\
\end{longtable}

\section*{Formes hermitiennes}

\begin{longtable}{rll}
  $\sigma_h$ & involution adjointe de $h$ & \\
  $h\perp h'$ & somme orthogonale de deux formes $\eps$-herm. & \\
  $\fdiag{a}_\sigma$ & forme diagonale sur $(A,\sigma)$ & (\ref{ex_prod_diag}) \\
  $\mathcal{H}(A,\sigma)$ & espace hyperbolique hermitien sur $(A,\sigma)$ & \\
  $\CBrh$ & 2-groupe de Brauer hermitien de $K$ & partie \ref{sec_brauer_herm} \\
  $(B,\tau)\xrightarrow {(V,h)} (A,\sigma)$ & $(V,h)$ est un bimodule $\eps$-hermitien & partie \ref{sec_brauer_herm} \\
  $h_A$ & identité de $(A,\sigma)$ dans $\CBrh$ & (\ref{eq_ha}) \\
  $h^*$ & dual de la forme $\eps$-hermitienne $h$ & (\ref{eq_h_star}) \\
  $\phi_{(A,\sigma)}^{(d)}$ & isom. $(A,\sigma)^{\otimes d}\to (A,\sigma)$ ou $(K,\Id)$  & (\ref{prop_prod_red}) \\
  $\mathbf{Mor}_h(R,\iota)$ & catégorie de Morita hermitienne de $(R,\iota)$ & (\ref{rem_cat_mor}) \\
  $f\oplus f'$ & somme orthogonale de morphismes dans $\CBrh$ & \\
  $SGW^\eps(A,\sigma)$ & monoïde des formes $\eps$-hermitiennes sur $(A,\sigma)$ & (\ref{def_gw_mixte})\\
  $GW^\eps(A,\sigma)$ & gpe de Grothendieck-Witt des formes $\eps$-herm. & \\
  $W^\eps(A,\sigma)$ & groupe de Witt des formes $\eps$-hermitiennes & \\
  $SGW^\pm(A,\sigma)$ & $SGW(A,\sigma)\oplus SGW^-(A,\sigma)$ & (\ref{def_gw_mixte}) \\
  $GW^\pm(A,\sigma)$ & $GW(A,\sigma)\oplus GW^-(A,\sigma)$ & \\
  $W^\pm(A,\sigma)$ & $W(A,\sigma)\oplus W^-(A,\sigma)$ & \\
  $\tld{SGW}(A,\sigma)$ & semi-anneau de Grothendieck-Witt mixte de $(A,\sigma)$ & (\ref{def_gw_mixte}) \\
  $\tld{GW}(A,\sigma)$ & anneau de Grothendieck-Witt mixte de $(A,\sigma)$ & (\ref{def_gw_mixte}) \\
  $\tld{W}(A,\sigma)$ & anneau de Witt mixte de $(A,\sigma)$ & (\ref{def_w_mixte}) \\
  $\Gamma$ & groupe de graduation $\{+,-,o,s\}$ & (\ref{eq_gamma}) \\
  $\tld{SGW}(A,\sigma)_+$ & $SGW(K)$ & p.\pageref{par_grad} \\
  $\tld{GW}(A,\sigma)_+$ & $GW(K)$ & p.\pageref{par_grad} \\
  $\tld{W}(A,\sigma)_+$ & $W(K)$ & p.\pageref{par_grad_w} \\
  $\tld{SGW}(A,\sigma)_-$ & $SGW^-(K)$ & p.\pageref{par_grad} \\
  $\tld{GW}(A,\sigma)_-$ & $GW^-(K)$ & p.\pageref{par_grad} \\
  $\tld{W}(A,\sigma)_-$ & 0 & p.\pageref{par_grad_w} \\
  $\tld{SGW}(A,\sigma)_o$ & composante orthogonale de $\tld{SGW}(A,\sigma)$ & p.\pageref{par_grad} \\
  $\tld{GW}(A,\sigma)_o$ & composante orthogonale de $\tld{GW}(A,\sigma)$ & p.\pageref{par_grad} \\
  $\tld{W}(A,\sigma)_o$ & composante orthogonale de $\tld{W}(A,\sigma)$ & p.\pageref{par_grad_w} \\
  $\tld{SGW}(A,\sigma)_s$ & composante symplectique de $\tld{SGW}(A,\sigma)$ & p.\pageref{par_grad} \\ 
  $\tld{GW}(A,\sigma)_s$ & composante symplectique de $\tld{GW}(A,\sigma)$ & p.\pageref{par_grad} \\  
  $\tld{W}(A,\sigma)_s$ & composante symplectique de $\tld{W}(A,\sigma)$ & p.\pageref{par_grad_w} \\
  $\tld{GW}(K)$ & $\tld{GW}(K,\Id)$ & (\ref{eq_gw_depl}) \\
  $\tld{W}(K)$ & $\tld{W}(K,\Id)$ & (\ref{eq_w_depl}) \\
  $\Alt^d(h)$ & puissance alternée d'une forme $\eps$-herm. & (\ref{def_alt_h}) \\
  $\Lambda^d(V)$ & puissance extérieure d'un $A$-module & (\ref{def_lambda_v}) \\
  $\lambda^d(h)$ & puissance extérieure d'une forme $\eps$-herm. & (\ref{def_lambda_v}) \\ 
  $\dim(h)$ & dimension réduite d'une forme $\eps$-herm. & (\ref{def_dim}) \\
  $\tld{GI}(A,\sigma)$ & idéal fondamental de $\tld{GW}(A,\sigma)$ & p.\pageref{par_gi} \\
  $\tld{I}(A,\sigma)$ & idéal fondamental de $\tld{W}(A,\sigma)$ & p.\pageref{par_gi} \\
  $\tld{GI}(A,\sigma)_o$ & $\tld{GI}(A,\sigma)\cap \tld{GW}(A,\sigma)_o$ & p.\pageref{par_gi} \\
  $\tld{I}(A,\sigma)_o$ & $\tld{I}(A,\sigma)\cap \tld{W}(A,\sigma)_o$ & p.\pageref{par_gi} \\
  $\tld{GI}(A,\sigma)_s$ & $\tld{GI}(A,\sigma)\cap \tld{GW}(A,\sigma)_s$ & p.\pageref{par_gi} \\
  $\tld{I}(A,\sigma)_s$ & $\tld{I}(A,\sigma)\cap \tld{W}(A,\sigma)_s$ & p.\pageref{par_gi} \\
  $\tld{H}^n(A,\sigma)$ & cohomologie mixte & \ref{def_filtr} \\
  $\tld{GI}(K)$ & $\tld{GI}(K,\Id)$ & p.\pageref{par_gi_depl} \\
  $\tld{I}(K)$ & $\tld{I}(K,\Id)$ & p.\pageref{par_gi_depl} \\
  $GI(K)$ & $\Ker(GW^\pm(K)\to \Z)$ & p.\pageref{par_gi_depl}\\
  $GJ(K)$ & $\Ker(GW^\pm(K)\to \Zd)$ & p.\pageref{par_gi_depl}\\
  $GJ^{(n)}(K)$ & $GI^{n-1}(K)GJ(K)$ ($n\pgq 1$) & p.\pageref{par_gi_depl}\\
  $\tld{H}^n(K)$ & $\tld{H}^d(K,\Id)$ & (\ref{cor_h_tilde_depl}) \\
  $I^n(A,\sigma)_o$ & $I^{n-1}(K)W(A,\sigma)_o$ & (\ref{prop_decomp_cohom}) \\
  $H^n(A,\sigma)_o$ & $I^n(A,\sigma)_o/I^{n+1}(A,\sigma)_o$ & (\ref{eq_h_gamma}) \\
  $I^n(A,\sigma)_s$ & $I^{n-1}(K)W(A,\sigma)_s$ & (\ref{prop_decomp_cohom}) \\
  $H^n(A,\sigma)_s$ & $I^n(A,\sigma)_s/I^{n+1}(A,\sigma)_s$ & (\ref{eq_h_gamma}) \\
  $\tld{GI}_r^n(A,\sigma)$ & élém. dans $\tld{GI}^n(A,\sigma)$ sur $K_r(A)$ & (\ref{def_gi_red}) \\
  $\tld{I}_r^n(A,\sigma)$ & élém. dans $\tld{I}^n(A,\sigma)$ sur $K_r(A)$ & (\ref{def_gi_red}) \\
\end{longtable}

\section*{Quaternions}

\begin{longtable}{rll}
  $(Q,\can)$ & algèbre de quaternions munie de son involution canonique & \\
  $n_Q$ & forme norme d'une algèbre de quaternions $Q$ & \\
  $\phi_{z_1,\dots,z_r}$ & $r$-forme de Pfister $\pfis{z_1^2,\dots,z_r^2} - \pfis{-1}^{r-2}n_Q$ & (\ref{prop_phi_quater}) \\
  $\p$ & point fermé de la variété de Severi-Brauer & p.\pageref{par_valu} \\
  $v_{\mathfrak{p}}$ & valuation induite par $\p$ & p.\pageref{par_valu} \\
  $\partial_{1,\mathfrak{p}}$ & premier résidu $W(F)\to W(K(\mathfrak{p}))$ & p.\pageref{par_valu} \\
  $\partial_{2,\mathfrak{p}}$ & deuxième résidu $W(F)\to W(K(\mathfrak{p}))$ & p.\pageref{par_valu} \\
  $\ext_\infty$ & extension des scalaires $W^-(Q,\can) \to W(F)$ & (\ref{eq_ext_a}) \\
  $W_{nr,1,\infty}(F)$ & premier groupe de Witt non ramifié & (\ref{eq_w1}) \\
  $W_{nr,2,\infty}(F)$ & deuxième groupe de Witt non ramifié & (\ref{eq_w2}) \\
  $\partial_{\mathfrak{p}}$ & résidu $H^d(F,\mu_2)\to H^{d-1}(K(\mathfrak{p}),\mu_2)$ & p.\pageref{par_valu_cohom} \\
  $H^d_{nr}(F,\mu_2)$ & cohomologie non ramifiée & (\ref{eq_h_nr}) \\
  $\tld{e}_d$ & invariant $\tld{I}_0^d(Q,\can) \to  H^{d-1}_{nr}(F,\mu_2)$ & (\ref{eq_ed_tilde}) \\
  $M_A^d(K)$ & 2-torsion de $H^d(K,\mu_4^{\otimes d-1})/[A]\cdot H^{d-2}(K,\mu_2)$ & (\ref{eq_depl_ind_2}) \\
  $N_A^d(K)$ & $H^d(K,\mu_2)/[A]\cdot H^{d-2}(K,\mu_2)$ & p.\pageref{par_m_n} \\
  $F_Q$ & formes anti-herm. sur $Q$ compatibles avec $F$ & p.\pageref{par_fq} \\
  $F_{\alpha,Q}$ & sous-foncteur de $F_Q$ des formes compatibles avec $\alpha$ & p.\pageref{par_fq_alpha} \\
  $\hat{\alpha}$ & invariant dans $\Inv(F_{\alpha,Q}, M_Q^d)$ induit par $\alpha$ & (\ref{prop_ber})
\end{longtable}

\section*{Produits croisés}

\begin{longtable}{rll}
  $K$ & $k$-algèbre étale & p.\pageref{par_etale} \\
  $K_\p$ & composante de $K$ correspondant à $\p\i \Spec(K)$ & p.\pageref{par_etale} \\  
  $u_g$ & générateur d'un produit croisé & p.\pageref{par_ug} \\
  $\alpha(g,h)$ & 2-cocycle décrivant un produit croisé & (\ref{eq_alpha_gh}) \\
  $T_g$ & opération semi-linéaire sur $V$ & (\ref{eq_def_tg}) \\
  $S_g$ & opération semi-linéaire sur $\End_L(V)$ & (\ref{eq_def_sg}) \\
  $\overline{g}$ & $\theta_{|L} g^{-1}\theta_{|L}$ & (\ref{eq_g_bar}) \\
  $\theta_g$ & $g\theta_{|L}$ & (\ref{eq_theta_g}) \\
  $\theta'_g$ & $g^{-1}\theta_{|L}$ & (\ref{eq_theta_g}) \\
  $\mu_g$ & élément de $L^*$ décrivant l'involution sur un produit croisé & (\ref{prop_mu}) \\
  $\pi_g$ & projection $E\to L$ liée à $g$ & (\ref{eq_pi_g}) \\
  $h_g(x,y)$ & $\pi_g(h(x,y))$  & (\ref{eq_hg}) \\
  $f_g(x,y)$ & $\pi_g(\theta(x)y)$ & (\ref{eq_fg}) \\
  $\eps_g$ & $g^{-1}(\mu_g)$ & (\ref{prop_morita_crois}) \\
\end{longtable}

\section*{Algèbre de Clifford}

\begin{longtable}{rll}
  $Q$ &  algèbre de quaternion équivalente à $A$ & \\
  $B$ & $Cl(A,\sigma)$ & \\
  $(e_i)$ & base orthogonale de $V$ & \\
  $\fdiag{z_1,\dots,z_r}$ & diagonalisation de $h$ & \\
  $L$ & sous-algèbre étale de $B$ & \\
  $G$ & $\Gal(L/k)$ & \\
  $H$ & $\Gal(L/Z(B))$ & \\
  $\tau_{p,q}$ & élément de $G$ & \\
  $K$ & sous-corps quadratique de $Q$ & \\
  $j$ & quaternion pur, $Q=K\oplus jK$ & \\
  $\pi$ & $j^2$ & \\
  $s$ & élément non trivial de $\Gal(K/k)$ & \\
  $b$ & forme bilinéaire, $\sigma_K\simeq \sigma_b$ & \\
  $\phi_b$ & identification $V\otimes_K V\to \End_K(V)$ & \\
  $\psi$ & morphisme canonique $V\otimes_Q ( ^\gamma V)\to Cl(A,\sigma)$ & \\
  $u(x,y)$ & $\frac{1}{2}\psi(x\otimes y + y\otimes x)$ & \\
  $u(x)$ & $\frac{1}{2}\psi(x\otimes x)$ & \\
  $\xi_i$ & $u(e_i)$ & \\
  $\xi$ & $\prod_i \xi_i$ & \\
  $u_{p,q}$ & $u(e_p,e_q)$ & 
\end{longtable}

\bibliographystyle{plain}
\bibliography{manuscrit}

\pagebreak

\thispagestyle{empty}
\newgeometry{top=4em}

\begin{center}
\textbf{Thèse de mathématiques}
\end{center}

\vspace{2em}

\small{\textbf{Mots-clés:} invariant cohomologique, algèbre à involution,
  opération lambda, anneau de Witt, filtration fondamentale, équivalence
  de Morita hermitienne.}

\vspace{2em}

\small{\textbf{Résumé:} Afin d'étudier certains types d'objets algébriques, et notamment
les groupes algébriques, Serre a introduit la notion d'invariants, en particulier
d'invariants cohomologiques. La construction d'invariants cohomologiques non triviaux
de groupes algébriques est un domaine actif de la recherche actuelle, et très peu
d'invariants sont connus en degré strictement supérieur à 3.

Dans le premier chapitre, on donne une description complète des invariants de Witt
et des invariants cohomologiques des foncteurs $I^n$ comme combinaison d'invariants
fondamentaux se comportant comme des puissances divisées, dont la construction repose
de façon cruciale sur les opérations lambda dans l'anneau de Grothendieck-Witt. On
étudie également le comportement des ces invariants vis-à-vis de diverses opérations
comme le produit ou les similitudes.

Le deuxième chapitre est consacré à la construction d'un anneau de Witt \og{}mixte\fg{}
associé à une algèbre à involution: l'idée fondamentale est de définir le produit de deux
formes $eps$-hermitiennes à l'aide d'une équivalence de Morita canonique donnée par la
forme trace d'involution associée à l'algèbre. On définit également des opérations
lambda sur l'anneau de Grothendieck-Witt mixte, ainsi qu'une filtration fondamentale
imitant le cas déployé. Une attention particulière est portée aux calculs explicites
dans le cas des algèbres d'indice 2.

On mobilise ces outils dans le troisième chapitre pour imiter les constructions du
chapitre 1 dans le cadre des formes hermitiennes, et on construit ainsi des invariants
cohomologiques d'algèbres à involution, avec des résultats plus détaillés en indice 2.
L'intérêt principal est de pouvoir en principe construire des invariants non triviaux
de degré arbitrairement grand, bien que l'indice de l'algèbre constitue une forme
d'obstruction.}

\rule{18em}{0.4pt}

\vspace{1em}

\small{\textbf{Title: Cohomological invariants of algebraic groups and
  algebras with involution}}

\vspace{2em}

\small{\textbf{Abstract:} In order to study certain algebraic objects, and
  notably algebraic groups, Serre introduced the notion on invariants, in
  particular cohomological invariants. The construction of non-trivial
  cohomological invariants of algebraic groups is an active area of modern
  research, and very few invariants are known in degree greater than 3.

  In the first chapter, we give a complete description of Witt and
  cohomological invariants of the functors $I^n$ as combinations of
  fundamental invariants behaving like divided powers, whose construction
  relies crucially on lambda operations in the Grothendieck-Witt ring. We
  also study the behaviour of these invariants with respect to various
  operations such as products or similitudes.

  The second chapter is dedicated to the construction of a ``mixed''
  Witt ring associated to an algebra with involution: the fundamental
  idea is to define the product of two $\eps$-hermitian forms using
  a Morita equivalence given by the involution trace form of the algebra.
  We also define lambda operations on the mixed Grothendieck-Witt ring,
  as well as a fundamental filtration imitating the split case. A particular
  attention is given to explicit calculations in the case of algebras of
  index 2.

  We use those tools in the third chapter to mimic the constructions of
  chapter 1 in the framework of hermitian forms, and thus construct cohomological
  invariants of algebras with involution, with more detailed results in index 2.
  The main interest is to be in principle able to define non-trivial invariants
  of arbitrarily high degree, although the index constitues a form of obstruction.
}

\vspace{2em}

\large{
\textit{Thèse préparée au LAGA (UMR 7539) - Institut Galilée}

\textit{99 avenue Jean Baptiste clément}

\textit{93430 Villetaneuse}}

\end{document}